\documentclass[amssymb,10pt]{article}

\usepackage{graphicx}
\usepackage{amssymb}
\usepackage{amsthm}
\usepackage{indentfirst}
\usepackage{amsmath}
\usepackage{color}
\usepackage{mathrsfs}
\usepackage{amssymb}
\usepackage{amsthm}
\usepackage{titletoc}
\usepackage{amssymb}
\usepackage{amsxtra}
\usepackage{amsmath,amstext,amsfonts,verbatim, url}
\usepackage{epstopdf}
\usepackage{subfig}
\usepackage{psfrag}
\usepackage{mathtools}
\usepackage{enumitem}
\usepackage[all]{xy}
\usepackage{hyperref}

\usepackage{tikz-cd}
\usepackage{bm}

\usepackage{bbm}

\setlist[enumerate]{label={\rm (\arabic*)}}
\numberwithin{equation}{section}

\newtheorem{thm}{Theorem}[section]
\newtheorem{lem}[thm]{Lemma}
\newtheorem{pro}[thm]{Proposition}

\newtheorem{dfn}[thm]{Definition}

\newtheorem{cor}[thm]{Corollary}
\newtheorem{rk}[thm]{Remark}

\newtheorem{theorema}{Theorem}[section]

\newcommand{\norm}[1]{\left\Vert#1\right\Vert}
\newcommand{\Norm}[1]{\left\vert#1\right\vert}

\def\D{{\cal D}}           
   \def\U{{\cal U}}    \def\X{{\cal X}}    \def\P{{\cal P}}
          
 \def\N{{\cal N}}    \def\H{{\cal H}} \def\T{{\cal T}}

\def\wt{\widetilde}  
\def\wh{\widehat}
       \def\nt{\noindent}
\def\la{\langle}         \def\ra{\rangle}

\def\ol{\overline}   

\def\ul{\underline}

\def\qed {\hfill $\Box$\vskip5pt}

\def\Sing{{\rm Sing}}

\def\loc{{\rm loc}}

\def\dist{{\rm dist}}

\def\GL{{\rm GL}}

\def\supp{{\rm supp}}

\def\dist{{\rm dist}}

\def\FB{{\rm{FB}^*}}

\def\NUH{{\rm NUH}}
\def\diam{{\rm diam}}
\def\dim{{\rm dim}}
\def\proj{{\rm proj}}

\def\Pes{{\rm PES}}
\def\BHC{{\rm BHC}}
\def\HC{{\rm HC}}
\def\Var{{\rm Var}}
\def\Pe{{\mathbb{P}_{\rm erg}}}

\def\bN{\mathbb{N}}
\def\bR{\mathbb{R}}
\def\bZ{\mathbb{Z}}

\def\over{\overset{\varepsilon}{\approx}}
\def\ov{\overset{\varepsilon}{\to}}

\def\eqL{\overset{N}{\sim}}

\begin{document}

\title{Thermodynamic formalism for singular flows}

\author{Ming Li and Xingzhong Liu}

\date{}
\maketitle

\begin{abstract}

We prove that $C^\infty$ smooth three-dimensional flows with positive topological entropy admit only finitely many ergodic measures of maximal entropy, even in the presence of singularities (zero-velocity points). Furthermore, for flows in a 
$C^\infty$-open and dense subset of the space of such three-dimensional flows, every ergodic measure of maximal entropy is rapidly mixing. To prove this result, we construct a novel symbolic coding framework for singular flows, which serves as the fundamental technical tool. We then introduce the property of strong positive recurrence (SPR) for singular flows, and prove that SPR flows can be coded by suspension flows over SPR countable discrete symbolic systems. Under appropriate conditions, we further establish that such suspension flows are rapidly mixing. These frameworks can be further generalized to a broad class of flows, including star flows, and the results naturally extend to general equilibrium states.


\end{abstract}

\begingroup
\linespread{0.8}
\selectfont

\setcounter{tocdepth}{1}
\tableofcontents
\endgroup

%%%%%%%%%%%%%%%%%%%%%%%%%%%%%%%%%%%%%%%%%%%%%%%%%%%%
\section{Introduction}\label{sec.intro}

%%%%%%%%%%%%%%%%%%%%%%%%%%%%%%%%%%%%%%%%%%%%%%%%%%%%

The ergodic theory of dynamical systems fundamentally seeks to characterize typical long-term behavior through invariant probability measures. For flows, this pursuit becomes particularly challenging when singularities---equilibrium points at which the velocity field vanishes---are present. In this paper, we develop a novel symbolic framework to characterize the statistical properties of singular flows with non-uniform hyperbolicity. Applied to three-dimensional smooth flows, this yields the following result:

\begin{thm}\label{thm.3flows}
	Every $C^\infty$ flow on a closed smooth $3$-dimensional Riemannian manifold $M^3$ with positive topological entropy has only finitely many ergodic measures of maximal entropy.

Moreover, there exists a $C^\infty$-open and dense subset of the space of all  $C^\infty$ flows with positive topological entropy on $M^3$ such that, for every flow in this subset, every ergodic measure of maximal entropy is rapidly mixing	
\end{thm}

Here the term \textit{rapid mixing} refers to superpolynomial decay of correlations: 
for any given $A>0$, 
there exists a constant $C > 0$ and some integer $m\ge 1$ such that for all observables $F, G\in C^m(M)$,
\[
\left| \int F \cdot (G \circ \varphi_t)  d\mu - \int F  d\mu \int G  d\mu \right| \leq C \|F\|_{_{C^m}} \|G\|_{_{C^m}}  |t|^{-A}
\]
holds for all $t \in \mathbb{R}$.

\subsection{Singular flows}

Throughout this paper, we assume that $M$ is a closed smooth Riemannian manifold of dimension $d\ge 3$, that is, $M$ is compact and without boundary.  A $C^1$ vector field $X$ on $M$ generates a $C^1$ flow $\varphi_t:M\to M$.

Compared to the case of diffeomorphisms, studying flows presents substantially greater challenges. A central difficulty arises from the presence of singularities (zero-velocity points of the flow). A flow that admits singularities is referred to as a {\it  singular flow}. In this paper, we also use the term singular flow generically to refer to all flows---both with and without singularities---to emphasize that singularities are explicitly permitted within the framework.

While singularities themselves support only trivial dynamics as fixed points of the flow, the global dynamics often become highly complex when recurrent orbits accumulate near them. This complexity can break standard dynamical structures, as famously exemplified by the classical Lorenz attractor \cite{Lor}. Due to the presence of its singularities, the Lorenz attractor is neither uniformly hyperbolic nor structurally stable, yet it exhibits robust non-trivial dynamical behavior.

It has long been recognized that hyperbolicity for flows manifests itself in the transverse directions rather than along the flow direction. To systematically analyze the hyperbolic structure of flows, Liao \cite{Liao} defined the linear Poincar\'e flow on the normal bundle. However, this flow remains undefined at singularities, which typically breaks the global compactness of the associated phase space. Moreover, near singularities, the flow speed approaches zero, causing orbits to spend arbitrarily long intervals of time in their vicinity. Simultaneously, both the domain of the Poincar\'e map and the influence range of the linear Poincar\'e flow tend to degenerate to zero.

 To resolve this compactness issue, the work in \cite{LGW} introduced the {\it extended linear Poincar\'e flow}, extending its definition naturally to singularities. This framework was later further generalized to the nonlinear Poincar\'e flow in \cite{CY}. For the second core challenge, Liao \cite{Liao1,Liao2,Liao3} pioneered the celebrated {\it scaling} technique, which was subsequently refined in \cite{GY,WW} and others. Nevertheless, these existing techniques do not resolve all difficulties associated with singular flows. In this paper, building upon previous works, we develop a novel approach for studying the statistical properties of singular flows.
 
 If there exists a global transverse section for the flow on the manifold, the corresponding first return map is a diffeomorphism. It is generally accepted that the dynamics of such a flow resemble those of its discrete first return map. However, global sections typically do not exist, which forces one to study the local dynamics near periodic orbits through local sections. This is particularly challenging for flows with singularities, where no return map induced by finitely many local sections can have uniformly bounded return times. The foundation of our work lies in a novel construction of a family of sections that is potentially countably infinite in total number but locally finite, with return times uniformly bounded both above and below.

 \subsection{Equilibrium states}
 
 The statistical properties of dynamical systems characterize how the random behaviors of a system are determined. This forms a fundamental problem in dynamical systems theory and is closely tied to the evolution of invariant measures. One of the most important concepts in the statistical theory of dynamical systems is the measure of maximal entropy, or more generally, the equilibrium state, defined as an invariant probability measure achieving the supremum in the variational principle. 
 
 Specifically, we say that a probability measure $\mu$ on $M$ is {\it $X$-invariant} (resp. {\it $X$-ergodic}) if $\mu$ is $\varphi_t$-invariant (resp. $\varphi_t$-ergodic) for any $t\in \bR$. Denote by $\mathbb P(X)$ (resp. $\Pe(X)$) the set of all $X$-invariant (resp. $X$-ergodic) probability measures. 
 
 For any $\mu\in\mathbb P(X)$, the {\it metric entropy} for $X$  with respect to $\mu$, denoted by $h_\mu(X)$, is defined as the metric entropy of the time-$1$ map $\varphi_1$  with respect to $\mu$.  According to the variational principle (see, e.g. \cite{Wal}),
  \[
 h_{\rm top}(X)=\sup\{h_\mu(X): \mu\in \mathbb P(X)\}=\sup\{h_\mu(X): \mu\in \Pe(X)\},
 \]
 where  $h_{\rm top}(X)$ is the {\it topological entropy} of $X$, which is defined as the topological entropy of the time-$1$ map $\varphi_1$. The {\it measure of maximal entropy} ({\it MME}) for $X$  is  an $X$-invariant probability measure $\mu$ such that $h_\mu(X)= h_{\rm top}(X)$.

 A {\it potential} is a bounded measurable function $\vartheta$ defined on $M$. The {\it metric pressure} of $\vartheta$ for $X$  with respect to $\mu\in \mathbb P(X)$ is defined as 
 \[
 P_{\mu}(X,\vartheta):=h_\mu(X)+\int \vartheta d\mu.
 \]
 Similarly, by the variational principle,
 \[
 P_{\rm top}(X,\vartheta)=\sup\{P_{\mu}(X,\vartheta): \mu\in \mathbb P(X)\}=\sup\{P_{\mu}(X,\vartheta): \mu\in \Pe(X)\},
 \]
where $P_{\rm top}(X,\vartheta)$ is the {\it topological pressure} of $\vartheta$ for $X$, which is defined as the topological pressure of $\vartheta$ for the time-$1$ map $\varphi_1$. The {\it equilibrium state} of $\vartheta$ for $X$ is a measure $\mu\in\mathbb P(X)$ such that $P_{\mu}(X,\vartheta)= P_{\rm top}(X,\vartheta)$.
The MME corresponds exactly to the equilibrium state of the zero potential $\vartheta\equiv 0$.

The existence and finiteness of ergodic equilibrium states for dynamical systems constitute a long-standing question. Research on equilibrium states began with uniformly hyperbolic systems \cite{Par,AW,Bow71,Sin,Bow74,Bow75,PP} and was later extended to non-uniformly hyperbolic systems. Substantial results have been established for various types of non-uniformly hyperbolic systems.

There are two primary approaches to addressing this problem. The first employs the specification property. Bowen \cite{Bow75} pioneered this method to prove the existence and uniqueness  of equilibrium state on uniformly hyperbolic transitive sets. Subsequently, Climenhaga-Thompson \cite{CT16} further developed this theory, establishing existence and uniqueness for a large class of non-uniformly  hyperbolic systems (see, e.g. \cite{CT14,CT16,CT18,CKW}), including geodesic flows on closed manifolds with nonpositive curvature \cite{BCFT}. 

Very recently, Pacifico-Yang-Yang further generalized the Climenhaga-Thompson framework to singular flows \cite{PYY22,PYYY}, proving that for a $C^1$-open and dense subset of star flows satisfying the {\it pressure gap condition} (see \S\,\ref{subsec.starflow}), ergodic equilibrium states exist and are finite \cite{PYY25a,PYY25b}.

In this article, we adopt a distinct symbolic approach: we construct a finite-to-one semi-conjugacy from a topological Markov flow to the main part of the original flow on the manifold, then leverage well-developed properties of symbolic dynamics (see, e.g. \cite{Gur}) to study the equilibrium states for flows. This symbolic coding method has been widely applied (see \cite{Buz97,BFSV,BF,LM,LS,LP} for instance). Recently, Buzzi-Crovisier-Sarig established the finiteness of equilibrium states for smooth surface diffeomorphisms \cite{BCS22}. The corresponding result was later extended to three-dimensional flows without singularities \cite{Zan,BLY}.

\subsection{Symbolic dynamics for non-uniformly hyperbolic systems}

The symbolic dynamics approach to studying hyperbolic behavior in dynamical systems has achieved significant success. This includes applications to uniformly hyperbolic systems \cite{AW,Sin68,Rat69,Bow70,Rat73,Bow08} and piecewise interval maps and beyond \cite{Hof,Buz97,Buz05}. 

In Sarig's groundbreaking work \cite{Sar13}, a symbolic dynamics framework was constructed for non-uniformly hyperbolic surface diffeomorphisms. The method has seen substantial development in recent years, both for broad classes of non-uniformly hyperbolic discrete-time systems \cite{LM,Ben,Lim,ALP} and for non-uniformly hyperbolic flows without singularities \cite{LS,BCL,LMN}.

In this paper, our central approach involves the construction of a symbolic dynamics framework for general non-uniformly hyperbolic flows, including those that admit singularities. 

A {\it topological Markov shift} $ (\Sigma, \sigma) $ is a symbolic system generated by a directed graph with at most countably many vertices. A {\it roof function} $r$ is a uniformly bounded positive continuous function defined on 
$\Sigma$. A {\it topological Markov flow} $(\Sigma_r, \sigma_r)$ is the suspension flow generated by a topological Markov shift $ (\Sigma, \sigma) $ and a roof function $r$ defined on $\Sigma$. Under the {\it Bowen-Walters metric}, this flow is continuous.

The {\it regular set} $\Sigma^{\#}_r$ of $\Sigma_r$ is the set of points whose symbolic coordinate sequence contains at least one symbol that appears infinitely often in the forward direction, and at least one (possibly distinct) symbol that appears infinitely often in the backward direction. A topological Markov flow is {\it locally compact} if its generating graph has finite out-degree and finite in-degree at every vertex. An {\it irreducible component} of a topological Markov flow is a maximal subset of $\Sigma_r$ for which any two states in the corresponding generating subgraph are connected by directed paths. Full formal definitions are provided in \S\,\ref{subsec.TMSTMF}.

An $X$-invariant probability measure $\mu$ on $M$ is \textit{$\chi$-hyperbolic} if, at $\mu$-almost every point, the absolute values of all Lyapunov exponents except the one corresponding to the flow direction are strictly greater than $\chi$. A measure $\mu$ is \textit{hyperbolic} if there exists $\chi > 0$ such that $\mu$ is $\chi$-hyperbolic. A measure $\mu$ is \textit{regular} if $\mu(\Sing(X)) = 0$, where $\Sing(X)$ denotes the set of all singularities of $X$. Two regular hyperbolic measures $\mu$ and $\nu$ are said to be  \textit{homoclinically related} if for $\mu$-almost every $x$ and $\nu$-almost every $y$, the stable manifold of $x$ transversely intersects the unstable manifold of $y$, and vice versa.

\begin{theorema}\label{main.coding}
Let $X$ be a $C^{1+\beta}$ vector field ($\beta>0$) on $M$. For any $\chi>0$, there exists a locally compact topological Markov flow $(\wh{\Sigma}_{\wh{r}},\wh \sigma_{\wh{r}})$  and a factor map $\wh{\pi}_{\wh{r}}:\wh{\Sigma}_{\wh{r}}\to M$ satisfying the following properties:
\begin{enumerate}
\item  $\wh{\pi}_{\wh{r}}\circ \wh{\sigma}_{\wh{r},t}=\varphi_t\circ \wh{\pi}_{\wh{r}}$ for all $t\in\mathbb{R}$.
\item  $\sup \wh{r}<\infty$ and $\inf \wh{r}>0$. Both $\wh{r}$ and $\wh{\pi}_{\wh{r}}$ are H\"older continuous.
\item  $\wh{\pi}_{\wh{r}}(\wh{\Sigma}^{\#}_{\wh{r}})$ has full measure for every regular $\chi$-hyperbolic $X$-invariant measure.
\item  $\wh{\pi}_{\wh{r}}$ is finite-to-one on $\wh{\Sigma}^{\#}_{\wh{r}}$.
\item For any regular hyperbolic ergodic measure $\mu$, there exists an irreducible component $\wh{\Sigma}'_{\wh{r}}$ which lifts all regular $\chi$-hyperbolic ergodic measures homoclinically related to $\mu$.
\end{enumerate}
\end{theorema}

See Theorems~\ref{thm.strongthm} and \ref{thm.irrcomp} for full formal statements.
To the best of our knowledge, Theorem~\ref{main.coding} constitutes the first symbolic coding result for general singular flows that is finite-to-one with uniformly bounded roof function. We note that in \cite{AV}, a symbolic system was constructed for the geometric Lorenz attractors (or general singular hyperbolic attractors) that contain singularities, where the corresponding roof function is not uniformly bounded.

After the first version of this paper was posted on the arXiv (arXiv:2604.26936), we learned that Buzzi, Crovisier, Lima, Luo, and Yang were also working on a project devoted to symbolic coding for singular flows. Their main result is essentially equivalent to our Theorem~\ref{main.coding}, and they use a different approach to handle singularities \cite{BCLLY}.

A flow is called {\it Bernoulli} if its time-1 map is a Bernoulli automorphism. We say that a flow is {\it Bernoulli up to a period} if it is Bernoulli, or isomorphic to the product of a Bernoulli flow and a rotational flow on a circle.
Let ${\rm Per}_T(X)$ denote the number of distinct prime periodic orbits of $X$ with minimal positive period less than $T$.

As a direct application of Theorem A, one can analyze the equilibrium states of singular flows using the constructed symbolic framework. The following theorem (see Theorem~\ref{thm.main.mme} for the full strengthened version) is a direct corollary of Theorem~\ref{main.coding} combined with results from \cite[Theorem~1.1]{LLS} and \cite[Theorem~4.7 and 5.1]{LS}. 

\begin{theorema}\label{main.mme}
	Let  $X$ be a $C^{1+\beta}$ vector field ($\beta>0$) on $M$. Then for any H\"older continuous potential $\vartheta$ on $M$ and any regular hyperbolic ergodic measure $\mu$, there is at most one regular hyperbolic ergodic equilibrium state of $\vartheta$ for $X$ that is homoclinically related to $\mu$. If such an equilibrium state exists,  it must be Bernoulli up to a period.  
	
	In particular, there exists  at most one regular hyperbolic ergodic MME for $X$  homoclinically related to $\mu$.
If there exists at least one MME, then there exists a constant $C_0>0$ such that ${\rm Per}_T(X)\ge\frac{C_0}{T} e^{T h_{\rm top}(X)}$ for all sufficiently large $T>0$.
\end{theorema}

\subsection{Strong positive recurrence}

For uniformly hyperbolic systems, coding by finite symbolic systems provides deep statistical insights via the spectral gap of the Ruelle operator. The situation changes significantly for non-uniformly hyperbolic systems, where many fundamental questions remain open. In their groundbreaking work \cite{BCS25}, Buzzi-Crovisier-Sarig introduced the concept of {\it Strong Positive Recurrence} ({\it SPR}) for diffeomorphisms. SPR diffeomorphisms are generic in settings such as smooth surface dynamics, and they exhibit rich statistical properties.

The concept of SPR, or its equivalent formulations, originated in the study of countable Markov shifts through  works by Gurevich, Zargaryan, Savchenko \cite{GZ,Gur96,GS}, Vere-Jones \cite{Ver}, Sarig \cite{Sar01}, and others. SPR symbolic systems possess strong statistical properties. Most notably, Cyr-Sarig established that the Ruelle operator admits a spectral gap for topologically mixing SPR one-sided countable Markov shifts \cite{CS}.

In \cite{BCS25}, Buzzi-Crovisier-Sarig developed the SPR theory for diffeomorphisms. By coding SPR diffeomorphisms with SPR topological Markov shifts, they transferred the profound statistical properties arising from the spectral gaps in symbolic dynamics to non-uniformly hyperbolic diffeomorphisms. This work constructs a refined ergodic theory framework for a broad class of non-uniformly hyperbolic diffeomorphisms, which is comparable to the well-developed framework for uniformly hyperbolic systems. They proved that such diffeomorphisms possess key statistical properties including, but not limited to, exponential mixing, large deviation principles, and central limit theorems.

In the present paper, we define the SPR property for singular flows. It should be emphasized that due to the presence of potential singularities, we use the scaled linear Poincar\'e flow to construct the required Pesin blocks (see Definition~\ref{def.pesin}). Furthermore, since singularities may lie in the closure of Pesin blocks---which prevents these Pesin blocks from being extended to compact closures---we appropriately adapt the definition of SPR to the singular flow setting.

 Specifically, we say that a flow is {\it strongly positively recurrent} ({\it SPR}) for a potential $\vartheta$, if there exists $\chi> 0$ satisfying:
 for each $\varepsilon > 0$, there exists  a $(\chi,\varepsilon)$-Pesin block, a compact subset $\Lambda$ of this Pesin block, and constants $P_0<P_{\rm top}(X,\vartheta)$, $\tau>0$ such that
 $\nu(\Lambda)>\tau$ for every ergodic measure $\nu$ with $P_{\nu}(X,\vartheta)>P_0$.  
We also verify that SPR singular flows can be coded by  suspension flows over a corresponding SPR topological Markov shift. As an immediate application, we obtain the following theorem. The full detailed statement is given in Theorem~\ref{thm.main.spr}.

\begin{theorema}\label{main.spr}
	Let $X$ be a $C^{1+\beta}$  vector field ($\beta>0$) on $M$ and $\vartheta$ be a H\"older continuous potential on $M$. 
	 If $X$ is SPR for $\vartheta$, then there exist only finitely many ergodic equilibrium states of $\vartheta$ for $X$, and each such state is hyperbolic, regular and Bernoulli up to a period.
\end{theorema}

\subsection{Rate of mixing}

For SPR diffeomorphisms, a key consequence of the spectral gap property is that ergodic measures of maximal entropy are exponentially mixing \cite{BCS25} (see also \cite{LY26} for corresponding results on ergodic equilibrium states for smooth surface diffeomorphisms under certain conditions). However, mixing properties differ fundamentally between flows and diffeomorphisms, even in the uniformly hyperbolic Anosov setting. Although Anosov flows exhibit strong mixing in  transverse directions, rigorous results on exponential mixing for such flows remained very limited prior to Dolgopyat's groundbreaking work \cite{Dol98a}.
 This foundational work has inspired extensive subsequent research on the mixing properties of Anosov flows, including geodesic flows on negatively curved manifolds, and a wide range of other related properties.
 Liverani \cite{Liv04} established exponential mixing with respect to the contact volume for contact Anosov flows, a class that naturally includes smooth geodesic flows on negatively curved manifolds. Baladi-Demers-Liverani \cite{BDL} proved exponential mixing for the Lebesgue invariant measure of finite-horizon Sinai billiard flows. Tsujii \cite{Tsu18} demonstrated that exponential mixing with respect to the volume holds for an open and dense set of volume-preserving three-dimensional Anosov flows. Butterley-War \cite{BW20} showed that the unique SRB measure is exponentially mixing for Anosov flows whose stable bundle is $C^1$ and satisfies the non-joint-integrability condition. Li-Pan \cite{LP2023} established exponential mixing for the unique MME of geodesic flows on geometrically finite hyperbolic manifolds with cusps. Most recently, Tsujii-Zhang proved exponential mixing for ergodic equilibrium states of three-dimensional smooth mixing Anosov flows \cite{TZ23}. Pollicott-Sharp \cite{PS} and Giulietti-Liverani-Pollicott \cite{GLP} systematically investigated the associated dynamical zeta functions and orbit counting problems for Anosov flows.

  The situation changes substantially for Axiom A flows. Although they remain exponentially mixing in transverse directions, Ruelle constructed examples of mixing Axiom A flows that fail to be exponentially mixing \cite{Rue83}.  Subsequently, Pollicott proved that mixing rates for mixing basic sets can be arbitrarily slow \cite{Pol85}. Dolgopyat developed a standard technique and proved that typical Axiom A flows are rapidly mixing \cite{Dol98}.  Field-Melbourne-T\"or\"ok \cite{FMT07} introduced methods to construct {\it good asymptotics} under shadowing and perturbations, establishing rapid mixing for ergodic equilibrium states on an open and dense subset of all Axiom A flows. In \cite{Melbourne2007}, Melbourne also extended Dolgopyat's approach to prove rapid mixing of SRB measures for non-uniformly hyperbolic flows in the framework of Young towers  \cite{Young1998,You99}.   Building on this, Melbourne \cite{Mel18} and B\'alint-Butterley-Melbourne \cite{BBM} employed new techniques to broaden the applicability of the prior results from \cite{Melbourne2007}, while Ara\'ujo-Melbourne \cite{AM2019} proved rapid mixing for SRB measures of singular hyperbolic flows.

For suspension semi-flows over interval maps with infinitely many branches,  Baladi-Vall\'e introduced a method to prove exponential mixing for SRB measures under uniform non-integrability conditions  \cite{BV2005}. 
Avila-Gou\"ezel-Yoccoz extended this method in \cite{AGY2006} and applied it to the study of Teichm\"uller flows. Subsequently, Ara\'ujo-Melbourne further developed this approach for Lorenz attractors, establishing exponential mixing for SRB measures when the stable bundle is sufficiently smooth \cite{AM2016}. Daltro-Varandas established later exponential mixing for Gibbs measures on Axiom A attractors in \cite{DV2021}.

Nevertheless, the question of establishing mixing rates for general equilibrium states associated to singular flows remains fully open, even in the classical case of the Lorenz attractor.
In the present work, we study rapid mixing for ergodic equilibrium states of general H\"older potentials for suitable SPR topological Markov flows, by extending the ``good asymptotics'' method developed by Field-Melbourne-T\"or\"ok. We further prove that ergodic equilibrium states having  good asymptotics for SPR singular flows are necessarily rapidly mixing. Full technical details are provided in Theorem~\ref{thm.rapidmixing}.

\subsection{Three-dimensional flows}

The thermodynamic formalism for two-dimensional diffeomorphisms has developed extensively in recent years, with a series of major works including \cite{Sar13,BCS22,BCS22b,Bur,BCS25,BuzLY,BLY,LY26}. Corresponding results for three-dimensional non-singular flows are established in \cite{LS,BCL,BLY,Zan}.

As an application of the general framework constructed in this paper, Theorem~\ref{main.3flowmme} below demonstrates that a three-dimensional smooth flow is SPR for any H\"older continuous potential $\vartheta$ that satisfies the small variation condition $\Var(\vartheta):= \sup_{x\in M^3} \vartheta(x)- \inf_{x\in M^3} \vartheta(x)< h_{\rm top}(X)$. This result implies the existence and finiteness of ergodic equilibrium states for the given potential $\vartheta$.
More precise statements can be found in Theorem~\ref{thm.3SPRset} and Corollary~\ref{thm.main.3flowmme}.

The small variation condition $\Var(\vartheta) < h_{\rm top}(X)$ guarantees two critical properties: it first ensures that all invariant measures with sufficiently large pressure (including all equilibrium states) are hyperbolic, and it also verifies the required pressure gap condition at singularities (see the details in \S\,\ref{subsec.starflow}).

\begin{theorema}\label{main.3flowmme}
	Let $M^3$ be a smooth closed $3$-dimensional Riemannian manifold, and let $X$ be a $C^{\infty}$ vector field on $M^3$.
 If $\vartheta$ is a  H\"older continuous potential on $M^3$ satisfying $\Var(\vartheta) < h_{\rm top}(X)$, then $X$ is SPR for $\vartheta$. As a direct corollary, there exist only finitely many ergodic equilibrium states of $\vartheta$ for $X$.  
 
 In particular, if $X$ has positive topological entropy, then there exist only finitely many ergodic MMEs.
\end{theorema}

Furthermore, we prove that for $C^\infty$-open and dense three-dimensional vector fields satisfying the small variation condition, each ergodic equilibrium state has good asymptotics, and is therefore necessarily rapidly mixing. Specifically, let $\vartheta$ be a  H\"older continuous potential on $M^3$, and denote by $\X_{\vartheta}^\infty(M^3)$ the set of all $C^\infty$ vector fields on $M^3$ satisfying  $\Var(\vartheta)< h_{\rm top}(X)$.

\begin{theorema}{\rm (Theorem~\ref{thm.main.3flowrm})}\label{main.3flowrm}
	If $\vartheta$ is a  H\"older continuous potential on $M^3$, then there exists a $C^\infty$-open and dense subset $\U$ of $\X_{\vartheta}^\infty(M^3)$ such that for every $X\in \U$,  every ergodic equilibrium state of $\vartheta$ for $X$ is rapidly mixing.
\end{theorema}

Our initial main result stated in Theorem~\ref{thm.3flows} is a direct special case of Theorems~\ref{main.3flowmme} and \ref{main.3flowrm} when the potential is taken to be identically zero, that is, $\vartheta\equiv 0$.

\subsection{Star flows}\label{subsec.starflow}

The concept of the  star system originates from the study of the renowned Stability Conjecture. During the investigation of this conjecture, Pliss, Liao, and Ma\~n\'e identified a pivotal condition: the system should never generate non-hyperbolic periodic orbits under arbitrary $C^1$-small perturbations (and for flows, non-hyperbolic singularities are also excluded simultaneously). Liao named this the {\it star condition}. The star condition appears quite weak; however, for diffeomorphisms, it implies Axiom A combined with the no-cycle condition. This landmark result was first established for the two-dimensional case by Liao \cite{Liao81} and Ma\~n\'e \cite{Man82}, and later extended to all general dimensions by Ma\~n\'e \cite{Man88}, Aoki \cite{Aok} and Hayashi \cite{Hay}. Finally, Gan and Wen proved an analogous result for flows without singularities \cite{GW}.

For flows that admit singularities, Axiom A fails to hold when singularities are approximated by periodic orbits. The geometric Lorenz attractor \cite{Guc76,ABS,GucW} serves as a canonical example: although it is not structurally stable, it exhibits robust dynamical properties under small perturbations. To properly characterize such hyperbolicity, Morales-Pujals-Pacifico introduced the concept of {\it singular hyperbolicity} \cite{MPP99}, which was later extended to higher dimensions via the concept of {\it sectional hyperbolicity} in \cite{MM}. Counterexamples constructed in \cite{DaL} demonstrate that singular hyperbolicity is not sufficient to describe generic star flows. To resolve this gap, Bonatti-da Luz introduced the concept of {\it multi-singular hyperbolicity} in \cite{BD}, and proved via techniques developed in \cite{SGW} that all $C^1$-open and dense star flows are multi-singular hyperbolic.

Recently, Pacifico-Yang-Yang  established that $C^1$-open and dense star flows satisfying the pressure gap condition $\sup_{x\in  \Sing(X)} \vartheta(x)< P_{\rm top}(X)$, where $\vartheta$ is the given H\"older continuous potential, admit only finitely many ergodic equilibrium states \cite{PYY25b}. Since $\Sing(X)$ is compact, the pressure gap condition is equivalent to the requirement that none of the atomic measures supported on the singularities are equilibrium states.
By the upper-semicontinuity of metric entropies,
the pressure gap condition also ensure that all large pressure measures do not accumulate to the singularities.

On the other hand, this pressure gap condition is sharp: Shi-Yang constructed examples showing that for certain singular star flows and carefully chosen potentials, atomic measures supported on singularities become equilibrium states, which may lead to multiple distinct ergodic equilibrium states within the same chain recurrent class \cite{SY}. In fact, the SPR property of a countable topological Markov shift corresponds to the existence of a pressure gap between the whole system and it at infinity, while in the scaling sense, singularities can be naturally regarded as the manifold's boundary at infinity---the two conditions share the same essence.

As another application of our framework, we also establish rapid mixing for open and dense star flows.  Let  $\vartheta$ be a  H\"older continuous potential, and denote by $\X_\vartheta^{*}(M)$ the set of all $C^2$ star vector fields satisfying $\sup_{x\in \Sing(X)} \vartheta(x)< P_{\rm top}(X)$.

\begin{theorema}{\rm (Theorem~\ref{thm.main.starrm})}\label{main.starrm}
 For any H\"older continuous potential $\vartheta$, there exists a $C^2$-open and $C^1$-dense subset $\U$ of $\X_\vartheta^{*}(M)$ such that for every $X\in \U$,  $X$ is SPR for $\vartheta$, and every ergodic equilibrium state of $\vartheta$ is rapidly mixing.
\end{theorema}

\subsection{Plan of the proof}

For non-singular flows, it is common to discretize the flow by constructing a return system generated by a finite collection of transverse sections. However, for flows that admit singularities, using only finitely many transverse sections inevitably leads to return times that can become arbitrarily large.  To resolve this fundamental issue, we introduce a novel return map $f$ generated by a countable collection of transverse sections  $\{D_j\}_{j\in \mathbb Z^+}$, with return times that admit uniform positive lower and upper bounds (Theorem~\ref{thm.poincaresection}).  Crucially,  this system of sections is locally finite: for each section  $D_j$,
 the number of sections that intersect $f(D_j)$ or $f^{-1}(D_j)$
  is uniformly bounded. Consequently, the local compactness arising from this locally finite structure can substitute for the role played by the global compactness of finite section systems in the coding process. This local compactness property is precisely the core reason that enables the successful coding for non-uniformly hyperbolic systems.
  
  Another fundamental issue in the symbolic coding construction arises near singularities, where both the size of the sections and the effective scope of hyperbolicity for the return map tend to degenerate to zero. This challenge is commonly encountered when handling singular flows. Liao introduced a pioneering and effective approach to resolve this: appropriately scaling the corresponding linear Poincar\'e flow to achieve certain uniform estimates. In this work, we also adopt this method and appropriately scale the corresponding linear return maps defined on the aforementioned newly constructed sections to complete the symbolic coding construction.  

The scaling process can be viewed as normalizing the metric on the tangent space at every non-singular point through multiplication by the reciprocal of the local flow speed.
 This operation expands the size (in the sense of metric) of hyperbolic influence near singularities to match its range in regions far from singularities. Correspondingly, Pesin theory for non-uniformly hyperbolic systems achieves uniform hyperbolicity by switching to the Lyapunov metric. We thus carry out two sequential metric modifications: first ensuring uniform local sizing via the scaling operation, then establishing uniform hyperbolicity through the Lyapunov metric.

Subsequently, we extend the framework in \cite{BCL} by applying our newly constructed sections and the scaled system defined above to perform symbolic coding for general  singular flows (Theorem~\ref{thm.strongthm}). Concurrently, we implement irreducible coding for the Borel homoclinic class of singular flows (Corollary~\ref{cor.Borelclass}) using the method introduced in  \cite{BCS22}, and perform SPR coding for SPR singular flows (Theorem~\ref{thm.SPRcoding}) following the approach in \cite{BCS25}. This valid SPR symbolic coding then allows us to transfer the rich set of well-developed statistical properties of SPR symbolic systems back to the original SPR flow. To realize this property transfer, we further prove that three-dimensional smooth flows satisfying the small variation condition (Theorem~\ref{thm.3SPRset}) and suitable star flows satisfying the pressure gap condition (Theorem~\ref{thm.main.starrm}) are SPR. The existence and finiteness of ergodic equilibrium states for such flows follow immediately as a direct consequence of the SPR symbolic dynamics.

Cyr-Sarig in \cite{CS} proved that the SPR one-sided countable topological Markov shift admits a spectral gap for its associated Ruelle operator, which serves as the fundamental foundation for our proof of rapid mixing. We adopt Dolgopyat's classical approach from \cite{Dol98}, where he established several equivalent conditions for rapid mixing in suspension flows over finite-type symbolic systems. However, these classical results cannot be directly applied in our infinite-alphabet setting for two key reasons: the functional norm used here to establish the spectral gap is distinct from the standard H\"older norm, and more critically, the infinite-alphabet environment leads to the absence of crucial uniform properties such as the global Gibbs property. To resolve these issues, we first establish a Lasota-Yorke type functional inequality adapted to this specific Banach norm that gives rise to the spectral gap property. To compensate for the lack of global uniformity, we restrict the required phase cancellation condition to a predetermined finite-type symbolic subsystem with finitely many symbols, and by introducing a tailored new functional norm, we complete the proof of the decay property of the Ruelle operator with the help of the local Gibbs property (Proposition~\ref{pro.Idabinv}).

When applying the above abstract results to three-dimensional  smooth flows and star flows, we employ the method introduced by Field-Melbourne-T\"or\"ok in \cite{FMT07}. This approach was originally developed to prove rapid mixing for open and dense subsets of Axiom A flows. The core technique involves constructing a carefully chosen sequence of periodic orbits that satisfies the required phase cancellation condition via shadowing arguments and arbitrarily small perturbations. Crucially, these periodic orbits lie within a small neighborhood of a transverse homoclinic orbit to some hyperbolic periodic orbit, and are contained in a uniformly hyperbolic set. This allows us to lift these periodic orbits to a compact symbolic subsystem, thereby confining the phase cancellation entirely within a well-defined finite-alphabet range, ultimately yielding the rapid mixing property  (Theorem~\ref{thm.rapidmixing}).

\subsection{Outline}

This paper is organized into four main parts.

In Part~\ref{part.section}, we construct a Poincar\'e return system induced by at most countably many transverse sections, which covers all non-singular points of the flow. This system satisfies the local finiteness property, and the corresponding roof function admits uniform positive lower and upper bounds. The local finiteness provides the foundation for the symbolic coding of singular flows.
In \S\,\ref{sec.poincaremap}, we recall fundamental concepts, including Poincar\'e maps and scale Poincar\'e maps, as well as results of uniform estimates.
In \S\,\ref{sec.section}, we complete the construction of the aforementioned Poincar\'e return system and present the proof of its corresponding properties.
In \S\,\ref{sec.linearpoincare}, we analyze the return maps and their scaling on the constructed sections, and build a global linear flow that is compatible with the return dynamics.

In Part~\ref{part.coding}, we perform symbolic coding for singular flows following the core methodology of Buzzi-Crovisier-Lima, leveraging the return system constructed in Part~\ref{part.section}.
In \S\,\ref{sec.prelim}, we review some related preliminary concepts and results.
In \S\,\ref{sec.nuh}--\S\,\ref{sec.finitetoone}, we give the required estimates for the subsequent symbolic coding construction.
In \S\,\ref{sec.codingthm}, we prove the main coding theorem for singular flows, together with the irreducible coding and the SPR coding.

In Part~\ref{part.thermdyn},  we employ the newly constructed symbolic coding system to transfer statistical properties from topological Markov flows to general singular flows.
In \S\,\ref{sec.equili}, we address the existence and finiteness of ergodic equilibrium states for the considered class of flows.
In \S\,\ref{sec.tranoper}, we discuss the Ruelle operator associated to the one-sided symbolic dynamics, and prove a Lasota-Yorke type inequality.
In \S\,\ref{sec.EstimRurlle}, we define a tailored new norm on the appropriate Banach space, and present the key decay estimates for the Ruelle operator on one-sided symbolic dynamics having good asymptotics.
In \S\,\ref{sec.rapidflow}, we prove rapid mixing property for ergodic equilibrium states of SPR flows having good asymptotics.

In Part~\ref{part.app}, we apply the aforementioned abstract results to three-dimensional smooth flows and star flows.
In \S\,\ref{sec.3flow}, we demonstrate that three-dimensional smooth flows satisfying the small variation condition are SPR for the given potential, and every ergodic equilibrium state for $C^\infty$-open and dense vector fields has good asymptotics.
In \S\,\ref{sec.starflow}, we complete the proof of Theorem~\ref{main.starrm} for $C^2$-open and $C^1$-dense star flows satisfying the pressure gap condition.

The appendix collects the details of some standard proofs that are omitted in the main text.

\section*{Acknowledgements}

Xingzhong Liu is supported by NSFC (Grant No. 12501236).

%%%%%%%%%%%%%%%%%%%%%%%%%%%%%%%%%%%%%%%%%%%%%%%%%%%%%%%%%%%%%%%%%%%%

%%%%%%%%%%%%%%%%%%%%%%%%%%%%%%%%%%%%%%%%%%%%%%%%%%%%%%%%%%%%%%%%%%%

%%%%%%%%%%%%%%%%%%%%%%%%%%%%%%%%%%%%%%%%%%%%%%%%%%%%%%%%%%%%%%%%%%%%

\part{Locally finite uniform Poincar\'{e} sections}\label{part.section}

In this part, we construct a family of Poincar\'e sections compatible with flows that possibly exhibit singularities. Such sections and the corresponding first return maps form the foundational framework for symbolic coding flows in Part~\ref{part.coding}. A key property of these sections is the {\it local finiteness} (see Definition~\ref{dfn.locallyfinite}). Furthermore, we introduce a global ``linear Poincar\'e flow'' adapted to the Poincar\'e section, as a necessary technical ingredient for the subsequent coding construction.

%%%%%%%%%%%%%%%%%%%%%%%%%%%%%%%%%%%%%%%%%%%%%%%%%%%%%%%%%%%%%%%%%%%

\section{Poincar\'e maps and scaled Poincar\'e maps}\label{sec.poincaremap}

In this section, we recall the definitions of Poincar\'e maps and scaled Poincar\'e maps, and collect several basic properties that will be used in subsequent arguments.

%%%%%%%%%%%%%%%%%%%%%%%%%%%%%%%%%%%%%%%%%%%%%%%%%%%%%%%%%%%%%%%%%%%

\subsection{Poincar\'e maps}

 Recall that $M$ is a smooth connected compact $d$-dimensional ($d\ge 3$) Riemannian manifold without boundary. Let $X$ be a $C^1$ vector field on $M$. It generates a $C^1$ flow $\varphi_t$ on $M$, along with its tangent flow  $d\varphi_t:TM\to TM$, where $TM$ is the tangent bundle of $M$.

 Denote by $\Sing(X):=\{x\in M: X(x)=0\}$ the set of singularities of $X$.
For any point $x\in M\setminus\Sing(X)$, the {\it normal space} $\N_x$ at $x$ is defined as 
\begin{equation}\label{def.normbundle}
\N_x:=\{v\in T_xM:\la v, X(x)\ra=0\},
\end{equation}
 where $\la\cdot,\cdot\ra$ is the inner product on $TM$ induced by the Riemannian metric. The disjoint union 
 \[
 \N:=\biguplus_{x\in M\setminus\Sing(X)}\N_x
 \] 
 forms a $(d-1)$-dimensional vector bundle over $M\setminus\Sing(X)$, which is called the {\it normal bundle} of $X$.

Let $\proj_{\N}: TM|_{M\setminus\Sing(X)}\to\N$ be the orthogonal projection onto the normal bundle along the flow direction. The {\it linear Poincar\'{e} flow} $\psi_t: \N\to\N$ is defined as the orthogonal projection of $d\varphi_t$ onto $\N$:
\[
\psi_t(v)=\proj_{\N}(d\varphi_t(v))=d\varphi_t(v)-\frac{\langle d\varphi_t(v),X(\varphi_t(x))\rangle}{|X(\varphi_t(x))|^2}X(\varphi_t(x)),
\]
 for any $x\in M\setminus \Sing(X)$ and $v\in \N_x$.
 
 Since $M$ is compact, there exists a constant  $\rho_M>0$ such that for any $x\in M$, the {\it exponential map}  $\exp_x: T_x M\to M$ is a $C^\infty$ diffeomorphism from
 $T_xM(\rho_M)$ to its image $M(x,\rho_M)$, where $T_xM(\rho_M)$ is the closed ball of radius $\rho_M$ centered at the origin in $T_xM$ and $M(x,\rho_M)$ is the closed ball of radius $\rho_M$ centered at $x$ in $M$. By reducing $\rho_M$ if necessary, we assume that 
 \begin{equation}\label{equ.dexp}
 \|d\exp_x\|\le 2 \quad \text{and} \quad \|d\exp^{-1}_x\|\le 2
\end{equation}
 for any $x\in M$.
 Denote by $N_x(\rho_M)=\exp_x(\N_x(\rho_M))$ the {\it normal manifold} at $x$ of radius $\rho_M$, where $\N_x(\rho_M)$ is the closed ball of radius $\rho_M$ centered at the origin in $\N_x$.
 
For any $x \in M \setminus \Sing(X)$ and $t>0$, the classical {\it Poincar\'e map} $P_{x,\varphi_t(x)}$ is the holonomy map generated by the flow from a neighborhood of $x$ in $N_x(\rho_M)$ to $N_{\varphi_t(x)}(\rho_M)$, that is, for any $z\in N_x(\rho_M)$ close enough to $x$,
\[
P_{x,\varphi_t(x)}(z)=\varphi_{\mathfrak{t}_{x,\varphi_t(x)}(z)}(z)\in N_{\varphi_t(x)}(\rho_M)
\]
for some $\mathfrak{t}_{x,\varphi_t(x)}(z)$ (which is called the {\it Poincar\'e time}) close to $t$. Using the exponential map, one may lift the Poincar\'e map locally to get the {\it sectional Poincar\'e map} on the normal bundle:
\[
\P_{x,\varphi_t(x)}:=\exp_{\varphi_t(x)}^{-1}\circ P_{x,\varphi_t(x)}\circ \exp_x.
\]

It follows directly from the continuity of the flow that if the orbit of $x\in M\setminus\Sing(X)$ intersects the normal manifold transversely at some point $y\in M\setminus\Sing(X)$, then the holonomy map generated by the flow from a neighborhood of $x$ in $N_x(\rho_M)$ to $N_{y}(\rho_M)$ is well-defined. We say that this holonomy map is the {\it Poincar\'e map} from $x$ to $y$, and denote it by $P_{x,y}$. Denote by $\mathfrak{t}_{x,y}(z)$ the {\it Poincar\'e time} satisfying $P_{x,y}(z)=\varphi_{\mathfrak{t}_{x,y}(z)}(z)$. 
We also define the {\it sectional Poincar\'e map} from $x$ to $y$ as 
 \[
\P_{x,y}:=\exp_{y}^{-1}\circ P_{x,y}\circ \exp_x.
\]

 \subsection{Scaled Poincar\'e maps and scaled tubular neighborhoods}
 
 As noted in Section~\ref{sec.intro}, the Poincar\'e map is locally defined and 
 the size of its domain tends to zero when the point approaches the singularities.
To address the issues arising from this phenomenon, Liao \cite{Liao1, Liao2, Liao3} introduced the {\it scaled linear Poincar\'e flow} $\psi^*_t: \N\to \N$ as
\[
\psi^*_t(v):=\frac{\psi_t(|X(x)|v)}{|X(\varphi_t(x))|}=\frac{|X(x)|}{|X(\varphi_t(x))|}\psi_t(v),
\]
where $x\in M\setminus\Sing(X)$ and $v\in \N_x$.

We will also consider the {\it scaled sectional Poincar\'e map}
\[
\P_{x,y}^*(v):=\frac{\P_{x,y}(|X(x)|v)}{|X(y)|},
\]
where $x,y\in M\setminus\Sing(X)$ and  $v\in\N_x$ is sufficiently small.

The scaling technique pioneered by Liao is fundamental for the analysis of singular flows, and all subsequent symbolic coding constructions will heavily rely on these scaled objects. 
 
For $x\in M\setminus \Sing(X)$, we define the {\it scaled tubular neighborhood} at $x$ of size $\gamma>0$ and length $\tau>0$ as
\[
\FB(x,\gamma,\tau):=\bigcup_{\begin{subarray}{c} y\in N_x(\gamma|X(x)|)\\t\in[-\tau, \tau] \end{subarray}}\varphi_{t}(y).
\]
Denote by
$ U^*(x,\gamma,\tau)=\{v+tX(x)\in T_{x}M:v\in \N_x(\gamma|X(x)|),\ |t|\le \tau\} $
the scaled-size flow box in the tangent space at $x$. Define the $ C^{1} $ map $F_{x}:U^*(x,\gamma,\tau)\rightarrow M$ as
\[
F_{x}(v+tX(x))=\varphi_{t}(\exp_{x}(v)),
\]
whose image is $\FB(x,\gamma,\tau)$.

\subsection{Uniform estimates}

The following lemma establishes that scaled tubular neighborhoods of sufficiently small size and length are flow boxes on the manifold with some uniform control.

\setlist[enumerate]{label={\rm (F\arabic*)}}
\begin{lem}\label{lem.C1Poincare}
Let  $X$ be a $C^1$ vector field on $M$. Then there exist $\rho_{F1}\in (0,\rho_M)$ and $K_{F1}>1$, such that for all $x\in M\setminus\Sing(X)$,
\begin{enumerate}
  \item\label{fb.box}  the map $F_{x}:U^*(x,\rho_{F1},\rho_{F1})\rightarrow M$ is an embedding whose image contains no singularity of $X$, and satisfies $ \|d_{w}F_{x}\| \leq 3 $ for every $w\in U^*(x,\rho_{F1},\rho_{F1})$ and $ \|d_{y}F^{-1}_{x}\|\leq 3$ for every $y\in \FB(x,\rho_{F1},\rho_{F1})$;
  \item\label{fb.length} $\left|\frac{|X(y)|}{|X(z)|}-1\right|\le K_{F1}\frac{\dist(y,z)}{|X(z)|}$, and in particular, $\frac{|X(y)|}{|X(z)|}\in(\frac{9}{10},\frac{10}{9})$ for any $y,z\in \FB(x,\rho_{F1},\rho_{F1})$;
  \item\label{fb.contain} $M(x,\frac{\gamma}{3}|X(x)|)\subset \FB(x,\gamma,\gamma)\subset M(x,3\gamma|X(x)|)$ for any $\gamma\in (0,\rho_{F1}]$;
  \item\label{fb.derivare} the maps $P_{x}:\FB(x,\rho_{F1},\rho_{F1})\to N_x(\rho_{F1}|X(x)|)$ and $ \mathfrak{t}_x:\FB(x,\rho_{F1},\rho_{F1})\to [-\rho_{F1},\rho_{F1}]$ with $P_{x}(y)=\varphi_{\mathfrak{t}_x(y)}(y)\in  N_x(\rho_{F1}|X(x)|)$ are well-defined $C^1$ maps  and satisfy that 
\[
\norm{dP_{x}}\le K_{F1}  \quad \text{and} \quad  \norm{d \mathfrak{t}_x}\cdot |X(x)| \le  K_{F1}; 
\]
\item \label{fb.sectPoincare} for any $y$ satisfying $\dist(x,y)\le \rho_{F1}|X(x)|$, $P_{y,x}$ is well-defined on $N_y(\rho_{F1}|X(y)|)$ and the Poincar\'e time function $\mathfrak{t}_{y,x}$ satisfies $\mathfrak{t}_{y,x}=\mathfrak{t}_x|_{N_y(\rho_{F1}|X(y)|)}$,
\[
\norm{dP_{y,x}}\le K_{F1}  \quad \text{and} \quad  \norm{d\mathfrak{t}_{y,x}}\cdot |X(x)|\leq K_{F1}\cdot\frac{\dist(x,y)}{|X(x)|}.
\]
\end{enumerate}
\end{lem}
\begin{proof}
	Item~\ref{fb.box} is proved in \cite[\S\,2]{WW}.
	The core estimates for Item~\ref{fb.length}, \ref{fb.contain}, and the first statement in Item~\ref{fb.derivare} are established in \cite{GY}, and a clear statement and proof are provided in \cite[\S\,2]{WW}.	
	The proof of Item~\ref{fb.sectPoincare}, as well as the uniform boundedness statements for $\norm{d\mathfrak{t}_x}$ and $\norm{d \mathfrak{t}_{y,x}}$ can be found in \cite[\S\,3]{LLL2024}.
\end{proof}
\setlist[enumerate]{label={\rm (\arabic*)}}

For the higher-order uniform estimates of $\P_{x,y}$, we have the following results.

\begin{lem}\label{lem.holderPoincare}\cite[Lemma~4.1]{LLL2024}
	If $X$ is a $C^{1+\beta}$ vector field on $M$, then there exists $\rho_{F}\in (0,\rho_{F1})$ and $K_{F}\ge K_{F1}$  such that for any $x,y\in M\setminus\Sing(X)$ with $\dist(x,y)\le \rho_F|X(x)|$ and $v_1, v_2\in \N_x(\rho_F)$, we have
	\[
	\|d\P_{x,y}^*(v_1)-d\P_{x,y}^*(v_2)\|\le K_F |v_1-v_2|^\beta.
	\]   
\end{lem}

%%%%%%%%%%%%%%%%%%%%%%%%%%%%%%%%%%%%%%%%%%%%%%%%%%%%%%%%%%%%%%%%%%%

\section{Compatible Poincar\'{e} sections for singular flows}\label{sec.section}
Throughout this section, we only require that  $X$ is a $C^1$ vector field. Our main goal is to construct an appropriate family of transverse sections. These sections, together with their associated Poincar\'e return maps, form the core technical foundation for the symbolic coding procedure developed in Part~\ref{part.coding}.

For vector fields $X$ without singularities, the existence of standard Poincar\'{e} sections is established in \cite{Bow73} (also see \cite{LS}). We thus focus on the singular case in what follows, while noting that the entire construction remains fully valid for non-singular flows as well.

\subsection{Poincar\'{e} sections}

  We begin by introducing some definitions.
A connected smooth codimension-one submanifold $D\subset M\setminus \Sing(X)$ is called a {\it transverse section} of $X$ if each orbit of $X$ intersects $D$ transversely or is disjoint from $D$.

We say that the disjoint union of at most countably many transverse sections ${\mathbb D}=\biguplus_{j\in \Gamma} D_j$ ($\Gamma$ is a finite set $\{1,2,\dots,\ell\}$ or $\bZ^+$) is a {\it Poincar\'e section} of $X$ if for any $x\in M\setminus \Sing(X)$, we have
\begin{itemize}
  \item when arranged in ascending order, the set $\{t>0: \varphi_t(x)\in \mathbb D \}$ forms a one-sided sequence that tends to positive infinity; and 
  \item when arranged in descending order, the set $\{t<0: \varphi_t(x)\in \mathbb D \}$ forms a one-sided sequence that tends to negative infinity.
\end{itemize}

If ${\mathbb D}=\biguplus_{j\in \Gamma} D_j$ is a Poincar\'e section of $X$, then one may define the {\it Poincar\'e return time} $R_{\mathbb D}: {\mathbb D}\to \bR^+$ as $R_{\mathbb D}(x)=\min\{t>0: \varphi_t(x)\in {\mathbb D}\}$, and the {\it Poincar\'e return map} $f_{\mathbb D}:{\mathbb D}\to{\mathbb D}$ as $f_{\mathbb D}(x)=\varphi_{R_{\mathbb D}(x)}(x)$, which is a bijection. From the perspective of the Poincar\'e section ${\mathbb D}$, the flow $\varphi_t$ on $M\setminus\Sing(X)$ can be regarded as a \emph{suspension flow} of the Poincar\'e return map $f_{\mathbb D}$ with the Poincar\'e return time $R_{\mathbb D}$.

A Poincar\'e section ${\mathbb D}$ of $X$ is called {\it uniform} if $\inf_{x\in {\mathbb D}} R_{\mathbb D}(x) >0$ and $\sup_{x\in {\mathbb D}} R_{\mathbb D}(x) <+\infty$.

\begin{dfn}\label{dfn.locallyfinite}
We say that a Poincar\'e section ${\mathbb D}=\biguplus_{j\in \Gamma} D_j$ of $X$ is  {\rm locally finite} if there exists a constant $I_{\max}\ge 1$ such that for any $j\in \Gamma$, we have
 \[
\#\{i\in\Gamma: f_{{\mathbb D}}(D_j)\cap D_i\neq\emptyset \text{ or } f_{{\mathbb D}}^{-1}(D_j)\cap D_i\neq\emptyset\}\le I_{\max}.
\]
\end{dfn}

The local finiteness property is essential in the subsequent symbolic coding construction.
The following main result of this part asserts the existence of locally finite uniform Poincar\'e sections of $X$.

\setlist[enumerate]{label={\rm (S\arabic*)}}
\begin{thm}\label{thm.poincaresection}
For any $C^1$ vector field $X$ on $M$ and any numbers $\delta>0$, $\kappa>1$, there exist two numbers $\gamma_0\in (0,\delta/\kappa)$, $I_{\max} \ge 1$ and at most countably many points $\{o_j\}_{j\in \Gamma}\subset M\setminus \Sing(X)$ such that for every $\gamma\in[\gamma_0,\kappa \gamma_0]$,
\begin{enumerate}
  \item\label{se.poin} {\rm Poincar\'e section:} $\{N_{o_j}(\gamma|X(o_j)|)\}_{j\in \Gamma}$ is a collection of mutually disjoint transverse sections, and ${\mathbb D}(\gamma):=\biguplus_{j\in \Gamma} N_{o_j}(\gamma|X(o_j)|)$ is a Poincar\'e section of $X$;
  \item\label{se.unif} {\rm Uniformity:} $R_{{\mathbb D}(\gamma)}(x) \ge \kappa \gamma_0$ and $R_{{\mathbb D}(\gamma)}(x) \le \delta$ for all $x\in{\mathbb D}(\gamma)$, where $R_{{\mathbb D}(\gamma)}$ is the corresponding Poincar\'e return time of ${\mathbb D}(\gamma)$;
  \item\label{se.fini} {\rm Local finiteness:} $\#\{i\in\Gamma: f_{{\mathbb D}(\gamma)}(N_{o_j}(\gamma|X(o_j)|))\cap N_{o_i}(\gamma|X(o_i)|)\neq\emptyset \text{ or } f_{{\mathbb D}(\gamma)}^{-1}(N_{o_j}(\gamma|X(o_j)|))\cap N_{o_i}(\gamma|X(o_i)|)\neq\emptyset\}\le I_{\max}$ for any $j\in \Gamma$,  where $f_{{\mathbb D}(\gamma)}$ is the corresponding Poincar\'e return map of ${\mathbb D}(\gamma)$;
  \item\label{se.sepa} {\rm Separated:} $M(o_{j},3\kappa \gamma_0|X(o_{j})|)\cap M(o_{i},3\kappa \gamma_0|X(o_{i})|)=\emptyset$ for any $j\neq i$.
\end{enumerate}
\end{thm}

\setlist[enumerate]{label={\rm (\arabic*)}}

%%%%%%%%%%%%%%%%%%%%%%%%%%%%%%%%%%%%%%%%%%%%%%%%%%%%%%%%%%%%%%%%%%%

\subsection{Construction of locally finite uniform Poincar\'{e} sections}

To prove Theorem~\ref{thm.poincaresection}, we first construct a locally finite cover of 
$M\setminus\Sing(X)$ by using scaled tubular neighborhoods of uniform size and length. To this end, we first establish the following technical lemma. Denote $\rho_0={10^{-8}K_{F1}^{-1}}\cdot \rho_{F1}$, where $\rho_{F1}$, $K_{F1}$ are given by Lemma~\ref{lem.C1Poincare}. We assume without loss of generality that $K_{F1}>2$.

\begin{lem}\label{lem.FBsubset} For any $0<\gamma\le \tau \le \rho_0$, if there exist two points $x,y\in M\setminus \Sing(X)$ such that $\FB(x,\gamma,\tau)\cap\FB(y,\gamma,\tau)\neq\emptyset$, 
then we have $\FB(y,\gamma,\tau)\subset \FB(x,3\gamma K_{F1} , 30\tau)$, where $K_{F1}$ is given by Lemma~\ref{lem.C1Poincare}.
\end{lem}
\begin{proof}
By Lemma~\ref{lem.C1Poincare}~\ref{fb.contain}, we have
\[
 \FB(x,\gamma,\tau)\subset \FB(x,\tau,\tau)\subset M(x,3\tau|X(x)|) \quad \text{and} \quad \FB(y,\gamma,\tau)\subset \FB(y,\tau,\tau)\subset M(y,3\tau|X(y)|).
\]

	Take a point  $z\in \FB(x,\gamma,\tau)\cap\FB(y,\gamma,\tau)\subset \FB(x,\tau,\tau)\cap\FB(y,\tau,\tau)$. By  Lemma~\ref{lem.C1Poincare}~\ref{fb.length}, we have $|X(y)|\ge \frac{9}{10}|X(z)|\ge \frac{81}{100}|X(x)|$ and $|X(y)|\le  \frac{100}{81}|X(x)|$. Then by Lemma~\ref{lem.C1Poincare}~\ref{fb.contain} again we obtain
\[
\FB(y,\gamma,\tau)\subset M(y,3\tau|X(y)|)\subset M(x,3\tau(|X(x)|+|X(y)|))\subset M(x,10\tau|X(x)|)\subset \FB(x,30\tau,30\tau).
\]

For any $y'\in \FB(y,\gamma,\tau)$, denote $y'_1=P_y(y')\in N_y(\gamma|X(y)|)$, where $P_y$ is the map given by Lemma~\ref{lem.C1Poincare}~\ref{fb.derivare}. We also denote $z_1=P_y(z)\in N_y(\gamma|X(y)|)$. Since $y_1'\in \FB(x,30\tau,30\tau)$ and $z\in \FB(x,\gamma,\tau)$, we have $P_x(y_1')=P_x(y')\in N_x(30\tau|X(x)|)$ and $P_x(z_1)=P_x(z)\in N_x(\gamma|X(x)|)$. Since $\dist (y_1',z_1)\le 2\gamma|X(y)|$, it follows from Lemma~\ref{lem.C1Poincare}~\ref{fb.derivare} that
\[
\dist(x,P_x(y'))\le \dist(x,P_x(z_1))+\dist(P_x(z_1),P_x(y_1'))\le \gamma|X(x)|+ 2\gamma K_{F1}|X(y)|<3\gamma K_{F1}|X(x)|. 
\]
 Thus $y'\in \FB(x,3\gamma K_{F1} , 30\tau)$. This completes the proof.
\end{proof}

Next we use the Vitali Covering Lemma to get a locally finite cover of $M\setminus \Sing(X)$.

\begin{lem}\label{lem.locallyfinitecover}
For any $\gamma'\in(0,\rho_0)$, there exists a constant $I_0\ge 1$ and at most countably many points $\{p_i\}_{i\in \Gamma_0}\subset M\setminus\Sing(X)$, where $\Gamma_0$ is a finite set or $\mathbb Z^+$, such that
\begin{itemize}
  \item $M\setminus\Sing(X)\subset \bigcup_{i\in \Gamma_0} \FB({p_i},\gamma',\gamma')$; and
  \item $\#\{i\in\Gamma_0: \FB({p_i},3\gamma'K_{F1}, 63000\gamma')\cap \FB(x,3\gamma'K_{F1}, 63000\gamma')\neq \emptyset\}<I_0$ for any $x\in M\setminus\Sing(X)$.
\end{itemize}
\end{lem}
\begin{proof}
Take $\gamma'\in(0,\rho_0)$.
Consider 
\[
\{M(x,\frac{\gamma'}{15}|X(x)|):x\in M\setminus \Sing(X)\},
\]
which is a cover of $M\setminus \Sing(X)$. Applying the standard Vitali Covering Lemma,
there exists a subset $\{M(p_i,\frac{\gamma'}{15}|X(p_i)|)\}_{i\in\Gamma_0}$ consisting of at most countably many mutually disjoint balls ($\Gamma_0$ is a finite set or $\mathbb Z^+$) such that
\[
M\setminus \Sing(X)\subset \bigcup_{i\in\Gamma_0}M(p_i,\frac{\gamma'}{3}|X(p_i)|).
\]
Then by Lemma~\ref{lem.C1Poincare}~\ref{fb.contain}, we have that
\[
M\setminus \Sing(X)\subset \bigcup_{i\in\Gamma_0}M(p_i,\frac{\gamma'}{3}|X(p_i)|)\subset \bigcup_{i\in \Gamma_0} \FB({p_i},\gamma',\gamma').
\]

On the other hand, for any $x\in M\setminus\Sing(X)$, if there is some $p_i$ such that $\FB({p_i},3\gamma'K_{F1}, 63000\gamma')\cap \FB(x,3\gamma'K_{F1}, 63000\gamma')\neq \emptyset$, then by Lemma~\ref{lem.FBsubset} and Lemma~\ref{lem.C1Poincare}~\ref{fb.contain}  we have 
\[
\FB({p_i},3\gamma'K_{F1}, 63000\gamma')\subset \FB(x,9\gamma'K_{F1}^2 , 1890000\gamma')\subset M(x,3(9K_{F1}^2+1890000)\gamma'|X(x)|).
\]
 It follows from 
Lemma~\ref{lem.C1Poincare}~\ref{fb.length} that $|X(p_i)|\ge \frac{9}{10}|X(x)|$.

Collect all scaled tubular neighborhoods $\FB({p_i},3\gamma'K_{F1}, 63000\gamma')$ intersecting $\FB(x,3\gamma'K_{F1}, 63000\gamma')$. By construction, each such scaled tubular neighborhood contains a ball $M(p_i,\frac{\gamma'}{15}|X(p_i)|)$, and these balls are pairwise disjoint. We lift all these balls to the tangent space $T_x M$ via the exponential map $\exp^{-1}_x$. Consequently, their images remain pairwise disjoint. Note that since $\|d\exp_x\|\le 2$ in $T_xM(\rho_M)$, each image $\exp^{-1}_x(M(p_i,\frac{\gamma'}{15}|X(p_i)|))$ contains a ball of radius at least $\frac{\gamma'|X(p_i)|}{30}$ ($\ge \frac{9\gamma'|X(x)|}{300}$), and these subballs are also pairwise disjoint. 

Since all such subballs are contained in $T_xM(3(9K_{F1}^2+1890000)\gamma'|X(x)|)$, their total volume is at most the volume of $T_xM(3(9K_{F1}^2+1890000)\gamma'|X(x)|)$. It follows that the number of scaled tubular neighborhoods $\FB({p_i},3\gamma'K_{F1}, 63000\gamma')$ intersecting $\FB(x,3\gamma'K_{F1}, 63000\gamma')$
\begin{equation*}
\begin{aligned}
	& \#\{i\in\Gamma_0: \FB({p_i},3\gamma'K_{F1}, 63000\gamma')\cap \FB(x,3\gamma'K_{F1}, 63000\gamma')
	 \neq \emptyset\}\\ \le &~\frac{\left(3(9K_{F1}^2+1890000)\gamma'|X(x)|\right)^d}{\left(\frac{9\gamma'|X(x)|}{300}\right)^d}=\left(900(K_{F1}^2+210000)\right)^d=:I_0.
\end{aligned}
 \end{equation*}

This completes the proof.
\end{proof}

It is immediate to verify that, if intersections of transverse sections are allowed, the family of normal manifolds $\{N_{p_i}(\gamma' |X(p_i)|)\}$ derived from the previous lemma already forms a Poincar\'e system. However, the corresponding Poincar\'e return time may fail to admit a uniform positive lower bound. We next  construct sections satisfying all the requirements of the theorem, based on the above family $N_{p_i}(\gamma'|X(p_i)|)$. 

For non-singular flows (where $\Gamma_0$ is a finite set), a valid construction method was given by Lima-Sarig in \cite{LS}. Thanks to the local finiteness property established in the previous lemma, this method remains applicable in the present setting after appropriate adjustments. While the proof involves repeated local adjustments across different transverse sections, the underlying combinatorial idea remains straightforward. 

On each $N_{p_i}(\gamma'|X(p_i)|)$, we define a lattice grid with spacing $\varepsilon|X(p_i)|$. We then push each lattice point along its orbit for a short, uniformly bounded time, and associate to each pushed lattice point a thin scaled tubular neighborhood of size proportional to $\varepsilon$ and appropriately extended length. By Lemma~\ref{lem.C1Poincare}, these new scaled tubular neighborhoods still cover the non-singular part of $M$. Thus, it suffices to push the lattice points to suitably chosen positions to ensure complete mutual separation of the new small sections.

Note that if two lattice points can approach each other during the pushing procedure, their distance when flowing to the same nearby large section is bounded above by a constant multiple of $\varepsilon$. Therefore, for each fixed lattice point, the number of lattice points on the same $N_{p_i}(\gamma'|X(p_i)|)$ that are at risk of mutual approach is uniformly bounded. Combined with the local finiteness property from the previous lemma, the total number of such at-risk points remains uniformly bounded, even though the total number of lattice points may be infinite. As a result, a suitable pushing configuration exists for all  sufficiently small $\varepsilon$.

\begin{proof}[Proof of Theorem~\ref{thm.poincaresection}]
Without loss of generality, we only need to prove the theorem for sufficiently small $\delta>0$ and sufficiently large $\kappa>1$. We thus assume that $\delta <\rho_0$ and $\kappa>10$. Let $I_0\ge 1$ and $\{p_i\}_{i\in\Gamma_0}\subset M\setminus\Sing(X)$ be given by Lemma~\ref{lem.locallyfinitecover} for $\gamma'=10^{-8}\delta$.

\medskip

\noindent Step 1: Selection of parameters and candidate points.

\medskip

We now choose a sufficiently small parameter $\varepsilon>0$, whose exact value will be fixed later in the proof. Define
\[
\gamma_0:=\left(\frac{20}{9}\sqrt{d-1}K_{F1}\right)\varepsilon,
\]
where $d$ is the dimension of $M$ and $K_{F1}$ is given by Lemma~\ref{lem.C1Poincare}. Choose $\varepsilon>0$ sufficiently small such that $9\kappa \gamma_0<\gamma'$ and $81\kappa \gamma_0K_{F1}^2<\gamma'$.

For each $i\in \Gamma_0$, select an orthonormal basis $\mathbf e_i$ of the normal space $\N_{p_i}$.
Denote 
\[
\varepsilon\mathbb Z^{d-1}(\gamma'):=\{\mathbf{m}: \frac{\mathbf m}{\varepsilon} \in \mathbb Z^{d-1},\ |\mathbf m|\le \gamma' \}
\]
and $p_i(\mathbf m)=\exp_{p_i}(|X(p_i)|(\mathbf{m}\cdot \mathbf{e}_i))$, $i\in \Gamma_0$, $\mathbf m\in \varepsilon\mathbb Z^{d-1}(\gamma')$. We say that $p_i(\mathbf m)$ is a {\it $\varepsilon$-lattice point} in $N_{p_i}(\gamma'|X(p_i)|)$. Since the set $\{p_i(\mathbf m):i\in \Gamma_0,\mathbf m\in \varepsilon\mathbb Z^{d-1}(\gamma')\}$ has at most countably many elements, we denote it by $\{q_j\}_{j\in\Gamma}$, where $\Gamma=\{1,2,\ldots, \ell\}$ or $\bZ^+$. 

\medskip

\noindent Step 2: Definition of the set of dangerous points.

\medskip

For any $j\in \Gamma$, denote the set of ``dangerous points'' for $q_j$ as
\[
\D(j)=\{j'\in\Gamma: \FB(q_j,27\kappa \gamma_0K_{F1}, 2100\gamma')\cap \FB(q_{j'},27\kappa \gamma_0K_{F1}, 2100\gamma')\neq\emptyset\}.
\]

\medskip

\nt {\bf Claim 1.} {\it  If $j_1\in  \Gamma\setminus\D(j)$, then $M(\varphi_{t}(q_j),3\kappa \gamma_0|X(\varphi_{t}(q_j))|)\cap M(\varphi_{t_1}(q_{j_1}),3\kappa \gamma_0|X(\varphi_{t_1}(q_{j_1}))|)=\emptyset$ for any $t,t_1\in [0,\gamma']$.}
\begin{proof}
We argue by contradiction. Suppose that	there exist $t,t_1\in [0,\gamma']$ such that
\[
M(\varphi_{t}(q_j),3\kappa \gamma_0|X(\varphi_{t}(q_j))|)\cap M(\varphi_{t_1}(q_{j_1}),3\kappa \gamma_0|X(\varphi_{t_1}(q_{j_1}))|)\neq\emptyset.
\]
Then by Lemma~\ref{lem.C1Poincare}~\ref{fb.contain}, we have  
\begin{equation}\label{eq.intersect}
  \FB(\varphi_{t}(q_j),9\kappa \gamma_0,9\kappa \gamma_0)\cap \FB(\varphi_{t_1}(q_{j_1}),9\kappa \gamma_0,9\kappa \gamma_0)\neq\emptyset.
\end{equation}
	
On the other hand, since $\FB(\varphi_{t}(q_j),9\kappa \gamma_0,\gamma')\cap \FB(q_j,9\kappa \gamma_0,\gamma')\neq\emptyset$, Lemma~\ref{lem.FBsubset} implies
\[
\FB(\varphi_{t}(q_j),9\kappa \gamma_0,\gamma')\subset \FB(q_j,27\kappa \gamma_0K_{F1},30\gamma').
\]
Similarly, we also have
\[
\FB(\varphi_{t_1}(q_{j_1}),9\kappa \gamma_0,\gamma')\subset \FB(q_{j_1},27\kappa \gamma_0K_{F1},30\gamma').
\]

Recall that $9\kappa \gamma_0<\gamma'$. Combining these inclusions with the intersection property \eqref{eq.intersect}, we directly conclude that $j_1\in\D(j)$.
This contradiction proves the claim.	
\end{proof}

The key observation here is that $\# \D(j)$ admits a uniform upper bound that is independent of the choice of $\varepsilon>0$ and the index $j\in \Gamma$.

\medskip

\nt {\bf Claim 2.} {\it There exists $I_{\max}$ independent of $\varepsilon$ and $j$ such that $\# \D(j)\le I_{\max}$.}
\begin{proof}
	Assume that $q_j=p_i(\mathbf m)$ for some $i\in \Gamma_0$ and $\mathbf m\in \varepsilon\mathbb Z^{d-1}(\gamma')$. For any $j'\in\D(j)$ (with corresponding lattice point $q_{j'}=p_{i'}(\mathbf m')$), applying Lemma~\ref{lem.FBsubset} and recalling $81\kappa \gamma_0K_{F1}^2<\gamma'$, we get 
	\[
	\FB(q_j,27\kappa \gamma_0K_{F1}, 2100\gamma')\subset \FB(q_{j'},81\kappa \gamma_0K_{F1}^2, 63000\gamma')\subset \FB(p_{i'}, 3\gamma'K_{F1}, 1890000\gamma').
	\]
	Projecting along the flow to $N_{p_{i'}}(\rho_0)$, we have 
\[
\dist(P_{p_{i'}}(q_j),q_{j'})= \dist(P_{p_{i'}}(P_{q_{j'}}(q_j)),P_{p_{i'}}(q_{j'}))\le 81\kappa \gamma_0K_{F1}^3|X(q_{j'})|\le 90\kappa \gamma_0K_{F1}^3|X(p_{i'})|,
\]	
where $P_{p_{i'}}$ and $P_{q_{j'}}$ are the projection maps given by Lemma~\ref{lem.C1Poincare}~\ref{fb.derivare}. Since $q_{j'}$ is an $\varepsilon$-lattice point, we have
	\[
	\#\{j'\in\D(j): q_{j'}\in  N_{p_{i'}}(\gamma'|X(p_{i'})|)\}\le \Big( \frac{180\kappa \gamma_0 K_{F1}^3}{\varepsilon} \Big)^{d-1}=\left( 1200\kappa \sqrt{d-1}  K_{F1}^4 \right)^{d-1}=:I_1.
	\]
	
	Applying Lemma~\ref{lem.FBsubset}  once more, we get 
	 $\FB(q_j,27\kappa \gamma_0K_{F1}, 2100\gamma')\subset \FB(p_i,3\gamma'K_{F1}, 63000\gamma')$ and \\$\FB(q_{j'},27\kappa \gamma_0K_{F1}, 2100\gamma') \subset \FB(p_{i'},3\gamma'K_{F1}, 63000\gamma')$. Thus 
	 \[
	 \FB(p_i,3\gamma'K_{F1}, 63000\gamma')
	 \cap \FB(p_{i'},3\gamma'K_{F1}, 63000\gamma')\neq \emptyset.
	 \]
	  It follows from Lemma~\ref{lem.locallyfinitecover} that
	 \[
	 \#\{i'\in\Gamma_0: \exists\, j'\in  \D(j) \text{ such that } q_{j'}\in  N_{p_{i'}}(\gamma'|X(p_{i'})|)\}\le I_0.
	 \]
Combining these two uniform bounds, we finally conclude $\# \D(j)\le I_0I_1=:I_{\max}$.	
\end{proof}

We now fix the choice of $\varepsilon$. Choose $\varepsilon>0$ sufficiently small such that
\[
 9\kappa \gamma_0<\gamma', \ 81\kappa \gamma_0K_{F1}^2<\gamma' \quad \text{and} \quad 30\kappa \gamma_0 I_{\max}<\gamma'.
\]

\medskip

\noindent Step 3: Construction of the transverse sections.

\medskip

\nt {\bf Claim 3.} {\it  There exists a sequence of real numbers $\{\theta_j \}_{j\in \Gamma}\subset [0,\gamma']$ such that for any $j\in\Gamma$ one has 
\[
  M(o_{j_1},3\kappa \gamma_0|X(o_{j_1})|)\cap M(o_{j},3\kappa \gamma_0|X(o_{j})|)=\emptyset
\]
for any $j_1\in \Gamma$ with $j_1<j$, where $o_j=\varphi_{\theta_j}(q_j)$. }

\begin{proof}
	We will construct the sequence $\{\theta_j \}$ by induction. Let $\theta_1=0$. 
	Assume that for all $j'<j$, a suitable $\theta_{j'}$ has been selected. We now construct $\theta_{j}$ as follows.
	
For any $j_1\in \D(j)$ with $j_1<j$, define
\[
T_j(j_1)=\{t\in[0,\gamma']: \varphi_{t}(q_j)\in \FB(o_{j_1},30\kappa \gamma_0,30\kappa \gamma_0) \}.
\]	
It is clear that $T_j(j_1)$ is an interval of length no more than $30\kappa \gamma_0$.
Then by Claim~2 we have that the total length of $\bigcup_{j_1\in\D(j),j_1<j}T_j(j_1)$ is at most $30\kappa \gamma_0 I_{\max}<\gamma'$. Therefore, we can take some $\theta_j\in [0,\gamma']\setminus \bigcup_{j_1\in\D(j),j_1<j}T_j(j_1)$. We now verify that this choice of $\theta_j$ satisfies the requirement.

Suppose for contradiction that there exists $j_1<j$ so that  
 $M(o_{j_1},3\kappa \gamma_0|X(o_{j_1})|)\cap M(o_{j},3\kappa \gamma_0|X(o_{j})|)\neq\emptyset$.
Then by Claim 1 we know that $j_1\in \D(j)$. Note that $\dist(o_{j_1},o_j)\le 3\kappa \gamma_0(|X(o_{j_1})|+|X(o_j)|)\le 10 \kappa \gamma_0|X(o_{j_1})|$. Thus $o_j\in M(o_{j_1},10\kappa \gamma_0|X(o_{j_1})|)\subset \FB(o_{j_1},30\kappa \gamma_0, 30\kappa \gamma_0)$. This directly contradicts our selection of $\theta_j$. This completes the inductive step.	
\end{proof}

\noindent Step 4: Verification of conclusions.

\medskip

We now verify that the given $\gamma_0$, $I_{\max}$ and $\{o_j\}$ satisfy the requirements of the theorem.  Item~\ref{se.sepa} follows directly from the inductive construction in Claim~3. Furthermore, by Lemma~\ref{lem.FBsubset} we have
\begin{equation}\label{eq.empty}
 \FB(o_{j_1},\kappa \gamma_0,\kappa \gamma_0)\cap \FB(o_{j_2},\kappa \gamma_0,\kappa \gamma_0)=\emptyset.
\end{equation}

For any $\gamma\in[\gamma_0,\kappa \gamma_0]$, since $\{q_j\}_{j\in\Gamma}\cap N_{p_i}(\gamma'|X(p_i)|)$ is a lattice grid with spacing $\varepsilon|X(p_i)|$ on $N_{p_i}(\gamma'|X(p_i)|)$, we obtain
\[
N_{p_i}(\gamma'|X(p_i)|)=\bigcup_{q_j\in N_{p_i}(\gamma'|X(p_i)|)} N_{p_i}(q_j,2\varepsilon  \sqrt{d-1}|X(p_i)|),
\]
where $ N_{p_i}(q_j,2\varepsilon  \sqrt{d-1}|X(p_i)|)$ is the closed ball of radius $2\varepsilon \sqrt{d-1}|X(p_i)|$ centered at $q_j$ in $N_{p_i}(\rho_0)$.

By Lemma~\ref{lem.FBsubset} and the construction of $o_j$, for any $j\in\Gamma$ satisfying $q_j\in N_{p_i}(\gamma'|X(p_i)|)$, we have
\[
\FB({p_i},\gamma',\gamma')\subset \FB(o_j, 3\gamma'K_{F1},30\gamma').
\]
Then for every $x\in  N_{p_i}(q_j,2\varepsilon \sqrt{d-1}|X(p_i)|)$, 
\[
\dist(P_{o_j}(x),o_j)\le 2\varepsilon K_{F1}\sqrt{d-1}|X(p_i)|\le \frac{20}{9} \varepsilon K_{F1}\sqrt{d-1}|X(o_j)|= \gamma_0|X(o_j)|.
\]
This immediately implies that 
\[
\FB({p_i},\gamma',\gamma')\subset \bigcup_{q_j\in N_{p_i}(\gamma'|X(p_i)|)}\FB(o_j, \gamma_0,30\gamma').
\]

Combining this with Lemma~\ref{lem.locallyfinitecover}, we obtain that
\begin{equation}\label{eq.fullcover}
 M\setminus\Sing(X)\subset \bigcup_{i\in \Gamma_0} \FB({p_i},\gamma',\gamma')\subset \bigcup_{j\in\Gamma}\FB(o_j,\gamma_0,30\gamma').
\end{equation}

For any $x\in M\setminus \Sing(X)$, we claim that there exists $t\in[10\gamma',70\gamma']$ so that
 $\varphi_{t}(x)\in  \bigcup_{j\in\Gamma} N_{o_j}(\gamma|X(o_j)|)$. 
 In fact by \eqref{eq.fullcover}, there exist $j\neq j'$ such that $x\in \FB(o_j, \gamma,30\gamma')$
 and $x':=\varphi_{40\gamma'}(x)\in \FB(o_{j'}, \gamma,30\gamma')$. Thus there is $t'\in [-30\gamma',30\gamma']$ such that $\varphi_{t'}(x')\in  N_{o_{j'}}(\gamma|X(o_{j'})|)$. So we get $t:=40\gamma'+t'\in [10\gamma',70\gamma']$ satisfying $\varphi_{t}(x)\in  N_{o_{j'}}(\gamma|X(o_{j'})|)$. Similarly, there exists $t^-\in[-70\gamma',-10\gamma']$ such that
 $\varphi_{t^-}(x)\in  \bigcup_{j\in\Gamma} N_{o_j}(\gamma|X(o_j)|)$. 
 
 Combining the above fact with \eqref{eq.empty}, we obtain Item~\ref{se.poin} and Item~\ref{se.unif}. In particular, $R_{{\mathbb D}(\gamma)}(x) \le 70\gamma' <\delta$.
 
 Next, we suppose that either $f_{{\mathbb D}(\gamma)}(N_{o_j}(\gamma|X(o_j)|))\cap N_{o_i}(\gamma|X(o_i)|)\neq\emptyset$ or $f_{{\mathbb D}(\gamma)}^{-1}(N_{o_j}(\gamma|X(o_j)|))\cap N_{o_i}(\gamma|X(o_i)|)\neq\emptyset$. Since $R_{{\mathbb D}(\gamma)}(x) \le 70\gamma'$, we have $\FB({o_i},\gamma,70\gamma')\cap \FB({o_j},\gamma,70\gamma')\neq\emptyset$. It follows from Lemma~\ref{lem.FBsubset} that
$\FB(q_i,3\gamma K_{F1},2100\gamma')\cap \FB(q_j,3\gamma K_{F1},2100\gamma')\neq\emptyset $. Thus $i\in\D(j)$. So Item~\ref{se.fini} follows from Claim~2. 

 This completes the proof of Theorem~\ref{thm.poincaresection}.
\end{proof}

%%%%%%%%%%%%%%%%%%%%%%%%%%%%%%%%%%%%%%%%%%%%%%%%%%%%%%%%%%%%%%%%%%%

\section{Induced Poincar\'e systems}\label{sec.linearpoincare}

As outlined in the introduction, our construction builds on the symbolic coding framework developed in \cite{BCL,LMN}. The core novelty of our work is that it accommodates infinite transverse sections, rather than the finite section setting treated in those references, with the additional critical requirement that nearly all involved objects must be scaled. Building directly on the locally finite uniform Poincar\'e sections constructed in \S\,\ref{sec.section}, we now establish a foundational technical framework for the corresponding Poincar\'e return maps. This framework serves as the basis for the symbolic coding procedure developed in Part~\ref{part.coding}.

In the remainder of this section, we assume that $X$ is a $C^{1+\beta}$ vector field on $M$, where the H\"older exponent satisfies $0<\beta\le 1$. To facilitate direct comparison with the existing results in \cite{BCL} and \cite{LMN}, we adopt a compatible system of notations.

Fix  $0<\rho<10^{-20}\cdot\min\{1,K_{F}^{-1} \rho_{F}\}$, where $\rho_{F}\le\rho_{F1}$, $K_{F}\ge K_{F1}>1$ are given by Lemma~\ref{lem.holderPoincare}. 
Let $\kappa=10 K_F$ and take a sufficiently small $\delta\in (0,\rho)$. By Theorem~\ref{thm.poincaresection}, there exist $\gamma_0\in (0,\delta/\kappa)$ and at most countably many $o_j\in M\setminus\Sing(X)$ satisfying the conclusion of the theorem.
Denote $D_j:=N_{o_j}(\gamma_0|X(o_j)|)$, $\wh D_j:=N_{o_j}(\kappa \gamma_0|X(o_j)|)$, $\mathbb D:=\bigcup_{j\in\Gamma} D_j$ and  $\wh{\mathbb D}:=\bigcup_{j\in\Gamma} \wh D_j$. For simplicity, we denote the Poincar\'e return map associated to  $\mathbb D$ by $f$.

%%%%%%%%%%%%%%%%%%%%%%%%%%%%%%%%%%%%%%%%%%%%%%%%%%%%%%%%%%%%%%%%%%%

\subsection{Holonomy maps}\label{subsec.Holomap}

 In the previous  section, we constructed locally finite uniform Poincar\'e sections for $X$. The main technical issue of the system generated by these Poincar\'e sections is that the associated Poincar\'e return map is discontinuous. This discontinuity arises when the orbits pass exactly through the boundaries of the transverse sections. However, by appropriately enlarging each transverse section, we can naturally extend the return map to a well-defined continuous holonomy map on the enlarged sections.

Assume that $x\in D_{j}$, $f(x)\in D_{i}$ and $f^{-1}(x)\in D_{i'}$. Then according to Lemma~\ref{lem.C1Poincare} and Theorem~\ref{thm.poincaresection}, the Poincar\'e maps $P_{o_j,o_i}$ and ${P}_{o_j,o_{i'}}$ are well-defined on $N_{o_j}(x,\gamma_0|X(o_j)|)$, where $N_{o_j}(x,\gamma_0|X(o_j)|)$ is the closed ball of radius $\gamma_0|X(o_j)|$ centered at $x$ in  $\wh D_j$. Denote the holonomy maps at $x$ as
 \[
 g^{+}_x:= {P}_{o_j,o_i}|_{N_{o_j}(x,\gamma_0|X(o_j)|)} \quad \text{and} \quad
 g^{-}_x:= {P}_{o_j,o_{i'}}|_{N_{o_j}(x,\gamma_0|X(o_j)|)}.
 \]
 
 It follows from Lemma~\ref{lem.C1Poincare} that $\|dg^{+}_x\|\le K_{F1}$ and  $g^{+}_x(N_{o_j}(x,\gamma_0|X(o_j)|))\subset N_{o_i}(f(x),\gamma_0K_{F1 }|X(o_i)|)\subset \wh D_i$. Thus  $g^{+}_x$ is a diffeomorphism from
 $N_{o_j}(x,\gamma_0|X(o_j)|)$ to its image in $\wh D_i$. Note that $g_x^+(y)$ and $f(y)$ need not coincide. 
 
  For any $x\in \mathbb{D}$, points $o_j$ and $o_i$ are uniquely associated to $x$ such that  $x\in D_j$ and $f(x)\in D_i$. This allows us to simply write the Poincar\'e time function $\mathfrak{t}_{o_j,o_i}|_{N_{o_j}(x,\gamma_0|X(o_j)|)}$ corresponding to $P_{o_j,o_i}|_{N_{o_j}(x,\gamma_0|X(o_j)|)}$ as $\tau_x$ without ambiguity.
 
 By differentiating the equation  $g^{+}_x(y)=\varphi_{\tau_x(y)}(y)$, we obtain  $d_yg^{+}_x=X(y)d_y\tau_x+d_y\varphi_{\tau_x(y)}$. According to  Lemma~\ref{lem.C1Poincare}~\ref{fb.length}, \ref{fb.contain} and \ref{fb.sectPoincare}, we may assume that $\delta>0$ is small enough such that
\begin{equation}\label{eq.boundtimedg}
 \left|\frac{|X(o_j)|}{|X(x)|}-1\right|< \frac{\rho}{16},\quad \norm{d\tau_x}< \min \Big\{\frac{1}{|X(x)|}, \frac{1}{|X(o_j)|}\Big\}, \quad \text{and}\quad \norm{dg^{+}_x}< 2
 \end{equation}
for any $x\in D_j$. We also assume $\norm{dg_x^{-}}< 2$ symmetrically.

 For any point $x\in D_j$, denote by $\wt{x}:=\exp_{o_j}^{-1}(x)\in \N_{o_j}(\gamma_0|X(o_j)|)$ the lift of $x$, and by 
\[
\wt{g}^{+}_x:=\exp_{o_i}^{-1}\circ g^{+}_x\circ \exp_{o_j}
\]
the lift map of $g^{+}_x$.
By the definition, we have $\wt{g}^{+}_x(\wt y)=\wt{g^{+}_x(y)}$.

We also define $\wt x^*:=\frac{\wt x}{|X(o_j)|}$ and
\[
\wt g^{+,*}_x(v):=\frac{1}{|X(o_i)|}\wt{g}^{+}_x(|X(o_j)|v).
\]

By construction, $\wt{g}^{+}_x$ is well-defined on $\N_{o_j}(\wt{x},\gamma_0|X(o_j)|)$ and $\wt g^{+,*}_x$ is well-defined on $\N_{o_j}(\wt x^*,\gamma_0)$. Since $\wt g^{+,*}_x=\P^*_{o_j,o_i}|_{\N_{o_j}(\wt x^*,\gamma_0)}$, by Lemma~\ref{lem.holderPoincare}, we have that for any $v_1,v_2\in \N_{o_j}(\wt x^*,\gamma_0)$
	\begin{equation}\label{eq.gstar}
  \|d_{v_1}\wt g^{+,*}_x-d_{v_2}\wt g^{+,*}_x\|\le K_F |v_1-v_2|^\beta.
\end{equation}

Similarly, we define  $\wt g^{-}_x$ and $\wt g^{-,*}_x$, and $\wt g^{-,*}_x$ satisfies the same H\"older continuity as stated in \eqref{eq.gstar}.

%%%%%%%%%%%%%%%%%%%%%%%%%%%%%%%%%%%%%%%%%%%%%%%%%%%%%%%%%%%%%%%%%%%

\subsection{Induced linear Poincar\'e flow}\label{subsec.inducePion}

During the coding process in Part~\ref{part.coding}, we require a globally defined linear flow adapted to the return map. However, the standard linear Poincar\'e flow does not coincide with $d g^{+}_x$, so we need to define a new adapted linear flow.

\subsubsection{Construction of the induced linear Poincar\'e flow}\label{subsubsec.inducedpoincare}

For any $x\in M$ and any unit vector $v_x\in  T_xM$, denote by $\N_{v_x}$ the orthogonal complement of $\langle v_x\rangle$ in $T_x M$, where $\langle v_x\rangle$ is the $1$-dimensional subspace spanned by $v_x$. 
Define $N_{v_x}(\rho_M):=\exp_x(\N_{v_x}(\rho_M))$, where $\N_{v_x}(\rho_M):=\{v\in \N_{v_x}:|v|\leq \rho_M\}$. 
Let $Z_{v_x}$ be the continuous unit normal vector field of $N_{v_x}$ such that $Z_{v_x}(x)=v_x$.

Using parallel translation along geodesics normal to the section, we can extend $Z_{v_x}$ to a unit vector field over $M(x,\rho_M)$. We still denote it by $Z_{v_x}$.
By the compactness of $M$, $\|\nabla Z_{v_x}\|$ is bounded by a constant independent of the choice of $x$ and $v_x$. 

Denote $Z_j=Z_{u_j}$, where $u_j:=\frac{X(o_j)}{|X(o_j)|}$, $j\in \Gamma$. Since $\|\nabla Z_j\|$ and $\|\nabla X\|$ are uniformly bounded and $Z_j(o_j)=\frac{X(o_j)}{|X(o_j)|}$, by reducing $\delta$ if necessary, we may assume that for any $x\in M(o_{j},3\kappa \gamma_0|X(o_{j})|)$,
\begin{equation}\label{eq.angleXZ}
\left|Z_j(x)-\frac{X(x)}{|X(x)|}\right|<\frac{\rho}{16}.
\end{equation}

Let $h: \bR^d(1)\to [0,1]$ be a $C^{\infty}$ bump function defined on the unit ball, satisfying $h(w)=1$ for any $w\in \bR^d(1/3)$ and $h(w)=0$ for any $w\in \bR^d(1)\setminus \bR^d(2/3)$, where $\bR^d(a)$ is the closed ball in $\bR^d$ of radius $a$ centered at the origin.
For each $j\in \Gamma$, choose a linear isometry $\iota_j: T_{o_j} M \to \bR^d$. For any $x\in M(o_{j},3\kappa \gamma_0|X(o_{j})|)$, denote
\[
h_j(x):=h\left(\frac{\iota_j(\exp_{o_j}^{-1}(x))}{3\kappa \gamma_0|X(o_{j})|} \right).
\]
Then $h_j$ is a bump function defined on $M(o_{j},3\kappa \gamma_0|X(o_{j})|)$.

According to Theorem~\ref{thm.poincaresection}~\ref{se.sepa},  $M(o_{j},3\kappa \gamma_0|X(o_{j})|)\cap M(o_{i},3\kappa \gamma_0|X(o_{i})|)=\emptyset$ for any $j\neq i$. 
We can thus define a vector field on $M$ as follows: 
\[
Y(x)=\left\{
\begin{array}{lcl}
h_j(x)|X(o_j)|\cdot Z_j(x)+(1-h_j(x))\cdot X(x), &&\text{if $x\in M(o_{j},3\kappa \gamma_0|X(o_{j})|)$ for some $j\in\Gamma$;}\\
X(x),&&\text{else.}
\end{array}\right.
\]

By construction, we have:
\begin{itemize}
\item $\Sing(Y)=\Sing(X)$;
\item for any $x\in \wh{D}_j$, we have $Y(x)=|X(o_j)|\cdot Z_j(x)$; consequently, $Y(x)$ is orthogonal to $\wh{D}_j$ with $|Y(x)|=|X(o_j)|$;
\item $Y$ is a $C^{1}$ vector field on $M\setminus\Sing(X)$ and $\|\nabla Y\|$ is uniformly bounded by some constant $K_Y$ on $M\setminus \Sing(X)$.
\end{itemize}

Moreover, for any $x\in M\setminus\Sing(X)$, it follows from \eqref{eq.boundtimedg} and \eqref{eq.angleXZ} that
\[
\frac{|Y(x)-X(x)|}{|X(x)|}\leq \frac{|h_j(x)|\cdot\left||X(o_j)|Z_j(x)-X(x)\right|}{|X(x)|}\leq
\left|\frac{|X(o_j)|}{|X(x)|}-1\right|\cdot |Z_j(x)| + \left|Z_j(x)-\frac{X(x)}{|X(x)|}\right|<\frac{\rho}{8}.
\] 
Thus
\begin{equation}\label{eq.ratiosec}
\frac{1}{2}<e^{-\frac{\rho}{4}}<1-\frac{\rho}{8}<\frac{|Y(x)|}{|X(x)|}<1+\frac{\rho}{8}<e^{\frac{\rho}{8}}<2, \text{ and}
\end{equation}
\begin{equation}\label{eq.unitXY}
\left|\frac{Y(x)}{|Y(x)|}-\frac{X(x)}{|X(x)|}\right|\leq \frac{|Y(x)-X(x)|}{|X(x)|}+\left|\frac{|Y(x)|}{|X(x)|}-1\right|<\frac{\rho}{4}.
\end{equation}

For any  $x\in M\setminus \Sing(X)$, denote by $\H_x$ the orthogonal complement of $\langle Y(x)\rangle$ in $T_x M$. Let $\H:=\biguplus_{x\in M\setminus \Sing(X)}\H_x$. Following the standard definition of the classical linear Poincar\'e flow, we define the {\it induced linear Poincar\'e flow} $\Phi:\H\to\H$ as
\[
\Phi_t(v):=\proj_{\H}\circ d\varphi_t(v), 
\]
where $x\in M\setminus\Sing(X)$, $v\in \H_x$ and  $\proj_{\mathcal{H}}: TM|_{M\setminus\Sing(X)}\to\mathcal{H}$ is the projection to $\H$ parallel to the flow direction $\langle X\rangle$ (note that this is not the projection parallel to $\langle Y \rangle$). According to \eqref{eq.unitXY}, we have
\begin{equation}\label{eq.projH}
\norm{\proj_{\mathcal{H}}}\leq e^{\frac{\rho}{2}}.
\end{equation}

It is clear that 
\begin{equation}\label{rel.Pp}
\Phi_t=\proj_{\mathcal{H}}\circ \psi_t\circ \proj_{\mathcal{N}},
\end{equation}
for any $t\in \bR$.

We also define the scaled induced linear Poincar\'e flow. Unlike the scaled linear Poincar\'e flow, to facilitate later calculations, we scale $\Phi$ by the length of the vector field $Y$. The {\it scaled induced linear Poincar\'e flow} is defined as
\begin{equation}\label{def.sf}
\Phi^*_t(v)=\frac{|Y(x)|}{|Y(\varphi_t(x))|}\cdot\Phi_t(v),
\end{equation}
where $x\in M\setminus\Sing(X)$,  $v\in\mathcal{H}_x$ and $t\in\mathbb{R}$.

\subsubsection{H\"older continuity of $\Phi^*$}

Denote by $d_{\rm Sas}(\cdot,\cdot)$ the {\it Sasaki metric} on $TM$, which is induced by the Riemannian metric on $M$ (see, e.g. \cite{BMW}). For any  $v'\in T_x M$ and $w'\in T_y M$ close enough, we have
\begin{equation}\label{eq.sasaki}
\frac{1}{2}(\dist(x,y)+|v'-T_{y,x}w'|)\le d_{\rm Sas}(v',w')\le 2(\dist(x,y)+|v'-T_{y,x}w'|),
\end{equation}
where $T_{y,x}$ is the  parallel translation along the unique geodesic from $y$ to $x$.

\begin{pro}\label{hpsi}
Let $X$ be a $C^{1+\beta}$ vector field  on $M$, $0<\beta\le 1$. Then 
\begin{enumerate}
\item \label{Phi.rel} $d_yg^{+}_{x}=\Phi_{\tau_x(y)}$ for any $x\in D_j$ and $y\in N_{o_j}(x,\gamma_0|X(o_j)|)$;
\item\label{Phi.hol} there exist ${\rho}_{F3}>0$ and ${K}_{F3}> 1$ such that for any $x,y\in M\setminus\Sing(X)$ with $\dist(x,y)\le {\rho}_{F3}|X(x)|$, any unit vectors $v\in \mathcal{H}_x$, $w\in \mathcal{H}_y$ and any $|t|\leq 1$,
\[
\left|\Phi^*_t(v)-T_{\varphi_t(y),\varphi_t(x)}\circ\Phi^*_t(w)\right|\leq {K}_{F3}\bigg(\Big(\frac{\dist(x,y)}{|X(x)|}\Big)^{\beta}+|v-T_{y,x}(w)|\bigg).
\]
\end{enumerate}
\end{pro}
\begin{proof}
Since $g^{+}_x(y)=\varphi_{\tau_x(y)}(y)$ for any $x\in D_j$ and $y\in N_{o_j}(x,\gamma_0|X(o_j)|)$, we have
\[
d_yg^{+}_x=d_y\varphi_{\tau_x(y)}+ d_y\tau_{x}\cdot \frac{d}{dt}\Big|_{t=\tau_x(y)}\varphi_t(y)=d_y\varphi_{\tau_x(y)}+ d_y\tau_{x}\cdot X(g^{+}_x(y)).
\]
Note that $d_yg^{+}_x(v)\in \H_{g^{+}_x(y)}$ for any $v\in \H_y$. Therefore, we have 
\[
d_yg^{+}_x(v)=\proj_{\H_{g^{+}_x(y)}}(d_y\varphi_{\tau_x(y)}(v)+ d_y\tau_{x}\cdot X(g^{+}_x(y)))=\proj_{\H_{g^{+}_x(y)}}(d_y\varphi_{\tau_x(y)}(v))=\Phi_{\tau_x(y)}(v).
\]

We next consider Item~\ref{Phi.hol}. 
By the standard continuous dependence theorem for ODE initial data, $d\varphi_t$  is uniformly $\beta$-H\"older continuous for $t\in [-1,1]$. Thus there exist $\rho_1>0$ and $K_1> 1$ such that
\begin{equation}\label{eq.holderdphi}
  |d\varphi_t(v)|\le K_1 \quad \text{and} \quad \left|d\varphi_t(v)-T_{\varphi_t(y),\varphi_t(x)}\circ d\varphi_t(w)\right|\leq {K}_{1}\left(({\dist(x,y)})^{\beta}+|v-T_{y,x}(w)|\right)
\end{equation}
for any $x,y\in M$ with $\dist(x,y)\le {\rho}_{1}$, any unit vectors $v\in T_x M$, $w\in T_y M$  and any $|t|\leq 1$. 

Recall that $\|\nabla Y\|<K_Y$ on $M\setminus\Sing(X)$. Thus there exists $\rho_2>0$
such that $M(x,\rho_2|Y(x)|)\subset M\setminus\Sing(X)$ for all $x\in M\setminus\Sing(X)$. There further exists $K_2>1$ such that for any $x,y\in M\setminus\Sing(X)$ with $\dist(x,y)\le \rho_2|Y(x)|$, we have
\begin{equation}\label{eq.YxandYy}
\big||Y(x)|-|Y(y)|\big|\le|Y(x)-T_{y,x}Y(y)|\le K_2\cdot\dist(x,y).
\end{equation}
This immediately implies that 
\begin{equation}\label{eq.Yunit}
   \left|\frac{Y(x)}{|Y(x)|}-T_{y,x}\left(\frac{Y(y)}{|Y(y)|}\right)\right|\le\frac{|Y(x)-T_{y,x}Y(y)|}{|Y(x)|}+\frac{|T_{y,x}Y(y)|}{|Y(y)|}\cdot\frac{\big||Y(y)|-|Y(x)|\big|}{|Y(x)|} \le 2K_2\frac{\dist(x,y)}{|Y(x)|}.
\end{equation}
Since $\|\nabla X\|$ is uniformly bounded, by shrinking $\rho_2>0$ and increasing $K_2$ if necessary, we may additionally assume 
\begin{equation}\label{eq.Xunit}
  \left|\frac{X(x)}{|X(x)|}-T_{y,x}\left(\frac{X(y)}{|X(y)|}\right)\right|\le 2K_2\frac{\dist(x,y)}{|X(x)|}.
\end{equation}

Recall that $\proj_{\mathcal{H}}: TM|_{M\setminus\Sing(X)}\to\mathcal{H}$ is the projection to $\H$ parallel to $\langle X\rangle$, where $\H_x$ is the orthogonal complement of $\langle Y(x)\rangle$ in $T_x M$.  It follows from \eqref{eq.unitXY}, \eqref{eq.Yunit} and \eqref{eq.Xunit} that there exist $\rho_3>0$ and $K_3>1$ such that
\begin{equation}\label{eq.holderproj}
  |\proj_{\H}\circ T_{y,x}(w')-T_{y,x}\circ\proj_{\H}(w')|\le K_3\frac{\dist(x,y)}{|X(x)|}
\end{equation}
for any $x,y\in M$ with $\dist(x,y)\le {\rho}_{3}|Y(x)|$ and $w'\in T_y M$ with $|w'|\le K_1$.

 By the Gronwall inequality (see, e.g., \cite{KSSV}), we have $\dist(\varphi_t(x),\varphi_t(y))\le e^{\norm{\nabla X}|t|}\dist(x,y)$.
Combining with \eqref{eq.holderdphi} and \eqref{eq.holderproj},  there exist $\rho_4>0$ and $K_4>1$ such that
\[
\begin{aligned}
\left|\Phi_t(v)-T_{\varphi_t(y),\varphi_t(x)}\circ\Phi_t(w)\right|
&\le \left|\proj_{\H}\circ d\varphi_t(v)-\proj_{\H}\circ T_{\varphi_t(y),\varphi_t(x)}\circ d\varphi_t(w)\right|\\
&\ \ \ \ \! +\left|\proj_{\H}\circ T_{\varphi_t(y),\varphi_t(x)}\circ d\varphi_t(w)-T_{\varphi_t(y),\varphi_t(x)}\circ\proj_{\H}\circ d\varphi_t(w)\right| \\
&\le \|\proj_{\H}\|\cdot {K}_{1}\left(({\dist(x,y)})^{\beta}+|v-T_{y,x}(w)|\right)+  K_4\frac{\dist(x,y)}{|X(x)|}  
\end{aligned}
\]
for any $x,y\in M\setminus\Sing(X)$ with $\dist(x,y)\le {\rho}_{4}|X(x)|$, any unit vectors $v\in \mathcal{H}_x$, $w\in \mathcal{H}_y$ and any $|t|\leq 1$.

On the other hand, it follows from \eqref{eq.holderdphi} and \eqref{eq.YxandYy} that there exist $\rho_5>0$ and $K_5>1$ such that
\[
\left|\frac{|Y(x)|}{|Y(\varphi_t(x))|}-\frac{|Y(y)|}{|Y(\varphi_t(y))|}\right|\le K_5\frac{\dist(x,y)}{|Y(x)|}
\]
for any $x,y\in M$ with $\dist(x,y)\le {\rho}_{5}|Y(x)|$ and any $|t|\le 1$.

Let $\rho_{F3}=\frac{1}{2}\min\{1,\rho_3,\rho_4,\rho_5\}$. Then by \eqref{eq.ratiosec} and the above estimates, for any $x,y\in M\setminus\Sing(X)$ with $\dist(x,y)\le {\rho}_{F3}|X(x)|$, any unit vectors $v\in \mathcal{H}_x$, $w\in \mathcal{H}_y$ and any $|t|\leq 1$, we have
\[
\begin{aligned}
&\ \ \ \ \!\left|\Phi^*_t(v)-T_{\varphi_t(y),\varphi_t(x)}\circ\Phi^*_t(w)\right|\\
&=\left|\frac{|Y(x)|}{|Y(\varphi_t(x))|}\Phi_t(v)-\frac{|Y(y)|}{|Y(\varphi_t(y))|}T_{\varphi_t(y),\varphi_t(x)}\circ\Phi_t(w)\right|\\
&\le\frac{|Y(x)|}{|Y(\varphi_t(x))|}\cdot \left|\Phi_t(v)-T_{\varphi_t(y),\varphi_t(x)}\circ\Phi_t(w)\right|+\left|\frac{|Y(x)|}{|Y(\varphi_t(x))|}-\frac{|Y(y)|}{|Y(\varphi_t(y))|}\right| \cdot \left|T_{\varphi_t(y),\varphi_t(x)}\circ\Phi_t(w)\right|\\
&\le 4 K_1 \Big(\|\proj_{\H}\|\cdot {K}_{1}\left(({\dist(x,y)})^{\beta}+|v-T_{y,x}(w)|\right)+  K_4\frac{\dist(x,y)}{|X(x)|} \Big)+K_5\frac{\dist(x,y)}{|Y(x)|}\cdot \|\proj_{\H}\|\cdot K_1 \\
&\le {K}_{F3}\bigg(\Big(\frac{\dist(x,y)}{|X(x)|}\Big)^{\beta}+|v-T_{y,x}(w)|\bigg) 
\end{aligned}
\]
for some constant ${K}_{F3}>1$ independent of the choice of $x$, $y$, $v$ and $w$.
This completes the proof.
\end{proof}

%%%%%%%%%%%%%%%%%%%%%%%%%%%%%%%%%%%%%%%%%%%%%%%%%%%%%%%%%%%%%%%%%%%
%%%%%%%%%%%%%%%%%%%%%%%%%%%%%%%%%%%%%%%%%%%%%%%%%%%%%%%%%%%%%%%%%%%

\part{Symbolic dynamics for singular flows}\label{part.coding}

%%%%%%%%%%%%%%%%%%%%%%%%%%%%%%%%%%%%%%%%%%%%%%%%%%%%%%%%%%%%%%%%%%%

\section{Preliminaries}\label{sec.prelim}

In this section, we recall some definitions and related results.

%%%%%%%%%%%%%%%%%%%%%%%%%%%%%%%%%%%%%%%%%%%%%%%%%%%%%%%%%%%%%%%%%%%

\subsection{Topological Markov shifts and Topological Markov flows}\label{subsec.TMSTMF}

\subsubsection{Topological Markov shifts}

Let $\mathscr{G}$ be a directed graph with vertex set $\mathscr{V}$ and edge set $\mathscr{E}=\{u\to v: u,v\in\mathscr{V}\}$. If every vertex has positive in-going degree and out-going degree, then we say $\mathscr{G}$ is {\it proper}. Throughout the paper, $\mathscr{G}$ is always assumed to be proper and have at most countably many vertices.

If every vertex has finite in-going degree and out-going degree, then the graph $\mathscr{G}$ is called {\it locally finite}. A {\it path of length} $n$ from $u$ to $v$ is an $n$-tuple $(u_1,\dots,u_n)\in \mathscr{V}^{n}$ such that $u_1=u$ and $u_n=v$ and $u_i\to u_{i+1}$ for all $1\leq i\leq n-1$. 
The graph $\mathscr{G}$ is called {\it connected} if for every $u,v\in\mathscr{V}$, there is a path from $u$ to $v$. In this case, the {\it period} of $\mathscr{G}$ is the integer ${\rm gcd}\{n:u_1\to\dots\to u_n \text{ with } u_1=u_n=a \text{ and } a\in\mathscr{V}\}$, which is independent of the choice of $a\in\mathscr{V}$.
A {\it maximal connected component} $\mathscr{G}'$ of $\mathscr{G}$ is a connected subgraph of $\mathscr{G}$ such that any other connected subgraph of $\mathscr{G}$ does not contain $\mathscr{G}'$. 

Define the {\it symbolic space} as
\[
\Sigma=\Sigma(\mathscr{G}):=\{\ul{x}=\{x_n\}_{n\in\mathbb{Z}}\in \mathscr{V}^{\mathbb{Z}}: x_n\to x_{n+1}\text{ for all $n\in\mathbb{Z}$}\},
\]
endowed with the metric $\dist(\ul{x},\ul{y}):=e^{-\min\{|n|\in\mathbb{Z}:x_n\neq y_n\}}$, where $\min \emptyset:=+\infty$.
Denote by $\sigma:\Sigma\to\Sigma$ the standard {\it left shift} defined by $(\sigma(\underline{x}))_n = x_{n+1}$ for all $n\in\mathbb{Z}$. The pair $(\Sigma,\sigma)$ is called a {\it topological Markov shift} (TMS). 
We say that a TMS is {\it locally compact} if $\mathscr{V}$ is locally finite. If $\mathscr{V}$ is finite then $\Sigma$ is compact.

A path $(u_0,\dots,u_n)$ is also called an {\it admissible word} in $\Sigma$. A {\it cylinder} is defined as
 \[
 {}_m[u_0,\dots,u_n]:=\{\ul{x}\in\Sigma: x_{i+m}=u_i, 0\leq i\leq n\},
 \]
where $(u_0,\dots,u_n)$ is an  admissible word. We also write $[u_0,\dots,u_n]:={}_0[u_0,\dots,u_n]$ for brevity.
 The {\it regular set} of $\Sigma$ is 
\[
\Sigma^{\#}:=\{\ul{x}\in\Sigma:\exists u,v\in \mathscr{V}, x_n=u \text{ for infinitely many $n\in\mathbb{Z}^{+}$ and } x_n=v\text{ for infinitely many $n\in\mathbb{Z}^{-}$}\}.
\]
By Poincar\'e Recurrence Theorem, $\mu(\Sigma^{\#})=1$ for any $\sigma$-invariant probability measure $\mu$.

A TMS $(\Sigma,\sigma)$ is called {\it irreducible} if the graph $\mathscr{G}$ is connected. A subshift  $\Sigma'\subset\Sigma$ generated by a proper, maximal connected component $\mathscr{G}'$ of $\mathscr{G}$ is called an {\it irreducible component} of $\Sigma$. In this case $(\Sigma',\sigma|_{\Sigma'})$ is an irreducible TMS. 

The {\it period} of an irreducible TMS is the period of $\mathscr{G}$. An irreducible TMS with period $1$ is called {\it irreducible and aperiodic}. 
We have the following  relations: 
\begin{itemize}
  \item $\sigma:\Sigma\to\Sigma$ is topologically transitive if and only if $\Sigma$ is irreducible;
  \item $\sigma:\Sigma\to\Sigma$ is topologically mixing if and only if $\Sigma$ is irreducible and aperiodic.
\end{itemize}

 The following spectral decomposition of an irreducible TMS is well known (see, e.g. \cite{BCS25}).

\begin{lem}\label{lem.spectdecomp}
An irreducible TMS $(\Sigma,\sigma)$ with period $\hbar>1$ can be decomposed into $\Sigma=\biguplus_{i=0}^{\hbar-1}\Sigma(i)$ such that each $\Sigma(i)$ is closed and $\sigma(\Sigma(i))=\Sigma(i+1)$, where $\Sigma(\hbar)=\Sigma(0)$. Moreover, for each $0\leq i\leq \hbar-1$, $\sigma^{\hbar}:\Sigma(i)\to\Sigma(i)$ is topologically conjugate to a topologically mixing TMS.
\end{lem}

Denote by $\mathbb{P}(\Sigma)$ the set of $\sigma$-invariant probability measures, and by $\Pe(\Sigma)$ the set of ergodic $\sigma$-invariant probability measures. Both of them are endowed with the standard weak-$*$ topology.

Let $h_{\mu}(\sigma)$ be the {\it metric entropy} of $\mu\in\mathbb{P}(\Sigma)$. The {\it Gurevich entropy} of $(\Sigma,\sigma)$ is defined as
\[
h_{\rm TOP}(\sigma):=\sup\{h_{\mu}(\sigma):\mu\in \mathbb{P}(\Sigma)\}.
\]

For a H\"older continuous potential $\vartheta$, the {\it top pressure} corresponding to $\vartheta$ is defined as 
\[
P_{\rm TOP}(\sigma,\vartheta):=\sup\{P_{\mu}(\sigma,\vartheta):\mu\in \mathbb{P}(\Sigma)\},
\]
where $P_{\mu}(\sigma,\vartheta):=h_{\mu}(\sigma)+\int_{\Sigma} \vartheta d\mu$ is the {\it metric pressure} of the $\sigma$-invariant probability measure $\mu$. 

If the supremum is achieved by some probability measure $\mu$, namely, $P_{\mu}(\sigma,\vartheta)=P_{\rm TOP}(\sigma,\vartheta)$, then $\mu$ is called an {\it equilibrium state} of $\vartheta$.
In particular, if $\vartheta\equiv 0$ then the top pressure is indeed the Gurevich entropy $h_{\rm TOP}(\sigma)$ of this TMS. In this case, the corresponding equilibrium state for $\vartheta\equiv 0$ is called a {\it measure of maximal entropy}. 
See \cite{Sarig1999} for more properties on the thermodynamic formalism of TMS.

\subsubsection{Topological Markov flows}\label{sssec.tmf} 

Let $(\Sigma,\sigma)$ be a TMS.
A {\it roof function} $r$ is a  bounded positive continuous function defined on 
$\Sigma$. For $ n \ge 1 $, the $ n $-th {\it Birkhoff sum} of $ r $ is defined as
$r_n = r + r \circ \sigma + \cdots + r \circ \sigma^{n-1}$. We also define $r_0=0$ and $r_{-n}:=-r_n \circ \sigma^{-n} $ for any $n>0$.

In what follows, we assume the roof function $r$ to be H\"older continuous and uniformly bounded away from zero.
Define 
\[
\Sigma_r:=\{(\ul{x},t): \ul{x}\in\Sigma\text{ and } t\in[0,r(\ul{x}))\},
\]
and the flow $\sigma_r$ on $\Sigma_r$ acts as
\[
\sigma_{r,t}(\ul{x},s):=(\sigma^n(\ul{x}),s+t-r_n(\ul{x})),
\]
where $n$ is the unique integer satisfying $r_n(\ul{x})\leq s+t<r_{n+1}(\ul{x})$.

 The {\it Bowen-Walters metric} $ \dist_r(\cdot, \cdot) $ endows $ \Sigma_r $ with a topology such that $ \sigma_r $ constitutes a continuous flow \cite{BW}.  
The pair $(\Sigma_r,\sigma_r)$ is called a {\it topological Markov flow} (TMF). The {\it  regular set} of $\Sigma_r$ is 
$$
\Sigma_r^{\#}: = \{(\ul{v}, t) \in \Sigma_r : \ul{v} \in \Sigma^\#\}.
$$

We say that $\Sigma'_r\subset\Sigma_r $ is an {\it irreducible component}, or $\Sigma'_r$ is {\it irreducible}, if there is an  irreducible component $ \Sigma' $ of $ \Sigma $ such that  $\Sigma'_r=\{ (\ul{v}, t) \in \Sigma_r: \ul{v}\in \Sigma'\}$.
 A TMF $(\Sigma_r,\sigma_r)$ is topologically transitive if and only if the base TMS $(\Sigma,\sigma)$ is topologically transitive. However, even though $\sigma$ is topologically mixing, it is possible that $\sigma_r$ is not. By Lemma~\ref{lem.spectdecomp}, every TMF with topologically transitive base can be recoded as one with topologically mixing base. We just need to replace $\Sigma$ by $\Sigma_0$ and $r$ by $r_p$ (the $p$-th Birkhoff sum of $r$).
 
 Denote by $\mathbb{P}(\Sigma_r)$ and $\Pe(\Sigma_r)$ the set of all $\sigma_r$-invariant probability measures and ergodic probability measures, respectively. Both of them are endowed with the standard weak-$*$ topology.
 
Each $\mu_{_\Sigma}\in\mathbb{P}(\Sigma)$  induces a $\sigma_r$-invariant measure $\mu$ on $\Sigma_r$ defined as
\begin{equation}\label{eq.measproj}
\mu:=\frac{\int_{_\Sigma}\int_0^{r(\ul{x})}\delta_{(\ul{x},t)}dtd\mu_{_\Sigma}(\ul{x}) }{\int_{_\Sigma} rd\mu_{_\Sigma}},
\end{equation}
where $\delta_{(\ul{x},t)}$ denotes the Dirac measure at $(\ul{x},t)$.
It is shown that if the roof function satisfies $\inf r>0$, then the above equation gives a bijection between the space of $\sigma$-invariant measures and the space of $\sigma_r$-invariant measures. 
In addition, $\mu$ is ergodic if and only if $\mu_{_\Sigma}$ is ergodic. We say that $\mu_{_\Sigma}$ is the {\it projection} of $\mu$.
 
The {\it  metric entropy} $ h_{\mu}(\sigma_r)$ for a TMF $\sigma_r$ of $\mu\in\mathbb{P}(\Sigma_r)$ is defined as the metric entropy of the time-1 map of $\sigma_r$.
It relates to the metric entropy of the projection $\mu_{_\Sigma}\in\mathbb{P}(\Sigma)$ via Abramov's formula \cite{Abr59}:
\begin{equation}\label{eq.Abramov}
  h_{\mu}(\sigma_r)=\frac{h_{\mu_{_\Sigma}}(\sigma)}{\int_{_\Sigma} r d\mu_{_\Sigma}}.
\end{equation}

Given a H\"older continuous potential $\vartheta$ and a $\sigma_r$-invariant probability measure $\mu$, the {\it metric pressure} of $\mu$ is defined as $P_{\mu}(\sigma_r,\vartheta):=h_{\mu}(\sigma_r)+\int_{\Sigma_r} \vartheta d\mu$. The {\it top pressure}  is  $P_{\rm TOP}(\sigma_r,\vartheta):=\sup\{P_{\mu}(\sigma_r,\vartheta):\mu\in \mathbb{P}(\Sigma_r)\}
$. If $P_{\mu}(\sigma_r,\vartheta)=P_{\rm TOP}(\sigma_r,\vartheta)$, then $\mu$ is called an {\it equilibrium state} of $\vartheta$. Moreover, if $\vartheta\equiv0$ then the top pressure is  the {\it Gurevich entropy} $h_{\rm TOP}(\sigma_r):=\sup\{h_\mu(\sigma_r):\mu\in \mathbb{P}(\Sigma_r)\}$ of the TMF and the corresponding equilibrium state (if exists) is called a {\it measure of maximal entropy}.

For any H\"older continuous function $\vartheta:\Sigma_r\to\mathbb{R}$, define 
\begin{equation}\label{eq.potential}
\Delta_\vartheta(\ul{u}):=\int_0^{r(\ul{u})}\vartheta(\ul{u},s)ds.
\end{equation} 
Then $\Delta_\vartheta:\Sigma\to\mathbb{R}$ is also a H\"older continuous function \cite{BS00}.
We have the following  result.
\begin{pro}\cite{BI2006, IV2021}\label{pro.relationmme}
Let $(\Sigma_r,\sigma_r)$ be a TMF and $\vartheta$ be a H\"older continuous potential on $\Sigma_r$ with finite top pressure. Then an invariant measure $\mu$ on $\Sigma_r$ is an equilibrium state of  $\vartheta$ if and only if its projection $\mu_{_\Sigma}$ on $\Sigma$ is an equilibrium state of the potential $\Delta_\vartheta-P_{\rm TOP}(\sigma_r,\vartheta)\cdot r$.   
\end{pro}

%%%%%%%%%%%%%%%%%%%%%%%%%%%%%%%%%%%%%%%%%%%%%%%%%%%%%%%%%%%%%%%%%%%

\subsection{Lyapunov exponents and hyperbolic measures}\label{subsec.Lyapunov}

Let $X$ be a $C^1$ vector field on $M$.
Recall that $\mathbb P(X)$ denotes the set of $X$-invariant Borel probability measures, and $\Pe(X)$ denotes the set of ergodic $X$-invariant Borel probability measures. A measure $\mu\in\mathbb P(X)$ is said to be regular if $\mu(\Sing(X))=0$.

According to Oseledec Theorem \cite{Ose}, there exists an invariant measurable set 
$O(X)\subset M \setminus\Sing(X)$ of full $\mu$-measure for every regular 
 $X$-invariant measure $\mu$ satisfying: for any $x\in O(X)$, there is an invariant splitting for the tangent flow $d\varphi$:
\[
T_x M=E^1_x\oplus \cdots\oplus E^{k(x)}_x,
\] 
along with distinct numbers $\chi_1(x)<\cdots<\chi_{k(x)}(x)$ and positive integers $d_1(x),\dots,d_{k(x)}(x)$
such that $\dim E^j_x=d_j(x)$ and
\[
\lim_{t\to \pm\infty}\frac{1}{t}\log|d\varphi_t(v)|=\chi_j(x)
\]
holds for all $v\in E^j_x\setminus\{0\}$, $j=1,\ldots,k(x)$.

A point $x\in O(X)$ is said to be an  {\it Oseledec regular point} for $X$. The numbers $\chi_1(x),\ldots,\chi_{k(x)}(x)$ are called the {\it Lyapunov exponents} of $d\varphi$ at $x$, and $d_j(x)$ is called the {\it multiplicity} of $\chi_j(x)$.
 All  functions $k$, $\chi_j$ and $d_j$ are  measurable in $x$ and invariant along the orbit of $x$. Moreover, if $\mu$ is ergodic, then  $k$, $\chi_j$ and $d_j$ are constants $\mu$-almost everywhere.

By Poincar\'e Recurrence Theorem, we can replace $O(X)$ with a full-measure subset such that every point in it is recurrent. Thus,
\begin{equation}\label{eq.zeroexp}
\lim_{t\to \pm\infty}\frac{1}{t}\log \|d\varphi_t|_{\langle X(x)\rangle}\|=0
\end{equation}
for every $x\in O(X)$, where $\langle X(x)\rangle$ denotes the one-dimensional linear space spanned by $X(x)$. We may also consider the Lyapunov exponents of the linear Poincar\'e flow $\psi$. For any $x\in O(X)$, the Lyapunov exponents of $\psi$ at $x$ (counting multiplicity) coincide with those of $d\varphi$ except the zero exponent corresponding to the flow direction.

Furthermore, since
\[
\psi^*_t(v)=\frac{|X(x)|}{|X(\varphi_t(x))|}\psi_t(v)=\frac{\psi_t(v)}{\|d\varphi_t|_{\langle X(x)\rangle}\|},
\]
it follows from \eqref{eq.zeroexp} that $\lim_{t\to \pm\infty}\frac{1}{t}\log|\psi^*_t(v)|=\lim_{t\to \pm\infty}\frac{1}{t}\log|\psi_t(v)|$ for any $v\in \N_x\setminus\{0\}$, $x\in O(X)$. Thus $\psi$ and $\psi^*$ have the same Lyapunov exponents at $x\in O(X)$ and corresponding splitting on $\N_x$.

A regular $X$-invariant measure $\mu$ is said to be {\it hyperbolic} if, for $\mu$-a.e. $x$, 
all Lyapunov exponents of the scaled linear Poincar\'e flow $\psi^*$ at $x$ are non-zero. This is equivalent to requiring that the zero Lyapunov exponent of the tangent flow $d\varphi$ at $x$ has multiplicity exactly one, which corresponds to the flow direction. For any $\chi>0$, a regular $X$-invariant measure is said to be {\it $\chi$-hyperbolic} if, for $\mu$-a.e. $x$, all Lyapunov exponents of $\psi^*$ at $x$ lie outside the interval $[-\chi,\chi]$. We say that a regular invariant measure $\mu$ is a {\it saddle-type} hyperbolic measure if it is hyperbolic, and it admits at least one positive and at least one negative Lyapunov exponent at almost every point.

The Lyapunov exponents of the induced linear Poincar\'e flows $\Phi$ and the scaled induced linear Poincar\'e flow $\Phi^*$ can be defined in the same way. We have the following conclusion.

\begin{lem}\label{lem.LyapexpforPhi}
For any $x\in O(X)$ and $w\in \H_x$,  we have
\[
\lim_{t\to \pm\infty}\frac{1}{t}\log|\Phi^*_t(w)|=\lim_{t\to \pm\infty}\frac{1}{t}\log|\Phi_t(w)|=\lim_{t\to \pm\infty}\frac{1}{t}\log|\psi_t(v)|=\lim_{t\to \pm\infty}\frac{1}{t}\log|\psi^*_t(v)|,
\]
where $v=\proj_{\N_x}(w)$ is the orthogonal projection of $w$ onto $\N_x$.
\end{lem}
\begin{proof}
Since $\|\proj_{\N}|_{_{\H}}\|$ and $\|\proj_{\H}|_{_{\N}}\|$ are uniformly  bounded on $M\setminus\Sing(X)$ (see \S\,\ref{subsubsec.inducedpoincare}), we have 
$\lim_{t\to \pm\infty}\frac{1}{t}\log|\Phi_t(w)|=\lim_{t\to \pm\infty}\frac{1}{t}\log|\psi_t(v)|$.
Then by \eqref{eq.ratiosec}, \eqref{def.sf} and \eqref{eq.zeroexp},
 $\lim_{t\to \pm\infty}\frac{1}{t}\log|\Phi^*_t(w)|=\lim_{t\to \pm\infty}\frac{1}{t}\log|\Phi_t(w)|$.
\end{proof}

%%%%%%%%%%%%%%%%%%%%%%%%%%%%%%%%%%%%%%%%%%%%%%%%%%%%%%%%%%%%%%%%%%%

\subsection{Pesin blocks and invariant manifolds}\label{subsec.Pesinblock}

By \eqref{eq.zeroexp}, the tangent flow $d\varphi$ typically lacks hyperbolicity along the flow direction.  It is therefore more natural to define Pesin blocks for flows using the   Poincar\'e flow. To achieve uniform estimates in singular flows, we employ the scaled linear  Poincar\'e flow.
 
\begin{dfn}\label{def.pesin}
Given $\chi>0$, $0<\varepsilon\ll \chi$ and $\ell\ge 1$, a {\rm $(\chi,\varepsilon)$-Pesin block} $\Pes_\chi(\ell,\varepsilon)=\Pes^X_\chi(\ell,\varepsilon)$ is the set of all points $x\in M\setminus \Sing(X)$ such that there exists a splitting $\mathcal{N}_x=\N^s_x\oplus \N^u_x$ satisfying
\begin{itemize}
\item $\|\psi_n^*|_{\psi_m^*(\N^s_{x})}\|\le \ell e^{-\chi n}e^{\varepsilon|m|}$ for any $m, n\in\mathbb{Z}$, $n\ge 0$;
\item $\|\psi_{-n}^*|_{\psi_m^*(\N^u_{x})}\|\le \ell e^{-\chi n}e^{\varepsilon|m|}$ for any $m, n\in\mathbb{Z}$, $n\ge 0$;
\item $\angle (\psi_m^*(\N^s_{x}),\psi_m^*(\N^u_{x}))\ge \ell^{-1} e^{-\varepsilon|m|}$ for any $m\in \bZ$.
\end{itemize}
\end{dfn}

The subspaces  $\N^s_x$ and $\N^u_x$ are continuous and invariant on $\Pes_\chi(\ell,\varepsilon)$ \cite{BP07}. Here the invariance means that if $x, \varphi_n(x)\in \Pes_\chi(\ell,\varepsilon)$, then $\psi_n^*(\N^{s/u}_x)=\N^{s/u}_{\varphi_n(x)}$. For any $0<\chi'\le \chi$, $\ell'\ge\ell$ and $\varepsilon'\ge \varepsilon$, we have $\Pes_{\chi}(\ell,\varepsilon)\subset \Pes_{\chi'}(\ell',\varepsilon')$.

Typically, the Pesin block is not invariant for the flow. Moreover, since singularities may exist within its closure, the Pesin block is generally non-compact\footnote{In fact, the {\it extended linear Poincar\'e flow} introduced in \cite{LGW} allows us to lift Pesin blocks to compact {\it extended Pesin blocks} on the one-dimensional Grassmannian manifold. See \cite[Lemma~4.4]{LLL2024} for the details.}. However, applying a classical argument, it follows that $\ol{\Pes_\chi(\ell,\varepsilon)}\setminus\Pes_\chi(\ell,\varepsilon)\subset\Sing(X)$ (see Lemma~\ref{lem.cptofpesinblock} for a more general version).

We define 
\[
\Pes_\chi:=\bigcap_{\varepsilon>0}\bigcup_{\ell>0}\Pes_\chi(\ell,\varepsilon) \text{\ and \ } \Pes:=\bigcup_{\chi>0} \Pes_\chi.
\]
Following the same proof as in \cite{Pes76}, we  have $\mu(\Pes)=1$ for any regular hyperbolic $X$-invariant measure $\mu$, and $\nu(\Pes_\chi)=1$ for   any regular $\chi$-hyperbolic $X$-invariant measure $\nu$.

Denote
\[
\Pes':=\bigcup_{\chi>0}\bigcap_{\varepsilon>0}\bigcup_{\ell>0}\{x\in M\setminus\Sing(X): \text{there are periodic points } y_n\in  \Pes_\chi(\ell,\varepsilon) \text{ such that } y_n\to x\}.
\]

If $X$ is a $C^{1+\beta}$ vector field ($\beta>0$) on $M$, then by using the closing lemma for singular flows \cite{LLL2024} and the same argument as in \cite{Kat80}, we may obtain  $\mu(\Pes')=1$ for any regular hyperbolic $X$-invariant measure $\mu$. 
According to Pesin's stable manifold theorem, every $x\in \Pes$ admits {\it (weak) stable {\rm and} (weak) unstable manifolds} 
\[
\begin{aligned}
	W^s(x)&:=\{y\in M: \exists \text{ an increasing homeomorphism } h: \mathbb{R}\to\mathbb{R} \text{ satisfying } \lim_{t\to+\infty}\dist(\varphi_t(x),\varphi_{h(t)}(y))=0\},\\
	W^u(x)&:=\{y\in M: \exists \text{ an increasing homeomorphism } h: \mathbb{R}\to\mathbb{R} \text{ satisfying } \lim_{t\to-\infty}\dist(\varphi_t(x),\varphi_{h(t)}(y))=0\}.
\end{aligned}
\]

%%%%%%%%%%%%%%%%%%%%%%%%%%%%%%%%%%%%%%%%%%%%%%%%%%%%%%%%%%%%%%%%%%%%%%%%%%%%%%%%%

\subsection{Homoclinic classes and Borel homoclinic classes}\label{subsec.BHC}

Assume that $X$ is a $C^{1+\beta}$ vector field ($\beta>0$) on $M$. Let $x,y\in \Pes$.
The  {\it transverse intersection} of $W^u(x)$ and $W^s(y)$ is the set
\[
W^u(x)\pitchfork W^s(y):=\{z\in W^u(x)\cap W^s(y):T_zM= T_{z}W^u(x) + T_{z}W^s(y)\}.
\]

We say that $x$ and $y$ are {\it homoclinically related}, denoted by $x\sim y$, if $W^u(x)\pitchfork W^s(y)\neq\emptyset$ and $W^u(y)\pitchfork W^s(x)\neq\emptyset$. The homoclinic relation is an equivalence relation on $\Pes'$ (see \cite[Proposition~2.6]{BCS25}).  Each equivalence class of the homoclinic relation $\sim$ is called a {\it Borel homoclinic class}. This notion was first introduced by Buzzi-Crovisier-Sarig for diffeomorphisms in \cite{BCS25}.

The following  properties of Borel homoclinic classes are proved in \cite{BCS25} for diffeomorphisms. The proof of the flow case is similar.
\begin{pro}\label{pro.proBHC}{\rm \cite[Proposition 2.8]{BCS25}}
The following statements hold:
\begin{enumerate}
\item The set of Borel homoclinic classes is a finite or countable partition of $\Pes'$ into $X$-invariant Borel sets.
\item Each Borel homoclinic class contains a dense set of hyperbolic periodic points.
\item Any regular hyperbolic ergodic measure is carried by a Borel homoclinic class.
\item Any measure carried by a Borel homoclinic class is regular and hyperbolic.
\end{enumerate}
\end{pro}

Let $p$ be a periodic point. The point $q\in W^u(p)\pitchfork W^s(p)$ is called a {\it transverse homoclinic point} of (the orbit of) $p$. Denote by $\BHC(p)$ the Borel homoclinic class containing $p$. The classical {\it homoclinic class} $\HC(p)$ of $p$ defined by Newhouse \cite{New72} (which is called {\it topological homoclinic class} in \cite{BCS22}) is the set
\[
\HC({p}):=\ol{\{ q: q \text{ is a hyperbolic periodic point such that } q\sim p \}}.
\]

Generally, the homoclinic class and the Borel homoclinic class of $p$ do not coincide.
Nevertheless, the identity $\HC(p)=\ol{\BHC(p)}$ holds.

The homoclinic relation can also be extended to regular hyperbolic measures. 
Precisely, let $\mu$, $\nu$ be two regular hyperbolic $X$-invariant measures. We say that $\mu$ is {\it homoclinically related} to $\nu$ (denoted by $\mu \sim \nu$) if $x\sim y$ for $\mu$-a.e. $x$ and $\nu$-a.e. $y$. It follows from Proposition~\ref{pro.proBHC} that any two measures carried by a common Borel homoclinic class are homoclinically related, and any two homoclinically related ergodic measures are carried by the same Borel homoclinic class.

%%%%%%%%%%%%%%%%%%%%%%%%%%%%%%%%%%%%%%%%%%%%%%%%%%%%%%%%%%%

\subsection{Equilibrium states on invariant sets}

Let $B$ be a measurable invariant set for the flow. Denote by  $\mathbb P(X|_B)$ (resp. $\Pe(X|_B)$) the set of $X$-invariant (resp. $X$-ergodic) probability measures $\mu$ such that $\mu(B)=1$.
 
  The {\it top entropy} for $X$ on $B$ is defined as
 \[
 h_{\rm TOP}(X|_B)=\sup\{h_\mu(X): \mu\in \mathbb P(X|_B)\}.
 \]
 According to the variational principle (see, e.g., \cite{Wal}), when $B$ is compact, this top entropy equals the topological entropy $h_{\rm top}(X|_B)$ for $X$ on $B$. A {\it measure of maximal entropy} ({\it MME}) for $X$  on $B$ is an $X$-invariant probability measure $\mu$ with $\mu(B)=1$ such that $h_\mu(X)= h_{\rm TOP}(X|_B)$. 
 
 Let $\vartheta$ be a potential  on $M$. The {\it top pressure} of $\vartheta$ for $X$ on $B$ is defined as
 \[
 P_{\rm TOP}(X|_B,\vartheta)=\sup\{P_{\mu}(X,\vartheta): \mu\in \mathbb P(X|_B)\}.
 \]
Similarly, according to the variational principle, if $B$ is compact and $\vartheta$ is continuous, then the top pressure coincides with the topological pressure $P_{\rm top}(X|_B,\vartheta)$ for $X$ on $B$. An {\it equilibrium state} of $\vartheta$ for $X$  on $B$ is a measure $\mu\in\mathbb P(X|_B)$ such that $P_{\mu}(X,\vartheta)= P_{\rm TOP}(X|_B,\vartheta)$.

%%%%%%%%%%%%%%%%%%%%%%%%%%%%%%%%%%%%%%%%%%%%%%%%%%%%%%%%%%%%%

\subsection{Strong positive recurrence}

Let $(\Sigma, \sigma)$ be a TMS. The property of strongly positive recurrence (SPR) for countable TMS has several equivalent characterizations. In this work, we adopt the extended definition introduced in \cite{BCS25}.

We say that  $(\Sigma, \sigma)$ is {\it  SPR for a H\"older continuous potential $\vartheta$} if it has finite Gurevich entropy and there exist a finite union of cylinders $\mathbf C$ and numbers $P_0<P_{\rm TOP}(\sigma,\vartheta)$, $\tau>0$ such that  $\nu(\mathbf C)>\tau$ holds for every ergodic $\sigma$-invariant probability measure $\nu$ with $P_{\nu}(\sigma,\vartheta)>P_0$.
In particular, we say $(\Sigma, \sigma)$ is {\it SPR} if it is SPR for $\vartheta\equiv 0$.

Let $(\Sigma_r,\sigma_r)$ be a TMF generated by a TMS $(\Sigma,\sigma)$ and $r$ be a roof function. Analogous to the definition of SPR property for TMS,
we say that  $(\Sigma_r,\sigma_r)$ is {\it  SPR for a H\"older continuous potential $\vartheta$} if it has finite Gurevich entropy, and there exists a finite union  of cylinders $\mathbf C$ in $\Sigma$ together with numbers $P_0<P_{\rm TOP}(\sigma_r,\vartheta)$, $\tau>0$ such that 
\[
\nu (\{(\ul{x},t)\in \Sigma_r:\ul{x}\in \mathbf C,\ t\in[0,r(\ul{x}))\})>\tau
\]
holds for every ergodic $\sigma_r$-invariant probability measure $\nu$ with $P_{\nu}(\sigma_r,\vartheta)>P_0$.
We also say $(\Sigma_r, \sigma_r)$ is {\it SPR} if it is SPR for $\vartheta\equiv 0$.

The SPR property of TMF $(\Sigma_r,\sigma_r)$ implies that of the corresponding TMS $(\Sigma,\sigma)$. 

\begin{lem}\label{lem.relCT}
If a TMF $(\Sigma_r,\sigma_r)$ is SPR for a  H\"older continuous potential $\vartheta$, then the corresponding TMS $(\Sigma,\sigma)$ is SPR for the  H\"older continuous potential $\Delta_\vartheta -P_{\rm TOP}(\sigma_r,\vartheta)\cdot r$, where $\Delta_\vartheta$ is defined in \eqref{eq.potential}.
\end{lem}
\begin{proof}
Let the finite union of cylinders  $\mathbf C$, numbers $P_0<P_{\rm TOP}(\sigma_r,\vartheta)$ and $\tau>0$ be given by the definition of the SPR property for $(\Sigma_r,\sigma_r)$. Denote $\vartheta_{_\Sigma}:=\Delta_\vartheta -P_{\rm TOP}(\sigma_r,\vartheta)\cdot r$, $P_{0,\Sigma}:=\inf  r\cdot (P_0-P_{\rm TOP}(\sigma_r,\vartheta))<0$ and $\tau_{_\Sigma}:=\tau\cdot\inf r/\sup r$. For any ergodic $\sigma$-invariant measure $\nu$ satisfying $P_{\nu }(\sigma, \vartheta_{_\Sigma})>P_{0,\Sigma}$, define
\begin{equation}\label{eq.pressure}
\mu:=\frac{\int_{_\Sigma}\int_0^{r(\ul{x})}\delta_{(\ul{x},t)}dtd\nu(\ul{x}) }{\int_{_\Sigma} rd\nu }.
\end{equation}
Then $\mu$ is an ergodic $\sigma_r$-invariant measure on $\Sigma_r$. By \eqref{eq.Abramov} and \eqref{eq.potential}, we have
\[
  P_\mu(\sigma_r,\vartheta)=\frac{P_{\nu }(\sigma, \vartheta_{_\Sigma})}{\int_{_\Sigma} rd\nu }+P_{\rm TOP}(\sigma_r,\vartheta).
\]

This implies that $P_{\rm TOP}(\sigma, \vartheta_{_\Sigma})=0$ and $P_\mu(\sigma_r,\vartheta)>P_0$. It follows from the SPR property for $(\Sigma_r,\sigma_r)$ that 
\[
\mu (\{(\ul{x},t)\in \Sigma_r:\ul{x}\in \mathbf C,\ t\in[0,r(\ul{x}))\})>\tau.
\]
By \eqref{eq.pressure}, we obtain $\nu(\mathbf C)>\tau_{_\Sigma}$. Thus $(\Sigma,\sigma)$ is SPR for $\vartheta_{_\Sigma}$.
\end{proof}

We now turn to singular flows. Let $X$ be a $C^1$ vector field on $M$ and $B\subset M$ be an invariant Borel set for $X$ which is not necessarily compact. We say that $X$ (or the generated flow $\varphi$) is {\it SPR on $B$ for a H\"older continuous potential $\vartheta :M\to\mathbb{R}$} if there exists $\chi>0$ satisfying the following property:  for each $\varepsilon>0$, there are a $(\chi,\varepsilon)$-Pesin block $\Pes_\chi(\ell,\varepsilon)$, a compact set $\Lambda\subset \Pes_\chi(\ell,\varepsilon)$ and numbers $P_0<P_{\rm TOP}(X|_B,\vartheta)$, $\tau>0$ such that $\nu(\Lambda)>\tau$ holds for every  $\nu\in\Pe(X|_B)$ with $P_{\nu}(X,\vartheta)>P_0$. 
In particular, if $B=M$ then we say that $X$ is {\it SPR  for $\vartheta :M\to\mathbb{R}$}  for simplicity. We also say that $X$ is {\it SPR on $B$} if it is SPR on $B$ for  $\vartheta \equiv0$.

%%%%%%%%%%%%%%%%%%%%%%%%%%%%%%%%%%%%%%%%%%%%%%%

%%%%%%%%%%%%%%%%%%%%%%%%%%%%%%%%%%%%%%%%%%%%%%%%%%%%%%%%%%%%%%%%%%%

\section{Non-uniformly hyperbolic locus}\label{sec.nuh}

From this point on, we always assume that $X$ is a $C^{1+\beta}$ vector field ($\beta>0$) on $M$. By replacing $X$ with $cX$ for some positive constant $c$ if necessary, we assume without loss of generality that $\|\nabla X\|\le 1/2$.
Then by Gronwall inequality, we have $\|d\varphi_t\|\le e^{\frac{|t|}{2}}$ for all $t\in \mathbb R$. This implies that all Lyapunov exponents of $d\varphi$ at every Oseledec regular point lie in the interval $[-1/2,1/2]$. On the other hand, it follows from \eqref{eq.ratiosec},  \eqref{eq.projH}, and \eqref{def.sf} that for all $t\in\mathbb{R}$,
\begin{equation}\label{eq.flowGI}
\|\psi_t\|\le e^{\frac{|t|}{2}},\ \|\Phi_t\|\le \|\proj_{\mathcal{H}}\|\cdot \|d\varphi_t\|\le e^{\frac{\rho+|t|}{2}},\ \|\psi^{*}_t\|\le \|\psi_t\| \cdot \|d\varphi_{-t}\|\le e^{|t|}, \text{ and }   \|\Phi^{*}_t\|\le e^{\rho+|t|}.
\end{equation}

For any positive numbers $a,b,\varepsilon$,
we adopt the notation $a \wedge b :=\min\{a,b\}$, and write $a=b e^{\pm \varepsilon}$ to indicate that $\frac{a}{b}\in [e^{-\varepsilon},e^\varepsilon]$.
We now fix $\mathbb D$, $\rho$, $\gamma_0$,  $\Phi^*$ as in \S\,\ref{sec.linearpoincare}. 

\subsection{Non-uniformly hyperbolic locus}

We  focus on the following set associated with non-uniform hyperbolicity. Note that we consider here the hyperbolicity of the scaled induced linear Poincar\'e flow $\Phi^*$, which is comparable to the hyperbolicity of the scaled linear Poincar\'e flow $\psi^*$.

\setlist[enumerate]{label={\rm (NUH\arabic*)}}
\begin{dfn}\label{dfn.NUH}
Let $\chi>0$. The {\rm non-uniformly hyperbolic locus} of $X$ associated with $\chi$, denoted by $\NUH_{\chi}$, is the set of all points $x\in M\setminus\Sing(X)$ admitting a (unique) splitting $\H_x=\H^{s}_x\oplus \H^{u}_x$ such that:
\begin{enumerate}
\item\label{ite.NUH1} for any $v^s\in\H^{s}_x$, $\liminf
_{t\to+\infty}\frac{1}{t}\log|\Phi^{*}_{-t}(v^s)|>0$ and $\limsup_{t\to+\infty}\frac{1}{t}\log|\Phi^{*}_{t}(v^s)|\le -\chi$;
\item\label{ite.NUH2} for any $v^u\in\H^{u}_x$, $\liminf_{t\to+\infty}\frac{1}{t}\log|\Phi^{*}_{t}(v^u)|>0$ and $\limsup_{t\to+\infty}\frac{1}{t}\log|\Phi^{*}_{-t}(v^u)|\le -\chi$;
\item\label{ite.NUH3} $s_\chi(x):=\sup_{v^s\in \H_x^s, |v^s|=1}s_\chi(x,v^s)<+\infty$ and $u_\chi(x):=\sup_{v^u\in \H_x^u, |v^u|=1}u_\chi(x,v^u)<+\infty$, where
\[
 s_\chi(x,v^s):=2e^{\rho}\Big(\int_0^{\infty}e^{2\chi t}|\Phi^{*}_t(v^s)|^2dt\Big)^{\frac{1}{2}} \quad\rm{and}\quad u_\chi(x,v^u):=2e^{\rho}\Big(\int_0^{\infty}e^{2\chi t}|\Phi^{*}_{-t}(v^u)|^2dt\Big)^{\frac{1}{2}}.
\]
\end{enumerate} 
\end{dfn}
\setlist[enumerate]{label={\rm (\arabic*)}}

\begin{lem}\label{lem.NUHFullMeas} We have the following:
\begin{enumerate}
  \item for any $\chi>0$, $\NUH_\chi$ is an invariant set for the flow;
  \item for any $\chi_1>\chi_2>0$, $\NUH_{\chi_1}\subset \NUH_{\chi_2}$; 
  \item for any regular $\chi$-hyperbolic $X$-invariant measure $\mu$, $\mu(\NUH_\chi)=1$.
\end{enumerate}
\end{lem}
\begin{proof}
Items~(1) and (2) follow directly from the definition. For each $x \in O(X)$, let $\mathcal{N}^s_x$ (resp. $\mathcal{N}^u_x$) be the direct sum of all subspaces of $\mathcal{N}_x$ with positive (resp. negative) Lyapunov exponents of $\psi^*$. Define $\mathcal{H}^s_x := \proj_{\mathcal{H}}(\mathcal{N}^s_x)$ and $\mathcal{H}^u_x := \proj_{\mathcal{H}}(\mathcal{N}^u_x)$. By Lemma~\ref{lem.LyapexpforPhi}, 
$$
\lim_{t\to \pm\infty}\frac{1}{t}\log|\Phi^*_t(v^s)| < -\chi \quad \text{and} \quad \lim_{t\to \pm\infty}\frac{1}{t}\log|\Phi^*_t(v^u)| > \chi
$$
for any $v^s \in \mathcal{H}^s_x$ and $v^u \in \mathcal{H}^u_x$. This implies conditions \ref{ite.NUH1}--\ref{ite.NUH3}. Therefore, $O(X)$ is contained in $\NUH_\chi$. 
Since $O(X)$ has full measure for all regular $\chi$-hyperbolic $X$-invariant measures, $\NUH_\chi$ also has full measure. 
This proves Item~(3).
\end{proof}

\begin{rk}\label{rk.NUH}
	By Lemma~\ref{lem.NUHFullMeas}, for any regular $\chi$-hyperbolic $X$-invariant measure $\mu$, the set $\NUH_\chi$ is essentially independent of the specific choice of $\Phi$, modulo a set of $\mu$-measure zero. In fact, even if we replace $\Phi^{*}$ with the scaled linear Poincar\'e flow $\psi^*$ in Definition~\ref{dfn.NUH}, the resulting set also has full $\mu$-measure.
	\end{rk}

%%%%%%%%%%%%%%%%%%%%%%%%%%%%%%%%%%%%%%%%%%%%%%%%%%%%%%%%%%%%%

\subsection{Lyapunov metric for the scaled induced linear Poincar\'e flow $\Phi^*$}

We now introduce a new metric to obtain uniform hyperbolicity for local maps. As mentioned in the introduction, unlike the classical Lyapunov metric in Pesin theory, we construct a Lyapunov metric adapted to the scaled flow $\Phi^{*}$.

For any $x\in \NUH_\chi$, the {\it Lyapunov inner product} $\langle \cdot,\cdot \rangle_\chi'$ on  $\H_x$ is defined as follows:
  \begin{itemize}
  \item  $\langle w^s,v^s\rangle_\chi'=4e^{2\rho}\int_0^{\infty}e^{2\chi t}\la \Phi^*_t(w^s),\Phi^*_t(v^s) \ra dt$, for any $w^s,v^s\in \H^s_x$;
\item $\langle w^u,v^u\rangle_\chi'=4e^{2\rho}\int_0^{\infty}e^{2\chi t}\la \Phi^*_{-t}(w^u),\Phi^*_{-t}(v^u) \ra dt$, for any $w^u,v^u\in \H^u_x$;
\item $\langle v^s,v^u\rangle_\chi'=0$, for any $v^s\in \H^s_x$ and $v^u\in \H^u_x$,
\end{itemize}
where $\langle \cdot,\cdot \rangle$ denotes the Riemannian inner product on $T_x M$.

By  \ref{ite.NUH1}--\ref{ite.NUH3}, these integrals converge, and thus the Lyapunov inner product is well-defined. The {\it Lyapunov norm} $|\cdot|_\chi'$ is the norm induced by the Lyapunov inner product $\langle \cdot,\cdot \rangle_\chi'$.

Analogous to the construction in Pesin theory, for each $x\in\NUH_{\chi}$ there exists a linear isomorphism $C_\chi(x):\mathbb{R}^{d-1}\to \H_x$ that depends measurably on $x$, such that $C_\chi(x)$ is an isometry between $(\mathbb R^{d-1}, \langle\cdot,\cdot\rangle_{\mathbb R^{d-1}} )$ and $(\H_x, \langle\cdot,\cdot\rangle_\chi' )$ satisfying
 $C_\chi(x)(\mathbb{R}^{d^s(x)}\times \{0\})=\H^s_x$ and $C_\chi(x)(\{0\}\times\mathbb{R}^{d^u(x)})=\H^u_x$, where $d^s(x)=\dim \H^s_x$ and $d^u(x)=\dim \H^u_x$. Note that both $d^s(x)$ and $d^u(x)$ are invariant under the flow.

The following lemma shows that $\Phi^*$ exhibits uniform hyperbolicity under the Lyapunov metric. 
This result is analogous to \cite[Lemma 3.2]{LMN}, which generalizes the result in \cite[Lemma 3.2]{BCL} to higher dimensions. The proof follows standard arguments in the framework of classical Pesin theory, and we include a brief sketch here for completeness.

\begin{lem}\label{lem.propofC}
Let $\chi\in (0,1)$. For any $x\in \NUH_{\chi}$, we have:
\begin{enumerate}
\item \label{fC.norm}$\|C_\chi (x)\|\le 1$ and 
\[
\|C_\chi (x)^{-1}\|=\sup_{v\in\H_x,|v|=1}\frac{|v|_\chi'}{|v|}=\sup_{\begin{subarray}{c}v^s\in \H^s_x,v^u\in\H^u_x,\\ |v^s+v^u|\ne 0\end{subarray}} \frac{\big(s_\chi^2(x,v^s)+u_\chi^2(x,v^u)\big)^{\frac{1}{2}} }{|v^s+v^u|};
\]
\item \label{fC.linearize}
the linear map $C_\chi(\varphi_t(x))^{-1}\circ\Phi^*_t\circ C_\chi(x):\bR^{d-1}\to \bR^{d-1}$ can be expressed as the block diagonal matrix  $\left(\begin{matrix}
 A_t(x) & 0 \\
0 & B_t(x)
\end{matrix}\right)$
  under the direct sum decomposition $\bR^{d-1}=\bR^{d^s(x)}\oplus \bR^{d^u(x)}$, where $A_t(x)\in \GL(d^s(x),\mathbb{R})$ and $B_t(x)\in\GL(d^u(x),\mathbb{R})$ satisfy
\[
e^{-4\rho}<\|A_t(x)\|<e^{-\chi t}\quad \text{and} \quad e^{\chi t}<\|B_t(x)\|<e^{4\rho}, \text{ for all $t\in[0,2\rho]$};
\]
\item \label{fC.Cratio}$\frac{\|C_\chi(\varphi_t(x))^{-1}\|}{\|C_\chi(x)^{-1}\|}=e^{\pm 2(\rho+|t|)}$ for all $t\in\bR$; in particular $\frac{\|C_\chi(\varphi_t(x))^{-1}\|}{\|C_\chi(x)^{-1}\|}=e^{\pm 6\rho}$ for $|t|\leq 2\rho$.
\end{enumerate}
\end{lem}
\begin{proof} By \eqref{eq.flowGI}, for any vector $v^{s/u}\in\H^{s/u}_x$, we have
\[
(|v^{s/u}|_\chi')^2=4e^{2\rho}\int_0^{\infty}e^{2\chi t} |\Phi^*_t(v^{s/u})|^2 dt\ge 4e^{2\rho}\int_0^{\infty}e^{2\chi t} e^{-2(\rho+t)}|v^{s/u}|^2 dt=\frac{2}{1-\chi}|v^{s/u}|^2.
\]
Thus for any $v\in \H$,
\[
(|v|'_\chi)^2=(|v^s|'_\chi)^2+(|v^u|'_\chi)^2\ge \frac{2}{1-\chi}(|v^{s}|^2+|v^{u}|^2)\ge \frac{1}{1-\chi}(|v^{s}|+|v^{u}|)^2\ge \frac{1}{1-\chi}|v|^2>|v|^2.
\]
It implies that $\|C_\chi (x)\|\le 1$. The rest of Item~\ref{fC.norm} follows directly from the definition.

To prove Item~\ref{fC.linearize}, we first observe that $C_{\chi}(x)$ and $C_{\chi}(\varphi_t(x))$ preserve the splitting $\mathbb{R}^{d^s(x)}\oplus\mathbb{R}^{d^u(x)}$, while $\Phi^*$ preserves $\mathcal{H}^s\oplus \mathcal{H}^u$. This implies that $C_{\chi}(\varphi_t(x))^{-1}\circ\Phi^*_t\circ C_{\chi}(x)$ has exactly the form claimed in the lemma. Next we estimate $\norm{A_t(x)}$. For every $w\in \mathbb{R}^{d^s(x)}\times\{0\}$, set $v^s:=C_{\chi}(x)w$. Then we have $|w|_{\mathbb R^{d-1}}=|v^s|_\chi'$ and $|A_t(x)(w)|_{\mathbb R^{d-1}}=|\Phi^*_t(v^s)|_\chi'$. Since
\[
(|v^s|_\chi')^2=4e^{2\rho}\int_0^t e^{2\chi t'}|\Phi^*_{t'}(v^s)|^2dt'+e^{2\chi t}(|\Phi^*_t(v^s)|_\chi')^2,
\]
we have 
\begin{equation}\label{equ.lyapnorm}
  \frac{|\Phi^*_t(v^s)|_\chi'}{|v^s|_\chi'}=e^{-\chi t}\left(1-(|v^s|_\chi')^{-2}\cdot 4e^{2\rho}\int_0^t e^{2\chi t'}|\Phi^*_{t'}(v^s)|^2dt'\right)^{\frac{1}{2}}<e^{-\chi t}.
\end{equation}

On the other hand, by \eqref{eq.flowGI} we have $\int_{jt}^{(j+1)t}e^{2\chi t''}|\Phi^*_{t''}(v^s)|^2dt''\geq e^{-2(1-\chi)jt-2\rho} \int_0^t e^{2\chi t'}|\Phi^*_{t'}(v^s)|^2dt'$.
It follows that
\[
(|v^s|_\chi')^2=4e^{2\rho}\sum_{j=0}^{+\infty}\int_{jt}^{(j+1)t}e^{2\chi t''}|\Phi^*_{t''}(v^s)|^2dt''\ge \frac{4}{1-e^{-2(1-\chi)t}} \int_0^t e^{2\chi t'}|\Phi^*_{t'}(v^s)|^2dt'.
\]
Combining with \eqref{equ.lyapnorm}, $\frac{|\Phi^*_t(v^s)|_\chi'}{|v^s|_\chi'}\geq e^{-\chi t}\left(1-e^{2\rho}(1-e^{-2(1-\chi)t})\right)^{\frac{1}{2}}>e^{-4\rho}$ for all $t\in[0,2\rho]$. 
The proof for $\norm{B_t(x)}$ is symmetric. 

Finally, by the same argument as above, we have $\frac{|\Phi^*_t(v^{s/u})|_\chi'}{|v^{s/u}|_\chi'}=e^{\pm(\rho+|t|)}$ for all $v^{s/u}\in\mathcal{H}^{s/u}_x$. It implies  
\begin{equation}\label{eq.fCbound}
|\Phi_t^*(v)|_\chi'=e^{\pm(\rho+|t|)}|v|_\chi'
\end{equation} 
for all $v\in\mathcal{H}$. Then for any $v_t\in \mathcal{H}_{\varphi_t(x)}$ and $v=\Phi^*_{-t}(v_t)\in\mathcal{H}_{x}$, 
\[
\frac{|C_{\chi}(\varphi_t(x))^{-1}v_t|_{\mathbb R^{d-1}}/|v_t|}{\norm{C_{\chi}(x)^{-1}}}
\leq \frac{|C_{\chi}(\varphi_t(x))^{-1}v_t|_{\mathbb R^{d-1}}}{|C_{\chi}(x)^{-1}v|_{\mathbb R^{d-1}}}\cdot \frac{|v|}{|v_t|}=\frac{|v_t|'}{|v|'}\cdot \frac{|v|}{|v_t|}.
\] 
By \eqref{eq.flowGI}, \eqref{eq.fCbound}, and the arbitrariness of $v_t$, we have $\frac{\norm{C_{\chi}(\varphi_t(x))^{-1}}}{\norm{C_{\chi}(x)^{-1}}}
\leq e^{2(\rho+|t|)}$. The inequality $\frac{\norm{C_{\chi}(\varphi_t(x))^{-1}}}{\norm{C_{\chi}(x)^{-1}}}
\geq e^{-2(\rho+|t|)}$ is obtained similarly. This completes the proof of Item~\ref{fC.Cratio}.
\end{proof}

Since all Lyapunov exponents of $d\varphi$ lie in $[-1/2,1/2]$ (see the first paragraph of \S\,\ref{sec.nuh}), we only need to consider the case $0<\chi<1$. From now on we fix a number $\chi>0$. Throughout most of the remainder of this part (except in the proof of Lemma~\ref{impv.fin}), $\chi$ is held constant. Certain subsequently defined parameters also rely on $\chi$, but their dependence on $\chi$ is not explicitly discussed in this paper. Therefore, we suppress $\chi$ in their notations for simplicity. We also simplify the notations $C_\chi$, $s_\chi$ and $u_\chi$ to $C$, $s$ and $u$ respectively.

%%%%%%%%%%%%%%%%%%%%%%%%%%%%%%%%%%%%%%%%%%%%%%%%%%%%%%%%%%%%%

\subsection{Quantification of hyperbolicity and recurrently non-uniformly hyperbolic locus}

\subsubsection{Quantification of hyperbolicity} \label{Q.para}
For any $x\in \NUH_{\chi}$ and any $0<\varepsilon\ll \chi\wedge \rho$ (recall that $a\wedge b=\min\{a,b\}$), denote by
\[
Q(x)=\varepsilon^{6/\beta}\norm{C(x)^{-1}}^{-48/\beta}
\]
 the {\it quantification of hyperbolicity} of $x$ with respect to $\varepsilon$. We also define by
\[\begin{aligned}
q(x)&=\varepsilon\inf\{e^{\varepsilon|t|}Q(\varphi_t(x)):t\in\mathbb{R}\},\\
q^s(x)&=\varepsilon\inf\{e^{\varepsilon|t|}Q(\varphi_t(x)):t\geq0\},\\
q^u(x)&=\varepsilon\inf\{e^{\varepsilon|t|}Q(\varphi_t(x)):t\leq0\}
\end{aligned}\]
the quantification of hyperbolicity of the orbit of $x$, the positive orbit of $x$, and the negative orbit of $x$  with respect to $\varepsilon$, respectively.

It follows directly from the definition that $0\leq q(x), q^s(x), q^u(x)\leq\varepsilon Q(x)$ and $q(x)=q^s(x)\wedge q^u(x)$. Moreover, $Q(x)$ admits a positive lower bound on every non-singular uniformly hyperbolic set.

\begin{rk}\label{rk.Qandq}
	Both parameters $Q$ and $q$ defined above depend on $\varepsilon$, but we suppress this dependence in the notations for simplicity. All subsequent  relevant results hold for all $\varepsilon>0$ small enough, that is, there exists 
$\varepsilon_0>0$ such that they hold for all $0<\varepsilon\le \varepsilon_0$. Finally, it suffices to take the minimum of the 
   $\varepsilon_0$ obtained from finitely many instances of ``holds for all $\varepsilon>0$ small enough".
\end{rk}

\begin{lem}\label{lem.estimateQ}
The following statements hold for all $\varepsilon>0$ small enough:
\begin{enumerate}
\item \label{est.Qrat}$\frac{Q(\varphi_t(x))}{Q(x)}=e^{\pm \frac{288\rho}{\beta}}$ for all $x\in \NUH_{\chi}$ and $t\in[-2\rho,2\rho]$;
\item\label{est.QC}  $Q(x)<\varepsilon^{6/\beta}$, $\norm{C(x)^{-1}}\cdot Q(x)^{\beta/48}\leq \varepsilon^{1/8}$ for all $x\in \NUH_{\chi}$, and $\|C(f(x))^{-1}\|Q(x)^{\beta/12}\le \varepsilon^{1/4}$ for all $x\in \NUH_{\chi}\cap \mathbb{D}$; (recall that $f$ is the Poincar\'e return map of $\mathbb D$ defined at the beginning of \S\,\ref{sec.linearpoincare}.)
\item \label{est.qrat}$q(\varphi_t(x))=e^{\pm\varepsilon|t|}q(x)$ for all $x\in \NUH_{\chi}$ and $t\in\mathbb{R}$.  
\end{enumerate}
\end{lem}
\begin{proof}
	Item~(1) follows immediately from the definition of $Q(x)$ and Lemma~\ref{lem.propofC}~\ref{fC.Cratio}. Item~(2) follows from the definition of $Q(x)$ and Item~(1). Item~(3) is derived from the definition directly.
	\end{proof}

\subsubsection{Recurrently non-uniformly hyperbolic locus}

\setlist[enumerate]{label={\rm (NUH\arabic*)}}
\begin{dfn}\label{dfn.recNUH}
The {\rm recurrently non-uniformly hyperbolic locus} $\NUH_{\chi}^{\#}$ with respect to $\chi$ is the set of all points $x\in\NUH_\chi$ satisfying:
\begin{enumerate}[start=4]
\item\label{ite.NUH4} $q(x)>0$;
\item\label{ite.NUH5} there exist $\gamma_x, \delta_x>0$ and sequences $t^+_{x,n}\to +\infty$, $t^-_{x,n}\to -\infty$ such that $\varphi_{t^\pm_{x,n}}(x)\not\in M(\Sing(X), \gamma_x)$ and $q(\varphi_{t^\pm_{x,n}}(x))\ge \delta_x$ for each $n$, where $M(\Sing(X), \gamma_x):=\{y\in M:\dist(y,\Sing(X))\le \gamma_x\}$.
\end{enumerate}
\end{dfn}
\setlist[enumerate]{label={\rm (\arabic*)}}

\begin{rk}\label{rk.NUH5}
	For non-singular flows, condition \ref{ite.NUH5} is stated in {\rm \cite{BCL,LMN}} as 
	\[
	\limsup_{t\to+\infty}q(\varphi_t(x))>0 \quad \text{and} \quad  \limsup_{t\to-\infty}q(\varphi_t(x))>0.
	\]
	To account for the presence of singularities, we further require that the subsequence $\varphi_{t^\pm_{x,n}}(x)$
  attaining the limit superior can be selected to stay uniformly bounded away from the set of singularities.
\end{rk}

\begin{lem}\label{lem.recNUH}
 $\NUH^\#_\chi$ is $\varphi_t$-invariant for any $t\in \bR$, and $\mu({\NUH}^{\#}_\chi)=1$ for any regular $\chi$-hyperbolic $X$-invariant measure $\mu$.
\end{lem}
\begin{proof}
The invariance property is clear. Denote $\NUH_\chi^4:=\{x\in \NUH_\chi: x \text{ satisfies \ref{ite.NUH4}} \}$.
It is established in \cite[Proposition~3.5]{BCL} that $\NUH_\chi^4$ has full measure for any regular $\chi$-hyperbolic $X$-invariant measure.  Let $\mu$ be a regular $\chi$-hyperbolic $X$-invariant measure. Define
\[
A_k:=\{x\in \NUH_\chi^4: q(x)\ge \frac{1}{k},\ x\not\in M(\Sing(X),\frac{1}{k})\}. 
\]
Since $\NUH_{\chi}\cap \Sing(X)=\emptyset$, we have $\NUH_{\chi}^4=\bigcup_{k\ge 1} A_k$. Thus $\mu(A_k)>0$ for each sufficiently large $k$. By the Poincar\'e recurrence theorem, there exists $B_k\subset A_k$ satisfying $\mu(B_k)=\mu(A_k)$ and such that for any $x\in B_k$, there exist sequences $t^+_{x,n}\to +\infty$, $t^-_{x,n}\to -\infty$ such that $\varphi_{t^\pm_{x,n}}(x)\in A_k$. This directly implies that every $x\in B_k$ satisfies condition \ref{ite.NUH5} (by taking $\gamma_x=\delta_x=1/k$). Therefore, $B_k\subset \NUH^\#_\chi$ for all sufficiently large $k$. We then conclude that $\mu(\NUH_{\chi}^{\#})=1$, since $\mu(\bigcup_{k\ge 1} B_k)=\mu(\bigcup_{k\ge 1} A_k)=1$.
\end{proof}

\subsubsection{$\mathbb{Z}$-indexed quantification of hyperbolicity}
Let $x\in \NUH_{\chi}$. For any bi-infinite sequence of times $\T=\{t_n\}_{n\in\mathbb{Z}}$ satisfying $\frac{1}{2}\inf_{y\in\mathbb D} R_{\mathbb D}(y)\le t_{n+1}-t_n\leq2\sup_{y\in \mathbb D} R_{\mathbb D}(y)$, we define the discrete hyperbolicity quantification parameters associated with $\T$ as
\[
\begin{aligned}
p^s(x,\mathcal{T},n)&:=\varepsilon\inf\{e^{\varepsilon(t_m-t_n)}Q(\varphi_{t_m}(x)):m\geq n\},\\
p^u(x,\mathcal{T},n)&:=\varepsilon\inf\{e^{\varepsilon(t_n-t_m)}Q(\varphi_{t_m}(x)):m\leq n\}.
\end{aligned}
\]

To simplify the notation, we denote $p^{s/u}(x,\mathcal{T},n)$ as $p^{s/u}(\varphi_{t_n}(x))$ when no ambiguity arises.

\begin{lem}\label{lem.pqcomp}
The following statements hold for all $\varepsilon>0$ small enough, all $x\in \NUH_{\chi}^{\#}$ and all time sequences $\T=\{t_n\}_{n\in\mathbb{Z}}$ satisfying $\frac{1}{2}\inf_{y\in\mathbb D} R_{\mathbb D}(y)\le t_{n+1}-t_n\leq2\sup_{y\in \mathbb D} R_{\mathbb D}(y)$:
\begin{enumerate}
\item  $p^{s/u}(x,\mathcal{T},n)\geq q^{s/u}(\varphi_{t_n}(x))$;
\item\label{ite.pqsucomp} define $\mathfrak{Y}=\varepsilon\rho+288\rho/\beta$; then for all $n\in\mathbb{Z}$ and $t\in[t_n,t_{n+1}]$, we have
\[
\frac{p^{s/u}(\varphi_{t_n}(x))}{q^{s/u}(\varphi_t(x))}=e^{\pm\mathfrak{Y}};
\]
\item for all $n\in\mathbb{Z}$ we have
\[
\begin{aligned}
p^s(\varphi_{t_n}(x))&=\min\{e^{\varepsilon(t_{n+1}-t_n)}p^s(\varphi_{t_{n+1}}(x)),\varepsilon Q(\varphi_{t_n}(x))\},\\
p^u(\varphi_{t_n}(x))&=\min\{e^{\varepsilon(t_{n}-t_{n-1})}p^u(\varphi_{t_{n-1}}(x)),\varepsilon Q(\varphi_{t_n}(x))\};
\end{aligned}
\]
in particular,
\[
\begin{aligned}
\varepsilon Q(\varphi_{t_n}(x))&\ge p^s(\varphi_{t_n}(x))\ge e^{-\varepsilon(t_n-t_m)}p^s(\varphi_{t_m}(x)), \text{ for all $n\geq m$,}\\
\varepsilon Q(\varphi_{t_n}(x))&\ge p^u(\varphi_{t_n}(x))\ge e^{-\varepsilon(t_m-t_n)}p^u(\varphi_{t_m}(x)), \text{ for all $m\geq n$;}
\end{aligned}
\]
\item $p^s(\varphi_{t_n}(x))=\varepsilon Q(\varphi_{t_n}(x))$ holds for infinitely many $n>0$ and $p^u(\varphi_{t_n}(x))=\varepsilon Q(\varphi_{t_n}(x))$ holds for infinitely many $n<0$.
\end{enumerate}
\end{lem}
\begin{proof}
	Item (1) follows directly from the definition.  
  The proofs of Item (2)--(4) are the same as those of \cite[Proposition~3.6]{BCL}, since they only rely on the hyperbolicity of $\Phi^*$.  
\end{proof}

\begin{rk}\label{rk.Qqp}
Although the notations $Q(x)$, $q(x)$, $p(x)$, as well as the previously defined $\NUH_\chi$, $s(x)$, $u(x)$, and $C(x)$ do not carry the superscript ``$^*$'', please note that they all characterize the hyperbolicity after scaling.
\end{rk}  
  
\subsection{Scaled Pesin charts}\label{subsec.scalpesinchart}

In this subsection, we define the scaled Pesin charts for each $x\in \mathbb D\cap\NUH_{\chi}$. 
Note that our definition here differs slightly from common usage. 

Conventionally, a small neighborhood of $x\in D_j$ in $\widehat D_j$ is lifted to $\H_x$ via the exponential map $\exp_x$ at $x$ with scaling applied (as $C(x)$ checks hyperbolicity after scaling) before mapping to Euclidean space through $C(x)$. 
 Our method instead lifts the entire section $\widehat D_j$ (which contains local neighborhoods for every $x\in D_j$) to $\H_{o_j}$ using the exponential map $\exp_{o_j}$ anchored at the central point $o_j$. Scaling is done in this unified linear space $\H_{o_j}$, then transferred to $\H_x$ through $d_x\exp_{o_j}$, and finally mapped to Euclidean space via $C(x)$. 

This methodology has several advantages: by regarding $\exp_{o_j}^{-1}$ as a natural coordinate chart, we transfer the analysis on the submanifold $\widehat D_j$ to the linear space $\H_{o_j}$, which enables direct manipulation of the scaled systems. This not only makes existing results for diffeomorphisms and non-singular flows applicable, but also reduces the computational complexity. 

Assume $x\in D_j=N_{o_j}(\gamma_0|X(o_j)|)$. Define 
\[
\wt{C}(x)=d_x\exp_{o_j}^{-1}\circ C(x),
\]
which is a linear isomorphism from $\mathbb{R}^{d-1}$ to $\N_{o_j}$.

For any vector $v$, define $\varsigma(o_j)(v):=\frac{v}{|X(o_j)|}$.  Recall that $\wt x^*=\frac{\exp^{-1}_{o_j}(x)}{|X(o_j)|}=\varsigma(o_j)\circ \exp_{o_j}^{-1}(x)$. The {\it scaled Pesin chart} at $x\in \mathbb D\cap\NUH_{\chi}$ is the affine map $\wt\Psi_x^* :\mathbb{R}^{d-1}\to \N_{o_j}$ defined as 
\begin{equation}\label{eq.scalepesinchart}
  \wt\Psi_x^*(\xi)=\wt x^*+\wt C(x)(\xi).
\end{equation}
Therefore, $d_\xi\wt\Psi_{x}^*=\wt{C}(x)$ for all $\xi\in \mathbb{R}^{d-1}$. It follows from \eqref{equ.dexp} and Lemma~\ref{lem.propofC}~\ref{fC.norm} that  
\begin{equation}\label{eq.bdPsi}
\norm{d\wt\Psi_{x}^*}=\norm{\wt{C}(x)}\leq 2 \quad \text{and} \quad  \norm{d(\wt\Psi^{*}_x)^{-1}}=\norm{\wt{C}(x)^{-1}}\leq 2 \norm{C(x)^{-1}}.
\end{equation}

We also define the  {\it Pesin chart} of $\mathbb D$, $\Psi_x: \mathbb R^{d-1}(\gamma_0)\to \widehat D_j=N_{o_j}(\kappa\gamma_0|X(o_j)|)$, $x\in D_j\cap \NUH_\chi$, as  
\begin{equation}\label{equ.pesinchart}
\Psi_x=\exp_{o_j}\circ \varsigma(o_j)^{-1}\circ \wt\Psi^*_x.
\end{equation}
By \eqref{eq.bdPsi}, we have 
\begin{equation}\label{eq.BdPsi}
\norm{d\Psi_x}\leq 4|X(o_j)| \quad \text{and} \quad 
\norm{d\Psi_x^{-1}} \leq \frac{4}{|X(o_j)|}\cdot \norm{C(x)^{-1}}.
\end{equation}

Denote $R[\gamma]:=\bR^{d^s(x)}(\gamma)\times \bR^{d^u(x)}(\gamma)$. We omit the explicit dependence on $x$ in the notation $R[\gamma]$, since the dimensions $d^s(x)$ and $d^u(x)$ can always be fixed later.

\begin{pro}\label{pro.Pesinlift}
For all $\varepsilon>0$ small enough and all $x\in \mathbb{D}\cap \NUH_{\chi}$, the lift map 
\[
f^{+,*}_x:=(\wt \Psi_{f(x)}^{*})^{-1}\circ \wt g^{+,*}_x\circ \wt \Psi_{x}^*
\] 
is a well-defined diffeomorphism from $R[10Q(x)]$ onto its image, and it satisfies:
\begin{enumerate}
\item $d_0f^{+,*}_x=C(f(x))^{-1}\circ \Phi^*_{R_{\mathbb D}(x)}\circ C(x)=
\left(\begin{matrix}
 A & 0 \\
0 & B
\end{matrix}\right)$,
 where $A=A_{R_{\mathbb D}(x)}(x)$ and $B=B_{R_{\mathbb D}(x)}(x)$ are given by Lemma~\ref{lem.propofC} for $t={R_{\mathbb D}(x)}$;
\item $f^{+,*}_x=d_0f^{+,*}_x+H^*_x$, where
\begin{enumerate}
\item[\rm (2a)]\label{remain1} $H^*_x(0)=0$ and $d_0H^*_x=0$;
\item[\rm (2b)]\label{remain2} $\|H^*_x\|_{C^{1+\frac{\beta}{2}}}<\varepsilon$ on $R[10Q(x)]$.
\end{enumerate}
\end{enumerate}
A symmetric conclusion holds for $f^{-,*}_x:=(\wt \Psi_{x}^{*})^{-1}\circ \wt g^{-,*}_{f(x)}\circ \wt \Psi_{f(x)}^*$.
\end{pro}
\begin{proof}
By \eqref{eq.bdPsi} and the definition of $\wt\Psi_{x}^*$, $\wt\Psi_{x}^*(R[10Q(x)])$ is contained in the ball of radius $20Q(x)$ centered at $\wt x^*$. 
By Lemma~\ref{lem.estimateQ}~\ref{est.QC}, we have $20Q(x)<\gamma_0$ for all $\varepsilon>0$ small enough and all $x\in \mathbb{D}\cap \NUH_{\chi}$. It follows from the construction of $\mathbb{D}$ and $\wh{\mathbb{D}}$ at the beginning of \S\,\ref{sec.linearpoincare} that $ (\wt g^{+,*}_x\circ \wt \Psi_{x}^*)(R[10Q(x)])$ is contained in the ball of radius $20K_F Q(x)$ centered at $\wt {f(x)}^*$. Thus $f^{+,*}_x$ is well-defined on $R[10Q(x)]$. 

It is clear that $d_0f^{+,*}_x=C(f(x))^{-1}\circ \Phi^*_{R_{\mathbb D}(x)}\circ C(x)$. The remaining part of Items (1) and (2a) can be directly derived from Lemma~\ref{lem.propofC}~\ref{fC.linearize} and the definition of $H^*_x$. We now prove Item (2b).  By \eqref{eq.gstar}, \eqref{eq.bdPsi} and Lemma~\ref{lem.estimateQ}~\ref{est.QC}, for any $\xi_1, \xi_2\in R[10Q(x)]$, we have
\begin{equation}\label{equ.dHHolder}
\begin{aligned}
\norm{d_{\xi_1}H^*_x-d_{\xi_2}H^*_x}&\le\norm{\wt C(f(x))^{-1}}\cdot K_F |\wt C(x)(\xi_1)-\wt C(x)(\xi_2)|^{\beta}\\
&\le 2\norm{C(f(x))^{-1}}\cdot K_F |2(\xi_1-\xi_2)|^{\beta}\\
&\le 2^{1+\beta}20^{\frac{\beta}{2}}K_F\cdot (\norm{C(f(x))^{-1}}Q(x)^{\frac{\beta}{12}})\cdot(Q(x))^{\frac{5\beta}{12}}\cdot|\xi_1-\xi_2|^{\frac{\beta}{2}}\\
&\le 2^{1+\beta}20^{\frac{\beta}{2}}K_F\cdot \varepsilon^\frac{1}{4}\varepsilon^{\frac{5}{2}}\cdot |\xi_1-\xi_2|^{\frac{\beta}{2}}.
\end{aligned}	
\end{equation}

Taking $\xi_2=0$ in \eqref{equ.dHHolder}, we have
\[
\norm{d_{\xi_1}H^*_x}\le 2\norm{C(f(x))^{-1}}\cdot K_F |2\xi_1|^{\beta}\le 2\cdot 20^\beta K_F\cdot (\norm{C(f(x))^{-1}}Q(x)^{\frac{\beta}{12}})\cdot(Q(x))^{\frac{11\beta}{12}}\le 2\cdot 20^\beta K_F\cdot \varepsilon^\frac{1}{4}\varepsilon^{\frac{11}{2}}.
\]
Then by the Mean Value Theorem, 
\[
|H_x^*(\xi_1)|\le  2\cdot 20^\beta K_F\cdot \varepsilon^\frac{1}{4}\varepsilon^{\frac{11}{2}}\cdot 10Q(x)\le  2\cdot 20^\beta K_F\cdot \varepsilon^\frac{1}{4}\varepsilon^{\frac{11}{2}}\cdot 10\varepsilon^{\frac{6}{\beta}}.
\]

Combining the above estimates, we conclude that $\|H^*_x\|_{C^{1+\frac{\beta}{2}}}<\varepsilon$ on $R[10Q(x)]$ for all $\varepsilon>0$ small enough.
This completes the proof.
\end{proof}

The following calculations are mostly performed within $\N_{o_j}$ by using the scaled Pesin charts, and we only project back to the manifold when necessary. 
 The commutative diagram in Figure~\ref{fig.CommuDiagram} illustrates the relationship among these mappings (where $\imath(x)$ denotes translation by $\wt x^*$ in $\N_{o_j}$).

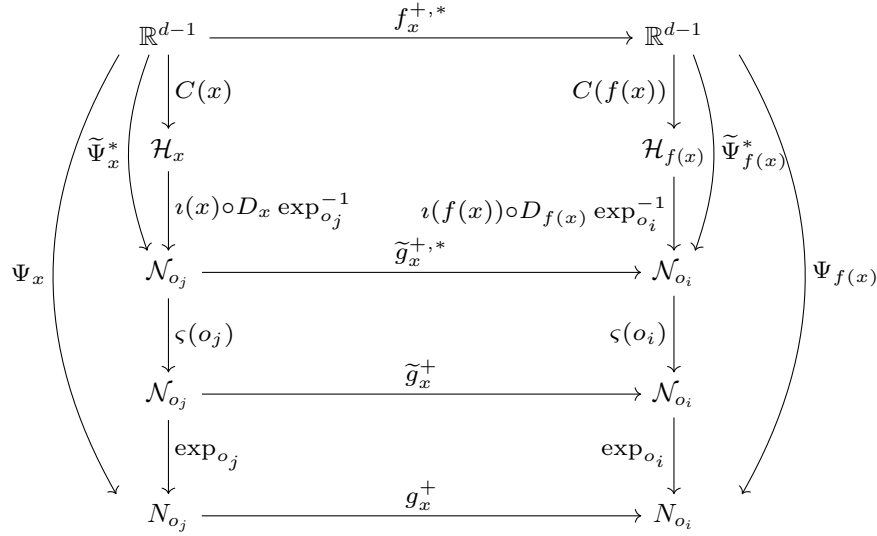
\begin{figure}[ht]
\centering
\begin{tikzcd}[column sep=5.6cm, row sep=1cm]
 \mathbb{R}^{d-1} \arrow[dddd, {xshift=-15pt}, bend right=30, "{\scalebox{1.3}{$\Psi_x$}}" left ] \arrow[dd, {xshift=-5pt}, bend right=20, "{\scalebox{1.3}{$\wt\Psi^*_x$}}" left] \arrow[d, "{\scalebox{1.3}{$C(x)$}}" right] \arrow[r, "{\scalebox{1.3}{$ f^{+,*}_x$}}" above] & \mathbb{R}^{d-1} \arrow[d, "{\scalebox{1.3}{$C(f(x))$}}" left]  \arrow[dddd, {xshift=21pt}, bend left=30, "{\scalebox{1.3}{$\Psi_{f(x)}$}}" right ] \arrow[dd, {xshift=5pt}, bend left=20, "{\scalebox{1.3}{$\wt\Psi^*_{f(x)}$}}" right]\\
\H_x \arrow[d, "{\scalebox{1.3}{$\imath(x)\circ D_x\exp_{o_j}^{-1}$}}" right] & \H_{f(x)} \arrow[d, "{\scalebox{1.3}{$\imath(f(x))\circ D_{f(x)}\exp_{o_i}^{-1}$}}" left]\\
 \N_{o_j} \arrow[d, "{\scalebox{1.3}{$\varsigma(o_j)$}}" right] \arrow[r, "{\scalebox{1.3}{$\wt g^{+,*}_x$}}" above] & \N_{o_i} \arrow[d, "{\scalebox{1.3}{$\varsigma(o_i)$}}" left] \\
 \N_{o_j} \arrow[d, "{\scalebox{1.3}{$\exp_{o_j}$}}" right] \arrow[r, "{\scalebox{1.3}{$\wt g^+_x$}}" above] & \N_{o_i}  \arrow[d, "{\scalebox{1.3}{$\exp_{o_i}$}}" left]\\
 N_{o_j} \arrow[r, "{\scalebox{1.3}{$g^+_x$}}" above] & N_{o_i}
\end{tikzcd}
\caption{Relationship between mappings}\label{fig.CommuDiagram}
\end{figure}

\subsection{Restricted Pesin charts and the overlapping condition}\label{sec.overlap}

For every $\eta\in(0,Q(x)]$, we denote by $\wt\Psi^{*,\eta}_{x}$ the restriction of the scaled Pesin chart at $x$ of scaled size $\eta$, that is, $\wt\Psi^{*,\eta}_{x}:=\wt\Psi^{*}_{x}\big|_{R[\eta]}$. Accordingly, we denote $\Psi^{\eta}_{x}:=\Psi_{x}\big|_{R[\eta]}$.

Subsequently, we will discuss the pseudo orbits and the shadowing property. Unlike diffeomorphisms and non-singular flows, we employ scaled methods to handle the lack of uniformity of sizes. 

\begin{dfn}\label{dfn.overlap}
Assume that $x_1,x_2\in \mathbb D\cap \NUH_\chi$. Two scaled Pesin charts $\wt\Psi^{*,\eta_1}_{x_1}$ and $\wt\Psi^{*,\eta_2}_{x_2}$ are said to be {\rm $\varepsilon$-overlapping}, denoted by $\wt\Psi^{*,\eta_1}_{x_1}\over\wt\Psi^{*,\eta_2}_{x_2}$, if 
\begin{enumerate}
\item $x_1,x_2\in D_j$ for some $j$;
\item $d^s(x_1)=d^s(x_2)$;
\item\label{dfoverlap.rat} $\frac{\eta_1}{\eta_2}=e^{\pm\varepsilon}$; and
\item\label{dfoverlap.dist} $|\wt x_1^*-\wt x_2^*|+\|\wt C(x_1)-\wt C(x_2)\|<(\eta_1\eta_2)^4$.
\end{enumerate}
\end{dfn}

We point out that $|\wt x_1^*-\wt x_2^*|$ is the distance after scaling. Thus by condition \ref{dfoverlap.dist} in Definition~\ref{dfn.overlap}, we have 
\begin{equation}\label{eq.disti}
|\wt x^*_1-\wt x^*_2|<(\eta_1\eta_2)^4 \quad \text{and} \quad \dist(x_1,x_2)<2(\eta_1\eta_2)^4|X(o_j)|.
\end{equation}

\begin{pro}\label{pro.overlap}
For all $\varepsilon>0$ small enough, if $\wt\Psi^{*,\eta_1}_{x_1}\over\wt\Psi^{*,\eta_2}_{x_2}$ and $x_1,x_2\in D_j\cap \NUH_\chi$, then we have
\begin{enumerate}
\item\label{ol.comp} $\wt \Psi_{x_i}^*(R[10Q(x_i)])\subset \N_{o_j}(\wt x_1^*, \gamma_0)\cap \N_{o_j}(\wt x_2^*, \gamma_0)$ for $i=1,2$;
\item \label{ol.inclu}$\wt \Psi_{x_i}^*(R[e^{-2\varepsilon}\eta_i])\subset \wt\Psi^{*}_{x_k}(R[\eta_k])$ for $i,k=1,2$;
\item\label{ol.Ccomp}$\norm{\wt{C}(x_1)^{-1}-\wt{C}(x_2)^{-1}}<(\eta_1\eta_2)^3$ and $\frac{\norm{\wt{C}(x_1)^{-1}}}{\norm{\wt{C}(x_2)^{-1}}}=e^{\pm(\eta_1\eta_2)^3}$;
\item\label{ol.trans}  $\norm{(\wt\Psi_{x_i}^*)^{-1}\circ\wt \Psi_{x_k}^*-{\rm Id}}_{C^2}<\varepsilon(\eta_1\eta_2)^{2}$ for $i, k=1, 2$, where the norm is taken in $R[\gamma_0]$.
\end{enumerate}
\end{pro}
\begin{proof}
For the case $i=1$ in Item~\ref{ol.comp}, as established in Proposition~\ref{pro.Pesinlift},
 we have $\wt \Psi_{x_1}^*(R[10Q(x_1)])\subset \N_{o_j}(\wt x_1^*,20Q(x_1))\subset \N_{o_j}(\wt x_1^*,\gamma_0)$ for all $\varepsilon>0$ small enough. On the other hand, the triangle inequality gives
$\wt\Psi_{x_1}^*(R[10Q(x_1)])\subset \N_{o_j}(x_2^{*},20Q(x_1)+|\wt x_1^{*}-\wt x_2^{*}|)$. 
Using $\eta_i\leq Q(x_i)<\varepsilon^{6/\beta}$ (by Lemma~\ref{lem.estimateQ}~\ref{est.QC}) and \eqref{eq.disti}, for all $\varepsilon>0$ small enough, we have $20Q(x_1)+|\wt x_1^{*}-\wt x_2^{*}|<20\varepsilon^{6/\beta}+(\eta_1\eta_2)^4<\gamma_0$. 
Thus $\wt\Psi_{x_1}^*(R[10Q(x_1)])\subset \N_{o_j}(\wt x_2^*,\gamma_0)$. The case $i=2$ is analogous.

For Item~\ref{ol.inclu}, we only prove $\wt\Psi_{x_1}^*(R[e^{-2\varepsilon}\eta_1])\subset \wt\Psi_{x_2}^*(R[\eta_2])$. For any $\xi\in R[e^{-2\varepsilon}\eta_1]$, a direct calculation gives
\begin{equation}\label{eq.Psidifference}
(\wt\Psi_{x_2}^*)^{-1}\circ\wt\Psi_{x_1}^*(\xi)=\wt{C}(x_2)^{-1}(\wt x_1^*-\wt x_2^*)+\wt{C}(x_2)^{-1}\circ(\wt{C}(x_1)-\wt{C}(x_2))(\xi)+\xi.
\end{equation}
Moreover, we have the following estimates:
\begin{enumerate}
\item[(Est1)] $|\xi|\leq e^{-2\varepsilon}\eta_1\leq e^{-\varepsilon}\eta_2$, by Definition~\ref{dfn.overlap}~\ref{dfoverlap.rat};
\item[(Est2)] $\big|\wt{C}(x_2)^{-1}(\wt x_1^*-\wt x_2^*)\big|\leq 2\norm{C(x_2)^{-1}}\cdot (\eta_1\eta_2)^4$, by \eqref{eq.bdPsi} and Definition~\ref{dfn.overlap}~\ref{dfoverlap.dist};
\item[(Est3)]$\big\|\wt{C}(x_2)^{-1}\circ(\wt{C}(x_1)-\wt{C}(x_2))\big\|\leq 2\norm{C(x_2)^{-1}}\cdot (\eta_1\eta_2)^4$, by \eqref{eq.bdPsi} and Definition~\ref{dfn.overlap}~\ref{dfoverlap.dist}. 
\end{enumerate}

By employing the same method used in the proof of Proposition~\ref{pro.Pesinlift} to apply Lemma~\ref{lem.estimateQ}~\ref{est.QC}, and noting that $\eta_i<Q(x_i)$, we conclude that $\big|(\wt\Psi_{x_2}^*)^{-1}\circ\wt\Psi_{x_1}^*(\xi)\big|<\eta_2$ for all $\varepsilon>0$ small enough. This proves Item~\ref{ol.inclu}.

Furthermore, for sufficiently small $\varepsilon>0$, we also have
\[
\left|\big\|\wt{C}(x_1)^{-1}\big\|-\big\|\wt{C}(x_2)^{-1}\big\|\right|\le\big\|\wt{C}(x_1)^{-1}-\wt{C}(x_2)^{-1}\big\|\leq \big\|\wt{C}(x_2)^{-1}\big\|\cdot\big\|\wt{C}(x_1)-\wt{C}(x_2)\big\|\cdot\big\|\wt{C}(x_1)^{-1}\big\|<\frac{1}{4}(\eta_1\eta_2)^3.
\]
Combined with $\big\|\wt{C}(x_2)^{-1}\big\|\ge 1/2$, this yields 
\[
e^{-(\eta_1\eta_2)^3}<1-\frac{1}{2}(\eta_1\eta_2)^3<\frac{\|\wt{C}(x_1)^{-1}\|}{\|\wt{C}(x_2)^{-1}\|}<1+\frac{1}{2}(\eta_1\eta_2)^3<e^{(\eta_1\eta_2)^3},
\] 
which proves Item~\ref{ol.Ccomp}.
Finally, differentiating both sides of \eqref{eq.Psidifference}, we obtain
\[
d((\wt\Psi_{x_2}^*)^{-1}\circ\wt\Psi_{x_1}^*-{\rm Id})=\wt{C}(x_2)^{-1}\circ(\wt{C}(x_1)-\wt{C}(x_2)).
\]
Therefore, the $C^2$ norm bound in Item~\ref{ol.trans} follows directly from (Est3) and Lemma~\ref{lem.estimateQ}~\ref{est.QC}.
\end{proof}

We now consider the transition between overlapping scaled Pesin charts.  
If $\wt\Psi^{*,\eta}_{f(x)}\over\wt\Psi^{*,\eta'}_{y}$, then it follows from Proposition~\ref{pro.overlap}~\ref{ol.comp} that the  map 
\[
f^{+,*}_{x,y}:=(\wt \Psi^*_{y})^{-1}\circ \wt g^{+,*}_x\circ \wt\Psi^*_{x}
\] 
is well-defined on $R[10Q(x)]$. Symmetrically, if $\wt\Psi^{*,\eta}_{x}\over\wt\Psi^{*,\eta'}_{f^{-1}(y)}$, then  the  map 
\[
f^{-,*}_{x,y}:=(\wt \Psi^*_{x})^{-1}\circ \wt g^{-,*}_y\circ \wt\Psi^*_{y}
\] 
is well-defined on $R[10Q(x)]$.

Combining \eqref{eq.boundtimedg} with Proposition~\ref{pro.overlap}, the proof of the next proposition is identical to those of  \cite[Theorem~3.8']{BCL} and \cite[Theorem~3.10]{LMN}. 

\begin{pro}\label{pro.overlift}
For all $\varepsilon>0$ small enough, if $x, y\in {\mathbb D}\cap \NUH_{\chi}$ and $\wt\Psi^{*,\eta}_{f(x)}\over\wt\Psi^{*,\eta'}_{y}$, then the lift map 
$f^{+,*}_{x,y}$ is a well-defined diffeomorphism from $R[10Q(x)]$ onto its image, and can be written as
 $f^{+,*}_{x,y}=
\left(\begin{matrix}
 A & 0 \\
0 & B
\end{matrix}\right)+H^*_{x,y}$, where $A, B$ are given by Proposition~\ref{pro.Pesinlift} and satisfy that:
\begin{enumerate}
\item $e^{-4\rho}<|A|<e^{-\chi R_{\mathbb D}(x)}$ and $e^{\chi R_{\mathbb D}(x)}<|B|<e^{4\rho}$;
\item $\norm{H^*_{x,y}(0)}<\varepsilon \eta$, $\norm{d_0H^*_{x,y}}<\varepsilon\eta^{\beta/3}$, and the $\frac{\beta}{3}$-H\"older constant $\text{\rm  H\"ol}_{\beta/3}(dH^*_{x,y})<\varepsilon$; consequently, $df^{+,*}_{x,y}$
 is $\frac{\beta}{3}$-H\"older continuous, with a uniform H\"older constant that is independent of $x$ and $y$.
 \end{enumerate}
 
 The symmetric statement holds for $f^{-,*}_{x,y}$ if $\wt\Psi^{*,\eta}_{x}\over\wt\Psi^{*,\eta'}_{f^{-1}(y)}$.
\end{pro}

%%%%%%%%%%%%%%%%%%%%%%%%%%%%%%%%%%%%%%%%%%%%%%%%%%%%%%%%%%%%%%%%%%%

\section{Graph transforms and shadowing lemma}\label{sec.graph}

\subsection{Pseudo-orbits}\label{sec.gpo}

Let $x\in {\mathbb D}\cap \NUH_{\chi}$. A {\it scaled $\varepsilon$-double chart} $\wt\Psi^{*,p^s,p^u}_{x}$ is defined as a pair of scaled Pesin charts $\wt \Psi^{*,p^s,p^u}_{x}=(\wt \Psi^{*,p^s}_{x},\wt\Psi^{*,p^u}_{x})$ with $p^s,p^u\in(0, \varepsilon Q(x)]$. The parameters $p^s$ and $p^u$ will be used to characterize the hyperbolicity in the stable and unstable directions, respectively.

Let $\wt v^*=\wt\Psi^{*,p^s,p^u}_{x}$ and $\wt w^*=\wt\Psi^{*,q^s,q^u}_{y}$ be two scaled $\varepsilon$-double charts satisfying $\wt\Psi^{*,q^s\wedge q^u}_{f(x)}\over \wt \Psi^{*,q^s\wedge q^u}_{y}$ and $\wt\Psi^{*,p^s\wedge p^u}_{f^{-1}(y)}\over\wt\Psi^{*,p^s\wedge p^u}_{x}$. The {\it transition time} from $\wt v^*$ to $\wt w^*$ is
\[
T(\wt v^*,\wt w^*):=\Big(\min\{\tau_x(z):z\in \Psi_x(R[\frac{1}{20}(p^s\wedge p^u)])\}\Big)\wedge \Big(\min\{\tau_{f^{-1}(y)}(g_x^-(z)):z\in\Psi_y(R[\frac{1}{20}(q^s\wedge q^u)])\}\Big),
\]
where $\tau_x$ is the Poincar\'e time corresponding to $g_x^+$ defined in \S\,\ref{subsec.Holomap}.  It is clear that $T(\wt v^*,\wt w^*)<\rho$.

\setlist[enumerate]{label={\rm (GPO\arabic*)}}
\begin{dfn}\label{dfn.gpo}
Let $\wt v^*=\wt\Psi^{*,p^s,p^u}_{x}$ and $\wt w^*=\wt\Psi^{*,q^s,q^u}_{y}$ be two scaled $\varepsilon$-double charts. Define $\wt v^*\ov \wt w^*$ if:
\begin{enumerate}
\item\label{item.GPOcond1} $\wt\Psi^{*,q^s\wedge q^u}_{f(x)}\over \wt \Psi^{*,q^s\wedge q^u}_{y}$ and $\wt\Psi^{*,p^s\wedge p^u}_{f^{-1}(y)}\over\wt\Psi^{*,p^s\wedge p^u}_{x}$; and
\item\label{item.GPOcond2}  the following estimates hold:
\begin{equation}\label{eq.GPO21}
e^{-\varepsilon p^s}\min\{e^{\varepsilon T(\wt v^*,\wt w^*)}q^s,e^{-\varepsilon}\varepsilon Q(x)\}\leq p^s\leq\min\{e^{\varepsilon T(\wt v^*,\wt w^*)}q^s,\varepsilon Q(x)\},
\end{equation}
\begin{equation}\label{eq.GPO22}
e^{-\varepsilon q^u}\min\{e^{\varepsilon T(\wt v^*,\wt w^*)}p^u,e^{-\varepsilon}\varepsilon Q(y)\}\leq q^u\leq\min\{e^{\varepsilon T(\wt v^*,\wt w^*)}p^u,\varepsilon Q(y)\}.
\end{equation}
\end{enumerate}
\end{dfn}
\setlist[enumerate]{label={\rm (\arabic*)}}

\begin{rk}\label{rk.Tminimal}
By Proposition~\ref{pro.overlift}, if $\wt\Psi^{*,p^s,p^u}_{x} \ov \wt\Psi^{*,q^s,q^u}_{y}$, then $g^{-}_y(\Psi_y(R[(q^s\wedge q^u)/20]))\subset \Psi_x(R[(p^s\wedge p^u)/15])$. Thus the minimal time $T(\wt\Psi^{*,p^s,p^u}_{x}, \wt\Psi^{*,q^s,q^u}_{y})$ can be achieved by some $z\in \Psi_x(R[(p^s\wedge p^u)/15])$.
\end{rk}

Note that condition \ref{item.GPOcond2} in Definition~\ref{dfn.gpo} follows the formulation originally introduced in \cite{BCL}. This flexible approach is well suited to address the difficulty arising from the non-discrete nature of the flow's time parameter (see Proposition~\ref{pro.SDR} in \S\,\ref{sec.Coarse}).

An {\it $\varepsilon$-generalized pseudo-orbit} ({\it $\varepsilon$-gpo}) is a sequence $\ul{\wt v}^*=\{\wt v^*_n\}_{n\in\mathbb{Z}}$ of scaled $\varepsilon$-double charts such that $\wt v^*_n\ov \wt v^*_{n+1}$ for all $n\in\mathbb{Z}$. A {\it positive $\varepsilon$-gpo} is defined as the sequence $\ul{\wt v}^{*,+}=\{\wt v^*_n\}_{n\ge 0}$ satisfying $\wt v^*_n\ov \wt v^*_{n+1}$ for all $n\ge 0$.
The definition of a {\it negative $\varepsilon$-gpo} is analogous.

Denote by $\ul{\wt v}^{*,+}=\{\wt v^*_n\}_{n\ge 0}$ (resp. $\ul{\wt v}^{*,-}=\{\wt v^*_n\}_{n\le 0}$) the positive (resp. negative) part of an $\varepsilon$-gpo $\ul{\wt v}^*=\{\wt v^*_n\}_{n\in\mathbb{Z}}$. 
An $\varepsilon$-gpo $\ul{\wt v}^*$ is {\it regular} if there are $\wt v^*, \wt w^*$ so that $\wt v^*_n=\wt v^*$ for infinitely many $n>0$ and $\wt v^*_n=\wt w^*$ for infinitely many $n<0$. We define a regular positive or negative $\varepsilon$-gpo in the analogous manner.

The following lemma, adapted from \cite[Lemma 4.2]{BCL}, quantifies the evolution of the minimal chart size along any $\varepsilon$-gpo.

\begin{lem}\label{lem.edge}
If $\wt\Psi^{*,p^s,p^u}_{x}\ov \wt\Psi^{*,q^s,q^u}_{y}$ then 
$\frac{p^s\wedge p^u}{q^s \wedge q^u}=e^{\pm2\varepsilon}$. Moreover, for any $\varepsilon$-gpo $\{\wt\Psi_{x_n}^{*,p^s_n,p^u_n}\}_{n\in\mathbb{Z}}$, we have $\frac{p_n^s\wedge p_n^u}{p_0^s \wedge p_0^u}=e^{\pm2|n|\varepsilon}$ for all $n\in\mathbb{Z}$. 
\end{lem}

\subsection{Graph transforms}\label{subsec.graphtran}

In this subsection, we introduce the graph transform method, which serves as the standard tool for constructing invariant manifolds in non-uniformly hyperbolic dynamics. We emphasize that the scaling procedure defined on $\N_{o_j}$, combined with the scaled hyperbolicity of the relevant mappings, makes the discussion of graph transforms under scaled Pesin charts identical to that for diffeomorphisms and non-singular flows. For this reason, we state all results in this subsection without repeating the standard proofs, and refer the reader to the corresponding references for full details. Furthermore, we only state conclusions for the stable direction, as the unstable direction follows by complete symmetry.

Recall that $\bR^k(\gamma)$ is the closed ball in $\bR^k$ of radius $\gamma$ centered at the origin. Let $\wt v^*=\wt\Psi^{*,p^s,p^u}_{x}$ be a scaled $\varepsilon$-double chart.
An {\it $s$-admissible manifold} at $\wt v^*$ is defined as the set 
\[
\wt V^*:=\wt\Psi_x^*\{(\xi,G(\xi)):\xi\in \bR^{d^s(x)}(p^s)\},
\] 
where $G:\bR^{d^s(x)}(p^s)\to\mathbb{R}^{d^u(x)}$ is a $C^{1+\frac{\beta}{3}}$ map, called the {\it representing function} of $\wt V^*$, satisfying:
\setlist[enumerate]{label={\rm (AM\arabic*)}}
\begin{enumerate}
\item $|G(0)|\leq 10^{-3}(p^s\wedge p^u)$;
\item\label{ite.AMcomd2} $\norm{d_0G}\leq \frac{1}{2}(p^s\wedge p^u)^{\beta/3}$;
\item\label{ite.AMcomd3} $\norm{dG}_{C^0}+\text{H\"ol}_{\beta/3}(dG)\leq1/2$, where the norms are taken in $\bR^{d^s(x)}(p^s)$. 
\end{enumerate}
\setlist[enumerate]{label={\rm (\arabic*)}}

We may define {\it $u$-admissible manifolds} at $\wt v^*$ in an analogous way. Moreover, if $\wt V^*_1$ and $\wt V^*_2$ are two $s$-manifolds at $\wt v^*$ with representing functions $G_1$ and $G_2$ respectively, then we define $d_{C^{i}}(\wt V^*_1,\wt V^*_2)=\norm{G_1-G_2}_{C^i}$ ($i\ge 0$), where the norm is taken in $\bR^{d^s(x)}(p^s)$. Denote by $\mathscr{M}^{s/u}(\wt v^*)$  the collection of all $s/u$-admissible manifolds at $\wt v^*$.

\begin{pro}\label{pro.graphinters}
For all $\varepsilon>0$ small enough, any $\varepsilon$-double chart $\wt v^*=\wt \Psi^{*,p^s,p^u}_{x}$, any $\wt V^{*,s}\in\mathscr{M}^{s}(\wt v^*)$ and any $\wt V^{*,u}\in\mathscr{M}^{u}(\wt v^*)$,  $\wt V^{*,s}$ and $\wt V^{*,u}$ intersect transversally at a unique point in $R[\frac{1}{500}(p^s\wedge p^u)]$.
\end{pro}

The proof can be found in \cite[Proposition 4.11]{Sar13} and \cite[Proposition 3.5]{Ben}.

Given an edge $\wt v^*\ov \wt w^*$ with $\wt v^*=\wt\Psi^{*,p^s,p^u}_{x}$ and $\wt w^*=\wt \Psi^{*,q^s,q^u}_{y}$, we have that both the stable and unstable dimensions of $x$ coincide with those of $y$. Denote $d^s:=d^s(x)=d^s(y)$ and $d^u:=d^u(x)=d^u(y)$. In particular, the stable and unstable dimensions are constants along any gpo. 

Moreover, for any $s$-admissible manifold $\wt V^*_1\in \mathscr{M}^s(\wt w^*)$ with representing function $F:\bR^{d^s}(q^s)\to\mathbb{R}^{d^u}$, there exists a unique $s$-admissible manifold $\wt V^*_2\in \mathscr{M}^s(\wt v^*)$ with representing function $G:\bR^{d^s}(p^s)\to\mathbb{R}^{d^u}$ such that $\{(\xi,G(\xi)): \xi\in\bR^{d^s}(p^s)\}\subset f^{-,*}_{x,y}\{(\xi,F(\xi)): \xi\in\bR^{d^s}(q^s)\}$. Define $\mathscr{F}^s_{\wt v^*,\wt w^*}:\mathscr{M}^s(\wt w^*)\to\mathscr{M}^s(\wt v^*)$ by $\mathscr{F}^s_{\wt v^*,\wt w^*}(\wt V^*_1)=\wt V^*_2$. Similarly, we may define $\mathscr{F}^u_{\wt v^*,\wt w^*}:\mathscr{M}^u(\wt v^*)\to\mathscr{M}^u(\wt w^*)$. We call both $\mathscr{F}^u_{\wt v^*,\wt w^*}$ and $\mathscr{F}^s_{\wt v^*,\wt w^*}$  \emph{graph transforms}.

The following standard proposition state a fundamental property of graph transforms. We refer to \cite[Proposition 4.12, 4.14]{Sar13}, \cite[Proposition 3.6, 3.8]{Ben} and \cite[Appendix]{ALP} for the proof.  

\begin{pro}\label{graphtrans}
The following statements hold for all $\varepsilon>0$ small enough. For any $\wt v^*\ov \wt w^*$, the graph transforms $\mathscr{F}^{s/u}_{\wt v^*,\wt w^*}$ are well-defined. In addition, if $\wt V^*_1, \wt V^*_2\in\mathscr{M}^s(\wt w^*)$ then:
\begin{enumerate}
\item $d_{C^0}(\mathscr{F}^s_{\wt v^*,\wt w^*}(\wt V^*_1),\mathscr{F}^s_{\wt v^*,\wt w^*}(\wt V^*_2))\leq e^{-\frac{1}{2}\chi\inf R_{\mathbb D}}d_{C^0}(\wt V^*_1,\wt V^*_2)$;
\item $d_{C^1}(\mathscr{F}^s_{\wt v^*,\wt w^*}(\wt V^*_1),\mathscr{F}^s_{\wt v^*,\wt w^*}(\wt V^*_2))\leq e^{-\frac{1}{2}\chi\inf R_{\mathbb D}}(d_{C^1}(\wt V^*_1,\wt V^*_2)+d_{C^0}(\wt V^*_1,\wt V^*_2)^{\beta/3})$.
\end{enumerate}
A similar statement holds for $\mathscr{F}^u_{\wt v^*,\wt w^*}$.
\end{pro}

The properties established in Proposition~\ref{graphtrans} motivate the following definition.

\begin{dfn}\label{def.invmanifoldforgpo}
Given a  positive $\varepsilon$-gpo $\ul{\wt v}^{*,+}=\{\wt v^*_n\}_{n\geq0}$. The {\rm stable manifold} $\wt V^{*,s}[\ul{\wt v}^{*,+}]$ of $\ul{\wt v}^{*,+}$ is defined by 
\[
\wt V^{*,s}[\ul{\wt v}^{*,+}]:=\lim_{n\to+\infty}(\mathscr{F}^s_{\wt v^*_0,\wt v^*_1}\circ\cdots\circ\mathscr{F}^s_{\wt v^*_{n-1},\wt v^*_n})(\wt V^*_n),
\]
for some (any) choice $\{\wt V^*_n\}_{n\geq0}$ with $\wt V^*_n\in\mathscr{M}^s(\wt v^*_n)$. The limit is taken in the $C^1$ topology. 

The {\rm unstable manifold} $\wt V^{*,u}[\ul{\wt v}^{*,-}]$ of a negative $\varepsilon$-gpo $\ul{\wt v}^{*,-}=\{\wt v^*_n\}_{n\leq0}$ is defined analogously.
\end{dfn}

It is proved in \cite[Proposition~4.15]{Sar13} that these stable and unstable manifolds are well-defined, and their definitions are independent of the choice of $\{\wt V^*_n\}$.

Let $\ul{\wt v}^*=\{\wt \Psi_{x_n}^{*,p^s_n,p^u_n}\}_{n\in\mathbb{Z}}$ be an $\varepsilon$-gpo. For notational simplicity, we define the following conventions:
\begin{itemize}
\item  
$\wt V^{*,s}_k:=\wt V^{*,s}[\{\wt v^*_n\}_{n\geq k}]$ and $\wt V^{*,u}_k:=\wt V^{*,u}[\{\wt v^*_n\}_{n\leq k}]$ for $k\in\mathbb{Z}$;
\item $\wt g^{+,*}_{x_0,n}:=\wt g^{+,*}_{x_{n-1}}\circ\cdots\circ \wt g^{+,*}_{x_0}$ for $n\geq1$ and $\wt g^{+,*}_{x_0,0}={\rm Id}$; 
\item $\wt g^{-,*}_{x_0,n}:=\wt g^{-,*}_{x_{n+1}}\circ\cdots\circ \wt g^{-,*}_{x_{0}}$ for $n\leq-1$  and $\wt g^{-,*}_{x_0,0}={\rm Id}$.
\end{itemize}
We also define $g^{\pm}_{x_0,n}$ in the analogous manner.

Recall that $\tau_{x_n}$ is the Poincar\'e time corresponding to $g^{+}_{x_n}$. For $z\in \Psi_{x_0}(\wt V^{*,s}[\ul{\wt v}^{*,+}])$ and $n\ge 1$, define 
\begin{equation}\label{eq.smnfdtime}
t_n(z):=\sum_{i=0}^{n-1}\tau_{x_i}(g^{+}_{x_0,i}(z)) \quad \text{and} \quad t_{-n}(z):=-\sum_{i=-n}^{-1}\tau_{x_i}(g^{-}_{x_0,i}(z)).
\end{equation}
We also define $t_0(z):=0$. By Definition~\ref{def.invmanifoldforgpo}, the function $t_n$ is well-defined. 

The properties of the graph transform yield the following invariant manifold theorem. The proof is standard, which can be found in \cite[Proposition 4.15]{Sar13}, \cite[Proposition 3.12]{Ben} and \cite[Proposition~4.11]{ALP}. 

\begin{pro}\label{sca.invmnfd}
The following statements hold for all $\varepsilon>0$ small enough. Let $\ul{\wt v}^*=\{\wt v^*_n\}_{n\in\mathbb{Z}}=\{\wt\Psi_{x_n}^{*,p^s_n,p^u_n}\}_{n\in\mathbb{Z}}$ be an $\varepsilon$-gpo, and let $\ul{\wt v}^{*,+}$ denote its positive part. 
\begin{enumerate}
\item  The set $\wt V^{*,s}_0$ is an $s$-admissible manifold at $\wt v^*_0$, and
\[\wt V^{*,s}_0=\{\wt x^*\in \wt \Psi_{x_0}^* (R[p^s_0]):\wt g^{+,*}_{x_0,n}(\wt x^*)\in \wt \Psi_{x_n}^*(R[10Q(x_n)]), \forall n\geq0\}.\]
\item $\wt g^{+,*}_{x_0,1}(\wt V^{*,s}_0)\subset \wt V^{*,s}_1$.
\item We have:
\begin{enumerate}
\item[\rm (3a)] for any $\wt y_1^*,\wt y_2^*\in \wt V^{*,s}_0$ and any $n\geq0$,
\[
|\wt g^{+,*}_{x_0,n}(\wt y_1^*)-\wt g^{+,*}_{x_0,n}(\wt y_2^*)|\leq 8p_0^s\cdot e^{-\frac{2\chi}{3}\cdot t_n(y_1)}, \text{ where }y_1=\Psi_{x_0}(\wt y_1^*);
\]
\item[\rm (3b)]  for any unit vector $w_s\in T_{\wt y^*} \wt V^{*,s}_0$,
\[\begin{aligned}
&\norm{d_{\wt y^*}\wt g^{+,*}_{x_0,n}(w_s)}\leq 8\norm{\wt C(x_0)^{-1}}e^{-\frac{2\chi}{3}\cdot t_n(y)},\quad \forall\, n\geq0,\\
&\norm{d_{\wt y^*} \wt g^{-,*}_{x_{0},n}(w_s)}\geq \frac{1}{8}(p^s_0\wedge p^u_0)^{\frac{\beta}{12}}e^{-\frac{2\chi}{3}\cdot t_n(y)-\frac{\beta\varepsilon}{6}|n|} \quad \forall\, n\leq0,
\end{aligned}\]
where $y=\Psi_{x_0}(\wt y^*)$.
\end{enumerate}

\item\label{sca.hol}The map $\ul{\wt v}^*\mapsto \wt V^{*,s}_0$ is H\"older continuous, that is, there are $K_{G1}>0$ and $\varpi\in(0,1)$, independent of the choice of gpos, such that for all $N\geq0$, if $\ul{\wt v}^*,\ul{\wt w}^*$ are two scaled $\varepsilon$-gpos with $\wt v^*_n=\wt w^*_n$ for $n=0,\dots,N$, then $d_{C^1}(\wt V^{*,s}[\ul{\wt v}^{*,+}],\wt V^{*,s}[\ul{\wt w}^{*,+}])\leq K_{G1}\cdot(\varpi)^N$. 
\end{enumerate}
The set $\wt V^{*,s}[\ul{\wt v}^{*,+}]$ is called the {\rm local scaled stable manifold} of $\ul{\wt v}^{*,+}$.

A symmetric conclusion holds for the unstable case.
\end{pro}

\subsection{The shadowing lemma}

We say that an $\varepsilon$-gpo $\{\wt \Psi_{x_n}^{*,p^s_n,p^u_n}\}_{n\in\mathbb{Z}}$ \emph{shadows} a point $\wt z^*\in\N_{o_j}(\kappa \gamma_0)$, where $x_0\in D_j$, if for all $n\ge 0$:
\[\wt g^{+,*}_{x_0,n}(\wt z^*)\in \wt\Psi_{x_n}^*(R[p^s_n\wedge p^u_n]) \quad \text{and} \quad \wt g^{-,*}_{x_0,-n}(\wt z^*)\in \wt \Psi_{x_{-n}}^*(R[p^s_{-n}\wedge p^u_{-n}]).
\]

Recall that
\[
\Psi_x=\exp_{o_j}\circ \varsigma(o_j)^{-1}\circ \wt\Psi^*_x.
\]
Since $\exp_{o_j}\circ \varsigma(o_j)^{-1}$ is a diffeomorphism from $\N_{o_j}(\kappa \gamma_0)$ to $\widehat{D}_j$, we also say that the $\varepsilon$-gpo $\{\wt \Psi_{x_n}^{*,p^s_n,p^u_n}\}_{n\in\mathbb{Z}}$  \emph{shadows} the point $z=\exp_{o_j}(|X(o_j)|\wt z^*)$ whenever it shadows  $\wt z^*$.
By the definition of $\wt{(\cdot)}^*$ and Lemma~\ref{lem.C1Poincare}~\ref{fb.length}, we have that $|\wt z^*-\wt x_0^*|\le \delta$ implies $\dist(z,x_0)\le 2\delta|X(o_j)|\le \frac{20}{9}\delta|X(x_0)|$.

Similarly, any $s$-admissible manifold $\wt V^*$  with the representing function $G$ corresponds to a unique submanifold  
\[
V:=\exp_{o_j}\circ \varsigma(o_j)^{-1}(\wt V^*)=\Psi_x(\{( \xi, G(\xi)):\xi\in \bR^{d^s(x)}(p^s)\})\subset \wh D_j.
\]
We refer to $V$ as an {\it $s$-admissible manifold on the section}. And we also define $V^{s/u}[\ul{\wt v}^{*,+}]$, $V^{s/u}_k$ as the projections of $\wt V^{*,s/u}[\wt{\ul{v}}^{*,+}]$, $\wt V^{*,s/u}_k$ via $\exp_{o_j}\circ \varsigma(o_j)^{-1}$ respectively.

The following shadowing lemma follows as a consequence of Propositions~\ref{sca.invmnfd} and \ref{pro.graphinters}. We refer the reader to \cite[Theorem~4.2]{Sar13}, \cite[Proposition~4.6]{BCL} and \cite[Proposition~4.6]{LMN} for the proofs.
\begin{pro}\label{pro.shadow}
For all $\varepsilon>0$ small enough, any $\varepsilon$-gpo $\ul{\wt v}^*$ with $\wt v_0^*=\wt \Psi_{x_0}^{*,p^s_0,p^u_0}$ shadows a unique point $\wt z^*=\wt V^{*,s}[\ul{\wt v}^{*,+}]\pitchfork \wt V^{*,u}[\ul{\wt v}^{*,-}]\in \wt \Psi_{x_0}^*(R[\frac{1}{500}(p^s_0\wedge p^u_0)])$. Equivalently,  $\ul{\wt v}^*$  shadows a unique point $ z=V^{s}[\ul{\wt v}^{*,+}]\pitchfork V^{u}[\ul{\wt v}^{*,-}]\in  \Psi_{x_0}(R[\frac{1}{500}(p^s_0\wedge p^u_0)])$. 
\end{pro}

\subsection{Local invariant  manifolds on sections}\label{sec.locinvarmani}

We wish to obtain more properties of admissible manifolds. To this end, we return to  the section $D_j$ and analyze the dynamics on the original manifold $M$.

Recall that $\wt g^{+,*}_x$ is the scaling of the local lift $\wt g_x^+$  of $g^{+}_x$. Thus the above results of admissible manifolds in $\N_{o_j}$ for $\wt g^{+,*}_x$ can be directly projected to $\widehat D_j$ for $g^{+}_x$.
The following proposition is easily verified from Proposition~\ref{sca.invmnfd} by using the relations between $g^{+}_x$ and $\wt g^{+,*}_x$ and \eqref{eq.BdPsi}.

\begin{pro}\label{invmnfd}
The following statements hold for all $\varepsilon>0$ small enough. Let $\ul{\wt v}^*=\{\wt v^*_n\}_{n\in\mathbb{Z}}=\{\wt\Psi_{x_n}^{*,p^s_n,p^u_n}\}_{n\in\mathbb{Z}}$ be an $\varepsilon$-gpo, and let $\ul{\wt v}^{*,+}$ denote its positive part. Assume that for every $n\in\mathbb{Z}$, the base point $x_n$ of $\ul{\wt v}^*$ lies in $D_{j_n}=N_{o_{j_n}}(\gamma_0|X(o_{j_n})|)$.
\begin{enumerate}
\item  The set $ V^{s}_0=V^{s}[\ul{\wt v}^{*,+}]$ is an $s$-admissible manifold on section, and
\[ V^{s}_0=\{x\in  \Psi_{x_0} (R[p^s_0]): g^{+}_{x_0,n}(x)\in  \Psi_{x_n}(R[10Q(x_n)]), \ \forall n\geq0\}.\]
\item $g^{+}_{x_0,1}( V^{s}_0)\subset  V^{s}_1$.
\item We have:
\begin{enumerate}
\item[\rm (3a)] for any $ y_1, y_2\in  V^{s}_0$ and any $n\geq0$,
\[
\dist (g^{+}_{x_0,n}( y_1),g^{+}_{x_0,n}( y_2))\leq 16\cdot p_0^s |X(o_{j_n})|e^{-\frac{2\chi}{3}\cdot t_n(y_1)};
\]
\item[\rm (3b)]  for any unit vector $w_s\in T_{ y}  V^{s}_0$,
\[\begin{aligned}
&\norm{d_y g^{+}_{x_0,n}(w_s)}\leq 32 \norm{ C(x_0)^{-1}}\frac{|X(o_{j_n})|}{|X(o_{j_0})|}e^{-\frac{2\chi}{3}\cdot t_n(y)},\ \forall\, n\geq0,\\
&\norm{d_y  g^{-}_{x_{0},n}(w_s)}\geq \frac{1}{32}(p^s_0\wedge p^u_0)^{\frac{\beta}{12}}\frac{|X(o_{j_n})|}{|X(o_{j_0})|}e^{-\frac{2\chi}{3}\cdot t_n(y)-\frac{\beta\varepsilon}{6}|n|}, \ \forall\, n\leq0.
\end{aligned}\]
\end{enumerate}
\end{enumerate}
The set $V^{s}[\ul{\wt v}^{*,+}]$ is called the {\rm local stable manifold} of $\ul{\wt v}^{*,+}$.

A symmetric conclusion holds for the unstable case.
\end{pro}

We now establish two additional properties of these local invariant manifolds. Firstly, the local stable and unstable manifolds of gpos are related to the local strong stable and unstable manifolds of the flow.
\begin{pro}\label{strmnfd}
The following statement holds for all $\varepsilon>0$ small enough. For any $\varepsilon$-gpo $\wt{\ul{v}}^*=\{\wt\Psi_{x_n}^{*,p^s_n,p^u_n}\}_{n\in\mathbb{Z}}$, let  $G:\bR^{d^s(x_0)}(p_0^s)\to\mathbb{R}^{d^u(x_0)}$ be the representing function of the local stable manifold $V^{s}[\ul{\wt v}^{*,+}]$. Then there exists a function $\Delta:\bR^{d^s(x_0)}(p_0^s)\to\mathbb{R}$ with $\Delta(0)=0$ such that the set $\wh{V}^s:=\{\varphi_{\Delta( \xi)}(\Psi_x( \xi,G( \xi))): \xi\in\bR^{d^s(x_0)}(p_0^s)\}$ satisfies $\dist(\varphi_t(\wh y),\varphi_t(\wh z))< e^{-\frac{\chi\inf R_{\mathbb D}}{4\sup R_{\mathbb D}}t} \cdot |X(x_n)|$ for all $\wh y,\wh z\in \wh{V}^s$, $t\in [t_{n},t_{n+1}]$ and $n\geq0$.
A symmetric conclusion holds for the unstable case.
\end{pro}
\begin{proof}
The proof follows a similar line of reasoning to that of \cite[Proposition~4.8]{BCL} and \cite[Proposition~4.8]{LMN}. The main new ingredient here, adapted to our singular setting, is that we use the uniform estimate $\|d\tau_x\|\cdot |X(x)|<1$ from \eqref{eq.boundtimedg} in place of the global bound $\|d\tau_x\|<L$ (for some constant $L$) that holds in the non-singular case. The detailed proof is given in Appendix~\ref{appsec.coding}.
\end{proof}

Moreover, we have the following coherent property of local scaled stable and unstable manifolds. The proof is the same as \cite[Proposition~4.9]{BCL} and \cite[Proposition~4.9]{LMN}. See also \cite[Proposition~6.4]{Sar13}.

\begin{pro}\label{2mnfd}
The following statement holds for all $\varepsilon>0$ small enough. If $\ul{\wt v}^*=\{{\wt v}^*_n\}_{n\in\mathbb{Z}}$, $\ul{\wt w}^*=\{{\wt w}^*_n\}_{n\in\mathbb{Z}}$ are two $\varepsilon$-gpos with ${\wt v}^*_0=\wt\Psi_{x}^{*,p^s,p^u}$ and ${\wt w}^*_0=\wt\Psi_{x}^{*,q^s,q^u}$, then $V^{s}=V^{s}[\{\wt v^*_n\}_{n\geq0}]$ and $U^{s}=V^{s}[\{\wt w^*_n\}_{n\geq0}]$ are either disjoint, or one is entirely contained within the other. A symmetric conclusion holds for the unstable case.
\end{pro}

%%%%%%%%%%%%%%%%%%%%%%%%%%%%%%%%%%%%%%%%%%%%%%%%%%%%

\section{First coding and inverse problem}\label{sec.firstcod}

\subsection{Coarse graining and the generated symbolic dynamics}\label{sec.Coarse}

In this subsection we construct the first coding. This coding  is usually infinite to one.

Every point $x\in \mathbb D\cap\NUH^{\#}_{\chi}$ lies in uncountably many scaled double charts. To facilitate subsequent arguments, we select countably many suitable charts. This selection method is first proposed in \cite{Sar13}. In \cite{BCL}, the difficulty arising from the non-discrete nature of the flow's time parameter is addressed by relaxing the requirements of condition \ref{item.GPOcond2} of Definition~\ref{dfn.gpo}.

However, our setting differs substantially from the existing literature due to the presence of infinitely many sections. Fortunately, the local finiteness property allows the argument to proceed after appropriate modifications. We get the ``local discreteness'' in Proposition~\ref{pro.SDR}~\ref{SDR.D} below. The local discreteness is weaker than the global discreteness, yet it suffices for subsequent applications. 

Moreover, we emphasize that the following discussion is conducted exclusively within the scaled framework on the normal spaces $\N_{o_j}$.

\begin{pro}\label{pro.SDR}
For all $\varepsilon>0$ small enough,
there exists a countable collection $\mathscr{A}$ of scaled $\varepsilon$-double charts such that:
\begin{enumerate}
\item\label{SDR.D} {\rm Local discreteness:} let $\mathscr{A}_j$ be the collection of scaled $\varepsilon$-double charts of $\mathscr{A}$ contained in $\widehat D_j$; then for every $t>0$ and every $j\in\Gamma$, the set $\{\wt\Psi_{x}^{*,p^s,p^u}\in\mathscr{A}_j: p^s,p^u>t\}$ is finite;
\item \label{Suff} {\rm Sufficiency:} for every $x \in \mathbb D\cap \NUH_{\chi}^{\#}$, there is a regular $\varepsilon$-gpo $\ul{v}^*\in\mathscr{A}^{\mathbb{Z}}$ that shadows $x$;
\item {\rm Relevance:} for every $v^*\in\mathscr{A}$, there is an  $\varepsilon$-gpo $\ul{v}^*\in\mathscr{A}^{\mathbb{Z}}$ such that $v^*_0=v^*$ and $\ul{v}^*$ shadows a point in $\mathbb D\cap\NUH_{\chi}^{\#}$. 
\end{enumerate}
\end{pro}
\begin{proof}
The construction proceeds via a pre-compactness argument, originally developed in \cite{Sar13}.
This argument is also used in \cite{Ben, LS, ALP, BCL, LMN}. 

Let $\mathbb N_0:=\mathbb Z^+\cup\{0\}$. Define  $\widehat Y:=
{\mathbb D}^3\times {\rm GL}(d-1,\mathbb{R})^3\times(0,1]^2\times \{0,1,\dots,d-1\}$.
For any $x\in \mathbb D\cap \NUH_{\chi}^{\#}$, denote $\Omega (x):=(\ul{x},\ul{\wt C},\ul{Q},d^s(x))\in \widehat Y$, where
\[
\ul{x}:=(f^{-1}(x),x,f(x)),\quad \ul{\wt C}:=(\wt{C}(f^{-1}(x)), \wt{C}(x), \wt{C}(f(x)))
\quad\text{and}\quad \ul{Q}:=(Q(x),q(x)) .
\] 

Define $Y:=\{\Omega(x):x\in \mathbb D\cap \NUH_{\chi}^{\#}\}$ and $Y^j:=\{\Omega(x)\in Y: x\in D_j\}$. We further partition $Y^j$ into disjoint subsets $Y^j_{\ul{\ell},m,k,d^s}$ for $j\in \Gamma$, $\ul{\ell}:=(\ell_{-1},\ell_{0},\ell_{1})\in\mathbb{N}_0^3$, $m,k\in\mathbb{N}_0$, and $d^s\in \{0,1,\dots,d-1\}$:
\[
Y^j_{\ul{\ell},m,k,d^s}:=\left\{\Omega(x)\in Y^j:
\begin{subarray}{c}
e^{\ell_i}\leq\norm{\wt{C}(f^i(x))^{-1}}< e^{\ell_i+1},\ i=0,\pm1\\
	e^{-m-1}\leq Q(x)<e^{-m}\\
	e^{-k-1}\leq q(x)< e^{-k}\\
	d^s(x)=d^s
\end{subarray}
\right\}.
\]
It is immediate that $Y^j=\bigcup_{\begin{subarray}{c} \ul{\ell}\in\mathbb{N}_0^3, m,k\in\mathbb{N}_0\\ d^s\in \{0,1,\dots,d-1\}\end{subarray}} Y^j_{\ul{\ell},m,k,d^s}$. We  perform the pre-compactness argument on each $Y^j_{\ul{\ell},m,k,d^s}$.

\smallskip
\smallskip
\nt{\bf Claim.} {\it
$Y^j_{\ul{\ell},m,k,d^s}$ is pre-compact in $\widehat Y$.
}
\begin{proof}
For any $j\in \Gamma$, $\ul{\ell}\in\mathbb{N}_0^3$, and $m,k\in\mathbb{N}_0$, we have the following:
\begin{itemize}
\item By the local finiteness from Theorem~\ref{thm.poincaresection}~\ref{se.fini}, there are at most $I_{\max}$ components of $D_{j'}$ satisfying $f(D_j)\cap D_{j'}\neq\emptyset$ or $f^{-1}(D_{j})\cap D_{j'}\neq\emptyset$. Define $\mathbb D_j$ as the union of $D_j$ and such $D_{j'}$. So $\ul x\in \mathbb D_j^3$ for any $\Omega(x)\in Y^j_{\ul{\ell},m,k,d^s}$. Note that  $\mathbb D_j^3$ is a compact set of $\mathbb D^3$. 
\item For each 
$i=0, \pm1$, $\norm{\wt C(f^i(x))}\le 2$ and $\norm{\wt{C}(f^i(x))^{-1}}\le e^{\ell_i+1}$. Thus $\wt{C}(f^{i}(x))$ lies in a compact set of ${\rm GL}(d-1,\mathbb{R})$. 
\item $\ul{Q}=(Q(x),q(x))$ is contained in the compact set $[e^{-m-1},e^{-m}]\times[e^{-k-1},e^{-k}]\subset [0,1]^2$.
\end{itemize}
Consequently, we have that $Y^j_{\ul{\ell},m,k,d^s}$ is a pre-compact set.
\end{proof}

The pre-compactness implies that there exists a finite set $Z^j_{\ul{\ell},m,k,d^s}\subset Y^j_{\ul{\ell},m,k,d^s}$ such that for every $\Omega(x)\in Y^j_{\ul{\ell},m,k,d^s}$, there exists $\Omega(y)\in Z^j_{\ul{\ell},m,k,d^s}$ satisfying:
\begin{enumerate}
\item[(a)] $|\wt{f^i(x)}^*-\wt{f^i(y)}^*|+\|\wt{C}(f^i(x))-\wt{C}(f^i(y))\|<\frac{1}{2}\cdot q(x)^8$ for $i=0,\pm1$;
\item[(b)] $\frac{Q(x)}{Q(y)}=e^{\pm\varepsilon/3}$ and $\frac{q(x)}{q(y)}=e^{\pm\varepsilon/3}$;
\item[(c)] $d^s(x)=d^s(y)$.
\end{enumerate}

\setlist[enumerate]{label={\rm (CG\arabic*)}}
Now take $\mathscr{A}_j$ be the collection of scaled double charts $\Psi_{x}^{*,p^s,p^u}$ such that:
\begin{enumerate}
\item\label{item.CG1} $\Omega(x)\in Z^j_{\ul{\ell},m,k,d^s}$ for some $(\ul{\ell},m,k,d^s)\in\mathbb{N}_0^3\times\mathbb{N}_0\times\mathbb{N}_0\times\{0,1,\dots,d-1\}$;
\item\label{item.CG2} $0<p^s,p^u\leq\varepsilon Q(x)$ and $p^s,p^u\in I_{\varepsilon,q(x)}:=\{e^{-n\varepsilon^2q(x)}:n\in \mathbb N_0\}$;
\item\label{item.CG3} $\frac{p^s\wedge p^u}{q(x)}=e^{\pm(\mathfrak{Y}+1)}$, where $\mathfrak{Y}$ is given by Proposition~\ref{lem.pqcomp}~(2).
\end{enumerate}
\setlist[enumerate]{label={\rm (\arabic*)}}

Finally, let $\mathscr{A}:=\bigcup_{j\in\Gamma}\mathscr{A}_j$. It is clear that $\mathscr{A}$ is a countable set. Next we show  that $\mathscr{A}$ satisfies the three conditions.

\begin{itemize}
  \item \emph{Local Discreteness}:
\end{itemize}

Fix $j\in\Gamma$ and $t>0$, assume that $\wt \Psi_x^{*,p^s,p^u}\in \mathscr{A}_j$ with $p^s,p^u>t$ and $\Omega(x)\in Z^j_{\ul{\ell},m,k,d^s}$.

Since $e^{\ell_0} \leq \|\wt C(x)^{-1}\|\leq 2\|C(x)^{-1}\| < 2Q(x)^{-1} < 2t^{-1}$, we have that $\ell_0 < |\log (2t)|$. For $i=\pm1$, it follows from Lemma~\ref{lem.propofC}~\ref{fC.Cratio} that  $e^{\ell_i} \leq \|\wt C(f^i(x))^{-1}\| \leq 2\|C(f^i(x))^{-1}\| \leq 2e^{6\rho}\|C(x)^{-1}\| <2 e^{6\rho}t^{-1}$. Thus $\ell_{-1},\ell_1 <6\rho+|\log (2t)|$. So $\ul \ell$ has only finitely many choices.

Since $e^{-m} > Q(x) > t$, we conclude that $m < |\log t|$. So there are only finitely many possible values for $m$.

Since $e^{-k} > q(x) \geq e^{-\mathfrak Y -1}(p^s \wedge p^u) > e^{-\mathfrak Y-1}t$, we have that $k \leq |\log t|+\mathfrak Y+1$ have only finitely many choices.

In summary, there exist only finitely many $(\ul{\ell},m,k)$ such that $\Omega(x)\in Z^j_{\ul{\ell},m,k,d^s}$. Moreover, since $p^s,p^u\in I_{\varepsilon,q(x)}$,  $p^s,p^u>t$, $p^s$ and $p^u$ have only finitely many choices. Then we finally obtain that $\{\wt\Psi_{x}^{*,p^s,p^u}\in\mathscr{A}_j: p^s,p^u>t\}$  is a finite set.

\begin{itemize}
  \item \emph{Sufficiency}:
\end{itemize}

Let $x\in \mathbb D \cap \NUH_{\chi}^{\#}$. We first have the following lemma. The proof is analogous to that of \cite[Theorem 5.1]{BCL} (in the ``Proof of sufficiency'' part), since the argument is unaffected by the scaling process and the existence of infinitely many $D_j$.

\begin{lem}\label{lem.chartsA}
There are sequences $\{x_n\}_{n\in\bZ}$, $\{p^s_n\}_{n\in\bZ}$ and $\{p^u_n\}_{n\in\bZ}$ such that $\frac{p_n^s \wedge p_n^u}{q(f^n(x))}= e^{\pm(\mathfrak Y+4\varepsilon/3 )}$, and $\{\wt\Psi_{x_n}^{*,p^s_n,p^u_n}\}_{n\in\mathbb{Z}}\in  \mathscr A^{\mathbb Z}$ is an $\varepsilon$-gpo that shadows $x$. 
\end{lem}

It remains only to prove that the $\varepsilon$-gpo $\{\wt\Psi_{x_n}^{*,p^s_n,p^u_n}\}_{n\in\mathbb{Z}}$ constructed in Lemma~\ref{lem.chartsA} is regular.
By \ref{ite.NUH5}, there exist $\gamma_x, \delta_x>0$ and a sequence $t^+_{x,i}\to +\infty$ such that $\varphi_{t^+_{x,i}}(x)\not\in M(\Sing(X), \gamma_x)$ and $q(\varphi_{t^+_{x,i}}(x))\ge \delta_x$. According to the construction of $\mathbb D$, there exists a subsequence $n_i\to +\infty$ such that $f^{n_i}(x)=\varphi_{t^+_{x,i}+s_i}(x)$ for some $s_i\in [0,2\rho)$. It follows from Lemma~\ref{lem.estimateQ}~\ref{est.qrat} that $q(f^{n_i}(x))\ge e^{-2\rho\varepsilon}\delta_x$. Since $\frac{p_n^s \wedge p_n^u}{q(f^n(x))}= e^{\pm(\mathfrak Y+4\varepsilon/3 )}$, we have $p_{n_i}^s, p_{n_i}^u>e^{-(2\rho\varepsilon+\mathfrak Y+4\varepsilon/3)}\delta_x/2$ for all $i$.

Note that there is $\gamma_x'>0$ such that $f^{n_i}(x)\not\in M(\Sing(X), \gamma_x')$.
 By the local finiteness of $\mathbb D$,  there are only finitely many components of $\mathbb D$ that intersect $M\setminus M(\Sing(X), \gamma_x')$. Then combining with the local discreteness of $\mathscr A$, we know that $\wt\Psi_{x_{n_i}}^{*,p^s_{n_i},p^u_{n_i}}$ repeats infinitely many times when $i\to +\infty$. The backward case is the same. This verifies the regularity of $\{\wt\Psi_{x_n}^{*,p^s_n,p^u_n}\}_{n\in\mathbb{Z}}$.

We finish the proof of the sufficiency of $\mathscr{A}$.

\begin{itemize}
  \item \emph{Relevance}:
\end{itemize}

We say that 
$\wt\Psi_{x_n}^{*,p^s_n,p^u_n}\in  \mathscr A$ is {\it relevant} if there exists $\ul{\wt v}^*\in  \mathscr A^{\mathbb Z}$ such that $\wt v_0^*=\wt\Psi_{x_n}^{*,p^s_n,p^u_n}$ and $\ul{\wt v}^*$ shadows a point in $\mathbb D\cap \NUH^{\#}_\chi$.
Replacing $\mathscr A$ with the subset consisting of all relevant elements in $\mathscr A$; the resulting set is then locally discrete, sufficient and relevant.

We complete the proof.
\end{proof}

Let $(\Sigma,\sigma)$ be the TMS generated by the graph with  vertex set $\mathscr{A}$ given by Proposition~\ref{pro.SDR} and with
 edges defined by the relation $\wt v^*\ov \wt w^*$. Then each element $\ul{\wt v}^*=\{\wt v^*_n\}_{n\in\mathbb{Z}}\in\Sigma$ assigns an $\varepsilon$-gpo. 
By Proposition~\ref{pro.shadow}, the following map is well-defined:
\[\begin{aligned}
&\wt\pi^*:\Sigma\to \bigcup_{j\in\Gamma}\N_{o_j}\\
&\wt\pi^*(\ul{\wt v}^*):=\wt V^{*,s}[\{\wt v^*_n\}_{n\geq0}]\cap \wt V^{*,u}[\{\wt v^*_n\}_{n\leq0}].
\end{aligned}\]

We also define $\pi:\Sigma\to\widehat {\mathbb D}$ as $\pi(\ul{\wt v}^*):= \exp_{o_j}(|X(o_j)|\cdot\wt\pi^*(\ul{\wt v}^*))$, where $\wt\pi^*(\ul{\wt v}^*)\in \N_{o_j}$.

\begin{pro}\label{1code}
The following holds for all $\varepsilon>0$ small enough.
\begin{enumerate}
\item Each $\wt v^*\in\mathscr{A}$ has finite degrees, hence $\Sigma$ is locally compact.
\item $\wt\pi^*:\Sigma\to\bigcup_{j\in\Gamma}\N_{o_j}$ is H\"older continuous.
\item\label{1code.incl}  $\mathbb D\cap \NUH_{\chi}^{\#}\subset \pi(\Sigma^{\#})$, where $\Sigma^{\#}$ is the regular part of $\Sigma$.
\end{enumerate}
\end{pro}
\begin{proof}
The first item is a direct consequence of Lemma~\ref{lem.edge} and the local discreteness of each $\mathscr{A}_j$ (Proposition~\ref{pro.SDR}~\ref{SDR.D}).
The second item follows from Proposition~\ref{sca.invmnfd}~\ref{sca.hol} and the last item follows from Proposition~\ref{pro.SDR}~\ref{Suff}.  
\end{proof}

%%%%%%%%%%%%%%%%%%%%%%%%%%%%%%%%%%%%%%%%%%%%%%%%%%%%%%%%%%%%%%%%%%%

\subsection{Improvement lemma and inverse theorem}\label{sec.inverse}

In this subsection,  we investigate whether the quantification of hyperbolicity remains consistent at each section when an orbit is shadowed by a gpo. This is a crucial property toward obtaining a finite-to-one coding, which was originally developed in \cite{Sar13}. Our approach closely follows the framework in \cite[\S\,6]{BCL} and \cite[\S\,6]{LMN}. 

Most of our discussions have taken place in the normal space so far (except in \S\,\ref{sec.locinvarmani}). In the present section, we return to the original manifold $M$. This transition is necessary because the quantification of hyperbolicity must be evaluated along entire orbits, rather than discretely at sections. This is a groundbreaking feature introduced in \cite{BCL}. We will thus frequently handle transformations between scaling multipliers at different points.

For high dimensional systems, \cite{Ben} introduced a method for comparing the hyperbolicity. It is also used in \cite{ALP, LMN}.

\begin{lem}\label{lem.comH}
The following statements hold for all $\varepsilon>0$ small enough. Let $\wt v^*=\wt \Psi^{*,p^s,p^u}_{x}$ be an $\varepsilon$-double chart with $x\in D_j$, and $V\subset \wh D_j$ be an $s$-admissible manifold on the section at $\wt v^*$.  Then for any $y\in V \cap \Psi_x(R[p^s \wedge p^u])$, there exists a linear isomorphism $\pi^s_{y,x}:T_yV\to\H^s_x$ such that 
\begin{enumerate}
\item\label{comH1} $\big\|d_x\exp_{o_j}^{-1}\circ \pi^s_{y,x}-d_y\exp_{o_j}^{-1}\big\|<\varepsilon^{1/16}\cdot (p^s \wedge p^u)^{5\beta/16}$;
\item\label{comH2} $\big\|d_x\exp_{o_j}^{-1}\circ \pi^s_{y,x}\circ d_{\wt{y}}\exp_{o_j}-{\rm Id}\big\|<\varepsilon^{1/16}\cdot (p^s \wedge p^u)^{5\beta/16}$, where $\wt{y}=\exp_{o_j}^{-1}(y)$;
\item\label{comH3} $\big\|\pi^s_{y,x}-T_{y,x}\big\|<\varepsilon^{1/16}\cdot (p^s \wedge p^u)^{5\beta/16}$.
\end{enumerate}

A similar statement holds for the unstable case.
\end{lem}
\begin{proof}
Let $G$ be the representing function of the $s$-admissible manifold $V$. Take an arbitrary $y\in V \cap \Psi_x(R[p^s \wedge p^u])$. Denote $\theta:=(\theta^s,G(\theta^s))=\Psi_x^{-1}(y)$. Since $V$ is the image of the graph of $G$ under $\Psi_x$, for any vector $\omega\in T_yV$, there exists a unique $\xi\in\mathbb{R}^{d^s(x)}$ such that $d_{\theta}\Psi_x(\xi,d_{\theta^s}G(\xi))=\omega$. Define $\pi^s_{y,x}(\omega):=d_0\Psi_x(\xi,0)$.  Recall that $C(x)(\mathbb{R}^{d^s(x)}\times \{0\})=\H^s_x$ and $d_0\Psi_x=\varsigma(o_j)^{-1}\circ C(x)$, where $\varsigma(o_j)(v)=\frac{v}{|X(o_j)|}$. Thus $\pi^s_{y,x}$ is a linear isomorphism from $T_{y}V$ to $\H^s_x$.

It follows from \eqref{eq.scalepesinchart} and \eqref{equ.pesinchart} that $d_y\exp_{o_j}^{-1}\circ d_{\theta}\Psi_x= \varsigma(o_j)^{-1}\circ d_{\theta}\wt\Psi_x^*=\varsigma(o_j)^{-1}\circ \wt C(x)$. Thus for any $\omega\in T_yV$, we have
\[
\big|d_x\exp_{o_j}^{-1}\circ \pi^s_{y,x}(\omega)-d_y\exp_{o_j}^{-1}(\omega)\big|=\big|d_{0}\wt\Psi_x^*(\xi,0) -d_{\theta}\wt\Psi_x^*(\xi,d_{\theta^s}G(\xi))\big|\cdot |X(o_j)|\le \big\|\wt C(x)\big\|\cdot|d_{\theta^s}G(\xi)|\cdot|X(o_j)|.
\]

By \ref{ite.AMcomd2} and \ref{ite.AMcomd3} in the definition of admissible manifold, we conclude that $\norm{dG}\leq (p^s\wedge p^u)^{\beta/3}$ on $\bR^{d^s(x)}(p^s \wedge p^u)$. It follows from \eqref{eq.BdPsi} that $|\xi|\le |d_y\Psi_x^{-1}(\omega)|\le 4|X(o_j)|^{-1} \big\|C(x)^{-1}\big\|\cdot|\omega|$.
Combining these estimates with \eqref{eq.bdPsi} and Lemma~\ref{lem.estimateQ}~\ref{est.QC}, and noting that $p^s\wedge p^u\le \varepsilon Q(x)$, we have
\begin{equation}\label{equ.dexppi-dexp}
	\big|d_x\exp_{o_j}^{-1}\circ \pi^s_{y,x}(\omega)-d_y\exp_{o_j}^{-1}(\omega)\big|\le 8(p^s\wedge p^u)^{\beta/3}\big\|C(x)^{-1}\big\|\cdot|\omega|<8\varepsilon^{\beta/48}\varepsilon^{1/8}\cdot (p^s\wedge p^u)^{5\beta/16}\cdot|\omega|.
\end{equation}
 Thus Item (1) holds for all $\varepsilon>0$ small enough. Since $\|d\exp\|\le 2$, Item (2) also follows from \eqref{equ.dexppi-dexp}.

Next we prove Item (3). Since $\|d^2\exp\|$ is uniformly bounded, there exists a constant $\ol K>0$ independent of both $x$ and $y$ such that $d_{\rm Sas}(\pi^s_{y,x}(\omega),\omega)\leq \ol K \big|d_x\exp_{o_j}^{-1}\circ \pi^s_{y,x}(\omega)-d_y\exp_{o_j}^{-1}(\omega)\big|$ for all $\omega\in T_yV$. By \eqref{eq.sasaki}, we have
\[
|\pi^s_{y,x}(\omega)-T_{y,x}(\omega)|\leq 2\ol K \big|d_x\exp_{o_j}^{-1}\circ \pi^s_{y,x}(\omega)-d_y\exp_{o_j}^{-1}(\omega)\big|.
\]
So Item (3) follows from \eqref{equ.dexppi-dexp}.
\end{proof}

Recall that $s(x,v^s)=|v^s|'=|C(x)^{-1}v^s|$ was introduced in Definition~\ref{dfn.NUH} for any $x\in \NUH_\chi$ and $v^s\in\H^{s}_x$. In fact, the function $s(\cdot,\cdot)$ can be formally extended by defining
\begin{equation}\label{equ.extenddefs}
	s(x,v):=2e^{\rho}\Big(\int_0^{\infty}e^{2\chi t}|\Phi^{*}_t(v)|^2dt\Big)^{\frac{1}{2}}
\end{equation}
for all $x\in M\setminus\Sing(X)$ and all $v\in \H_x$. Note that $s$ may take infinite values. However, in the subsequent discussion, we only consider $(x,v)$ for which $s(x,v)$ is finite.

Let $\wt v_0^*=\wt\Psi^{*,p_0^s,p_0^u}_{x_0}$ with $x_0\in D_{i}$ and $\wt v_1^*=\wt\Psi^{*,p_1^s,p_1^u}_{x_1}$ with $x_1\in D_{j}$ satisfy $\wt v_0^*\ov\wt v_1^*$. 
Take an arbitrary $s$-admissible manifold $\wt V_1^{*,s}\in \mathscr{M}^s[\wt v_1^*]$. Denote $\wt V_0^{*,s}=\mathscr{F}^s_{\wt v_0^*,\wt v_1^*}(\wt V_1^{*,s})\in \mathscr{M}^s[\wt v_0^*]$. Let $V_0^s$ and $V_1^s$ be the projections of $\wt V_0^{*,s}$ and $\wt V_1^{*,s}$ via $\exp_{o_i}\circ \varsigma(o_i)^{-1}$ and $\exp_{o_j}\circ \varsigma(o_j)^{-1}$, respectively.
For any $y_0\in V_0^s\cap \Psi_{x_0}(R[p_0^s \wedge p_0^u])$ satisfying $y_1:=g_{x_0}^+(y_0)\in V_1^s\cap \Psi_{x_1}(R[p_1^s \wedge p_1^u])$, and any $\omega_0\in T_{y_0}V_0^s$, by Proposition~\ref{hpsi}~\ref{Phi.rel} and Proposition~\ref{invmnfd}~(3), we have $\omega_1:=\Phi^*_{\tau_1}(\omega_0)\in T_{y_1}V_1^s$, where $\tau_1:=\tau_{x_0}(y_0)\in (0,2\rho)$ satisfying $\varphi_{\tau_1}(y_0)=y_1$ (see \S\,\ref{subsec.Holomap}).

We now state the following important Improvement Lemma, which asserts that the ratio of Lyapunov norms on the stable direction improves along backward orbits.

\begin{lem}[Improvement Lemma]\label{impv.srat}
The following holds for all $\varepsilon>0$ small enough. Let $\wt v_k^*=\wt\Psi^{*,p_k^s,p_k^u}_{x_k}$, $V_k^{*,s}\in \mathscr{M}^s[\wt v_k^*]$, $y_k\in V_k^s\cap \Psi_{x_k}(R[p_k^s \wedge p_k^u])$, and $\omega_k\in T_{y_k} V_k^s$, $k=0,1$, be as above.
  For any $\zeta\geq \sqrt{\varepsilon}$, if $\frac{s(y_1,\omega_1)}{s(x_1,\pi_{y_1,x_1}^s(\omega_1))}=e^{\pm\zeta}$ then 
  \[
  \frac{s(y_0,\omega_0)}{s(x_0,\pi_{y_0,x_0}^s(\omega_0))}=e^{\pm(\zeta-Q(x_0)^{\frac{\beta}{4}})}.
  \]
     
 A similar statement holds for the unstable case.
\end{lem}
\begin{proof}
	The proof strategy follows essentially the same lines as that of \cite[Lemma~6.2]{BCL} and \cite[Lemma~6.3]{LMN}. Note that $\omega_0$ and $\omega_1$ are scaled vectors. The proof details are placed in Appendix~\ref{appsec.coding}.
\end{proof}

The following lemma shows that $s(\cdot,\cdot)$ (resp. $u(\cdot,\cdot)$) is uniformly bounded on the corresponding $s$-admissible (resp. $u$-admissible) manifold. Such result for the high-dimensional case was first proved in \cite{Ben}. Here we adopt the arguments developed in \cite[Lemma~6.6]{ALP} and \cite[Lemma 6.5]{LMN}, and the proof is provided in Appendix~\ref{appsec.coding}. 

\begin{lem}\label{impv.fin}
Let $\ul{\wt v}^*=\{\wt\Psi^{*,p^s_n,p^u_n}_{x_n}\}_{n\in\mathbb{Z}}$ be a gpo and $V^s:=V^s[\ul{\wt v}^{*,+}]$. If $V^s\cap\NUH_{\chi}^{\#}\neq \emptyset$, then we have $\sup_{y\in V^s}s(y)<\infty$. A similar statement holds for the unstable case.
\qed
\end{lem}

Define the roof function $r:\Sigma\to\mathbb{R}$ by $r(\ul{\wt v}^*):=\tau_{x_0}(z)\geq\inf R_{\widehat{\mathbb D}}>0$, where $\ul{\wt v}^*=\{\wt\Psi_{x_n}^{*,p^s_n,p^u_n}\}_{n\in\mathbb{Z}}\in\Sigma$ and $z:=\pi(\ul{\wt v}^*)$ is the shadowed point in $\widehat{\mathbb D}$.  Let $r_n$ be the $n$-th Birkhoff sum with respect to the left shift $\sigma$. Then for any $n\in\mathbb{Z}$, $\varphi_{r_n(\ul{\wt v}^*)}(z)=\pi(\sigma^{n}(\ul{\wt v}^*))$ and $g^{+}_{x_n}(\varphi_{r_n(\ul{\wt v}^*)}(z))=\varphi_{r_{n+1}(\ul{\wt v}^*)}(z)$. 
To simplify notations, we denote $z_n:=\varphi_{r_n(\ul{\wt v}^*)}(z)$ for all $n\in\mathbb{Z}$.

Based on Lemma~\ref{impv.srat} and \ref{impv.fin}, we can get the following inverse theorem, which plays a key role in constructing finite-to-one coding. The proof 
is the same as that of \cite[Lemma~6.1]{BCL} and \cite[Lemma~6.1]{LMN}.

\begin{thm}\label{invs}
For all $\varepsilon>0$ small enough, if $\ul{\wt v}^*=\{\wt \Psi_{x_n}^{*,p^s_n,p^u_n}\}_{n\in\mathbb{Z}}\in\Sigma^{\#}$ and 
 $z=\pi(\ul{\wt v}^*)\in \widehat{\mathbb D}$, then $z\in\NUH^{\#}_{\chi}$ and 
\begin{enumerate}
\item $\dist(z_n,x_n)<\frac{1}{50}(p^s_n\wedge p^u_n)|X(x_n)|$.

\item\label{invs.QCcomp} $\frac{\norm{C(x_n)^{-1}}}{\norm{C(z_n)^{-1}}}=e^{\pm2\sqrt{\varepsilon}}$. Consequently, $\frac{Q(x_n)}{Q(z_n)}=e^{\pm\sqrt[3]{\varepsilon}}$.

\item\label{invs.pspucomp} $\frac{p^s_n}{p^s(z_n)}=e^{\pm\sqrt[3]{\varepsilon}}$ and $\frac{p^u_n}{p^u(z_n)}=e^{\pm\sqrt[3]{\varepsilon}}$.

\item $\Psi^{-1}_{x_n}\circ\Psi_{z_n}=(\wt\Psi^*_{x_n})^{-1}\circ \wt\Psi^*_{z_n}=O_n+\delta_n+\Xi_n$ on $R[10Q(x_n)]$, where $\delta_n<\frac{1}{50}(p^u_n\wedge p^s_n)$, $O_n$ is an orthogonal linear map preserving the splitting $\mathbb{R}^{d^s(x_n)}\times \mathbb{R}^{d^u(x_n)}$, $\Xi_n$ is an affine vector field such that $\Xi_n(\ul{0})=\ul{0}$ and $\Vert d\Xi_n\Vert_{C^0}<5\varepsilon^{1/2}$ on $R[10Q(x_n)]$. A similar statement holds for $\Psi_{z_n}^{-1}\circ\Psi_{x_n}$.
\end{enumerate}
\end{thm}

%%%%%%%%%%%%%%%%%%%%%%%%%%%%%%%%%%%%%%%%%%%%%%%%%%%%%%%%%%%%%%%%%%%

\section{Second coding}\label{sec.finitetoone}

In the preceding sections, we have constructed a symbolic system $(\Sigma,\sigma)$ with countably many symbols. Each bi-infinite symbol sequence corresponds to a gpo as defined in \S\,\ref{sec.gpo}, and determines a unique point in $\widehat{\mathbb D}$ via the shadowing lemma. Furthermore, \S\,\ref{sec.inverse} establishes that the quantification of hyperbolicity of the gpo and its shadowed point is essentially identical.

This section is devoted to deriving our main theorem from the preceding results, following the methodologies developed in \cite{BCL} and \cite{LMN}. Crucially, all ingredients required in their proofs for non-singular flows have now been fully established in our framework. We further observe that the infinitude of components of the global section does not affect most of the arguments. Accordingly, we state the corresponding results without proof.

Special attention must be paid to Proposition~\ref{proL}~\ref{finL}, which relies on the local finiteness  of $\widehat{\mathbb D}$. This marks a main difference from the non-singular flow setting.

\subsection{A countable cover $\mathscr{L}$}
 
Let $\varepsilon>0$ be small enough. Denote $\mathscr{L}:=\{Z(\wt v^*):\wt v^*\in\mathscr{A}\}$ and $Z(\wt v^*):=\{\pi(\ul{\wt v}^*):\ul{\wt v}^*\in\Sigma^{\#} \text{ and } \wt v^*_0=\wt v^*\}$. 
For notational simplicity, we continue to denote the union $\bigcup_{\wt v^*\in\mathscr{A}}Z(\wt v^*)$ by the same symbol $\mathscr{L}$. No ambiguity will arise in the sequel, as the intended meaning will always be clear from context. We also adopt the shorthand $Z:=Z(\wt v^*)$ whenever no confusion can arise.

Fix any $x\in Z$, and take some (and hence any) $\ul{\wt v}^*=\{\wt v^*_n\}_{n\in\mathbb{Z}}\in\Sigma^{\#}$ satisfying $\pi(\ul{\wt v}^*)=x$ and $\wt v^*_0=\wt v^*$. The intersection $W^s(x,Z):=V^s(x,Z)\cap Z$ is called the \emph{$s$-fiber of $x$ in $Z$}, where $V^s(x,Z):=V^s[\{\wt v^*_n\}_{n\geq 0}]$.  Similarly, the \emph{$u$-fiber of $x$ in $Z$} is defined by $W^u(x,Z):=V^u[\{\wt v^*_n\}_{n\leq 0}]\cap Z$, where $V^u(x,Z):=V^u[\{\wt v^*_n\}_{n\leq 0}]$. 

 We are now in a position to define the Poincar\'e return map on $\mathscr{L}$. For any $x=\pi(\ul{\wt v}^*)\in\mathscr{L}$ with $\ul{\wt v}^*\in\Sigma^{\#}$, we have $\varphi_{r_n(\ul{\wt v}^*)}(x)=\pi(\sigma^n(\ul{\wt v}^*))\in\mathscr{L}$ for all $n\in\mathbb{Z}$. Define the return time function $r_{\mathscr{L}}:\mathscr{L}\to(0,\rho)$ by $r_{\mathscr{L}}(x):=\min\{t>0:\varphi_t(x)\in\mathscr{L}\}$, and then define the Poincar\'e return map $H:\mathscr{L}\to\mathscr{L}$ by $H(x)=\varphi_{r_{\mathscr{L}}(x)}(x)$.
 
\begin{pro}\label{proL}
The following statements hold:
\begin{enumerate}
\item\label{proL.cov} $\mathscr{L}$ is a cover of $\mathbb D\cap \NUH_{\chi}^{\#}$.
\item\label{finL} For every $Z\in\mathscr{L}$,  
\[
\#\{Z'\in\mathscr{L}:\varphi_{[-\rho,\rho]}(Z)\cap Z'\neq\emptyset\}<\infty.
\]

In particular,
\[
\#\{Z'\in\mathscr{L}:\big(\bigcup_{n=-1,0,1}H^n(Z)\big)\cap Z'\neq\emptyset\}<\infty.
\]
\item\label{proL.inter} For every $Z\in\mathscr{L}$ and any $x, y\in Z$, the intersection $W^s(x,Z)\cap W^u(y,Z)$ consists of exactly one point in $Z$. More precisely, if $\ul{\wt v}^*,\ul{\wt u}^*\in\Sigma$ satisfy that $\wt v^*_0=\wt u^*_0$, $\pi(\ul{\wt v}^*)=x$ and $\pi(\ul{\wt u}^*)=y$ (so that $Z=Z(\wt v^*_0)=Z(\wt u^*_0)$), then the unique intersection point of $W^s(x,Z)$ and $W^u(y,Z)$ is exactly $\pi(\ul{\wt w}^*)$, where $\ul{\wt w}^*=\{\wt w^*_n\}_{n\in\mathbb{Z}}$ is defined by
\[
\wt w^*_n=\left\{
\begin{array}{lcl}
\wt v^*_n&&n\geq 0,\\
\wt u^*_n,&&n<0.
\end{array}\right.
\] 
\item\label{proL.coin} Any two $s$-fibers (resp. $u$-fibers) are either identical or disjoint.
\item\label{proL.hyp} If $x=\pi(\ul{\wt v}^*)\in\mathscr{L}$ with $\ul{\wt v}^*=\{\wt v^*_n\}_{n\in\mathbb{Z}}=\{\wt\Psi^{*,p^s_n,p^u_n}_{x_n}\}_{n\in\mathbb{Z}}\in\Sigma^{\#}$, then
\[
g^{+}_{x_0}\big(W^s(x,Z(\wt v^*_0))\big)\subset W^s\big(g^{+}_{x_0}(x), Z(\wt v^*_1)\big),\
g^{-}_{x_1}\big(W^u(g^{+}_{x_0}(x),Z(\wt v^*_1))\big)\subset W^u\big(x, Z(\wt v^*_0)\big). 
\]
\end{enumerate}
\end{pro}
\begin{proof}
 As noted at the beginning of this section, we focus exclusively on the second item. The proofs of the remaining items are identical to those of \cite[Proposition 7.1]{BCL} and \cite[Proposition 7.1]{LMN}.
 
 Assume that $Z=Z(\wt\Psi_x^{*,p^s,p^u})$, $Z'=Z(\wt\Psi_y^{*,q^s,q^u})\in \mathscr{L}$ satisfy $\varphi_{[-\rho,\rho]}(Z)\cap Z'\neq\emptyset$. Then there exists $z\in Z$ such that $z'=\varphi_t(z)\in Z'$ for some $t\in [-\rho,\rho]$. By Theorem~\ref{invs}~\ref{invs.pspucomp} and Lemma~\ref{lem.pqcomp}~\ref{ite.pqsucomp}, we have $\frac{p^s\wedge p^u}{q(z)}=e^{\pm(\sqrt[3]{\varepsilon}+\mathfrak{Y})}$ and $\frac{q^s\wedge q^u}{q(z')}=e^{\pm(\sqrt[3]{\varepsilon}+\mathfrak{Y})}$. Moreover, Lemma~\ref{lem.estimateQ}~\ref{est.qrat} gives that $\frac{q(z)}{q(z')}= e^{\pm\rho\varepsilon}$. Combining these estimates yields $\frac{p^s\wedge p^u}{q^s\wedge q^u}=e^{\pm(2\sqrt[3]{\varepsilon}+2\mathfrak{Y}+\rho\varepsilon)}$. 
 
 For any $j\in \Gamma$, by the local discreteness of $\mathscr{A}_j$ (see Proposition~\ref{pro.SDR}), the set $\{\wt\Psi^{*,q^s,q^u}_y\in\mathscr{A}_j:q^s\wedge q^u\geq e^{-(2\sqrt[3]{\varepsilon}+2\mathfrak{Y}+\rho\varepsilon)}(p^s\wedge p^u)\}$ is finite. Moreover, by Theorem~\ref{thm.poincaresection}, $\varphi_{[-\rho,\rho]}(Z)$ intersects only finitely many components of $\widehat{\mathbb D}$. 
 It implies that $\{Z'\in\mathscr{L}:(\varphi_{[-\rho,\rho]}(Z))\cap Z'\neq\emptyset\}$ is a finite set.
 
 Since $r_{\mathscr{L}}(x)\le \rho$ for all $x\in \mathscr L$, the ``in particular" part follows as a direct corollary.
\end{proof}

Let $Z=Z(\wt v^*)$ and $Z'=Z(\wt w^*)$, where $\wt v^*=\wt\Psi^{*,p^s,p^u}_x$ and $\wt w^*=\wt\Psi^{*,q^s,q^u}_y$ are contained in $\mathscr{A}$. Suppose further that $Z\cap\varphi_{[-2\rho,2\rho]}(Z')\neq\emptyset$. Assume  $Z\subset \wh{D}_{i}$ and $Z'\subset \wh{D}_{j}$. For any $z\in Z$ and $z'\in Z'$, define
\[
[z,z']_Z:=V^s(z,Z)\cap {P}_{o_i}(V^u(z',Z')),\
[z,z']_{Z'}:= {P}_{o_j}(V^s(z,Z))\cap V^u(z',Z').
\]
Here the projections $ {P}_{o_i}$, $ {P}_{o_j}$ are given by Lemma~\ref{lem.C1Poincare}.

The following proposition establishes the Markov property on $\mathscr{L}$. Its proof is identical to that of \cite[Proposition 7.2, 7.3]{BCL} and \cite[Proposition 7.2, 7.3]{LMN},  as the argument does not depend on the infinitude of the components of $\mathbb D$.

\begin{pro}\label{markovp}
The following statements hold for all $\varepsilon>0$ small enough.
\begin{enumerate}
\item Let $Z=Z(\wt v^*)\subset \wh{D}_{i}$ and $Z'=Z(\wt w^*)\subset \wh{D}_{j}$ with $\wt v^*=\wt\Psi^{*,p^s,p^u}_x\in\mathscr{A}$ and $\wt w^*=\wt\Psi^{*,q^s,q^u}_y\in\mathscr{A}$. If $Z\cap \varphi_{[-2\rho,2\rho]}(Z')\neq\emptyset$, then:
\begin{enumerate}
\item[{\rm (1a)}]\label{mark.incl} ${P}_{o_j}\circ \Psi_x(R[\frac{1}{2}(p^s\wedge p^u)])\subset \Psi_y(R[q^s\wedge q^u])$;
\item[{\rm (1b)}]\label{mark.Pincl} if $z\in Z$ and $z'={P}_{o_j}(z)\in Z'$, then ${P}_{o_j}(W^{s/u}(z,Z))\subset V^{s/u}(z',Z')$;
\item[{\rm (1c)}] for any $z\in Z$ and $z'\in Z'$, the sets $[z,z']_{Z}$ and $[z,z']_{Z'}$ each consist of exactly one point, and satisfy $[z,z']_Z={P}_{o_i}([z,z']_{Z'})$.
\end{enumerate}

\item \label{3Z} Let $Z,Z',Z''\in\mathscr{L}$ satisfy $Z\cap \varphi_{[-2\rho,2\rho]}(Z')\neq\emptyset$ and $Z\cap \varphi_{[-2\rho,2\rho]}(Z'')\neq\emptyset$. Assume that there exists $z'\in Z$ such that $\varphi_t(z')\in Z''$ for some $|t|\leq 2\rho$. Then for every $z\in Z$, $[z,z']_{Z}=[z,\varphi_t(z')]_Z$.
\end{enumerate}
\end{pro}

\subsection{A refined procedure}
With all the preceding preparations in place, the construction of the topological Markov flow can now be carried out as in \cite{BCL}.

\subsubsection{The refined partition $\mathscr{R}$}\label{subsubsec.refinedpart}

The following lemma is a direct consequence of the fact that the Poincar\'e return times $R_{\mathbb{D}}$ and $R_{\wh{\mathbb{D}}}$ have uniform upper and lower bounds.

\begin{lem}\label{Rfinite}
There exists $N\in\mathbb{Z}^+$ such that for any $x=\pi(\ul{\wt v}^*)\in\mathscr{L}$ with $\ul{\wt v}^*=\{\wt\Psi_{x_n}^{*,p^s_n,p^u_n}\}_{n\in\mathbb{Z}}$, there exists $1\leq n\leq N$ satisfying $g^{+}_{x_0}(x)=H^n(x)$. 
\qed
\end{lem}

For each $Z\in\mathscr{L}$, define $\mathscr{I}_Z:=\{Z'\in\mathscr{L}:\varphi_{[-\rho,\rho]}(Z)\cap Z'\neq\emptyset\}$. Proposition~\ref{proL}~\ref{finL} guarantees that this set is finite. Assume $Z\subset\wh{D}_i$. By Proposition~\ref{proL}, for all sufficiently small $\varepsilon>0$, we have that
\[
\text{if $Z'\in\mathscr{I}_Z$, then $Z'\subset\varphi_{[-2\rho,2\rho]}(\wh{D}_i)$.}
\]
Hence the restriction ${P}_{o_i}|_{Z'}$ is well-defined. For each $Z'\in\mathscr{I}_Z$, we partition $Z$ into four subsets as below:
\[
\begin{aligned}
E^{su}_{Z,Z'}&=\{x\in Z: W^s(x,Z)\cap {P}_{o_i}(Z')\neq\emptyset,\; W^u(x,Z)\cap {P}_{o_i}(Z')\neq\emptyset\},\\
E^{s\emptyset}_{Z,Z'}&=\{x\in Z: W^s(x,Z)\cap {P}_{o_i}(Z')\neq\emptyset,\; W^u(x,Z)\cap  {P}_{o_i}(Z')=\emptyset\},\\
E^{\emptyset u}_{Z,Z'}&=\{x\in Z: W^s(x,Z)\cap {P}_{o_i}(Z')=\emptyset,\; W^u(x,Z)\cap  {P}_{o_i}(Z')\neq\emptyset\},\\
E^{\emptyset\emptyset}_{Z,Z'}&=\{x\in Z: W^s(x,Z)\cap {P}_{o_i}(Z')=\emptyset,\; W^u(x,Z)\cap  {P}_{o_i}(Z')=\emptyset\}.	
\end{aligned}
\]

Denote this partition by $\mathscr{E}_{Z,Z'}=\{E^{su}_{Z,Z'},E^{s\emptyset}_{Z,Z'},E^{\emptyset u}_{Z,Z'},E^{\emptyset\emptyset}_{Z,Z'}\}$. It is immediate that $E^{su}_{Z,Z'}=Z\cap{P}_{o_i}(Z')$.
Let $\mathscr{E}_Z$ be the coarsest partition of $Z$ that refines every $\mathscr{E}_{Z,Z'}$ for $Z'\in\mathscr{I}_Z$. We then introduce an equivalence relation $\eqL$ on $\mathscr{L}$:

for any $x,y\in\mathscr{L}$, we write $x\eqL y$ if for all integers $|k|\leq N$, the following two conditions hold:
\begin{enumerate}
\item[(H1)] for every $Z\in\mathscr{L}$, $H^k(x)\in Z\Leftrightarrow H^k(y)\in Z$;
\item[(H2)] for every $Z\in\mathscr{L}$ containing both $H^k(x)$ and $H^k(y)$,  $H^k(x)$ and $H^k(y)$ lie in the same atom of $\mathscr{E}_Z$.
\end{enumerate} 

Let $\mathscr{R}$ denote the collection of all equivalence classes of $\eqL$. Then $\mathscr{R}$ is a refinement of the cover $\mathscr{L}$. The following lemma is the exact analogue of \cite[Lemma 8.2]{BCL} and \cite[Lemma 8.2]{LMN}.

\begin{lem}\label{finrefine}
The partition $\mathscr{R}$ satisfies the following properties:
\begin{enumerate}
\item for every $Z\in\mathscr{L}$, $\#\{R\in\mathscr{R}: R\subset\varphi_{[-\rho,\rho]}(Z)\}<\infty$;
\item for every $R\in\mathscr{R}$, $\#\{Z\in\mathscr{L}: R\subset\varphi_{[-\rho,\rho]}(Z)\}<\infty$.
\end{enumerate}
\end{lem}

\subsubsection{The Markov property of $\mathscr{R}$}
Parallel to the construction for $\mathscr{L}$, we can define $s/u$-fibers for the refined partition $\mathscr{R}$. Given $x\in R\in\mathscr{R}$, the \emph{$s/u$-fiber of $x$ in $R$} is defined by:
\[
W^s(x,R)=\bigcup_{Z\in\mathscr{L},R\subset Z}V^s(x,Z)\cap R \quad \text{and} \quad W^u(x,R)=\bigcup_{Z\in\mathscr{L},R\subset Z}V^u(x,Z)\cap R.
\]

We extend the Markov property to the partition $\mathscr{R}$, as formulated in Proposition~\ref{MarkovR}, whose statement is adapted from \cite[Proposition 8.3]{BCL} and \cite[Proposition 8.3]{LMN}.

\begin{pro}\label{MarkovR}
The following properties hold for all $R, R'\in\mathscr{R}$:
\begin{enumerate}
\item\label{MR.inter} for all $x,y\in R$, the intersection $[x,y]:=W^s(x,R)\cap W^u(y,R)$ contains exactly one point, and this point lies in $R$;
\item\label{MR.coin} any two $s$-fibers are either identical or disjoint; the analogous property holds for $u$-fibers;
\item\label{MR.incl} if $x\in R\cap H^{-1}(R')$, then $H(W^s(x,R))\subset W^s(H(x),R')$ and $H^{-1}(W^u(H(x),R'))\subset W^u(x,R)$; 
\item\label{MR.hyp} for any $z,w\in W^s(x,R)$, we have $\dist(H^n(z), H^n(w))\to 0$ exponentially as $n\to\infty$; similarly, for any $z,w\in W^u(x,R)$, we have $\dist(H^{-n}(z), H^{-n}(w))\to 0$ exponentially.
\end{enumerate}
\end{pro}

\subsection{Second coding}
Let $\mathscr{R}$ be the refined partition constructed in the previous section. Via the lifting and scaling procedure, we may also regard $\mathscr{R}$ as a collection of rectangles on the scaled space.
For $R,S\in\mathscr{R}$, we write $R\to S$ whenever $R\cap H^{-1}S\neq  \emptyset$. Using this relation, we define a directed graph $\wh{\mathscr{G}}$ with vertex set $\mathscr{R}$ and edge set $\{(R,S)\in\mathscr{R}\times\mathscr{R}:R\to S\}$. The graph induces a symbolic system $(\wh{\Sigma},\wh{\sigma})$. By Proposition~\ref{proL}~\ref{finL}, every vertex in $\wh{\mathscr{G}}$ has finite in-going degree and out-going degree, which implies that $\wh{\Sigma}$ is locally compact.

Define
\[
\begin{aligned}
	{}_{\ell}[R_m,\dots,R_n]&:=H^{-\ell}(R_m)\cap H^{-\ell-1}(R_{m+1})\cap\dots\cap H^{-\ell-(n-m)}(R_n),\\
    Z_{\ell}[\wt v^*_m,\dots,\wt v^*_n]&:=\{\pi(\ul{\wt u}^*):\ul{\wt u}^*\in\Sigma^{\#}, \wt u^*_\ell=\wt v^*_m, \dots, \wt u^*_{\ell+n-m}=\wt v^*_n\}.
\end{aligned}
\]
 
 The following relation between $\Sigma$ and $\wh{\Sigma}$ is established in \cite[Proposition 9.2]{BCL}. 

\begin{pro}\label{RProp}
The following properties hold:
\begin{enumerate}
\item \label{RP.nemp} if $R_m\to\dots\to R_n$ is a finite path on $\wh{\mathscr{G}}$, then
${}_{\ell}[R_m,\dots,R_n]\neq  \emptyset$;

\item \label{RP.RZ} for every $\ul{R}=\{R_n\}_{n\in\mathbb{Z}}\in\wh{\Sigma}$ and any $Z\in\mathscr{L}$ with $R_0\subset Z$, there exists $\ul{\wt v}^*=\{\wt v^*_k\}_{k\in\mathbb{Z}}$ satisfying $Z({\wt v}^*_0)=Z$, together with an integer sequence $\{n_k\}_{k\in\mathbb{Z}}$ with $n_0=0$ and $1\leq n_k-n_{k-1}\leq N$ for all $k\in\mathbb{Z}$, such that for every $k\geq 1$, the inclusion ${}_{n_{-k}}[R_{n_{-k}},\dots,R_{n_k}]\subset Z_{-k}[{\wt v}^*_{-k},\dots,{\wt v}^*_k]$ holds; in addition, $R_{n_k}\subset Z({\wt v}^*_k)$ for all $k\in\mathbb{Z}$.
\end{enumerate}
\end{pro}

Combining Proposition~\ref{RProp}~\ref{RP.RZ} with Proposition~\ref{invmnfd}, we have that for all $k\geq 0$, 
\[
\diam\overline{_{n_{-k}}[R_{n_{-k}},\dots,R_{n_k}]}\leq \diam\overline{Z_{-k}[{\wt v}^*_{-k},\dots,{\wt v}^*_k]}\leq K\vartheta^k.
\]
Thus for any $\ul{R}=\{R_n\}_{n\in\mathbb{Z}}\in\wh{\Sigma}$, $\bigcap_{n\ge 0}\overline{_{-n}[R_{-n},\dots,R_{n}]}$ consists of exactly one point. Denote by $\wh{\pi}(\ul{R})$ this point. Then we get a map $\wh{\pi}:\wh{\Sigma}\to\wh{\mathbb D}$.

By an argument analogous to that for Proposition~\ref{1code}, we obtain the following properties of $\wh{\pi}$. The proof can be found in \cite[Proposition 9.2]{BCL}.

\begin{pro}\label{pro.pi}
Let $\ul{R}\in\wh{\Sigma}$ and $\ul{\wt v}^*\in\Sigma$ be as in Proposition~\ref{RProp}. Then the following statements hold:
\begin{enumerate}
\item\label{pi.1} $\wh{\pi}(\ul{R})=\pi(\ul{\wt v}^*)$; in particular, if $\ul{R}\in\wh{\Sigma}^{\#}$, then $\ul{\wt v}^*\in\Sigma^{\#}$ and $\wh{\pi}(\wh{\Sigma}^{\#})=\pi(\Sigma^{\#})$;

\item\label{pi.2} the map $\wh{\pi}$ is scaled H\"older continuous, that is, there exist $\wh K>1$ and $\varpi>0$ such that for any $\ul{R},\ul{S}\in\wh{\Sigma}$ with $R_i=S_i$ for $i=0, 1$, we have 
\[\dist(\wh{\pi}(\ul{R}),\wh{\pi}(\ul{S}))\le \wh K\cdot \dist(\ul{R},\ul{S})^{\varpi}\cdot|X(\wh{\pi}(\ul{R}))|.\]
Consequently, $\wh{\pi}$ is H\"older continuous.
Moreover, the finite subsequence $\{\wt v^*_i\}_{|i|\leq k}$ is completely determined by $\{R_j\}_{|j|\leq kN}$ for each $k\geq 1$. 
\end{enumerate}
\end{pro}

%%%%%%%%%%%%%%%%%%%%%%%%%%%%%%%%%%%%%%%%%%%%%%%%%%%%%%%%%%%%%%%%%%%

%%%%%%%%%%%%%%%%%%%%%%%%%%%%%%%%%%%%%%%%%%%%%%%%%%%%%%%%%%%%%%%%%%%

\section{Coding for singular flows}\label{sec.codingthm}

%%%%%%%%%%%%%%%%%%%%%%%%%%%%%%%%%%%%%%%%%%%%%%%%%%%%%%%%%%%%%%%%%%%%

\subsection{Finite-to-one topological Markov flow}

We now proceed to define the suspension flow over $\wh{\Sigma}$. Define the roof function $\wh{r}: \wh{\Sigma} \to (0,\infty)$ by 
\[
\wh r(\ul{R}) := \min\left\{ t>0 : \varphi_t(\wh{\pi}(\ul{R})) = \wh{\pi}(\wh{\sigma}(\ul{R})) \right\}.
\] 
Using $\wh{r}$ as the roof function over $\wh{\Sigma}$, we construct the topological Markov flow $(\wh{\Sigma}_{\wh{r}}, \wh{\sigma}_{\wh{r}})$. 
The factor map $\wh{\pi}_{\wh{r}} : \wh{\Sigma}_{\wh{r}} \to M$ is then given by $\wh{\pi}_{\wh{r}}(\ul{R},s) = \varphi_s(\wh{\pi}(\ul{R}))$.

We reiterate that the proof follows exactly the same procedure as that in \cite[Theorem 9.1]{BCL} and \cite[Theorem 9.1]{LMN}, with the only exception being the H\"older continuity of $\wh r$, and all other prerequisite results have been established in preceding sections. For this reason, we only present a concise outline here, and refer the reader to those references for the complete technical details.

\begin{thm}\label{thm.strongthm}
Let $X$ be a $C^{1+\beta}$ vector field ($\beta>0$) on $M$. For any given $\chi>0$, there exists a locally compact TMF $(\wh{\Sigma}_{\wh{r}},\wh \sigma_{\wh{r}})$ generated by the directed graph $\wh{\mathscr{G}}=(\wh V,\wh E)$, a roof function $\wh{r}$, and a factor map $\wh{\pi}_{\wh{r}}:\wh{\Sigma}_{\wh{r}}\to M$ satisfying the following properties:
\begin{enumerate}
\item $\wh{\pi}_{\wh{r}}\circ \wh{\sigma}_{\wh{r},t}=\varphi_t\circ \wh{\pi}_{\wh{r}}$ holds for all $t\in\mathbb{R}$;
\item\label{cod1} $\sup \wh{r}<\infty$ and $\inf \wh{r}>0$; moreover, $\wh{r}$ is H\"older continuous, and $\wh{\pi}_{\wh{r}}$ is scaled H\"older continuous with respect to the Bowen-Walters metric $\dist_{\wh{r}}$ on $\wh{\Sigma}_{\wh{r}}$, that is, there exist $\wh{K}_{\wh{r}}>1$ and $\varpi>0$ such that for any $(\ul{R},t)$ and $(\ul{S},s)$ with $R_0=S_0$ and $R_1=S_1$, we have 
\[
\dist(\wh{\pi}_{\wh{r}}(\ul{R},t),\wh{\pi}_{\wh{r}}(\ul{S},s))\le \wh{K}_{\wh{r}}\cdot \dist_{\wh{r}}((\ul{R},t),(\ul{S},s))^{\varpi}\cdot |X(\wh{\pi}_{\wh{r}}(\ul{R},t))|;
\]

\item\label{cod2} $\wh{\pi}_{\wh{r}}(\wh{\Sigma}^{\#}_{\wh{r}})=\NUH_{\chi}^{\#}$, which has full measure with respect to every regular $\chi$-hyperbolic $X$-invariant probability measure; furthermore, for any regular ergodic $\chi$-hyperbolic $X$-invariant probability measure $\mu$, there exists an ergodic $\wh{\sigma}_{\wh{r}}$-invariant probability measure $\wh{\mu}$ on $\wh{\Sigma}_{\wh{r}}$  such that $\wh{\mu}\circ \wh{\pi}_{\wh{r}}^{-1}=\mu$ and $h_{\mu}(X)=h_{\wh{\mu}}(\wh{\sigma}_{\wh{r}})$;

\item\label{cod3} for any $R,S\in \wh V$, there exists $C(R,S)\in \bZ^+$ such that:
if $(\ul{R},t)\in\wh{\Sigma}_{\wh{r}}$ satisfies $R_n=R$ for infinitely many $n>0$ and $R_m=S$ for infinitely many $m<0$, then the cardinality of the preimage set $ \wh{\pi}_{\wh{r}}^{-1}( \wh{\pi}_{\wh{r}}(\ul{R},t))$ is bounded above by $C(R,S)$;

\item \label{cod4} there exists a constant $\lambda>0$ such that for every point
$x\in\wh{\pi}_{\wh{r}}(\wh{\Sigma}_{\wh{r}})$, there exists a $(d-1)$-dimensional subspace $\H_x\subset T_x M$ that admits a unique splitting $\H_x=\H^s_x\oplus\H^u_x$, satisfying 
\[\begin{aligned}
&\limsup_{t\to+\infty} \frac{1}{t}\log\norm{\Phi^*_t|_{\H^s_x}}\leq -\lambda,\  \liminf_{t\to+\infty} \frac{1}{t}\log\norm{\Phi^*_{-t}|_{\H^s_x}}\geq \lambda,\\
&\limsup_{t\to+\infty} \frac{1}{t}\log\norm{\Phi^*_{-t}|_{\H^u_x}}\leq -\lambda,\  \liminf_{t\to+\infty} \frac{1}{t}\log\norm{\Phi^*_{t}|_{\H^u_x}}\geq \lambda;
\end{aligned}\]
the splitting is $\Phi^*$-invariant, and the maps $z\mapsto \H^{s/u}_{\wh{\pi}_{\wh{r}}(z)}$ are H\"older continuous on $\wh{\Sigma}_{\wh{r}}$;

\item \label{cod5} there exists  a constant $\alpha>0$ satisfying the following property: for every $z\in\wh{\Sigma}_{\wh{r}}$, set $x=\wh{\pi}_{\wh{r}}(z)$; then there exist two $C^1$ submanifolds $V^{cs}(z)$, $V^{cu}(z)$ passing through $x$ such that
\begin{enumerate} 
\item[\rm (6a)] $T_x V^{cs}(z)=\H_x^s\oplus \la X(x)\ra$ and $T_x V^{cu}(z)=\H_x^u\oplus \la X(x)\ra$;
\item[\rm (6b)] for any $y\in V^{cs}(z)$, there exists $\tau\in\mathbb{R}$ such that $\dist(\varphi_t(x),\varphi_{t+\tau}(y))\leq e^{-\alpha t}|X(\varphi_t(x))|$ for all $t\geq0$;
\item[\rm (6c)] for any $y\in V^{cu}(z)$, there exists $\tau\in\mathbb{R}$ such that $\dist(\varphi_{-t}(x),\varphi_{-t+\tau}(y))\leq e^{-\alpha t}|X(\varphi_t(x))|$ for all $t\geq0$;
\end{enumerate}

\item  \label{cod6} there exists a symmetric binary relation ``$\sim$'' on $\wh{V}$ such that:
\begin{enumerate}
\item[\rm (7a)]  $\#\{S\in\wh{V}:R\sim S\}<\infty$ for every $R\in\wh{V}$;
\item[\rm (7b)] the relation $\sim$ is a {\rm Bowen relation} for $(\wh{\sigma}_{\wh{r}},\wh{\pi}_{\wh{r}}|_{\wh{\Sigma}^{\#}_{\wh{r}}},\varphi_t)$, that is, 
\begin{itemize}
  \item if $\wh\pi_{\wh r}(\omega)=\wh\pi_{\wh r}(\omega')$ for some $\omega,\omega'\in \wh\Sigma^{\#}_{\wh r}$, then $0(\omega)\sim 0(\omega')$, where $0(\ul R,s):=R_0$;
  \item there exists $T_0>0$ such that, if $0(\wh{\sigma}_{\wh{r},t}(\omega))\sim 0(\wh{\sigma}_{\wh{r},t}(\omega'))$ holds for some $\omega,\omega'\in \wh\Sigma^{\#}_{\wh r}$ and all $t\in \bR$, then $\wh\pi_{\wh r}(\omega)=\varphi_{t_0}(\wh\pi_{\wh r}(\omega'))$ for some $|t_0|<T_0$; 
\end{itemize}
\end{enumerate}

\item  \label{cod7} there exists a measurable set $\mathscr{R}$ with a measurable partition $\{R:R\in\wh{V}\}$ indexed by $\wh{V}$ such that:
\begin{enumerate}
\item[\rm (8a)]  the orbit of every point $x\in \NUH_{\chi}^{\#}$ intersects $\mathscr{R}$;
\item[\rm (8b)]  the first return map $H:\mathscr{R}\to\mathscr{R}$ induced by the flow $\varphi_t$ is a well-defined bijection;
\item[\rm (8c)]  for any $x\in\mathscr{R}$, if there exists $\ul{R}=\{R_n\}_{n\in\mathbb{Z}}$ satisfying $H^n(x)\in R_n$ for all $n\in\mathbb{Z}$, then $(\ul{R},0)\in\wh{\Sigma}^{\#}_{\wh{r}}$ and $\wh{\pi}_{\wh{r}}(\ul{R},0)=x$;
\end{enumerate}
\item  \label{cod8} if $\Lambda_0\subset M\setminus\Sing(X)$ is a compact transitive invariant hyperbolic set, and every ergodic $X$-invariant measure supported on $K$ is $\chi$-hyperbolic, then there exists a transitive $\wh{\sigma}_{\wh{r}}$-invariant compact subset $\wh A_0\subset\wh{\Sigma}_{\wh{r}}$ such that $\wh{\pi}_{\wh{r}}(\wh A_0)=\Lambda_0$.
\end{enumerate}
\end{thm}
\begin{proof}
Item (1) follows immediately from the definition of the factor map. 
Since ${\mathbb D}$ and $\wh{\mathbb D}$ are locally finite uniform Poincar\'e sections, we directly obtain $\sup \wh{r}<\infty$ and $\inf \wh{r}>0$. It follows from 
Proposition~\ref{pro.pi}~\ref{pi.2} that $\wh{\pi}$ is scaled H\"older continuous. Using the argument in the proof of \cite[Lemma~5.9]{LS}, we conclude that $\wh{\pi}_{\wh{r}}$ is scaled H\"older continuous with respect to the Bowen-Walters metric on $\wh{\Sigma}_{\wh{r}}$. 

For any $\ul{R}\in\wh{\Sigma}$,  define the neighborhood $\wh U:=\{\ul{S}\in\wh{\Sigma}:S_0=R_0,\, S_1=R_1\}$. Denote $x:=\wh \pi (\ul{R})$ and $y:=\wh \pi (\wh\sigma(\ul{R}))$. There exist $i,j\in\Gamma$ such that for all $\ul{S}\in \wh{\Sigma}$, we have $\wh \pi (\ul{S})\in \wh D_i$ and $ \wh \pi (\wh\sigma(\ul{S}))\in \wh D_j$. This implies that $\wh r|_{\wh U}=\mathfrak{t}_{x,y}\circ \wh \pi|_{\wh U}$, where $\mathfrak{t}_{x,y}$ is the Poincar\'e time function defined in Lemma~\ref{lem.C1Poincare}~\ref{fb.sectPoincare}.  Combining Lemma~\ref{lem.C1Poincare}~\ref{fb.sectPoincare} and Proposition~\ref{pro.pi}~\ref{pi.2}, we have
\[
\big|\wh r(\ul{R})-\wh r(\ul{S})\big|\le \norm{d\mathfrak{t}_{x,y}}\cdot\dist(\wh{\pi}(\ul{R}),\wh{\pi}(\ul{S}))\le \rho \wh K K_{F1}  \cdot \dist(\ul{R},\ul{S})^{\varpi}.
\]
This completes the proof of Item~\ref{cod1}.

For Item~\ref{cod2}, $\wh{\pi}_{\wh{r}}(\wh{\Sigma}^{\#}_{\wh{r}})=\NUH_{\chi}^{\#}$ follows directly from the construction of $\mathscr{R}$ and Proposition~\ref{proL}~\ref{proL.cov}. By Lemma~\ref{lem.recNUH}, ${\NUH}^{\#}_\chi$ has full measure for every regular $\chi$-hyperbolic $X$-invariant measure.

We say that $R, S\in\mathscr{R}$ are \emph{affiliated}, denoted by $R\sim S$, if there exist $Z, Z'\in\mathscr{L}$ such that $R\subset Z$, $S\subset Z'$ and $Z'\in \mathscr{I}_Z$. This symmetric relation satisfies all the properties stated in Item~\ref{cod6}. We refer the reader to \cite[Section 9.4~(Part 6)]{BCL} and \cite[Section 9.4~(Part 6)]{LMN} for the details.

 For each $R\in\mathscr{R}$, define the set
\[
A(R):=\{(S,Z')\in\mathscr{R}\times \mathscr{L}:R\sim S\text{ and } S\subset Z'\}.
\]
It follows directly from Lemma~\ref{finrefine} that the cardinality $N(R):=\#A(R)$ is finite. 

Now define $C(R,S):=N(R)\cdot N(S)$. Then this $C(R,S)$ satisfies the requirement in Item~\ref{cod3}. The proof originates from the idea in \cite{Bow78}, making use of  Proposition~\ref{markovp} and the Markov property established in Proposition~\ref{MarkovR}. The proof of \cite[Theorem~9.6]{BCL} extends the argument for diffeomorphisms in \cite{Sar13} to $3$-dimensional non-singular flows, and the same method applies to both high-dimensional non-singular flows in \cite{LMN} and singular flows considered here.

For any regular ergodic $\chi$-hyperbolic $X$-invariant measure $\mu$, define the lifted measure $\wh{\mu}$ on $\wh{\Sigma}_{\wh{r}}$ by 
\begin{equation}\label{liftmeas}
\wh{\mu}(E)=\int_{\NUH^{\#}_{\chi}}\frac{1}{\#\{\wh{\pi}_{\wh{r}}^{-1}(x)\}}\sum_{\wh{\pi}_{\wh{r}}(\ul{R})=x}\mathbbm{1}_{E}(\ul{R})d\mu(x).
\end{equation}
It follows from Item~\ref{cod3} that $\wh{\pi}_{\wh{r}}$ has finitely many preimages for every $x\in\NUH^{\#}_{\chi}$. Thus the measure $\wh{\mu}$ defined in \eqref{liftmeas} is a well-defined ergodic $\wh{\sigma}_{\wh{r}}$-invariant measure on $\wh{\Sigma}_{\wh{r}}$,  satisfying $\wh{\mu}\circ \wh{\pi}_{\wh{r}}^{-1}=\mu$ and $h_{\mu}(X)=h_{\wh{\mu}}(\wh{\sigma}_{\wh{r}})$. We refer to \cite[\S\,13]{Sar13} for the details.
This completes the proof of the remaining part of Item~\ref{cod2}.

Item~\ref{cod4} follows from Proposition~\ref{invmnfd} and Proposition~\ref{MarkovR}. Item~\ref{cod5} is a direct consequence of Proposition~\ref{strmnfd}, and we refer to \cite[Section 9.4~(Part 5)]{BCL} and \cite[Section 9.4~(Part 5)]{LMN} for further details. 
Item~\ref{cod7} follows immediately from Proposition~\ref{pro.pi} combined with the standard suspension flow construction, with full arguments presented in \cite[Section 9.4~(Part 7)]{BCL} and \cite[Section 9.4~(Part 7)]{LMN}.
Finally, the statement in Item~\ref{cod8} concerning compact transitive hyperbolic sets fits exactly into the framework established in \cite{BCL,LMN}, and we refer to \cite[Section 9.4~(Part 8)]{BCL} and \cite[Section 9.4~(Part 8)]{LMN} for the complete proof.
\end{proof}

%%%%%%%%%%%%%%%%%%%%%%%%%%%%%%%%%%%%%%%%%%%%%%%%%%%%%%%%%%%%%%%%%%%

%%%%%%%%%%%%%%%%%%%%%%%%%%%%%%%%%%%%%%%%%%%%%%%%%%%%%%%%%%%%%%%%%%%

\subsection{Irreducible and SPR codings}

In this subsection, we first show that all regular ergodic $\chi$-hyperbolic measures carried by a Borel homoclinic class can be lifted to a common irreducible component of the TMF given in the previous subsection. The existence of such irreducible component was first proved for diffeomorphisms in \cite[Section 3]{BCS22}, and then proved for non-singular flows on three-dimensional and high-dimensional manifolds respectively in \cite[Section 10]{BCL} and \cite[Section 10]{LMN}. The  arguments are unaffected by the presence of singularities. Therefore, by employing the exact same proof, the following theorem can be obtained.

\begin{thm}\label{thm.irrcomp}
Assume that $X$ is a $C^{1+\beta}$ vector field ($\beta>0$) on $M$, and $\mu$ is a regular ergodic $\chi$-hyperbolic measure for some $\chi>0$. Let $(\wh{\Sigma}_{\wh{r}},\wh \sigma_{\wh{r}})$ be given by Theorem~\ref{thm.strongthm}. Then there exists an irreducible component $\wh{\Sigma}'_{\wh{r}}\subset \wh{\Sigma}_{\wh{r}}$ such that every regular $\chi$-hyperbolic ergodic measure $\nu$ on $M$ that is homoclinically related to $\mu$ admits some ergodic lift supported on $\wh{\Sigma}'_{\wh{r}}$.
\end{thm}

Since any two measures carried by a common Borel homoclinic class are homoclinically related (see \S\,\ref{subsec.BHC}), we immediately obtain the following Corollary.

\begin{cor}\label{cor.Borelclass}
Assume that $X$ is a $C^{1+\beta}$ vector field ($\beta>0$) on $M$, and $B$ is a Borel homoclinic class of $X$. For any $\chi>0$, let $(\wh{\Sigma}_{\wh{r}},\wh \sigma_{\wh{r}})$ be given by Theorem~\ref{thm.strongthm}. Then there exists an irreducible component $\wh{\Sigma}'_{\wh{r}}\subset \wh{\Sigma}_{\wh{r}}$ such that each regular $\chi$-hyperbolic ergodic measure carried by $B$ has some ergodic lift supported on $\wh{\Sigma}'_{\wh{r}}$.
\qed
\end{cor}

Next we lift an SPR flow on $M$ to an SPR TMF. The argument follows the strategy in \cite[Proposition 10.4]{BCS25}, where the authors introduce a key property of symbolic models called {\it $\chi$-bornology}. Since our objective here is only to construct an SPR coding, we only need to adapt the first part of their proof.

\begin{thm}\label{thm.SPRcoding}
Let $B=M$ or $B\subset M$ be a Borel homoclinic class. Assume that the $C^{1+\beta}$ vector field $X$ is SPR on $B$ for a given H\"older potential $\vartheta : M\to\mathbb{R}$. Let $\chi>0$ be the parameter associated to the SPR property, and let $(\wh\Sigma_{\wh r},\wh\sigma_{\wh r})$ be the locally compact TMF given by Theorem~\ref{thm.strongthm} or Corollary~\ref{cor.Borelclass} for $\chi$, respectively. Then  $(\wh\Sigma_{\wh r},\wh \sigma_{\wh r})$ is SPR for the H\"older potential $\vartheta \circ \wh\pi_{\wh r}$.
\end{thm}
\begin{proof}
By the definition of the SPR property, for every $\varepsilon>0$, there exists a Pesin block $\Pes_{\chi}(\ell,\varepsilon)$, a compact subset $\Lambda\subset \Pes_{\chi}(\ell,\varepsilon)$, and constants $P_0<P_{\rm TOP}(X|_B,\vartheta)$, $\tau\in(0,1)$ such that for any ergodic $X$-invariant probability measure $\nu$ with  $\nu(B)=1$ and $P_{\nu}(X,\vartheta)>P_0$, we have $\nu(\Lambda)>\tau$.

From the definition of Pesin blocks (Definition~\ref{def.pesin}), there exists $\bar\ell>\ell$ such that $\varphi_t(x)\in \Pes_{\chi}(\bar\ell,\varepsilon)$ holds for all $x\in \Pes_{\chi}(\ell,\varepsilon)$ and all $|t|\le \sup \wh r$. 

Following the same line of argument as in \cite[Lemma 10.2]{BCS25}, one can prove that there exists a constant $c_0>0$ such that for  
any $\ul{\wt v}^*=\{\wt \Psi_{x_n}^{*,p^s_n,p^u_n}\}_{n\in\mathbb{Z}}\in\Sigma^{\#}$ satisfying $\pi(\ul{\wt v}^*)\in\Pes_{\chi}(\bar\ell,\varepsilon)$, we have $p_0^s\wedge p_0^u>c_0$.

Define
\[
\Lambda_{\wh{\mathbb{D}}}=\big\{\varphi_{t_0}(x): x\in \Lambda, \ t_0=\max\{t\le 0: \varphi_{t}(x)\in \wh{\mathbb{D}}\} \big\}\subset \wh{\mathbb{D}}.
\]
Since the compact set $\Lambda\subset  \Pes_{\chi}(\ell,\varepsilon)$ is disjoint from $\Sing(X)$, the set $\Lambda_{\wh{\mathbb{D}}}\subset \Pes_{\chi}(\bar\ell,\varepsilon)$ is also compact and stays away from $\Sing(X)$. By the local finiteness of the Poincar\'e section $\wh{\mathbb{D}}$ (Theorem~\ref{thm.poincaresection}~\ref{se.fini}), only finitely many components of $\wh{\mathbb{D}}$ intersect $\Lambda_{\wh{\mathbb{D}}}$. 

Denote $A:=\{\ul{\wt v}^*=\{\wt \Psi_{x_n}^{*,p^s_n,p^u_n}\}_{n\in\mathbb{Z}}\in\Sigma^{\#}: \pi(\ul{\wt v}^*)\in \Lambda_{\wh{\mathbb{D}}}\}$. As every element in $A$ satisfies $p_0^s\wedge p_0^u>c_0$, combining this with the local discreteness property (Proposition~\ref{pro.SDR}~\ref{SDR.D}), we conclude that $A$ is a finite set. Lemma~\ref{finrefine} then implies that there exists a finite union of cylinders $\mathbf{C}$ in $\wh\Sigma$ such that
\[
\wh\Sigma^{\#}\cap \wh\pi^{-1}(\Lambda_{\wh{\mathbb{D}}}\cap \NUH^{\#}_{\chi})\subset \mathbf{C}.
\]

Let $\mathbf{C}_r=\{(\ul{x},t)\in \wh\Sigma_{\wh r}:\ul{x}\in \mathbf C,\ t\in[0,\wh r(\ul{x}))\}$ be the suspension of $\mathbf{C}$ in $\wh\Sigma_{\wh r}$. It follows that
\[
\wh\Sigma_{\wh r}^{\#}\cap \wh\pi_{\wh r}^{-1}(\Lambda\cap \NUH^{\#}_{\chi})\subset \mathbf{C}_r.
\]

Now take any ergodic probability measure $\bar{\nu}\in \mathbb{P}(\wh\Sigma_{\wh r})$  satisfying $P_{\bar{\nu}}(\wh \sigma_{\wh r},\vartheta \circ \wh\pi_{\wh r})>P_0$, and let $\nu=\bar{\nu}\circ \wh\pi_{\wh r}^{-1}$ be its projection to $M$.
By Theorem~\ref{thm.strongthm}~\ref{cod2}, the projected measure $\nu$ is an ergodic $X$-invariant probability measure with  $\nu(B)=1$ and $P_{\nu}(X,\vartheta)>P_0$. Applying the original SPR property on $B$, we get $\nu(\Lambda)>\tau$, which implies $\bar{\nu}(\mathbf{C}_r)>\tau$. This verifies the SPR property of $(\wh\Sigma_{\wh r},\wh \sigma_{\wh r})$, and completes the proof of the theorem.
  \end{proof}

%%%%%%%%%%%%%%%%%%%%%%%%%%%%%%%%%%%%%%%%%%%%%%%%%%%%%%%%%%%%%%%%%%%
%%%%%%%%%%%%%%%%%%%%%%%%%%%%%%%%%%%%%%%%%%%%%%%%%%%%%%%%%%%%%%%%%%%

\part{Statistical properties of singular flows}\label{part.thermdyn}

In this part, we present some statistical properties of singular flows derived from
symbolic dynamics.

%%%%%%%%%%%%%%%%%%%%%%%%%%%%%%%%%%%%%%%%%%%%%%%%%%%%%%%%%%%%%%%%%%%

\section{Finiteness of equilibrium states}\label{sec.equili}

In this section, we consider the finiteness of equilibrium states for singular flows and prove Theorems~\ref{main.mme} and \ref{main.spr}. We say that a vector field is of class $C^{1+}$ if it is $C^{1+\beta}$ for some $\beta>0$.

\subsection{Equilibrium states on Borel homoclinic classes}

Let $B$ be a measurable invariant set for the flow, and let $\vartheta: M\to \bR$ be a potential  on $M$. Recall that 
an  equilibrium state of $\vartheta$ for $X$  on $B$ is a measure $\mu\in\mathbb P(X|_B)$ such that $P_{\mu}(X,\vartheta)= P_{\rm TOP}(X|_B,\vartheta)$.
Denote  by ${\rm Per}_T(X|_B)$ the number of distinct prime periodic orbits contained in $B$ with minimal positive period less than $T$.

As shown in \S\,\ref{subsec.BHC}, any two homoclinically related measures are carried by the same Borel homoclinic class. Therefore,  the following theorem is a stronger version of Theorem~\ref{main.mme}.

\begin{thm}\label{thm.main.mme}
	Let  $X$ be a $C^{1+}$ vector field on $M$. Then for every H\"older continuous potential $\vartheta$ on $M$ and every Borel homoclinic class $B$ of $X$, there exists at most one regular hyperbolic ergodic equilibrium state of $\vartheta$ for $X$  on $B$. If such an equilibrium state exists, it is Bernoulli up to a period.  
	
	In particular, there exists  at most one regular hyperbolic ergodic MME for $X$  on $B$.
Whenever this MME exists, there is $C>0$ such that ${\rm Per}_T(X|_B)\ge \frac{C}{T} e^{T\cdot h_{\rm TOP}(X|_B)}$ holds for all $T$ large enough.
\end{thm}
\begin{proof}
	We first prove the uniqueness statement by contradiction. Suppose there exist two distinct regular hyperbolic ergodic equilibrium states $\mu$ and $\nu$ of the potential $\vartheta$ for $X$ on a Borel homoclinic class $B$, where $\mu$ is $\chi_1$-hyperbolic and $\nu$ is $\chi_2$-hyperbolic. Set $\chi=\min\{\chi_1,\chi_2\}$. 
	
	By Theorem~\ref{thm.irrcomp}, there exists an irreducible component $\wh \Sigma'_{\wh r}$ of the associated TMF such that $\mu$ and $\nu$ admit lifts $\wh\mu$ and $\wh\nu $ supported on this component, respectively. Since the entropy preservation property in Theorem~\ref{thm.strongthm}~\ref{cod2} gives $h_{\mu}(X)=h_{\wh{\mu}}(\wh{\sigma}_{\wh{r}})$, and the change-of-variable formula ensures $\int \vartheta d\mu=\int \vartheta\circ \wh\pi_{\wh r} d\wh\mu $, both lifted measures $\wh\mu$ and $\wh\nu $ are equilibrium states of the lifted potential $\vartheta\circ \wh\pi_{\wh r}$ for $\wh{\sigma}_{\wh{r}}$ on the irreducible component $\wh{\Sigma}'_{\wh{r}}$, where the projection map $\wh\pi_{\wh r}$ is given by Theorem~\ref{thm.strongthm}. This contradicts the uniqueness of equilibrium states for irreducible TMFs with H\"older continuous roof functions, whose complete proof can be found in \cite[Theorem~6.2]{LS} and \cite[Corollary~3.3]{BCS22}. 
	
	Now let $\mu$ be any regular hyperbolic ergodic equilibrium state of $\vartheta$ on $B$. 
	Its lift $\wh\mu$ is an equilibrium state of $\vartheta\circ \wh\pi_{\wh r}$ for $\wh{\sigma}_{\wh{r}}$ on $\wh{\Sigma}'_{\wh{r}}$. By \cite[Theorems~4.7 and 5.1]{LLS}, the lift $\wh \mu$ is Bernoulli up to a period, and this property descends to the projected measure $\mu$ on $M$.
	
	Finally, note that MMEs are exactly the equilibrium states of the zero potential $\vartheta\equiv 0$. The uniqueness result above immediately implies that there is at most one regular hyperbolic ergodic MME for $X$  on $B$. Whenever such an MME exists, the lower bound growth estimate for periodic orbits follows directly from the symbolic model argument in \cite[Theorem 8.1]{LS}.
\end{proof}

\begin{rk}\label{rk.quasiholder}
\begin{enumerate}
  \item The argument in the preceding proof only requires the H\"older continuity of the lifted potential $\vartheta\circ \wh\pi_{\wh r}$.
 The notion of {\rm  quasi-H\"older continuity}, first introduced in \cite{BCS25}, extends the scope of corresponding theorems to a much broader class of potentials that includes geometric potentials. For singular flows, even though the normal bundle does not enjoy uniform continuity, it is H\"older continuous in the scaled sense (see \cite[\S\,4.2]{LLL2024}). Combining this with the fact that the projection map  $\wh\pi_{\wh r}$ itself is H\"older continuous, one can naturally extend the definition of quasi-H\"older continuity to the singular flow setting.
 \item The coding provided by Theorem~\ref{thm.irrcomp} cannot be directly applied to derive upper bounds on the periodic orbit growth rate in general. This is because the Borel homoclinic class $B$ may contain a large family of periodic orbits whose hyperbolicity is strictly weaker than the parameter $\chi$ used to construct the symbolic model.
\end{enumerate}	
\end{rk}

\subsection{Equilibrium states on SPR sets}

We now turn to the SPR setting. The following result gives a detailed formulation of Theorem~\ref{main.spr}.

\begin{thm}\label{thm.main.spr}
	Let $X$ be a $C^{1+}$ vector field  on $M$ and $\vartheta$ be a H\"older continuous potential on $M$. 
	\begin{itemize}
  \item If $X$ is SPR for $\vartheta$ on a Borel homoclinic class $B$, then there exists a unique   equilibrium state of $\vartheta$ for $X$  on $B$, which is Bernoulli up to a period.
  \item If $X$ is SPR for $\vartheta$ on $M$, then there exist only finitely many ergodic equilibrium states of $\vartheta$ for $X$, each of which is Bernoulli up to a period.
    \end{itemize}
\end{thm}
\begin{proof}
	The first statement follows directly from combining Theorem~\ref{thm.SPRcoding} and Theorem~\ref{thm.main.mme} with \cite[Theorem~11.6]{BCS25}. We therefore focus on the proof of the second assertion.
	
	We proceed by contradiction. Suppose there exists a sequence of distinct ergodic equilibrium states $\{\mu_n\}$ of the potential $\vartheta$ for $X$. Since $X$ satisfies the SPR property for $\vartheta$, there exists a Pesin block $\Pes_{\chi}(\ell,\varepsilon)$, a compact subset $\Lambda\subset \Pes_{\chi}(\ell,\varepsilon)$ and a constant $\tau\in(0,1)$ such that $\mu_n(\Lambda)>\tau$ holds for every $n\in\bZ^+$. 
	
	Since $\mu_n(\Lambda)>0$ and each $\mu_n$ is ergodic, it follows that each $\mu_n$ is regular. By Proposition~\ref{pro.proBHC}, every $\mu_n$ is carried by some Borel homoclinic class $B_n$, and the uniqueness statement in Theorem~\ref{thm.main.mme} ensures that the classes $B_n$ are pairwise disjoint.
	
	For each $n$, pick a point $x_n\in B_n\cap \Lambda$.
By passing to a convergent subsequence if necessary, we may assume $x_n\to x\in \Lambda$. By  classical Pesin Theory (see Lemma~\ref{lem.Pesinmanifold}), on the compact nonsingular set $\Lambda$, both the sizes of local stable/unstable manifolds and the angles between  stable and unstable subspaces are uniformly bounded from below by a positive constant. As a consequence, for all sufficiently large $i,j \in\bZ^+$, the points $x_i$ and $x_j$  are homoclinically related. This contradicts the pairwise disjoint property of $\{B_n\}$, which completes the proof.
\end{proof}

%%%%%%%%%%%%%%%%%%%%%%%%%%%%%%%%%%%%%%%%%%%%%%%%%%%%%%%%%%%%%%%%%%%

%%%%%%%%%%%%%%%%%%%%%%%%%%%%%%%%%%%%%%%%%%%%%%%%%%%%%%%%%%%%%%%%%%%

\section{Ruelle operator of one-side TMS}\label{sec.tranoper}

In the subsequent three sections, we focus on the mixing rate of equilibrium states. This part is relatively self-contained.
 To minimize unnecessary notational overhead, we redefine a small number of previously introduced notations, with all new meanings clearly specified for these three sections only.  The original interpretations of these notations will not be used in these sections.

The Ruelle operator is a powerful tool for analyzing the statistical properties of symbolic systems. In this section, we will examine the properties of the Ruelle operator for SPR TMS.

\subsection{One-side symbolic systems, Ruelle operator and spectral gap}

\subsubsection{One-sided TMS and TMF}
Let $\mathscr{G}$ be a directed graph with vertex set $\mathscr{V}$ and edge set $\mathscr{E}=\{u\to v: u,v\in\mathscr{V}\}$. Assume that $\mathscr{G}$ is proper (see \S\,\ref{subsec.TMSTMF}) and has at most countably many vertices. 
 We say that $\mathscr{G}$ is {\it forward locally finite} if every vertex has finite out-going degree.
Define (recall that $\bN_0=\{0\}\cup \bZ^+$)
\[
\Sigma^{+}=\Sigma^{+}(\mathscr{G})=\{\ul{x}=(x_0,x_1,\dots)\in\mathscr{V}^{\bN_0}:x_n\to x_{n+1} \text{ for all } n\in\bN_0\},
\]
which is endowed with a metric $\dist(\ul{x},\ul{y})=e^{-\min\{n\geq 0:x_n\neq y_n\}}$, where we recall that $\min\emptyset:=+\infty$. Let $\sigma: \Sigma^+\to \Sigma^+$ be the {\it left shift map} on $\Sigma^+$. The system $(\Sigma^{+},\sigma)$ is called a {\it one-sided TMS} induced by the graph $\mathscr{G}$. 

As before, denote by $(\Sigma,\sigma)$ the (two-sided) TMS induced by $\mathscr{G}$. Here we use the same notation $\sigma$ to denote the left shift map on both $\Sigma$ and $\Sigma^+$, which is a standard abuse of notation causing no ambiguity. Similar to the two-sided case, we say that $(\Sigma^{+},\sigma)$ is {\it irreducible} if $\mathscr{G}$ is connected. The {\it period} of an irreducible one-sided TMS is the period of $\mathscr{G}$. 

By arguments analogous to the two-sided case, one can establish the following properties:
\begin{itemize}
  \item $(\Sigma^{+},\sigma)$ is topologically transitive if and only if it is irreducible.
  \item $(\Sigma^{+},\sigma)$ is topologically mixing if and only if it is irreducible with period $1$.
  \item $(\Sigma^{+},\sigma)$ is locally compact if and only if $\mathscr{G}$ is forward locally finite. And in this case, each cylinder of $\Sigma^+$ is compact.
  \item The spectral decomposition theorem (Lemma~\ref{lem.spectdecomp}) holds for one-sided TMS.
\end{itemize}

We say that a function $\phi^{+}:\Sigma^{+}\to\mathbb{R}$ is {\it $\theta$-H\"older continuous} ($\theta>0$) if 
\[
\norm{\phi^+}_{\theta}:=\sup_{\ul{x}\in \Sigma^+}|\phi^+(\ul{x})|
+ \sup\left\{\frac{|\phi^+(\ul{x})-\phi^+(\ul{y})|}{\dist(\ul{x},\ul{y})^{\theta}}:\ul{x}\neq \ul{y}\in \Sigma^+\right\}<+\infty.
\]
Denote by $C_{\theta}(\Sigma^+)$ the set of all $\theta$-H\"older continuous functions on $\Sigma^+$. The norm $\|\cdot\|_\theta$ for $\theta$-H\"older continuous functions on $\Sigma$ is defined similarly. We say that a function $\vartheta:\Sigma^+\to\mathbb{R}$ is {\it locally $\theta$-H\"older continuous} if there is $C>0$ such that $|\vartheta(\ul{x})-\vartheta(\ul{y})|\leq C\cdot \dist(\ul x,\ul y)^{\theta}$ for all $\ul{x},\ul{y}\in\Sigma^+$ with $x_0=y_0$.
A function is $\theta$-H\"older continuous if and only if it is bounded and locally $\theta$-H\"older continuous.

There is a natural semi-conjugacy $\mathbf{p}^{+}: \Sigma\to\Sigma^{+}$ defined as $\mathbf{p}^{+}\big((u_n)_{n\in\mathbb{Z}}\big)=(u_n)_{n\geq 0}$. 
This semi-conjugacy induces a bijection between the set of invariant probability measures of $(\Sigma,\sigma)$ and $(\Sigma^+,\sigma)$ given by $\mu^{+}=\mu\circ (\mathbf{p}^{+})^{-1}$. The corresponding $\mu^{+}$ and $\mu$ have the same metric entropy, and $\mu^{+}$ is ergodic (resp. mixing) if and only if $\mu$ is ergodic (resp. mixing).

Every function $\phi^{+}:\Sigma^{+}\to\mathbb{R}$ can be lifted to a {\it one-sided function} $\phi^{+}\circ\mathbf{p}^{+}:\Sigma\to\mathbb{R}$. 
In general, a function on $\Sigma$ is not one-sided. But we have the following Sinai Lemma:
\begin{lem}{\rm \cite{Bow08}}\label{Sinailem}
For any $\theta>0$, there exists $C({\theta})>0$ such that for every $\theta$-H\"older continuous function $\phi$ on $\Sigma$, there are $\theta/2$-H\"older continuous functions $\phi^{+}:\Sigma^{+}\to\mathbb{R}$ and $\varsigma :\Sigma\to\mathbb{R}$ such that $\phi^{+}\circ\mathbf{p}^{+}=\phi+\varsigma-\varsigma\circ\sigma$,  and $\phi^{+}$ satisfies $\norm{\phi^{+}}_{\theta/2}\leq C({\theta})\norm{\phi}_{\theta}$.
\end{lem}

 We say that a one-sided TMS $(\Sigma^+,\sigma)$ is {\it SPR} for $\phi^{+}:\Sigma^{+}\to\mathbb{R}$ if the associated two-sided TMS $(\Sigma,\sigma)$ is SPR for $\phi^{+}\circ \mathbf{p}^{+}:\Sigma\to\mathbb{R}$. 
Let $\phi:\Sigma\to \bR$ and $\phi^{+}:\Sigma^{+}\to\bR$ be as in Lemma~\ref{Sinailem}. It is easy to see that $\mu$ is an equilibrium state of $\phi$ for $(\Sigma,\sigma)$ if and only if $\mu^{+}:=\mu\circ (\mathbf{p}^{+})^{-1}$ is an equilibrium state of $\phi^{+}$ for $(\Sigma^+,\sigma)$.
Moreover, if $(\Sigma,\sigma)$ is SPR for $\phi$, then $(\Sigma^{+},\sigma)$ is SPR for $\phi^{+}$.

\smallskip

Let $r:\Sigma^{+}\to (0,\infty)$ be a roof function (see Section 5.1.2). Then the suspension semi-flow $(\Sigma^{+}_r,\sigma_r)$ based on $(\Sigma^{+},\sigma)$ with respect to the roof function $r$ is called a {\it one-sided TMF}. The suspension semi-flow is well-defined here since the base map is a one-sided shift, and it preserves all the key  properties of the corresponding two-sided TMF.

\subsubsection{Ruelle operators}
Let $(\Sigma^{+},\sigma)$ be an irreducible one-sided TMS, and $\phi^+$ be a function on $\Sigma^+$.
The {\it Ruelle operator} $\mathcal{L}_{\phi^+}$ with respect to $\phi^+$ is defined by
the formal sum (may not converge)
\[
(\mathcal{L}_{\phi^+}f)(\ul{x}):=\sum_{\sigma(\ul{y})=\ul{x}}e^{\phi^+(\ul{y})}f(\ul{y}).
\]
Under the forward locally finite condition of the underlying graph, the sum is finite at every point $\ul{x}\in \Sigma^+$ and thus the operator is well-defined.

To study the  mixing rate of suspension semi-flows, we also consider complex Ruelle operators.
Let $r$ be a roof function on $\Sigma^+$.
For any $c=a+ib\in\mathbb{C}$ ($a,b\in\mathbb{R}$), 
the {\it complex Ruelle operator} $\mathcal{L}_{{\phi^+},c}$ with respect to $\phi^+$ and $c$ is defined by the formal sum (may not converge)
\[
(\mathcal{L}_{{\phi^+},c}f)(\ul{x}):=\sum_{\sigma(\ul{y})=\ul{x}}e^{{\phi^+}(\ul{y})+cr(\ul{y})} f(\ul{y}).
\]

\subsubsection{Spectral gap}

In the remaining part of this section, we always assume that $(\Sigma^{+},\sigma)$ is a locally compact topologically mixing one-sided TMS with finite Gurevich entropy, and that it is SPR for a $\theta_0$-H\"older continuous potential $\phi^+:\Sigma^{+}\to \bR$. Cyr-Sarig \cite{CS} proved the existence of spectral gap for the Ruelle operator in an appropriate Banach space. Note that an alternative definition of the SPR property is adopted in \cite{CS}. It is equivalent to the definition presented here. See \cite[Proposition 7.13]{BCS25} and \cite{RH}. 

For further use, we briefly explain how to construct the Banach space. Let $\mu$ be the unique equilibrium state of $\phi^+$. First, we have the following proposition:

\begin{pro}{\rm \cite[Proposition B.4]{BCS25}}\label{pro.BCS1to1}
	There exists a locally $\theta_0$-H\"older continuous function $\varsigma_1:\Sigma^+\to \bR$ such that 
\begin{enumerate}
	 \item\label{ite.barphi} $\sum_{\sigma(\ul{y})=\ul{x}}e^{\bar\phi(\ul{y})}=1$ for all $\ul{x}\in\Sigma^+$, where $\bar\phi:=\phi^++\varsigma_1-\varsigma_1\circ\sigma-P_{\mu}(\sigma,\phi^+)$; 
  \item\label{ite.transferopera} $\mathcal{L}_{\bar\phi}:L^1(\mu)\to L^1(\mu)$ is the {\it transfer operator} of $\mu$, that is, $\int \mathcal{L}_{\bar\phi}f\cdot g d\mu=\int f\cdot g\circ\sigma  d\mu$ for all $f\in L^1(\mu)$ and $g\in L^{\infty}(\mu)$.
\end{enumerate}
\end{pro}
It follows that $\mathcal{L}_{\bar{\phi}}\mathbbm{1}=\mathbbm{1}$, $\mathcal{L}_{\bar{\phi}}^*\mu=\mu$ and $P_{\mu}(\sigma,\bar\phi)=0$, where $\mathbbm{1}$ denotes the constant function taking value 1 everywhere on $\Sigma^+$.
Denote by $[v]$ the cylinder $\{\ul{x}\in \Sigma^+:x_0=v\}$. 

For any $\theta\in (0,\theta_0]$ and $s\in \mathscr{V}$, it is proved in \cite[\S\,3]{BCS25} that there exist numbers $p_s>0$ and $\varepsilon_{s}>0$ such that ${p_s}\in(0,\theta)$ and $(\Sigma^{+},\sigma)$ is SPR for the potential $\chi_s:=\bar{\phi}+\varepsilon_{s}-p_s\cdot \mathbbm{1}_{[s]}$ with top pressure $P_{\rm TOP}(\sigma,\chi_s)=0$. Then by  the Ruelle-Perron-Frobenius theorem, there exists a positive continuous $h_{s,0}:\Sigma^{+}\to\mathbb{R}$ and a $\sigma$-invariant measure $\nu_s$ such that $\mathcal{L}_{\chi_s}h_{s,0}=h_{s,0}$, $\mathcal{L}_{\chi_s}^*\nu_s=\nu_s$ and $\norm{h_{s,0}}_{s}'=1$,  
 where $\norm{\cdot}_{s}'$ is the $L^1$-semi-norm with respect to $\nu_s$. Furthermore, there exists $C_{s,1}>1$ such that for all $v\in\mathscr{V}$ and $\ul{x}\in[v]$, 
 \begin{equation}\label{equ.hs0Cs0}
 	h_{s,0}[v]\le C_{s,1}h_{s,0}(\ul{x}), 
 \end{equation}
  where $h_{s,0}[v]:=\sup \{h_{s,0}(\ul{x}):\ul{x}\in[v]\}$.

Let $\mathscr{L}_s$ be the collection of continuous functions $f: \Sigma^+\to\mathbb{R}$ such that 
\[
\norm{f}_{s}:=\sup_{v\in \mathscr{V}}\frac{1}{h_{s,0}[v]}\left( \sup_{\ul{x}\in[v]}|f(\ul{x})|
+ \sup\left\{\frac{|f(\ul{x})-f(\ul{y})|}{e^{-t_s(\ul{x},\ul{y})\cdot\theta }}:\ul{x}\neq \ul{y}\in[v]\right\}\right)<\infty,
\]
where $t_s(\ul{x},\ul{y}):=\#\{0\leq j\leq t(\ul{x},\ul{y}):x_j=y_j=s\}$ and $t(\ul{x},\ul{y}):=\min\{n\geq0: x_n\neq y_n\}$. It is clear that $(\mathscr{L}_s, \norm{\cdot}_s)$ is a Banach space. Moreover, we have the following theorem:

\begin{thm}{\rm (\cite{CS} and \cite[Theorem B.7]{BCS25})}\label{thm.CS2009}
	We have:
\begin{enumerate}
		\item $(\mathcal{L}_{\bar\phi}f)(\ul{x})$ converges for all $f\in \mathscr{L}_s$ and $\ul{x}\in \Sigma^+$;
		\item $\mathcal{L}_{\bar{\phi}}$ is a positive, bounded linear operator, which preserves $\mathscr{L}_s$; 
  		\item\label{CS.spectgap} $\mathcal{L}_{\bar{\phi}}= P_s + N_s$, where $P_s$ and $N_s$ are bounded linear operators on $\mathscr{L}_s$ such that $P_sN_s = N_sP_s = 0$, $P_s^2 = P_s$, ${\rm rank}(P_s) = 1$ and $N_s$ has spectral radius strictly less than $1$.		
		\item $\mathscr{L}_s$ contains all $\theta$-H\"older continuous functions on $\Sigma^+$; moreover, if $g\in C_{\theta}(\Sigma^{+})$ and $f\in\mathscr{L}_s$, then $\norm{f\cdot g}_{s}\leq \norm{g}_{\theta}\cdot \norm{f}_{s}$;	
		\item\label{CS.Pisint} $\mathscr{L}_s\subset L^1(\mu)$ and $P_sf = (\int f \, d\mu) \cdot \mathbbm{1}$ for all $f\in\mathscr{L}_s$;
		\item if $f\in\mathscr{L}_s$ then $|f|\in\mathscr{L}_s$ and $\norm{|f|}_s\le \norm{f}_s$.
\end{enumerate}
\end{thm}

Since $P_s$ is bounded, there exists $C_{s,2}>1$ such that 
\begin{equation}\label{equ.boundPs}
	\norm{P_sf}_s\le C_{s,2}\norm{f}_s
\end{equation}
for every $f\in\mathscr{L}_s$. By Items~\ref{CS.spectgap} and \ref{CS.Pisint} of the above theorem, we have $\mathcal{L}^n_{\bar{\phi}}= P_s + N^n_s$, and there exist $C_{s,3}>1$ and $\delta_s\in(0,1)$ such that for any $f\in\mathscr{L}_s$,
\begin{equation}\label{equ.SpectralGap}
  \norm{\mathcal{L}^n_{\bar{\phi}}f- \int f \, d\mu \cdot \mathbbm{1}}_s=\norm{\mathcal{L}^n_{\bar{\phi}}f- P_sf}_s=\norm{N^n_sf}_s\le C_{s,3}\cdot \delta_s^n\norm{f}_{s}.
\end{equation}

\subsection{A Lasota-Yorke type inequality}\label{subsec.LasoYork}

Let $r:\Sigma^+\to \bR$ be a H\"older continuous roof function such that $\inf r>0$. Fix some $\theta\in (0,1)$ such that both $\bar\phi$ and $r$ are locally $\theta$-H\"older continuous.
In the following, we write $\mathcal{L}_c=\mathcal{L}_{\bar{\phi},c}$ for simplicity. In particular, $\mathcal{L}_0=\mathcal{L}_{\bar{\phi}}$. For any $s\in\mathscr{V}$ and any $f\in\mathscr{L}_s$, denote 
\[
\Upsilon_{s,1}(f)=\sup_{v\in\mathscr{V}}\left(\frac{1}{{h}_{s,0}[v]}\sup_{\ul{x}\in[v]}|f(\ul{x})|\right) \quad \text{and} \quad
\Upsilon_{s,2}(f)=\sup_{v\in\mathscr{V}}\left(\frac{1}{{h}_{s,0}[v]}\sup\left\{\frac{|f(\ul{x})-f(\ul{y})|}{e^{-t_s(\ul{x},\ul{y})\cdot \theta}}:\ul{x}\neq\ul{y}\in[v]\right\}\right). 
\]

It is straightforward to verify that $\Upsilon_{s,1}(|f|)=\Upsilon_{s,1}(f)$, $\Upsilon_{s,2}(|f|)\le \Upsilon_{s,2}(f)$ and $0\leq \Upsilon_{s,1}(f),\Upsilon_{s,2}(f)\leq \norm{f}_{s}\leq \Upsilon_{s,1}(f)+\Upsilon_{s,2}(f)$.
Moreover, we present the following estimates for $\Upsilon_{s,1}(\mathcal{L}^n_{ib}f)$ and $\Upsilon_{s,2}(\mathcal{L}^n_{ib}f)$.

\begin{lem}\label{lem.Upsilon1}
For any $s\in\mathscr{V}$, there exists $C_{s,4}>0$ such that for any $f\in\mathscr{L}_s$, $b\in\bR$ and $k,n\in\bZ^+$, 
\[
\Upsilon_{s,1}(\mathcal{L}^n_{ib}f)\leq \Upsilon_{s,1}(\mathcal{L}^n_{0}|f|)\leq C_{s,4}\cdot e^{(k+1)p_s}\left(\norm{f}'_s+(e^{-k\theta}+e^{-n\varepsilon_s})\norm{f}_s\right).
\]	
\end{lem}
\begin{proof}
	Note that for any $\ul{x}\in \Sigma^+$, 
\begin{equation}\label{equ.LibfL0absf}
\big|(\mathcal{L}_{ib}^nf)(\ul{x})\big|= \bigg|\sum_{\sigma^n(\ul{y})=\ul{x}}e^{\bar{\phi}_n(\ul{y})+ibr_n(\ul{y})} f(\ul{y})\bigg|\le 
\sum_{\sigma^n(\ul{y})=\ul{x}}e^{\bar{\phi}_n(\ul{y})} |f(\ul{y})|=(\mathcal{L}^n_{0}|f|)(\ul{x}).
\end{equation}	
Thus $\Upsilon_{s,1}(\mathcal{L}^n_{ib}f)\leq \Upsilon_{s,1}(\mathcal{L}^n_{0}|f|)$. 
Recall that $\norm{|f|}_s\le \norm{f}_s$ and $\norm{\cdot}_{s}'$ is the $L^1$-semi-norm with respect to $\nu_s$. It follows from \cite[Formula (3.5)]{CS} that there exists $C_{s,4}>0$ such that for any $v\in \mathscr{V}$, $\ul{x}\in [v]$, and $k, n\in\bZ^+$,
\begin{equation}\label{equ.CyrSarig3.5}
\begin{aligned}
(\mathcal{L}^n_{0}|f|)(\ul{x})&\le C_{s,4}\cdot e^{(k+1)p_s}\left(\norm{|f|}'_s+(e^{-k\theta}+e^{-n\varepsilon_s})\norm{|f|}_s\right)h_{s,0}[v]\\
&\le C_{s,4}\cdot e^{(k+1)p_s}\left(\norm{f}'_s+(e^{-k\theta}+e^{-n\varepsilon_s})\norm{f}_s\right)h_{s,0}[v].	
\end{aligned}
\end{equation}
This implies the second inequality.
\end{proof}

\begin{lem}\label{lem.Upsilon2}
For any $s\in\mathscr{V}$, there exists $C_{s,5}>1$ such that for all $f\in\mathscr{L}_s$, $b\in \bR$, and $k, n\in\bZ^+$, 	
\[
\Upsilon_{s,2}(\mathcal{L}^n_{ib}f)\leq C_{s,5}(1+|b|) e^{(k+1)p_s}\norm{f}'_s+C_{s,5}\left(e^{-n\varepsilon_s}+ e^{(k+1)p_s}(1+|b|) (e^{-k\theta}+e^{-n\varepsilon_s})\right)\norm{f}_s.
\]
\end{lem}
\begin{proof}
For any $v\in \mathscr{V}$ and $n\geq 1$, denote $P_n(v):=\{\ul{\varpi}=(\varpi_0,\dots,\varpi_{n-1}):(\varpi_{n-1},v) \text{ is an admissible word}\}$ and  $n_s(\ul{\varpi})=\#\{0\leq i\leq n-1:\varpi_i=s\}$. Then for any $\ul{x}\in[v]$ and $n\geq 1$,
\[
(\mathcal{L}_{ib}^nf)(\ul{x})=\sum_{\ul{\varpi}\in P_n(v)}e^{\bar{\phi}_n(\ul{\varpi}\ul{x})+ibr_n(\ul{\varpi}\ul{x})}\cdot f(\ul{\varpi}\ul{x}),
\]
where $\ul{\varpi}\ul{x}:=(\varpi_0,\dots,\varpi_{n-1},v,v_1,v_2,\dots)$ for $\ul{x}=(v,v_1,v_2,\dots)$.
Therefore, for any $\ul{x},\ul{y}\in[v]$,
\begin{equation}\label{eq.Ups2}
  \begin{aligned}
\big|(\mathcal{L}^n_{ib}f)(\ul{x})-(\mathcal{L}_{ib}^nf)(\ul{y})\big|\leq 
&\sum_{\ul{\varpi}\in P_n(v)}e^{\bar{\phi}_n(\ul{\varpi}\ul{y})}|f(\ul{\varpi}\ul{x})-f(\ul{\varpi}\ul{y})|\\
+&\sum_{\ul{\varpi}\in P_n(v)}e^{\bar{\phi}_n(\ul{\varpi}\ul{x})}|f(\ul{\varpi}\ul{x})|\cdot \left|1-e^{(\bar{\phi}_n(\ul{\varpi}\ul{y})-\bar{\phi}_n(\ul{\varpi}\ul{x}))+ib(r_n(\ul{\varpi}\ul{y})-r_n(\ul{\varpi}\ul{x}))}\right|.
\end{aligned}
\end{equation}

Note that $t_s(\ul{\varpi}\ul{x},\ul{\varpi}\ul{y})=n_s(\ul{\varpi})+t_s(\ul{x},\ul{y})$ and $h_{s,0}[\varpi_0]\le C_{s,1}h_{s,0}(\ul{\varpi}\ul{y})$ (by \eqref{equ.hs0Cs0}). Thus
\[
|f(\ul{\varpi}\ul{x})-f(\ul{\varpi}\ul{y})|\leq  C_{s,1}h_{s,0}(\ul{\varpi}\ul{y})e^{-(n_s(\ul{\varpi})+t_s(\ul{x},\ul{y}))\theta}\norm{f}_s.
\]
Combining the above estimate with the fact $\bar{\phi}_n(\ul{\varpi}\ul{y})=(\chi_s)_n(\ul{\varpi}\ul{y})-n\varepsilon_s+p_sn_s(\ul{\varpi})$, we have
\begin{align}
&\sum_{\ul{\varpi}\in P_n(v)}e^{\bar{\phi}_n(\ul{\varpi}\ul{y})}|f(\ul{\varpi}\ul{x})-f(\ul{\varpi}\ul{y})| \notag\\
 \leq\, &e^{-n\varepsilon_s}e^{-t_s(\ul{x},\ul{y})\cdot\theta} C_{s,1} \norm{f}_s \sum_{\ul{\varpi}\in P_n(v)} e^{(p_s-\theta)n_s(\ul{\varpi})} e^{(\chi_s)_n(\ul{\varpi}\ul{y})}h_{s,0}(\ul{\varpi}\ul{y}) \notag\\
 \leq\, & e^{-n\varepsilon_s-t_s(\ul{x},\ul{y})\cdot\theta} C_{s,1} \norm{f}_s h_{s,0}(\ul{y})  \quad\quad\quad\text{(because ${p_s}<\theta$ and $\mathcal{L}_{\chi_s}h_{s,0}=h_{s,0}$)}\notag\\
 \leq\, & e^{-n\varepsilon_s-t_s(\ul{x},\ul{y})\cdot\theta} C_{s,1} \norm{f}_s h_{s,0}[v].\label{eq.varphi21}
\end{align}

On the other hand, it follows from the locally $\theta$-H\"older property of $\bar{\phi}$ and $r$ that 
\[
|\bar{\phi}_n(\ul{\varpi}\ul{y})-\bar{\phi}_n(\ul{\varpi}\ul{x})|\leq \frac{C'_1}{1-e^{-\theta}}\cdot e^{-t(\ul{x},\ul{y})\cdot \theta} \quad \text{and} \quad |r_n(\ul{\varpi}\ul{y})-r_n(\ul{\varpi}\ul{x})|\leq \frac{C'_2}{1-e^{-\theta}}\cdot e^{-t(\ul{x},\ul{y})\cdot \theta},
\]
where $C'_1$ and $C'_2$ are the local $\theta$-H\"older constants of  $\bar{\phi}$ and $r$ respectively.
Define $C'_3:=(1-e^{-\theta})^{-1}\cdot\max \{C'_1,C'_2\}$. Then
\begin{equation}\label{equ.constantC4p}
	\begin{aligned}
	& \left|1-e^{(\bar{\phi}_n(\ul{\varpi}\ul{y})-\bar{\phi}_n(\ul{\varpi}\ul{x}))+ib(r_n(\ul{\varpi}\ul{y})-r_n(\ul{\varpi}\ul{x}))}\right|\\
	  \leq & \left|(\bar{\phi}_n(\ul{\varpi}\ul{y})-\bar{\phi}_n(\ul{\varpi}\ul{x}))+ib(r_n(\ul{\varpi}\ul{y})-r_n(\ul{\varpi}\ul{x}))\right|\cdot\max \{1, e^{\bar{\phi}_n(\ul{\varpi}\ul{y})-\bar{\phi}_n(\ul{\varpi}\ul{x})}\} \\
	  \leq &\, C'_3(1+|b|)e^{-t(\ul{x},\ul{y})\cdot \theta}\cdot\max \{1, e^{C'_3\cdot e^{-t(\ul{x},\ul{y})\cdot \theta}}\}\\
	  \leq &\, e^{C'_3}\cdot C'_3(1+|b|)e^{-t(\ul{x},\ul{y})\cdot \theta}=:C'_4 (1+|b|)e^{-t(\ul{x},\ul{y})\cdot \theta}.
\end{aligned}
\end{equation}
Combining \eqref{equ.constantC4p} with \eqref{equ.CyrSarig3.5}, we have that for any $k\in \bZ^+$,
\begin{equation}\label{eq.varphi22}
\begin{aligned}
	& \sum_{\ul{\varpi}\in P_n(v)}e^{\bar{\phi}_n(\ul{\varpi}\ul{x})}|f(\ul{\varpi}\ul{x})|\cdot \left|1-e^{(\bar{\phi}_n(\ul{\varpi}\ul{y})-\bar{\phi}_n(\ul{\varpi}\ul{x}))+ib(r_n(\ul{\varpi}\ul{y})-r_n(\ul{\varpi}\ul{x}))}\right|\\
	\leq & \, C'_4 (1+|b|)e^{-t(\ul{x},\ul{y})\cdot \theta}\cdot (\mathcal{L}^n_0|f|)(\ul{x})\\
	\leq &\, C'_4 (1+|b|)e^{-t_s(\ul{x},\ul{y})\cdot \theta}\cdot h_{s,0}[v] C_{s,4}\cdot e^{(k+1)p_s}\left(\norm{f}'_s+(e^{-k\theta}+e^{-n\varepsilon_s})\norm{f}_s\right).
\end{aligned}
\end{equation} 

Take $C_{s,5}:=\max\{C_{s,1},C'_4C_{s,4}\}>1$. Then the conclusion of the lemma follows immediately by combining \eqref{eq.Ups2} with \eqref{eq.varphi21} and \eqref{eq.varphi22}.
\end{proof}

For subsequent derivation, we also need the following estimate for $\mathcal{L}_{ib}$.

\begin{lem}\label{lem.onestepestim}
	There exists $C_{s,6}>1$ such that for all $f\in\mathscr{L}_s$ and $|b|\ge 1$, we have
	\begin{enumerate}
  \item\label{item.onestepU1}  $\Upsilon_{s,1}(\mathcal{L}_{ib}f)\le C_{s,6}\cdot \Upsilon_{s,1}(f)$;
  \item\label{item.onestepU2} $\Upsilon_{s,2}(\mathcal{L}_{ib}f)\le C_{s,6}|b|\cdot \Upsilon_{s,1}(f)+C_{s,6}\cdot \Upsilon_{s,2}(f)$.
\end{enumerate}
\end{lem}
\begin{proof}
The proof of this lemma is analogous to that of the previous two lemmas. It suffices to replace $\norm{f}_s$ in the original proofs with the more precise $\Upsilon_{s,1}(f)$ and $\Upsilon_{s,2}(f)$. 

By \eqref{equ.hs0Cs0} and the definition of $\Upsilon_{s,1}$, we have $|f(\ul{z})|\le C_{s,1}\Upsilon_{s,1}(f)h_{s,0}(\ul{z})$ for all $\ul{z}\in \Sigma^+$.
	Recall that $\bar{\phi}(\ul{z})=\chi_s(\ul{z})-\varepsilon_{s}+p_s\cdot \mathbbm{1}_{[s]}(\ul{z})\le \chi_s(\ul{z})-\varepsilon_{s}+p_s$ and $\mathcal{L}_{\chi_s}h_{s,0}=h_{s,0}$. Then for any $v\in \mathscr{V}$ and $\ul{x}\in [v]$, we have
\begin{equation}\label{equ.onstepL0}
\begin{aligned}
(\mathcal{L}_{0}|f|)(\ul{x})= 
\sum_{\sigma(\ul{z})=\ul{x}}e^{\bar{\phi}(\ul{z})} |f(\ul{z})|&\le e^{p_s-\varepsilon_s}\sum_{\sigma(\ul{z})=\ul{x}}e^{{\chi_s}(\ul{z})} |f(\ul{z})|\\
&\le e^{p_s-\varepsilon_s} C_{s,1}\Upsilon_{s,1}(f)\sum_{\sigma(\ul{z})=\ul{x}}e^{{\chi_s}(\ul{z})} h_{s,0}(\ul{z})\\
&\le e^{p_s-\varepsilon_s} C_{s,1}\Upsilon_{s,1}(f) h_{s,0}(\ul{x}).	
\end{aligned}	
\end{equation}
 Combining the above estimate with the fact $\big|(\mathcal{L}_{ib}f)(\ul{x})\big|\le (\mathcal{L}_{0}|f|)(\ul{x})$, we conclude that
\begin{equation}\label{equ.onestepU1}
\Upsilon_{s,1}(\mathcal{L}_{ib}f)\leq \Upsilon_{s,1}(\mathcal{L}_{0}|f|)\le e^{p_s-\varepsilon_s} C_{s,1}\Upsilon_{s,1}(f).
\end{equation}

Recall that $P_1(v):=\{\varpi:(\varpi,v) \text{ is an admissible word}\}$. As in \eqref{eq.Ups2}, for any $\ul{x},\ul{y}\in[v]$, we have
\begin{equation}\label{eq.oneUps2}
  \begin{aligned}
\big|(\mathcal{L}_{ib}f)(\ul{x})-(\mathcal{L}_{ib}f)(\ul{y})\big|\leq 
&\sum_{{\varpi}\in P_1(v)}e^{\bar{\phi}({\varpi}\ul{y})}|f({\varpi}\ul{x})-f({\varpi}\ul{y})|\\
+&\sum_{{\varpi}\in P_1(v)}e^{\bar{\phi}({\varpi}\ul{x})}|f({\varpi}\ul{x})|\cdot \left|1-e^{(\bar{\phi}({\varpi}\ul{y})-\bar{\phi}({\varpi}\ul{x}))+ib(r({\varpi}\ul{y})-r({\varpi}\ul{x}))}\right|.
\end{aligned}
\end{equation}
Since $t_s({\varpi}\ul{x},{\varpi}\ul{y})\ge t_s(\ul{x},\ul{y})$ and $h_{s,0}[\varpi]\le C_{s,1}h_{s,0}({\varpi}\ul{y})$ (by \eqref{equ.hs0Cs0}), we conclude that
\[
|f({\varpi}\ul{x})-f({\varpi}\ul{y})|\leq h_{s,0}[\varpi]e^{-t_s({\varpi}\ul{x},{\varpi}\ul{y})\theta}\Upsilon_{s,2}(f)\leq  C_{s,1}h_{s,0}({\varpi}\ul{y})e^{-t_s(\ul{x},\ul{y})\theta}\Upsilon_{s,2}(f).
\]
Therefore, using  $\bar{\phi}(\varpi\ul{y})\le \chi_s(\varpi\ul{y})-\varepsilon_{s}+p_s$ and $\mathcal{L}_{\chi_s}h_{s,0}=h_{s,0}$ again, we have
\begin{equation}\label{equ.onestepU21stsum}
	\sum_{{\varpi}\in P_1(v)}e^{\bar{\phi}({\varpi}\ul{y})}|f({\varpi}\ul{x})-f({\varpi}\ul{y})|  \leq e^{p_s-\varepsilon_s}e^{-t_s(\ul{x},\ul{y})\cdot\theta} C_{s,1} \Upsilon_{s,2}(f) h_{s,0}(\ul{y}).
\end{equation}

On the other hand, it follows from  \eqref{equ.constantC4p} (taking $n=1$) and \eqref{equ.onstepL0} that
\begin{equation}\label{equ.onestepU22ndsum}
\begin{aligned}
	& \sum_{{\varpi}\in P_1(v)}e^{\bar{\phi}({\varpi}\ul{x})}|f({\varpi}\ul{x})|\cdot \left|1-e^{(\bar{\phi}({\varpi}\ul{y})-\bar{\phi}({\varpi}\ul{x}))+ib(r({\varpi}\ul{y})-r({\varpi}\ul{x}))}\right|\\
	\leq & \, C'_4 (1+|b|)e^{-t(\ul{x},\ul{y})\cdot \theta}\cdot e^{p_s-\varepsilon_s} C_{s,1}\Upsilon_{s,1}(f) h_{s,0}(\ul{x})\\
	\leq &\, 2C'_4 |b|e^{-t_s(\ul{x},\ul{y})\cdot \theta}\cdot e^{p_s-\varepsilon_s} C_{s,1}\Upsilon_{s,1}(f) h_{s,0}(\ul{x}). \quad\quad\quad\text{(since $|b|\ge 1$)}
\end{aligned}
\end{equation} 
Substituting the two estimates above into \eqref{eq.oneUps2}, we have
\begin{equation}\label{equ.onestepU2}
	\Upsilon_{s,2}(\mathcal{L}_{ib}f)\leq 2C_4'|b|e^{p_s-\varepsilon_s} C_{s,1}\Upsilon_{s,1}(f)+e^{p_s-\varepsilon_s} C_{s,1}\Upsilon_{s,2}(f).
\end{equation}

The proof is completed by taking $C_{s,6}:=\max\{2,e^{p_s-\varepsilon_s} C_{s,1},2C_4'e^{p_s-\varepsilon_s} C_{s,1}\}>1$.
\end{proof}

Using the same method, we can also estimate the perturbation of $\mathcal{L}_{ib}$. Let $C_5'>1$ be the uniform supremum bound of the roof function $r$.

\begin{lem}\label{lem.estimofaib}
	There exists $C_{s,7}>0$ such that $\norm{\mathcal{L}_{a+ib}-\mathcal{L}_{ib}}_{s}\le C_{s,7}|a||b|$ for all $|b|\ge 1$ and $|a|\le 1/C_5'$.
\end{lem}
\begin{proof}
Recall that the roof function $r$ has local $\theta$-H\"older constant $C_2'$ and upper bound $C_5'$. For any $|a|\le 1/C_5'$, by the Lipschitz property of the exponential function, we have $|e^{ar(\ul{z})}-1|\le 2|a|C_5'$ for all $\ul{z}\in \Sigma^+$, and  
$|e^{ar(\ul{z_1})}-e^{ar(\ul{z_2})}| \le e|a|C_2' e^{-t(\ul{z_1},\ul{z_2}) \theta}$ for all 
$\ul{z_1},\ul{z_2}\in \Sigma^+$.

Thus for any $|a|\le 1/C_5'$, $|b|\ge 1$, $f\in\mathscr{L}_s$, $v\in \mathscr{V}$, and $\ul{x}\in [v]$, we have
\[
	\big|\big((\mathcal{L}_{a+ib}-\mathcal{L}_{ib})f\big)(\ul{x})\big|=\bigg|\sum_{\sigma(\ul{z})=\ul{x}}\big(e^{\bar{\phi}(\ul{z})+ibr(\ul{z})} f(\ul{z})\cdot(e^{ar(\ul{z})}-1)\big)\bigg|\le 2|a|C_5'\cdot(\mathcal{L}_0|f|)(\ul{x}).
\] 
By \eqref{equ.onestepU1}, we conclude that
\begin{equation}\label{equ.aibU1}
\Upsilon_{s,1}((\mathcal{L}_{a+ib}-\mathcal{L}_{ib})f)\leq 2|a|C_5'\Upsilon_{s,1}(\mathcal{L}_{0}|f|)\le 2|a|C_5' e^{p_s-\varepsilon_s} C_{s,1}\Upsilon_{s,1}(f)\le  2|a|C_5' e^{p_s-\varepsilon_s} C_{s,1} \norm{f}_s.
\end{equation}

For any $\ul{x},\ul{y}\in [v]$, we have
\[
\begin{aligned}
	&\, \big|\big((\mathcal{L}_{a+ib}-\mathcal{L}_{ib})f\big)(\ul{x})-\big((\mathcal{L}_{a+ib}-\mathcal{L}_{ib})f\big)(\ul{y})\big| \\
	=&\, \bigg|\sum_{\varpi\in P_1(v)}\big(e^{\bar{\phi}(\varpi\ul{x})+ibr(\varpi\ul{x})} f(\varpi\ul{x})\cdot(e^{ar(\varpi\ul{x})}-1)-e^{\bar{\phi}(\varpi\ul{y})+ibr(\varpi\ul{y})} f(\varpi\ul{y})\cdot(e^{ar(\varpi\ul{y})}-1)\big)\bigg|\\
	\le  &\, \sum_{\varpi\in P_1(v)}\big|e^{\bar{\phi}(\varpi\ul{y})+ibr(\varpi\ul{y})}\big|\cdot \big| f(\varpi\ul{y})\big|\cdot \big|e^{ar(\varpi\ul{x})}-e^{ar(\varpi\ul{y})}\big|\\
	&\, +\sum_{\varpi\in P_1(v)}\big|e^{\bar{\phi}(\varpi\ul{y})+ibr(\varpi\ul{y})}\big| \cdot\big|f(\varpi\ul{x})-f(\varpi\ul{y})\big|\cdot |e^{ar(\varpi\ul{x})}-1|\\
	&\, +\sum_{\varpi\in P_1(v)}\big|e^{\bar{\phi}(\varpi\ul{x})+ibr(\varpi\ul{x})} \big|\cdot \big|f(\varpi\ul{x})\big|\cdot\big|1-e^{(\bar{\phi}(\varpi\ul{y})-\bar{\phi}(\varpi\ul{x}))+i(br(\varpi\ul{y})-br(\varpi\ul{x}))} \big|\cdot |e^{ar(\varpi\ul{x})}-1|.
\end{aligned}
\] 
For the first summand, by \eqref{equ.onstepL0}, we have
\[
\begin{aligned}
	\sum_{\varpi\in P_1(v)}\big|e^{\bar{\phi}(\varpi\ul{y})+ibr(\varpi\ul{y})}\big|\cdot \big| f(\varpi\ul{y})\big| \big|e^{ar(\varpi\ul{x})}-e^{ar(\varpi\ul{y})}\big| &\le e|a|C_2' e^{-(t(\ul{x},\ul{y})+1) \theta}\cdot(\mathcal{L}_0|f|)(\ul{y})\\
	&\le e|a|C_2' e^{-t_s(\ul{x},\ul{y}) \theta}\cdot e^{p_s-\varepsilon_s} C_{s,1}\Upsilon_{s,1}(f) h_{s,0}(\ul{y}).
\end{aligned}
\]
By \eqref{equ.onestepU21stsum} and \eqref{equ.onestepU22ndsum}, the sum of the last two summands is bounded by 
\[
2|a|C_5'\cdot \big( e^{p_s-\varepsilon_s}e^{-t_s(\ul{x},\ul{y})\cdot\theta} C_{s,1} \Upsilon_{s,2}(f) h_{s,0}(\ul{y}) +2C'_4 |b|e^{-t_s(\ul{x},\ul{y})\cdot \theta}\cdot e^{p_s-\varepsilon_s} C_{s,1}\Upsilon_{s,1}(f) h_{s,0}(\ul{x}) \big).
\]

Therefore, 
\begin{equation}\label{equ.aibU2}
\Upsilon_{s,2}((\mathcal{L}_{a+ib}-\mathcal{L}_{ib})f)\leq |a|e^{p_s-\varepsilon_s} C_{s,1}\big(e C_2'+2C_5'+4C_5'C_4'|b|  \big)\norm{f}_s.
\end{equation}
Combining \eqref{equ.aibU1} with \eqref{equ.aibU2}, we conclude that
\[
\norm{\mathcal{L}_{a+ib}-\mathcal{L}_{ib}}_{s}\le C_{s,7}|a||b|,
\]
where $C_{s,7}:=e^{p_s-\varepsilon_s} C_{s,1}\big( 2C_5'+eC_2'+2C_5'+4C_5'C_4'\big)$.
\end{proof}

Combining Lemmas~\ref{lem.Upsilon1} with \ref{lem.Upsilon2}, we obtain the following uniform norm bound for  $\mathcal{L}_{ib}^nf$.

\begin{lem}\label{lem.norml}
There exists $R_s>1$ such that $\norm{\mathcal{L}_{ib}^nf}_s\leq R_s|b|\cdot\norm{f}_s$ for all $f\in\mathscr{L}_s$, $|b|\geq 1$, and $n\in\bZ^+$.
\end{lem}
\begin{proof}
Note that for any  $f\in\mathscr{L}_s$ and $\ul{x}\in[v]$,
\[
|f(\ul{x})|\leq \frac{|f(\ul{x})|}{{h}_{s,0}(\ul{x})}\cdot {h}_{s,0}(\ul{x})\leq C_{s,1}\cdot \frac{|f(\ul{x})|}{{h}_{s,0}[v]}\cdot {h}_{s,0}(\ul{x})\leq C_{s,1}\Upsilon_{s,1}(f)\cdot {h}_{s,0}(\ul{x}).
\]
Integrating both sides of the above inequality with respect to the measure $\nu_s$ and recalling that $\int h_{s,0}d\nu_s=1$, we obtain
\begin{equation}\label{eq.normfC}
  \norm{f}'_s\leq C_{s,1}\Upsilon_{s,1}(f)\le  C_{s,1}\norm{f}_s.
\end{equation}

Then according to Lemmas~\ref{lem.Upsilon1} and \ref{lem.Upsilon2}, for any $|b|\ge 1$ and $k, n\in\bZ^+$, we have
\[
\begin{aligned}
\norm{\mathcal{L}_{ib}^nf}_s \le &\, \Upsilon_{s,1}(\mathcal{L}_{ib}^nf)+\Upsilon_{s,2}(\mathcal{L}_{ib}^nf)\\
 \le &\, C_{s,4}\cdot e^{(k+1)p_s}(C_{s,1}+e^{-k\theta}+e^{-n\varepsilon_s})\norm{f}_s\\
& + \left(C_{s,5}(1+|b|) e^{(k+1)p_s}C_{s,1}+C_{s,5}(e^{-n\varepsilon_s}+ e^{(k+1)p_s}(1+|b|) (e^{-k\theta}+e^{-n\varepsilon_s}))\right)\norm{f}_s.
\end{aligned}
\]

Fix $k=1$. Since $|b|\ge 1$ and $e^{-n\varepsilon_s}<1$, there exists a constant  $R_s>1$ such that $\norm{\mathcal{L}_{ib}^nf}_s\leq R_s|b|\cdot\norm{f}_s$. This completes the proof.
\end{proof}

We also get the following  Lasota-Yorke type inequality. 

\begin{lem}\label{lem.LYineq}
There exist numbers $\lambda_{s,0}>1$ and $\varrho_{s,0}>0$ such that for all $|b|\ge 3$, $n\geq \varrho_{s,0}\log|b|$, and $f\in\mathscr{L}_s$, we have
\[
\Upsilon_{s,2}(\mathcal{L}^n_{ib}f)\leq |b|^{\lambda_{s,0}}\Upsilon_{s,1}(f)+\frac{1}{2}\norm{f}_{s}.
\]
\end{lem}
\begin{proof}
It follows from Lemma~\ref{lem.Upsilon2} and \eqref{eq.normfC} that for any $f\in\mathscr{L}_s$, $|b|\ge 1$ and $k, n\in\bZ^+$, 	
\[
\begin{aligned}
	\Upsilon_{s,2}(\mathcal{L}^n_{ib}f)& \leq C_{s,5}(1+|b|) e^{(k+1)p_s}C_{s,1}\Upsilon_{s,1}(f)+C_{s,5}\left(e^{-n\varepsilon_s}+ e^{(k+1)p_s}(1+|b|) (e^{-k\theta}+e^{-n\varepsilon_s})\right)\norm{f}_s\\
	& \le 2C_{s,5}C_{s,1} e^{(k+1)p_s}|b|\Upsilon_{s,1}(f)+\left(C_{s,5}e^{-n\varepsilon_s}+(2C_{s,5} |b| e^{p_s}( e^{p_s-\theta})^k+2C_{s,5} |b| e^{(k+1)p_s-n\varepsilon_s})\right)\norm{f}_s.
\end{aligned}
\]

Recall that $p_s<\theta$. Define $k_0:=\lfloor\frac{\log(8C_{s,5}\cdot e^{p_s}|b|)}{\theta -p_s}\rfloor+1$, where $\lfloor a \rfloor$ denotes the greatest integer less than or equal to $a$. Then we have 
\begin{equation}\label{equ.normpart1}
	2C_{s,5} |b| e^{p_s}( e^{p_s-\theta})^{k_0}<1/4.
\end{equation}

 Let
\[
\lambda_{s,0}:=\log (2C_{s,5}C_{s,1})+\left(\frac{\log(8C_{s,5}\cdot e^{p_s})}{\theta -p_s}+2 \right)p_s+\frac{p_s}{\theta-p_s}+1>1.
\]
Then for any $|b|\ge 3$,
\[
(\lambda_{s,0}-1)\log |b|\ge \log (2C_{s,5}C_{s,1})+\left(\frac{\log(8C_{s,5}\cdot e^{p_s})}{\theta-p_s}+2 \right)p_s+\frac{p_s}{\theta-p_s}\log|b|\ge \log (2C_{s,5}C_{s,1})+(k_0+1)p_s.
\]
This directly implies the desired exponential bound for the first term:
\begin{equation}\label{equ.Ups1lambda}
	2C_{s,5}C_{s,1} e^{(k+1)p_s}|b|\le |b|^{\lambda_{s,0}}.
\end{equation}

Denote
\[
\varrho_{s,0}:=\frac{1}{\varepsilon_s}\left(\log(16C_{s,5})+\left(\frac{\log(8C_{s,5}\cdot e^{p_s})}{\theta-p_s}+2 \right)p_s+\frac{p_s}{\theta-p_s}+1 \right).
\]
Following an argument analogous to the derivation of \eqref{equ.Ups1lambda}, for any $|b|\geq 3$ and $n\geq \varrho_{s,0}\log|b|$, we have 
\begin{equation}\label{equ.normpart2and3}
	C_{s,5}e^{-n\varepsilon_s}<2C_{s,5} |b| e^{(k+1)p_s-n\varepsilon_s}<1/8.
\end{equation}

 Substituting \eqref{equ.normpart1}, \eqref{equ.Ups1lambda} and \eqref{equ.normpart2and3} into the estimate for $\Upsilon_{s,2}(\mathcal{L}^n_{ib}f)$ given at the beginning of the proof, the lemma follows immediately.
\end{proof}

%%%%%%%%%%%%%%%%%%%%%%%%%%%%%%%%%%%%%%%%%%%%%%%%%%%%%%%%%%%%%%%%%%%
\section{Estimates of the Ruelle operator with good asymptotics}\label{sec.EstimRurlle}

In \cite{FMT07}, Field-Melbourne-T\"or\"ok introduced the concept of good asymptotics to establish that open and dense Axiom A flows are rapidly mixing. They proved that the presence of good asymptotics excludes the existence of approximate eigenfunctions. Dolgopyat \cite{Dol98} established that the absence of approximate eigenfunctions implies rapid mixing. However, applying their framework to our symbolic coding for singular flows faces several key challenges arising from the non-uniform hyperbolicity, including the lack of the uniformity of both the Gibbs property and the norm $\|\cdot\|_{s}$. To address this, we restrict all relevant analysis to a local region associated with only finitely many vertices, which allows us to recover the required uniform estimates in the localized region.

\subsection{Good asymptotics for one-sided TMFs}

First, we introduce the concept of good asymptotics for one-sided TMF.

\begin{dfn}\label{dfn.goodasymptotics}
Let $(\Sigma^+_r,\sigma_r)$ be a one-sided TMF based on a one-sided TMS $(\Sigma^+,\sigma)$.	We say that $(\Sigma^+_r,\sigma_r)$ has {\rm good asymptotics} if there exists a sequence of periodic points $\{\ul{p_N}\}_{N\ge 0}$ for $(\Sigma^+,\sigma)$ satisfying: there exists a set $S\subset \mathscr{V}$ consisting of finitely many vertices such that $\sigma^n(\ul{p_N})\in  \bigcup_{s\in S}[s]$ for all $n\in \bN_0$ and $N\in \bN_0$, and there exists $\tau_N>0$, which is a period of $(\ul{p_N},0)$ (not necessarily minimal) with respect to $\sigma_r$, constants $\Pi>0$, $\gamma\in(0,1)$, $\Theta\in[0,\pi)$, $\Omega>0$ and sequences $\{E_N\}_{N\ge 0}$, $\{\zeta_N\}_{N\ge 0}$ such that
	\begin{itemize}
  \item $\tau_N=N\tau_0+\Pi+E_N\cdot \gamma^N\cos(N\Theta+\zeta_N)+o(\gamma^N)$;
  \item  $|E_N|<\Omega$ for every $N\geq 1$;
  \item either  $\Theta=0$ and $\zeta_N\equiv 0$, or  $\Theta\in(0, \pi)$ and $\zeta_N\in(\Theta_0-\pi/12,\Theta_0+\pi/12)$ for every $N\geq 1$, where $\Theta_0$ is a fixed constant independent of $N$;
  \item $\liminf_{N\to\infty}|E_N|> 0$.
\end{itemize}
\end{dfn}

The definition naturally extends to two-sided TMFs in a straightforward manner.

In the remaining part of this section, we always assume that $(\Sigma^+_r,\sigma_r)$ is a  one-sided TMF based on a locally compact topologically mixing one-sided TMS $(\Sigma^+,\sigma)$, and its roof function $r$ is $\theta$-H\"older continuous satisfying $\inf r>0$. Moreover, we also assume that $(\Sigma^+_r,\sigma_r)$ has good asymptotics, and $(\Sigma^+,\sigma)$ is SPR for a $\theta$-H\"older continuous potential $\phi^+:\Sigma^{+}\to \bR$. Let $\{\ul{p_N}\}_{N\ge 0}$ and $S$ be given by Definition~\ref{dfn.goodasymptotics}.

By the topologically mixing property of $(\Sigma^+,\sigma)$ and the finiteness of $S$, there exists $l_0\in\mathbb{Z}^{+}$ such that $\sigma^n[v_1]\cap[v_2]\neq \emptyset$ for all $n\geq l_0$ and all $v_1, v_2\in S$.
For each pair $(v_1,v_2)\in S\times S$, choose a directed path of length $l_0$ connecting $v_1$ to $v_2$ in the directed graph. Denote by $V\subset \mathscr{V}$ the collection of all vertices appearing in all such directed paths. Note that $V$ is also finite and $S\subset V$.

Denote $\mathcal{O}:=\{\sigma^n(\ul{p_N}):n\in\bN_0,\ N\in\bN_0\}$. Recall that  $\mathcal{O}\subset \bigcup_{s\in S}[s]$. Define $R_0:=\max_{s\in S} R_{s}>1$ and $\lambda_0:=\max_{s\in S}\lambda_{s,0}>0$, where $R_s$ is given by Lemma~\ref{lem.norml} and $\lambda_{s,0}$ is given by Lemma~\ref{lem.LYineq}.

\begin{pro}\label{pro.appfunct} 
 For every $\lambda>0$ and $K>0$, there exist $\varrho>0$ and $b_0\ge 3$ satisfying the following: For any $b\in\bR$ with $|b|\geq b_0$, and any $g\in \bigcap_{s\in S}\mathscr{L}_s$, define $m_b:=\lfloor \varrho\log|b|\rfloor$ and $(\mathcal{L}_{ib}^{jm_b}g)(\ul{x})=:\xi_j^b(\ul{x})\cdot e^{i\omega_j^b(\ul{x})}$  ($j=0,1,2$). If  $\norm{\mathcal{L}_{ib}^{n}g}_s\leq 9R_0|b|^{\lambda_0+1}$ for all $s\in S$ and $n\geq 0$, $|\xi_0^b(\ul{p_0})|=|g(\ul{p_0})|\geq K$ and $|\xi_1^b(\ul{p_0})|=|(\mathcal{L}_{ib}^{m_b}g)(\ul{p_0})|\geq K$, then there exists  $\ul{x_0}=\{x_{0,0},x_{0,1},\dots\}$ such that 
 \begin{itemize}
  \item either $\ul{x_0}\in\mathcal{O}$, or $x_{0,n}\in V$ for all $n\ge 0$, $\sigma^{m_b}(\ul{x_0})\in \mathcal{O}$, $\xi_0^b(\ul{x_0})\geq K/2$, $\xi_1^b(\ul{x_0})\geq K/2$; and
  \item either $|e^{i\omega^b_1(\sigma^{m_b}(\ul{x_0}))}-e^{i\omega^b_0(\ul{x_0})}e^{ibr_{m_b}(\ul{x_0})}|\geq |b|^{-\lambda}$, or $|e^{i\omega^b_2(\sigma^{m_b}(\ul{x_0}))}-e^{i\omega^b_1(\ul{x_0})}e^{ibr_{m_b}(\ul{x_0})}|\geq |b|^{-\lambda}$.
\end{itemize}
\end{pro}
\begin{proof}
Denote by $T_N$ the minimal positive period of $\ul{p_N}$ under $\sigma$. Define $H:=\max_{s_1,s_2\in S}h_{s_1,0}[s_2]$.

We prove this statement by contradiction. Suppose there exist $\lambda>0$ and $K>0$ such that for
\[
\varrho = T_0 \cdot \frac{\log(18R_0 H K^{-1}) + \lambda + \lambda_0 + 1}{\theta} + l_0,
\]
where $l_0$ is given in the construction of $V$, and for any $k\ge 3$, 
there exist $b_k\in \bR$ with $|b_k|\geq k$, and $g_k\in \bigcap_{s\in S}\mathscr{L}_s$ satisfying $\norm{\mathcal{L}_{ib_k}^{n}g_k}_s\leq 9R_0|b_k|^{\lambda_0+1}$ for all $s\in S$ and $n\geq 0$,
$|\xi^k_0(\underline{p_0})| = |g_k(\underline{p_0})| \geq K$ and $|\xi^k_1(\underline{p_0})| = |(\mathcal{L}_{ib_k}^{m_k}g_k)(\underline{p_0})| \geq K$,
(where $m_k:=\lfloor \varrho\log|b_k|\rfloor$ and  
$(\mathcal{L}_{ib_k}^{jm_k}g_k)(\underline{x})=:\xi_j^k(\underline{x})e^{i\omega_j^k(\underline{x})}$ ($j=0,1,2$)),
such that for every $\ul{y}\in\mathcal{O}$ and every $\underline{y}=\{y_0,y_1,\dots\}$ satisfying $y_{n}\in V$ for all $n\ge 0$, $\sigma^{m_k}(\ul{y})\in \mathcal{O}$, $\xi_0^k(\ul{y})\geq K/2$, $\xi_1^k(\ul{y})\geq K/2$, we have
\begin{equation}\label{eq.assumption}
|e^{i\omega_1^k(\sigma^{m_{k}}(\ul{y}))}-e^{i\omega_0^k(\ul{y})}e^{ib_kr_{m_{k}}(\ul{y})}|<|b_k|^{-\lambda} \quad \text{and} \quad
|e^{i\omega_2^k(\sigma^{m_{k}}(\ul{y}))}-e^{i\omega_1^k(\ul{y})}e^{ib_kr_{m_{k}}(\ul{y})}|< |b_k|^{-\lambda}.
\end{equation}

Assume $\ul{p_0}=\{p_{0,0},p_{0,1},\dots,p_{0,T_0-1},p_{0,0},p_{0,1},\dots,p_{0,T_0-1},\dots\}$.
For any $k\ge 3$, note that $m_k=\lfloor \varrho\log|b_k|\rfloor\ge l_0$, thus for every $\ul{x}\in \mathcal{O}$, we may define
\[
q(\ul{x})=(\underbrace{p_{0,0},p_{0,1},\dots,p_{0,T_0-1},\cdots,p_{0,0},p_{0,1},\dots,p_{0,T_0-1},p_{0,0},\dots,p_{0,j}}_{\text{first } m_k-l_0+1\text{ positions}},d_1,\dots,d_{l_0-1}, \ul{x})\in\Sigma^+,
\] 
where $d_1,\dots, d_{l_0-1}\in V$ is the selected directed path from $p_{0,j}$ to $x_0$ in the construction of $V$. Then $\sigma^{m_k}(q(\ul{x}))=\ul{x}\in \mathcal{O}$ and all vertices appearing in $q(\ul{x})$ are contained in $V$. Moreover, we have the following estimates.

\begin{lem}\label{lem.r0r1xi}
	We have $|\xi_0^k(q(\ul{x}))|\geq K/2$ and $|\xi_1^k(q(\ul{x}))|\geq K/2$.
\end{lem} 
\begin{proof}
Denote $s_0:=p_{0,0}\in S$. Since $q(\ul{x}),\, \ul{p_0}\in[s_{0}]$ and $\Upsilon_{s_0,2}(g_k)\le \norm{g_k}_{s_0}\leq 9R_0|b_k|^{\lambda_0+1}$, we have
\[
|g_k(q(\ul{x}))-g_k(\ul{p_0})|\le \Upsilon_{s_0,2}(g_k)\cdot h_{s_0,0}[s_0]e^{-t_{s_0}(q(\ul{x}),\ul{p_0})\cdot \theta}  \leq 9R_0 |b_k|^{\lambda_0+1}H e^{-t_{s_0}(q(\ul{x}),\ul{p_0})\cdot \theta}.
\]
It follows from the construction of $q(\ul{x})$ that
\[
t_{s_0}(q(\ul{x}),\ul{p_0})\geq\frac{m_k- l_0+1}{T_0}> \frac{(\varrho- l_0)\log|b_k|}{T_0} \ge\frac{\log(18R_0HK^{-1})+(\lambda+\lambda_0+1)\log|b_k|}{\theta}.
\]
Substituting into the above inequality, we obtain that
\begin{equation}\label{eq.r0est}
|g_k(q(\ul{x}))-g_k(\ul{p_0})|<\frac{K}{2}\cdot |b_k|^{-\lambda}<\frac{K}{2}.
\end{equation}

Combining this estimate with the assumption $|g_k(\ul{p_0})|\ge K$, we conclude that $|\xi_0^k(q(\ul{x}))|=|g_k(q(\ul{x}))|\geq K/2$. The proof for $|\xi_1^k(q(\ul{x}))|\geq K/2$ follows by an identical argument, replacing $g_k$ with $\mathcal{L}_{ib_k}^{m_k}g_k$.
\end{proof}

With the aid of $q(\ul{x})$, we can derive the following estimate.

\begin{lem}\label{lem.approxfun}
	There exists $\eta_k\in\mathbb{R}$ such that for all $\ul{x}\in\mathcal{O}$, we have
\[
|e^{ib_kr_{m_k}(\ul{x})}e^{-i\omega_1^k(\sigma^{m_k}(\ul{x}))}-e^{i\eta_k}e^{-i\omega^k_1(\ul{x})}|<5|b_k|^{-\lambda}.
\]
\end{lem}
\begin{proof}
By Lemma~\ref{lem.r0r1xi}, $\ul{y}=q(\ul{x})$ satisfies the condition for the validity of \eqref{eq.assumption}, so we have \begin{equation}\label{equ.assumofqx}
|e^{i\omega_1^k(\ul{x})}-e^{i\omega_0^k(q(\ul{x}))}e^{ib_kr_{m_{k}}(q(\ul{x}))}|<|b_k|^{-\lambda} \quad \text{and} \quad
|e^{i\omega_2^k(\ul{x})}-e^{i\omega_1^k(q(\ul{x}))}e^{ib_kr_{m_{k}}(q(\ul{x}))}|< |b_k|^{-\lambda}.
\end{equation}

Using the elementary inequality for complex numbers on the unit circle, we conclude that
\[
|e^{i\omega}-e^{i\wt{\omega}}|^2=4\sin^2\frac{\omega-\wt{\omega}}{2}=\frac{|\xi e^{i\omega}-\wt{\xi}e^{i\wt{\omega}}|^2-(\xi-\wt{\xi})^2}{\xi\cdot\wt{\xi}} \leq \frac{|\xi e^{i\omega}-\wt{\xi}e^{i\wt{\omega}}|^2}{\xi\cdot\wt{\xi}} \]
for any $\omega,\wt{\omega}\in\bR$ and $\xi, \wt{\xi}\in\mathbb{R}^+$.
Combining with Lemma~\ref{lem.r0r1xi} and \eqref{eq.r0est}, we have 
\begin{equation}\label{equ.omega0qp0}
|e^{i\omega_0^k(q(\ul{x}))}-e^{i\omega_0^k(\ul{p_0})}|\leq \frac{|g_k(q(\ul{x}))-g_k(\ul{p_0})|}{(\xi_0^k(q(\ul{x}))\xi_0^k(\ul{p_0}))^{\frac{1}{2}}}< \frac{\sqrt{2}}{K}\cdot \frac{K}{2}\cdot |b_k|^{-\lambda}<|b_k|^{-\lambda}.
\end{equation}
Similarly, we also have
\begin{equation}\label{equ.omega1qp0}
|e^{i\omega_1^k(q(\ul{x}))}-e^{i\omega_1^k(\ul{p_0})}|<|b_k|^{-\lambda}.
\end{equation}

Denote $\eta_k:=\omega_1^k(\ul{p_0})-\omega_0^k(\ul{p_0})$ and $z:=e^{ib_kr_{m_k}(q(\ul{x}))}$. Then for any $\ul{x}\in \mathcal{O}$, it follows from \eqref{equ.assumofqx}, \eqref{equ.omega0qp0} and \eqref{equ.omega1qp0} that
\[\begin{aligned}
|e^{i\omega_2^k(\ul{x})}-e^{i\omega_1^k(\ul{x})}e^{i\eta_k}|
\leq &\,|e^{i\omega_2^k(\ul{x})}-ze^{i\omega_1^k(q(\ul{x}))}|+|z(e^{i\omega_1^k(q(\ul{x}))}-e^{i\omega_1^k(\ul{p_0})})|+|z(e^{i\omega_1^k(\ul{p_0})}-e^{i\omega_0^k(\ul{p_0})}e^{i\eta_k})|\\
&+|ze^{i\eta_k}(e^{i\omega_0^k(\ul{p_0})}-e^{i\omega_0^k(q(\ul{x}))})|+|e^{i\eta_k}(ze^{i\omega_0^k(q(\ul{x}))}-e^{i\omega_1^k(\ul{x})})|\\
\leq &\, 4|b_k|^{-\lambda}.
\end{aligned}\]

Substituting $\sigma^{m_k}(\ul{x})$ for $\ul{x}$ in the above inequality, we obtain $|e^{i\omega_2^k(\sigma^{m_k}(\ul{x}))}-e^{i\omega_1^k(\sigma^{m_k}(\ul{x}))}e^{i\eta_k}|\leq 4|b_k|^{-\lambda}$. Taking $\ul{y}=\ul{x}\in\mathcal{O}$ in \eqref{eq.assumption}, we have $|e^{i\omega_2^k(\sigma^{m_k}(\ul{x}))}-e^{ib_kr_{m_k}(\ul{x})}e^{i\omega_1^k(\ul{x})}|<|b_k|^{-\lambda}$. Thus
\[
\begin{aligned}
	&\,|e^{ib_kr_{m_k}(\ul{x})}e^{-i\omega_1^k(\sigma^{m_k}(\ul{x}))}-e^{i\eta_k}e^{-i\omega^k_1(\ul{x})}|\\
	=&\,|e^{-i\omega_1^k(\sigma^{m_k}(\ul{x}))}e^{-i\omega^k_1(\ul{x})}|\cdot |e^{ib_kr_{m_k}(\ul{x})}e^{i\omega^k_1(\ul{x})}-e^{i\eta_k}e^{i\omega_1^k(\sigma^{m_k}(\ul{x}))}|\\
	\le &\,|e^{ib_kr_{m_k}(\ul{x})}e^{i\omega^k_1(\ul{x})}-e^{i\omega_2^k(\sigma^{m_k}(\ul{x}))}|+|e^{i\omega_2^k(\sigma^{m_k}(\ul{x}))}-e^{i\eta_k}e^{i\omega_1^k(\sigma^{m_k}(\ul{x}))}|\\
	\le &\,5|b_k|^{-\lambda}.
\end{aligned}
\]
\end{proof}

To complete the contradiction argument following the strategy developed in \cite{FMT07}, we need to replace the real constant
$\eta_k$ from the previous lemma with a rational number such that the conclusion still holds for some $\lambda>0$. To achieve this step, we establish the following lemma via Dirichlet's rational approximation theorem.

\begin{lem}\label{lem.dirichlet}
	There exist $u_k\in\mathbb{Z}^{+}$ with $u_k\leq |b_k|^{\lambda/2}$, and $t_k\in\mathbb{Z}$ such that for all $\ul{x}\in \mathcal{O}$,
\[
|e^{ib_kr_{m_k}(\ul{x})}e^{-i\omega_1^k(\sigma^{m_k}(\ul{x}))}-e^{i\cdot 2\pi\frac{t_k}{u_k}}e^{-i\omega_1^k(\ul{x})}|\leq \frac{20}{u_k}\cdot |b_k|^{-\frac{\lambda}{2}}.
\]
\end{lem}
\begin{proof}
By  Dirichlet Theorem, there exist $u_k, t_k\in\mathbb{Z}$ with $1\leq u_k\leq \lfloor |b_k|^{\lambda/2}\rfloor$ such that 
\[
\left|\frac{\eta_k}{2\pi}-\frac{t_k}{u_k}\right|\leq \frac{1}{u_k\cdot \lfloor |b_k|^{\lambda/2}\rfloor},
\]
 where $\eta_k$ is given by Lemma~\ref{lem.approxfun}. It implies that
\[
|e^{i\eta_k}-e^{i\cdot 2\pi\frac{t_k}{u_k}}|=2\left|\sin \frac{\eta_k-2\pi\frac{t_k}{u_k}}{2}\right|\le \left|\eta_k-2\pi\frac{t_k}{u_k}\right| \leq \frac{2\pi}{u_k\cdot \lfloor |b_k|^{\lambda/2}\rfloor}.
\]

Combining this approximation bound with Lemma~\ref{lem.approxfun}, we have
\[
\begin{aligned}
	&\,|e^{ib_kr_{m_k}(\ul{x})}e^{-i\omega_1^k(\sigma^{m_k}(\ul{x}))}-e^{i\cdot 2\pi\frac{t_k}{u_k}}e^{-i\omega_1^k(\ul{x})}|\\
	\le &\,|e^{ib_kr_{m_k}(\ul{x})}e^{-i\omega_1^k(\sigma^{m_k}(\ul{x}))}-e^{i\eta_k}e^{-i\omega^k_1(\ul{x})}|+|e^{-i\omega^k_1(\ul{x})}|\cdot |e^{i\eta_k}-e^{i\cdot 2\pi\frac{t_k}{u_k}}|\\
	\le &\,5|b_k|^{-\lambda}+2\pi u_k^{-1} \lfloor |b_k|^{\lambda/2}\rfloor^{-1}
	\leq \frac{20}{u_k}\cdot |b_k|^{-\frac{\lambda}{2}}.
\end{aligned}
\]
\end{proof}

Proposition~\ref{pro.appfunct} is finally established via the following lemma.

\begin{lem}\label{lem.FMT}
	The conclusion of Lemma~\ref{lem.dirichlet} leads to a contradiction with the fact that $\{\ul{p_N}\}_{N\ge 0}$ satisfy Definition~\ref{dfn.goodasymptotics}.   
\end{lem}
\begin{proof}
	The proof is similar to that of \cite[Theorem 1.6(a)]{FMT07}. The uniform rational phase relation derived in Lemma \ref{lem.dirichlet} would contradict the non-degenerate oscillatory condition $\liminf_{N\to\infty}|E_N|> 0$. We place the detail in Appendix~\ref{appsec.decay}.
\end{proof}
\end{proof}

Using a simplified variant of the contradiction argument from Lemma~\ref{lem.FMT}, we may prove that all equilibrium states  of  H\"older continuous potentials for a topological transitive two-sided TMF having good asymptotics are mixing.
We rely on the following theorem which is essentially proved in \cite{LS}.

\begin{thm}{\rm \cite[Theorem~7.2]{LS}}\label{thm.LS}
	Let $(\Sigma_r,\sigma_r)$ be a topological transitive two-sided TMF. If there exists an equilibrium state of a H\"older continuous potential that is not mixing, then for any finite subset $S_0\subset \mathscr{V}$, there exists a full-measure invariant component $\Sigma'_r\subset\Sigma_r$ such that $(\Sigma'_r,\sigma_r)$ is topologically conjugate to a TMF with constant roof function, and $\Sigma'_r$ contains all periodic orbits intersecting $\bigcup_{s_0\in S_0}[s_0]\times \{0\}$.
\end{thm}
\begin{proof}
	We only need to explain why we can choose such $\Sigma'_r$ containing all periodic orbits which intersect $\bigcup_{s_0\in S_0}[s_0]\times \{0\}$. The construction of $\Sigma_r'$ in the original proof of \cite[Theorem~7.2]{LS} is as follows: take a point $\ul{y}\in \Sigma$ with a dense orbit, we may choose a cylinder ${}_{-m}[y_{-m},\dots,y_n]$ satisfying conditions (i), (ii) and (iii) listed in the proof, then $\Sigma_r'$ is selected as the set of points whose orbit intersects  ${}_{-m}[y_{-m},\dots,y_n]\times \{0\}$ infinitely often in the past and in the future. It is not hard to check that $m,n>0$ can be chosen arbitrarily large. Since the orbit of $\ul{y}$ is dense, every vertex $s\in \mathscr{V}$ appears in the symbol sequence of $\ul{y}$. Thus we may select sufficiently large $m,n>0$ satisfying conditions (i), (ii) and (iii) such that $S_0\subset \{y_j\}_{j=-m}^n$. It follows that all periodic orbits intersecting $\bigcup_{s_0\in S_0}[s_0]$ are contained in $\Sigma_r'$. The remaining proof is the same as that of \cite[Theorem~7.2]{LS}.
\end{proof}

\begin{lem}\label{lem.mixing}
Let $(\Sigma_r,\sigma_r)$ be a topological transitive two-sided TMF having good asymptotics, and $\vartheta_1$ be a H\"older continuous potential on $\Sigma_r$. Then the unique equilibrium state $\nu $ of  $\vartheta_1$ for $(\Sigma_r,\sigma_r)$ is mixing.
\end{lem}
\begin{proof}
	Suppose by contradiction that the equilibrium state $\nu$ is not mixing. Let $\{\ul{p_N}\}$ and $S\subset \mathscr V$ be given by the definition of the good asymptotic property.
	 By Theorem~\ref{thm.LS}, there exists $\Sigma'_r\subset\Sigma_r$ of full measure such that $(\Sigma'_r,\sigma_r)$ is topologically conjugate to a TMF with constant roof function, and $\Sigma'_r$ contains all periodic orbits intersecting $\bigcup_{s\in S}[s]\times \{0\}$. In particular, the orbit of $(\ul{p_N},0)$ is contained in $\Sigma_r'$ for all $N\ge 0$. Since the roof function of the conjugated TMF is constant, there  exists a constant $T>0$ such that $\tau_N/T\in \bZ^+$ for all $N\ge 0$. Taking the limit as $N\to +\infty$ in the equation
		\[
	\tau_N/T=N(\tau_0/T)+\Pi/T+(E_N/T)\cdot \gamma^N\cos(N\Theta+\zeta_N)+o(\gamma^N),
	\]
we obtain that $\Pi/T\in \bZ$. Thus $\dist((E_N/T)\cdot \gamma^N\cos(N\Theta+\zeta_N),\bZ)=o(\gamma^N)$.
	Denote $M(k):=\lfloor \frac{\log(2k\Omega/T)}{-\log\gamma}\rfloor +1$. 
	For any given $j\ge 0$, taking $N=M(k)+j$ and multiplying both sides by $k$, we have
	\[
	\dist((kE_{M(k)+j}/T)\cdot \gamma^{M(k)+j}\cos((M(k)+j)\Theta+\zeta_{M(k)+j}),\bZ)=o(k\gamma^{M(k)+j}).
	\]
	
	Since $(kE_{M(k)}/T)\cdot \gamma^{M(k)+j}\le 1/2$, $(k/T)\cdot \gamma^{M(k)+j}\ge \gamma^{j+1}/(2\Omega)$ and $\liminf_{N\to+\infty}|E_N|> 0$, we have $\lim_{k\to +\infty}\cos((M(k)+j)\Theta+\zeta_{M(k)+j})=0$. Thus $\lim_{k\to\infty}\dist(j\Theta+\zeta_{M(k)+j}-\zeta_{M(k)},\pi\mathbb{Z})=0$ for all $j\ge 0$. This contradicts the constraints of $\Theta$ and $\zeta_N$ (see the proof of Lemma~\ref{lem.FMT} for details).  
	
\end{proof}

%%%%%%%%%%%%%%%%%%%%%%%%%%%%%%%%%%%%%%%%%%%%%%%%%%%%%%%%%%%%%%%%%%%%%%%%

\subsection{Estimates under the new norm}

In this subsection, we give more estimates of the complex Ruelle operator. To this end, we first introduce a new Banach space.

Recall that $S$ is the finite set in the definition of good asymptotics.
Define 
\[
\mathscr{L}:=\bigcap_{s\in S} \mathscr{L}_{s},\ \text{ and }  \|f\|_{\mathscr{L}}:=\max_{s\in S} \|f\|_s \text{ for every } f\in \mathscr{L}.
\]
Then $(\mathscr{L},\|\cdot\|_{\mathscr{L}})$ is a Banach space. We also denote 
\begin{equation}\label{equ.defUps12}
  \Upsilon_{1}(f):=\max_{s\in S}\Upsilon_{s,1}(f) \quad \text{and} \quad \Upsilon_{2}(f):=\max_{s\in S}\Upsilon_{s,2}(f).
\end{equation}

Recall that $C_{s,2}>1$, $C_{s,3}>1$, and $\delta_s\in (0,1)$ are given by \eqref{equ.boundPs} and \eqref{equ.SpectralGap}, $C_{s,7}>0$ is given by Lemma~\ref{lem.estimofaib}, and $\varrho_{s,0}>0$ is given by Lemma~\ref{lem.LYineq}. Denote
\[
C_2:=\max_{s\in S}C_{s,2}>1,\ C_3:=\max_{s\in S}C_{s,3}>1,\ C_7:=\max_{s\in S}C_{s,7},\ \delta:=\max_{s\in S} \delta_s\in (0,1) \ \text{and}\  \varrho_0:=\max_{s\in S}\varrho_{s,0}.
\]

Recall that $R_0:=\max_{s\in S}R_s>1$ and $\lambda_0:=\max_{s\in S}\lambda_{s,0}>1$, where $R_s>1$ is given by Lemma~\ref{lem.norml} and $\lambda_{s,0}>1$ is given by Lemma~\ref{lem.LYineq}. Let  $\mu$ be the unique equilibrium state of  $\phi^+$ for $(\Sigma^+,\sigma)$.
We summarize the following fundamental properties of this new Banach space.

\begin{lem}\label{lem.propUp12norm}
	For any $f\in \mathscr{L}$, we have
	\begin{enumerate} 
	\item\label{ite.Up1Up2norm} $\Upsilon_1(f), \Upsilon_2(f)\le \norm{f}_{\mathscr{L}}\le \Upsilon_1(f)+\Upsilon_2(f)$;
  \item\label{ite.absless} $\Upsilon_1(|f|)=\Upsilon_1(f)$, $\Upsilon_2(|f|)\le \Upsilon_2(f)$, and $\norm{|f|}_{\mathscr{L}}\le \norm{f}_{\mathscr{L}}$;
  \item\label{ite.aib-ib} $\norm{\mathcal{L}_{a+ib}-\mathcal{L}_{ib}}_{\mathscr{L}}\le C_{7}|a||b|$ for all $|b|\ge 1$ and $|a|\le 1/C_5'$, where $C_5'>1$ is a bound of the roof function $r$;
  \item\label{ite.normR0}  $\norm{\mathcal{L}_{ib}^nf}_{\mathscr{L}}\leq R_0|b|\norm{f}_{\mathscr{L}}$, for all $|b|\geq 1$ and $n\in\bZ^+$;
  \item\label{ite.LasoYor}  $\Upsilon_{2}(\mathcal{L}^n_{ib}f)\leq |b|^{\lambda_{0}}\Upsilon_{1}(f)+\frac{1}{2}\norm{f}_{\mathscr{L}}$ for all $|b|\ge 3$ and $n\geq \varrho_{0}\log|b|$;
  \item\label{ite.nonemptypres} $C_{\theta}(\Sigma^+)\subset\mathscr{L}\subset L^1(\mu)$; moreover, if $g\in C_{\theta}(\Sigma^{+})$ and $f\in\mathscr{L}$, then $\norm{f\cdot g}_{\mathscr{L}}\leq \norm{g}_{\theta}\cdot \norm{f}_{\mathscr{L}}$;
  \item\label{ite.decompose}  $\mathcal{L}_{0}= P + N:\mathscr{L}\to \mathscr{L}$, where $P=P_s|_\mathscr{L}$ and $N=N_s|_\mathscr{L}$ are independent of the choice of $s\in S$, satisfying $P N= NP=0$, $P^2 = P$, ${\rm rank}(P) = 1$, and the spectral radius of $N$ is less than $1$;
  \item\label{ite.Pfisint} $P f = \int f \, d\mu \cdot \mathbbm{1}$ and $\Upsilon_1(Pf)=|\int fd\mu|\cdot \norm{\mathbbm{1}}_{\mathscr{L}}=\norm{Pf}_\mathscr{L}\le C_{2}\norm{f}_\mathscr{L}$;
  \item\label{ite.Nlessthan1} $\mathcal{L}^n_{0}= P+ N^n$ and $ \norm{\mathcal{L}^n_{0}f- \int f \, d\mu \cdot \mathbbm{1}}_\mathscr{L}=\norm{\mathcal{L}^n_{0}f- Pf}_\mathscr{L}=\norm{N^nf}_\mathscr{L}\le C_{3}\cdot \delta^n\norm{f}_{\mathscr{L}}$.
\end{enumerate}
\end{lem}
\begin{proof}
	Items~\ref{ite.Up1Up2norm} and \ref{ite.absless} follow immediately from the definitions. Items~\ref{ite.aib-ib}, \ref{ite.normR0}, and \ref{ite.LasoYor} are direct consequences of  Lemmas~\ref{lem.estimofaib}, \ref{lem.norml}, and \ref{lem.LYineq}, respectively. Items~\ref{ite.nonemptypres}--\ref{ite.Nlessthan1} can be derived directly from Theorem~\ref{thm.CS2009}, \eqref{equ.boundPs} and \eqref{equ.SpectralGap}, since all constants are defined as the maximum over the finite set $S$. 
\end{proof}

For each $|b|\ge 1$, we also consider the following equivalent norm on $\mathscr{L}$. For any $f\in\mathscr{L}$, define
\[
\norm{f}_{(b)}=\max\Big\{\Upsilon_1(P|f|),\Upsilon_1(f),\frac{\Upsilon_2(f)}{8|b|^{\lambda_0}}\Big\}.
\]
It is straightforward to verify that $\norm{\cdot}_{(b)}$ defines a well-defined equivalent norm on $\mathscr{L}$, and satisfies 
\begin{equation}\label{eq.relnorms}
\norm{f}_{(b)}\leq C_2 \norm{f}_{\mathscr{L}}  \quad \text{and} \quad \norm{f}_{\mathscr{L}}\leq (8|b|^{\lambda_0}+1)\cdot \norm{f}_{(b)}\le 9|b|^{\lambda_0} \norm{f}_{(b)}.
\end{equation}

Let $C_{s,6}>1$ be given by Lemma~\ref{lem.onestepestim}. We have the following estimate for the new norm of $\mathcal{L}_{ib}$.

\begin{lem}\label{lem.bnormofonestep}
	For any $|b|\ge 1$, we have $\norm{\mathcal{L}_{ib}}_{(b)}\le 2C_{6}$, where $C_6:=\max_{s\in S}C_{s,6}>1$. 
\end{lem}
\begin{proof}
	By Lemma~\ref{lem.onestepestim} and the definitions of $\Upsilon_1$, $\Upsilon_2$ and $C_6$, for any $f\in\mathscr{L}$ and $|b|\ge 1$ we have  
	\[
	\Upsilon_{1}(\mathcal{L}_{ib}f)\le C_{6} \Upsilon_{1}(f)\quad \text{and}\quad \Upsilon_{2}(\mathcal{L}_{ib}f)\le C_{6}|b| \Upsilon_{1}(f)+C_{6} \Upsilon_{2}(f).
	\]
It follows that (note that $\lambda_0>1$ and $|b|\ge 1$)
\[
\Upsilon_{1}(\mathcal{L}_{ib}f)\le C_{6} \norm{f}_{(b)}\quad \text{and}\quad \frac{\Upsilon_{2}(\mathcal{L}_{ib}f)}{8|b|^{\lambda_0}}\le C_{6} \Upsilon_{1}(f)+C_{6} \frac{\Upsilon_{2}(f)}{8|b|^{\lambda_0}}\le 2C_{6} \norm{f}_{(b)}.
\]

Moreover, it follows from Lemma~\ref{lem.propUp12norm}~\ref{ite.decompose} that $P\mathcal{L}_{0}=P(P+N)=P$. Thus
\[
\Upsilon_1(P|\mathcal{L}_{ib}f|)=\int |\mathcal{L}_{ib}f|d\mu\cdot \norm{\mathbbm{1}}_{\mathscr{L}}\le \int \mathcal{L}_{0}|f|d\mu \cdot \norm{\mathbbm{1}}_{\mathscr{L}}=\Upsilon_{1}(P(\mathcal{L}_{0}|f|))=\Upsilon_{1}(P|f|)\le \norm{f}_{(b)}.
\]
To sum up, the proof is completed.
\end{proof}

In the following, we derive uniform contraction estimates for the three components that define the norm $\norm{\cdot}_{(b)}$ respectively.

\begin{lem}\label{lem.newUpsilon2}
There exist $\lambda_1>0$ and $\varrho_1>0$ such that for all $|b|\geq 3$, $n\geq \varrho_1\log|b|$, and $f\in\mathscr{L}$,  we have 
	 \[
	 \frac{\Upsilon_2(\mathcal{L}_{ib}^nf)}{8|b|^{\lambda_0}}\leq (1-|b|^{-\lambda_1})\norm{f}_{(b)}.
	 \]
\end{lem}
\begin{proof}
	Define $\lambda_1:=\frac{\log4}{\log 3}$, which is chosen to satisfy $1-|b|^{-\lambda_1}\ge 3/4$ for all $|b|\ge 3$. Set $\varrho_1:=\varrho_0$, where $\varrho_0$ 
	satisfies Lemma~\ref{lem.propUp12norm}~\ref{ite.LasoYor}. By Lemma~\ref{lem.propUp12norm}~\ref{ite.LasoYor}, \eqref{eq.relnorms} and the definition of $\norm{\cdot}_{(b)}$, for any $f\in\mathscr{L}$, $|b|\ge 3$ and $n\geq \varrho_1\log|b|$, we have
	\[
	\Upsilon_{2}(\mathcal{L}^n_{ib}f)\leq |b|^{\lambda_{0}}\Upsilon_{1}(f)+\frac{1}{2}\norm{f}_{\mathscr{L}}\le |b|^{\lambda_{0}}\norm{f}_{(b)}+\frac{1}{2}(8|b|^{\lambda_0}+1) \norm{f}_{(b)}=(5|b|^{\lambda_0}+\frac{1}{2}) \norm{f}_{(b)}.
	\]
	Therefore,
	\[
	 \frac{\Upsilon_2(\mathcal{L}_{ib}^nf)}{8|b|^{\lambda_0}}\le \frac{3}{4} \norm{f}_{(b)}\leq (1-|b|^{-\lambda_1})\norm{f}_{(b)}.
	\]
\end{proof}

\begin{lem}\label{lem.newUpsilonpi}
	There exist $\lambda_2>0$, $\varrho_2>0$ and $\widetilde{b}\ge 3$ such that for all $|b|\geq \wt{b}$, $n\geq \varrho_2\log|b|$ and $f\in\mathscr{L}$, we have
	\[
	\Upsilon_1(P|\mathcal{L}_{ib}^{n}f|)\leq (1-|b|^{-\lambda_2})\norm{f}_{(b)}.
	\]
\end{lem}
\begin{proof}
For any $f\in\mathscr{L}$ and $n\in \bZ^+$, by Lemma~\ref{lem.propUp12norm}~\ref{ite.decompose} and \ref{ite.Pfisint}, we have  $\Upsilon_1(P(\mathcal{L}_0^{n}f))=\Upsilon_1(Pf)=|\int fd\mu|\cdot \norm{\mathbbm{1}}_{\mathscr{L}}$.
Since $|\mathcal{L}_{ib}^nf|(\ul{x})\leq \mathcal{L}_0^n|f|(\ul{x})$ for all $\ul{x}\in \Sigma^+$ (see \eqref{equ.LibfL0absf}), we obtain
	\[
	\Upsilon_1(P|\mathcal{L}_{ib}^nf|)\leq \Upsilon_1(P(\mathcal{L}_0^n|f|))=\int |f|d\mu\cdot \norm{\mathbbm{1}}_{\mathscr{L}}.
	\]

If $\int |f|d\mu\cdot \norm{\mathbbm{1}}_{\mathscr{L}}\leq \frac{1}{2}\norm{f}_{(b)}$, then the conclusion of this lemma is clear. Hence in the following we only need to consider $f\in\mathscr{L}$ satisfying
\begin{equation}\label{eq.nasp}
\int |f|d{\mu}\cdot \norm{\mathbbm{1}}_{\mathscr{L}}> \frac{1}{2}\norm{f}_{(b)}.
\end{equation}

\noindent {\bf Claim 1.}  {\it  There exists $\varrho_1'$ such that for any $|b|\ge 3$, $\ell\geq \varrho_1'\log|b|$, $\ul{x}\in \mathcal{O}$ and $f\in\mathscr{L}$ satisfying \eqref{eq.nasp}, we have
\[
(\mathcal{L}_0^\ell|f|)(\ul{x})\geq \frac{\norm{f}_{(b)}}{4\norm{\mathbbm{1}}_{\mathscr{L}}}.
\]}
\begin{proof}
Recall that $H:=\max_{s_1,s_2\in S}h_{s_1,0}[s_2]$. Take a big $\lambda'>0$ such that  $|b|^{\lambda'}>4\norm{\mathbbm{1}}_{\mathscr{L}}C_3(8|b|^{\lambda_0}+1)H$ for all $|b|\ge 3$, and take $\varrho_1':=-\frac{\lambda'}{\log\delta}$. By Lemma~\ref{lem.propUp12norm}~\ref{ite.Nlessthan1},
\[
\Upsilon_1(\mathcal{L}_0^\ell|f|-P|f|)\leq  C_3\delta^\ell\norm{f}_{\mathscr{L}}\leq C_3\delta^\ell(8|b|^{\lambda_0}+1)\norm{f}_{(b)}.
\]
For any $\ul{x}\in \mathcal{O}$, there is $s\in S$ such that $\ul{x}\in [s]$. Then we have
\[
\Big|(\mathcal{L}_0^\ell|f|)(\ul{x})-\int |f|d{\mu}\Big|\le \Upsilon_{s,1}(\mathcal{L}_0^\ell|f|-P|f|)\cdot h_{s,0}[s]
\le C_3\delta^\ell(8|b|^{\lambda_0}+1)\norm{f}_{(b)}\cdot H.
\]

 It follows that $\big|(|\mathcal{L}_0^\ell|f|)(\ul{x})-\int |f|d{\mu}\big|\le \frac{\norm{f}_{(b)}}{4\norm{\mathbbm{1}}_{\mathscr{L}}}$ for all $\ell\geq \varrho_1'\log|b|$. Combining with \eqref{eq.nasp}, we obtain that $(\mathcal{L}_0^\ell|f|)(\ul{x})\ge \frac{\norm{f}_{(b)}}{4\norm{\mathbbm{1}}_{\mathscr{L}}}$. 
\end{proof}

\noindent {\bf Claim 2.} {\it  There exist $\lambda_2'>0$, $\varrho_2>\varrho_1'$ and $\wt{b}\ge 3$ such that for all $|b|\geq \wt{b}$ and $f\in\mathscr{L}$ satisfying \eqref{eq.nasp}, there exists a point $\ul{x_1}\in \mathcal{O}$ and an integer $n_0\in(\varrho_1'\log|b|,\varrho_2\log|b|)$ such that 
\[
|(\mathcal{L}_{ib}^{n_0}f)(\ul{x_1})|\leq (\mathcal{L}_0^{n_0}|f|)(\ul{x_1})-|b|^{-\lambda_2'}\norm{f}_{(b)}.
\]
}
\begin{proof}
By normalizing $f$ to $f/\norm{f}_{(b)}$ if necessary, we may assume without loss of generality that $\norm{f}_{(b)}=1$. Fix some $\lambda>0$ and take $K:=\frac{1}{8\norm{\mathbbm{1}}_{\mathscr{L}}}$. There exist $\varrho>0$ and $b_0\ge 3$ such that the conclusion of Proposition~\ref{pro.appfunct} holds. Denote $\Phi:=\max \{|\bar{\phi}(\ul{x})|:\ul{x}\in [s],\ s\in S\}$. Take a large enough $\lambda_2'>0$ such that $|b|^{\lambda_2'-(2\lambda+\varrho \Phi)}>32\norm{\mathbbm{1}}_{\mathscr{L}}$ for all $|b|\geq b_0$. Define $\varrho_2:=\varrho_1'+2\varrho+1$ and $\wt b:=b_0$. 
There are two cases:

\vspace{5pt}
{\bf Case 1:} There exists an integer $n_0\in(\varrho_1'\log|b|$, $\varrho_2\log|b|)$ and $\ul{x_1}\in \mathcal{O}$ such that $|(\mathcal{L}_{ib}^{n_0}f)(\ul{x_1})|<\frac{1}{8\norm{\mathbbm{1}}_{\mathscr{L}}}$.

By Claim 1, we have  $(\mathcal{L}_{0}^{n_0}|f|)(\ul{x_1})\geq \frac{1}{4\norm{\mathbbm{1}}_{\mathscr{L}}}$.
Note that $|b|^{-\lambda_2'}\norm{f}_{(b)}=|b|^{-\lambda_2'}<\frac{1}{32\norm{\mathbbm{1}}_{\mathscr{L}}}$. It implies the conclusion of this claim in this case.

\vspace{5pt}
{\bf Case 2:} For every integer $\ell\in(\varrho_1'\log|b|,\varrho_2\log|b|)$ and any $\ul{x}\in \mathcal{O}$, we have $|(\mathcal{L}_{ib}^{\ell}f)(\ul{x})|\geq \frac{1}{8\norm{\mathbbm{1}}_{\mathscr{L}}}$.

In this case, take $\ell_0:=\lfloor \varrho_1'\log|b|\rfloor +1$ and $g:=\mathcal{L}_{ib}^{\ell_0}f$. Denote $m_b:=\lfloor\varrho\log|b|\rfloor$ and $\mathcal{L}_{ib}^{jm_b}g(\ul{x})=:\xi_j^b(\ul{x})\cdot e^{i\omega_j^b(\ul{x})}$ for $j=0,1,2$. It follows from Lemma~\ref{lem.norml} and \eqref{eq.relnorms} that  $\norm{\mathcal{L}_{ib}^ng}_{\mathscr{L}}=\norm{\mathcal{L}_{ib}^{\ell_0+n}f}_{\mathscr{L}}\leq R_0|b|\norm{f}_{\mathscr{L}}\leq 9R_0|b|^{\lambda_0+1}$ for all $n\geq 0$. Furthermore, by the hypothesis of this case, we have $|\xi_0^b(\ul{p_0})|\geq K$ and $|\xi_1^b(\ul{p_0})|\geq K$. 

By Proposition~\ref{pro.appfunct}, there exists $\ul{x_0}$ satisfying 
 either $\ul{x_0}\in\mathcal{O}$, or $x_{0,n}\in V$ for all $n\ge 0$, $\sigma^{m_b}(\ul{x_0})\in \mathcal{O}$, $\xi_0^b(\ul{x_0})\geq K/2$, $\xi_1^b(\ul{x_0})\geq K/2$,
such that either
\begin{equation}\label{eq.dic1}
|e^{i\omega^b_1(\sigma^{m_b}(\ul{x_0}))}-e^{i\omega^b_0(\ul{x_0})}e^{ibr_{m_b}(\ul{x_0})}|\geq |b|^{-\lambda},
\quad \text{or}\quad  
|e^{i\omega^b_2(\sigma^{m_b}(\ul{x_0}))}-e^{i\omega^b_1(\ul{x_0})}e^{ibr_{m_b}(\ul{x_0})}|\geq |b|^{-\lambda}.
\end{equation}

Suppose that the first inequality of \eqref{eq.dic1} holds. The proof for the other case is identical. Note that in both cases of the location of $\ul{x_0}$, we have $|g(\ul{x_0})|=\xi_0^b(\ul{x_0})\geq K/2=\frac{1}{16\norm{\mathbbm{1}}_{\mathscr{L}}}$. Denote $\ul{x_1}:=\sigma^{m_b}(\ul{x_0})\in\mathcal{O}$. 
The real number
\[
\begin{aligned}
	(\mathcal{L}_0^{m_b}|g|)(\ul{x_1})-\xi_1^b(\ul{x_1})&=\sum_{\sigma^{m_b}(\ul{y})=\ul{x_1}}\big(e^{\bar{\phi}_{m_b}(\ul{y})}|g(\ul{y})|-e^{\bar{\phi}_{m_b}(\ul{y})+ibr_{m_b}(\ul{y})}g(\ul{y})\cdot e^{-i\omega_1^b(\ul{x_1})}\big) \\
	&= \sum_{\sigma^{m_b}(\ul{y})=\ul{x_1}}|g(\ul{y})|e^{\bar{\phi}_{m_b}(\ul{y})}\cdot \Re \big(1-e^{ibr_{m_b}(\ul{y})}e^{i\omega_0^b(\ul{y})}e^{-i\omega_1^b(\ul{x_1})}\big),
\end{aligned}
\]
where $\Re(c)$ is the real part of $c$.  Every term in this sum is non-negative, and since $\sigma^{m_b}(\ul{x_0})=\ul{x_1}$, the sum is no less than $|g(\ul{x_0})|e^{\bar{\phi}_{m_b}(\ul{x_0})}\cdot \Re \big(1-e^{ibr_{m_b}(\ul{x_0})}e^{i\omega_0^b(\ul{x_0})}e^{-i\omega_1^b(\ul{x_1})}\big)$. Note that
$\bar{\phi}_{m_b}(\ul{x_0})\ge -\Phi\cdot \varrho\log|b|$. It follows from the first inequality of \eqref{eq.dic1} and the fact $\Re(1-e^{i\omega})=|1-e^{i\omega}|^2/2$ that
\[
(\mathcal{L}_0^{m_b}|g|)(\ul{x_1})-\xi_1^b(\ul{x_1})\geq \frac{1}{16\norm{\mathbbm{1}}_{\mathscr{L}}}\cdot e^{-\Phi \varrho \log |b|}\cdot  \frac{|b|^{-2\lambda}}{2}\ge |b|^{-\lambda'_2}.
\]
Therefore, 
\[
|(\mathcal{L}_{ib}^{\ell_0+m_b}f)(\ul{x_1})|=\xi_1^b(\ul{x_1})\leq (\mathcal{L}_0^{m_b}|g|)(\ul{x_1})-|b|^{-\lambda'_2}\leq (\mathcal{L}_0^{\ell_0+m_b}|f|)(\ul{x_1})-|b|^{-\lambda'_2}.
\]
Note that $n_0:=\ell_0+m_b<(\varrho'_1+\varrho)\log|b|+1<\varrho_2\log|b|$. This proves Claim 2. 
\end{proof}

In general, the equilibrium state $\mu$ of  $\phi^+$ for the infinite-state TMS $(\Sigma^+,\sigma)$ has no global Gibbs property. However, by the SPR property, it has the following {\it local Gibbs property} (see, e.g. \cite[Theorem 3.1]{Sar11}):  there exists a constant $G_S>1$ such that for any cylinder ${}_0[v_{0},v_{1},\dots,v_{n-1}]$ of length $n$ satisfying $v_0, v_{n-1}\in S$ and any $\ul{x}\in {}_0[v_{0},v_{1},\dots,v_{n-1}]$, we have
\[
G_S^{-1}\cdot e^{\varphi_n^+(\ul{x})-nP_{\rm TOP}(\sigma,\varphi^+)}\le \mu({}_0[v_{0},v_{1},\dots,v_{n-1}])\le G_S\cdot e^{\varphi_n^+(\ul{x})-nP_{\rm TOP}(\sigma,\varphi^+)}.
\]

\noindent  {\bf Claim 3.} {\it There exists $\lambda_2''>0$ such that for all $|b|\geq \wt{b}$ and $f\in\mathscr{L}$ satisfying \eqref{eq.nasp}, we have 
\[
\int |\mathcal{L}_{ib}^{n_0}f|d{\mu}\leq \int \mathcal{L}_0^{n_0}|f|d{\mu}-|b|^{-\lambda_2''}\norm{f}_{(b)},
\]
where $\wt{b}$ and $n_0$ are given by Claim 2.}
\begin{proof}
Define $\Psi:= \sup |\varphi^+|+P_{\rm TOP}(\sigma, \varphi^+)$. Moreover, define
\[
\begin{aligned}
\lambda''_2&:=\lambda'_2+\log(2G_S)+\bigg(\frac{\#S}{\theta}\Big(\log(36\big(R_0+C_2+C_3)H\big)+(\lambda_0+\lambda_2'+1)\Big)+1\bigg)\cdot\Psi,\\
m_0&:=\Big\lfloor \frac{\#S}{\theta}\Big(\log(36\big(R_0+C_2+C_3)H\big)+(\lambda_0+\lambda_2'+1)\log|b|\Big)\Big\rfloor +1,
\end{aligned}
\]
where $\lambda'_2$ is given by Claim~2 and $R_0, C_2, C_3>1$ satisfy the statements in Lemma~\ref{lem.propUp12norm}. It is clear that for any $|b|\ge \wt b\ge 3$,
\begin{equation}\label{equ.lambda2pp}
 |b|^{\lambda''_2-\lambda'_2}\ge 2G_S\cdot e^{m_0\Psi}.
\end{equation}

For any $|b|\ge \wt{b}$, by Claim~2, there exists $\ul{x_1}\in\mathcal{O}$ such that 
\begin{equation}\label{equ.byclaim2}
	|\mathcal{L}_{ib}^{n_0}f(\ul{x_1})|\leq \mathcal{L}_0^{n_0}|f|(\ul{x_1})-|b|^{-\lambda_2'}\norm{f}_{(b)}.
\end{equation}

Write $\ul{x_1}=(x_{1,0},x_{1,1},x_{1,2},\dots)$. Define $\mathbf{C}_0:={}_0[x_{1,0},x_{1,1},\dots,x_{1,m_0-1}]$, which is a cylinder containing $\ul{x_1}$. Note that $x_{1,n}\in S$ for all $n\ge 0$. By the local Gibbs property, we have
\begin{equation}\label{equ.localgibbs}
	\mu(\mathbf{C}_0)\ge G_S^{-1}\cdot e^{-m_0\Psi}.
\end{equation}

Let $s\in S$ be the vertex that appears most frequently among $x_{1,0},x_{1,1},\dots,x_{1,m_0-1}$. Its occurrence count is at least $\frac{m_0}{\#S}$. Thus for any $\ul{y}\in \mathbf{C}_0$, we have $t(\ul{x_1},\ul{y})\geq m_0$ and $t_s(\ul{x_1},\ul{y})\ge \frac{m_0}{\#S}$. It implies that
\[
\begin{aligned}
|(\mathcal{L}_{ib}^{n_0}f)(\ul{x_1})-(\mathcal{L}_{ib}^{n_0}f)(\ul{y})|
&\leq  {h}_{s,0}[x_{1,0}]\cdot e^{-t_s(\ul{x_1},\ul{y})\cdot \theta}\cdot \Upsilon_{s,2} (\mathcal{L}_{ib}^{n_0}f)& \text{(since $\mathcal{L}_{ib}^{n_0}f\in\mathscr{L}\subset \mathscr{L}_s$)}\\
&\le H e^{^{-\frac{m_0}{\#S}\cdot \theta}}\norm{\mathcal{L}_{ib}^{n_0}f}_{\mathscr{L}}&\text{(since $\Upsilon_{s,2}(\cdot)\le \norm{\cdot}_s\le\norm{\cdot}_{\mathscr{L}}$)}\\
&\le 9H e^{^{-\frac{m_0}{\#S}\cdot \theta}}R_0|b|^{\lambda_0+1}\norm{f}_{(b)}& \text{(by Lemma~\ref{lem.propUp12norm}~\ref{ite.normR0} and \eqref{eq.relnorms})}\\
&\le \frac{1}{4}|b|^{-\lambda'_2}\norm{f}_{(b)}. & \text{(by the choice of $m_0$)}
\end{aligned}
\]

On the other hand, for any $\ul{y}\in \mathbf{C}_0$, we also have
\[
\begin{aligned}
|(\mathcal{L}_{0}^{n_0}|f|)(\ul{x_1})-(\mathcal{L}_{0}^{n_0}|f|)(\ul{y})|
&\le H e^{^{-\frac{m_0}{\#S}\cdot \theta}}\norm{\mathcal{L}_{0}^{n_0}|f|}_{\mathscr{L}}&\\
&\le H e^{^{-\frac{m_0}{\#S}\cdot \theta}}\norm{P|f|+N^{n_0}|f|}_{\mathscr{L}}&\\
&\le 9H e^{^{-\frac{m_0}{\#S}\cdot \theta}}(C_2+C_3)|b|^{\lambda_0}\norm{f}_{(b)}& \text{(by Lemma~\ref{lem.propUp12norm}~\ref{ite.Pfisint}, \ref{ite.Nlessthan1} and \eqref{eq.relnorms})}\\
&\le \frac{1}{4}|b|^{-\lambda'_2}\norm{f}_{(b)}. & \text{(by the choice of $m_0$)}
\end{aligned}
\]

Combining the above two inequalities with \eqref{equ.byclaim2}, it follows that for all $\ul{y}\in \mathbf{C}_0$,
\[
\begin{aligned}
(\mathcal{L}_0^{n_0}|f|)(\ul{y})-|(\mathcal{L}_{ib}^{n_0}f)(\ul{y})|
\ge \frac{1}{2}|b|^{-\lambda'_2}\norm{f}_{(b)}.
\end{aligned}
\]

Recall that $|(\mathcal{L}_{ib}^{n_0}f)(\ul{x})|\le (\mathcal{L}_0^{n_0}|f|)(\ul{x})$. Then by \eqref{equ.localgibbs} and \eqref{equ.lambda2pp}, we conclude
\[
\begin{aligned}
\int \mathcal{L}_0^{n_0}|f|d{\mu}-\int |\mathcal{L}_{ib}^{n_0}f|d{\mu}
&\ge \int_{\mathbf{C}_0}(\mathcal{L}_0^{n_0}|f|-|\mathcal{L}_{ib}^{n_0}f|)d{\mu}\\
&\ge \frac{1}{2}|b|^{-\lambda'_2}\norm{f}_{(b)}\cdot G_S^{-1}\cdot e^{-m_0\Psi}\\
&\ge |b|^{-\lambda_2''}\norm{f}_{(b)}.
\end{aligned}
\]

The proof of Claim 3 is completed.
\end{proof}

We continue the proof of Lemma~\ref{lem.newUpsilonpi}. Take $\lambda_2>\lambda_2''$ such that $|b|^{-(\lambda_2-\lambda_2'')}\le \norm{\mathbbm{1}}_{\mathscr{L}}$ for all $|b|\ge \wt{b}$. Let $f\in\mathscr{L}$ satisfy \eqref{eq.nasp}. By Lemma~\ref{lem.propUp12norm}~\ref{ite.decompose} and \ref{ite.Pfisint}, we have 
$\int \mathcal{L}_0^{n_0}|f|d{\mu}\cdot \norm{\mathbbm{1}}_{\mathscr{L}}=\Upsilon_1(P(\mathcal{L}_0^{n_0}|f|))=\Upsilon_1(P|f|)\leq \norm{f}_{(b)}$. By Claim 3 we have 
\[
\Upsilon_1(P|\mathcal{L}_{ib}^{n_0}f|)=\int|\mathcal{L}_{ib}^{n_0}f|d{\mu}\cdot \norm{1}_{\mathscr{L}}\le \int \mathcal{L}_0^{n_0}|f|d{\mu}\cdot \norm{1}_{\mathscr{L}}-|b|^{-\lambda_2''} \norm{1}_{\mathscr{L}}\norm{f}_{(b)}\le (1-|b|^{-\lambda_2})\norm{f}_{(b)}.
\]

On the other hand, by \eqref{equ.LibfL0absf}, for any $n\ge \varrho_2\log|b|>n_0$, we have 
\[
\int |\mathcal{L}_{ib}^nf|d{\mu}=  \int |\mathcal{L}_{ib}^{n-n_0}(\mathcal{L}_{ib}^{n_0}f)|d{\mu} \le \int \mathcal{L}_{0}^{n-n_0}|\mathcal{L}_{ib}^{n_0}f|d{\mu}.
\]
Hence
\[
\Upsilon_1(P|\mathcal{L}_{ib}^{n}f|)\le \Upsilon_1(P(\mathcal{L}_{0}^{n-n_0}|\mathcal{L}_{ib}^{n_0}f|)) =\Upsilon_1(P|\mathcal{L}_{ib}^{n_0}f|)\leq (1-|b|^{-\lambda_2})\norm{f}_{(b)}.
\]
This proves the lemma.
\end{proof}

\begin{lem}\label{lem.newUpsilon1}
	There exist $\lambda_3>0$ and $\varrho_3>0$ such that for all $|b|\geq \wt{b}$ (where $\wt{b}$ is given by Lemma~\ref{lem.newUpsilonpi}), $n\geq \varrho_3\log|b|$ and $f\in\mathscr{L}$, we have
	\[
	\Upsilon_1(\mathcal{L}_{ib}^{n}f)\leq (1-|b|^{-\lambda_3})\norm{f}_{(b)}.
	\]
\end{lem}
\begin{proof}
Denote 
\[
\varrho_3:=\frac{\log(18R_0C_3)+\lambda_0+\lambda_2}{-\log\delta}+\varrho_2+1,\quad \lambda_3:=\lambda_2+1, \quad\text{and}\quad  n_1:=\lfloor \varrho_2\log|b|\rfloor +1,
\]
where $C_3$ and $\delta$ satisfy Lemma~\ref{lem.propUp12norm}~\ref{ite.Nlessthan1}, $\lambda_2$ and $\varrho_2$ are given by Lemma~\ref{lem.newUpsilonpi}, $R_0$ and $\lambda_0$ satisfy Lemma~\ref{lem.propUp12norm}~\ref{ite.normR0} and \ref{ite.LasoYor}. Then for all $|b|\geq \wt{b}$, $n\geq \varrho_3\log|b|$ and $f\in\mathscr{L}$, we have
\[\begin{aligned}
\Upsilon_1(\mathcal{L}_{ib}^{n}f)&\leq \Upsilon_1(\mathcal{L}_{0}^{n-n_1}|\mathcal{L}_{ib}^{n_1}f|)& \text{(by \eqref{equ.LibfL0absf})}\\
&\leq \norm{\mathcal{L}_{0}^{n-n_1}|\mathcal{L}_{ib}^{n_1}f|}_{\mathscr{L}}&\\
&\leq \norm{P|\mathcal{L}_{ib}^{n_1}f|}_{\mathscr{L}}+\norm{N^{n-n_1}|\mathcal{L}_{ib}^{n_1}f|}_{\mathscr{L}}&\text{(by Lemma~\ref{lem.propUp12norm}~\ref{ite.decompose})}\\
&\leq \Upsilon_1(P|\mathcal{L}_{ib}^{n_1}f|)+C_3\delta^{n-n_1}\cdot \norm{|\mathcal{L}_{ib}^{n_1}f|}_{\mathscr{L}}& \text{(by Lemma~\ref{lem.propUp12norm}~\ref{ite.Pfisint} and \ref{ite.Nlessthan1})}\\
&\leq (1-|b|^{-\lambda_2})\norm{f}_{(b)}+9R_0C_3|b|^{\lambda_0+1}\delta^{n-n_1}\norm{f}_{(b)}& \text{(by Lemma~\ref{lem.newUpsilonpi}, \ref{lem.propUp12norm}~\ref{ite.absless}, \ref{ite.normR0} and \eqref{eq.relnorms})}\\
&\leq (1-|b|^{-\lambda_2})\norm{f}_{(b)}+\frac{1}{2}|b|^{-\lambda_2}\norm{f}_{(b)}& \text{(by the choice of $\varrho_3$)}\\
&\le (1-|b|^{-\lambda_3})\norm{f}_{(b)}.
\end{aligned}\]
\end{proof}

Combining the three previous estimates for the three components of the norm $\norm{\cdot}_{(b)}$, we have the following corollary.

\begin{cor}\label{cor.bnorm}
	There exist $\lambda_4>1$ and $\varrho_4>1$ such that for all $|b|\geq \wt{b}$ (where $\wt{b}$ is given by Lemma~\ref{lem.newUpsilonpi}),  $f\in\mathscr{L}$ and $n\geq \varrho_4\log|b|$, we have
	\[
	\norm{\mathcal{L}_{ib}^nf}_{(b)}\leq (1-|b|^{-{\lambda_4}})\norm{f}_{(b)}.
	\]
\end{cor}
\begin{proof}
	Take  $\lambda_4:=\max\{2,\lambda_1,\lambda_2,\lambda_3\}$ and $\varrho_4:=\max\{2,\varrho_1,\varrho_2,\varrho_3\}$, where $\lambda_1,\lambda_2,\lambda_3$ and $\varrho_1,\varrho_2,\varrho_3$ are given by Lemmas~\ref{lem.newUpsilon2}, \ref{lem.newUpsilonpi} and \ref{lem.newUpsilon1}. This choice guarantees that all three components in the definition of $\norm{\cdot}_{(b)}$ satisfy the uniform contraction simultaneously, hence the conclusion follows immediately.
\end{proof}

Now we prove the main result of this subsection.

\begin{pro}\label{pro.Idabinv}
	There exist constants $\wt C_0>0$ and $\wt \lambda_0>0$ such that for any $|b|\geq \wt{b}$ (where $\wt{b}$ is given by Lemma~\ref{lem.newUpsilonpi}), $|a|\leq |b|^{-\wt{\lambda}_0}$, and $n\in \bN_0$, we have
\[
\norm{\mathcal{L}_{a+ib}^n}_{\mathscr{L}}\le \wt C_0|b|^{\wt \lambda_0}e^{-n|b|^{-\wt\lambda_0}}.
\]   	
\end{pro}
\begin{proof}
	Denote $m:=\lfloor \varrho_4\log|b| \rfloor+1$, where $\varrho_4>1$ is given by Corollary~\ref{cor.bnorm}.	For any $n\in \bN_0$, we have $n=km+m'$, where $ k:=\lfloor \frac{n}{m}\rfloor$ and $0\le m'<m$. For any $|b|\geq \wt{b}$, by Lemma~\ref{lem.bnormofonestep} and 
	Corollary~\ref{cor.bnorm}, 
	\[
	\norm{\mathcal{L}_{ib}^n}_{(b)}\le \norm{\mathcal{L}_{ib}^m}_{(b)}^k\cdot \norm{\mathcal{L}_{ib}}_{(b)}^{m'}\le (1-|b|^{-{\lambda_4}})^k(2C_6)^{m'}.
	\]
	
Since $m'<m\le 2\varrho_4\log|b|$ and $k> \frac{n}{m}-1\ge  \frac{n}{2\varrho_4\log|b|}-1$, we have $(2C_6)^{m'}\le (2C_6)^{ 2\varrho_4\log|b|}=|b|^{ 2\varrho_4\log(2C_6)}$ and $(1-|b|^{-{\lambda_4}})^k\le (1-|b|^{-{\lambda_4}})^{-1}(1-|b|^{-{\lambda_4}})^{ \frac{n}{2\varrho_4\log|b|}}\le \frac{3}{2}e^{\log (1-|b|^{-{\lambda_4}})\cdot \frac{n}{2\varrho_4\log|b|}}$. Using the fact $\log |b|<|b|$ and $\log (1-x)\le -x$ for all $0\le x<1$, we finally obtain
\begin{equation}\label{equ.ibnbnorm}
\norm{\mathcal{L}_{ib}^n}_{(b)}\le	|b|^{ 2\varrho_4\log(2C_6)}\cdot  \frac{3}{2}e^{-n(2\varrho_4)^{-1} |b|^{-(\lambda_4+1)}}	
\end{equation}	

We next consider $\mathcal{L}_{a+ib}^n$. Write $L:=\mathcal{L}_{ib}$ and $D:=\mathcal{L}_{a+ib}-\mathcal{L}_{ib}$ for brevity.
\begin{equation}\label{equ.sumLDn}
	\mathcal{L}_{a+ib}^n=(L+D)^n=
\sum_{\ell=0}^{n}
\sum_{\substack{j_1,\dots,j_{\ell+1}\ge 0\\ j_1+\cdots+j_{\ell+1}=n-\ell}}
L^{j_1}DL^{j_2}D\cdots DL^{j_{\ell+1}}.
\end{equation}
By Lemma~\ref{lem.propUp12norm}~\ref{ite.aib-ib} and \eqref{eq.relnorms}, we have $\norm{D}_{(b)}\le C_2 9|b|^{\lambda_0} \norm{\mathcal{L}_{a+ib}-\mathcal{L}_{ib}}_{\mathscr{L}}\le 9 C_2 C_7 |a||b|^{\lambda_0+1}$ for all $|a|\le 1/C_5'$. Then by \eqref{equ.ibnbnorm}, for the norm of each summand in \eqref{equ.sumLDn}, we have
\[
\begin{aligned}
	\|L^{j_1}DL^{j_2}D\cdots DL^{j_{\ell+1}}\|_{(b)}&\le \|L^{j_1}\|_{(b)}\cdots \|L^{j_{\ell+1}}\|_{(b)}\cdot \|D\|_{(b)}^{\ell}\\
	&\le \Big(|b|^{\lambda_5}\cdot  \frac{3}{2}e^{-j_1\lambda_6 |b|^{-(\lambda_4+1)}}\Big)\cdots \Big(|b|^{\lambda_5}\cdot  \frac{3}{2}e^{-j_{\ell+1}\lambda_6 |b|^{-(\lambda_4+1)}} \Big)\cdot (9 C_2 C_7 |a||b|^{\lambda_0+1})^\ell\\
	&=\frac{3}{2}|b|^{\lambda_5}\cdot  \big(\frac{3}{2}|b|^{\lambda_5} \big)^\ell \cdot e^{-(n-\ell)\lambda_6 |b|^{-(\lambda_4+1)}}\cdot (9 C_2 C_7 |a||b|^{\lambda_0+1})^\ell\\
	&=\frac{3}{2}|b|^{\lambda_5}e^{-n\lambda_6 |b|^{-(\lambda_4+1)}}\cdot  \big(\frac{27}{2}C_2C_7|a||b|^{\lambda_5+\lambda_0+1} e^{\lambda_6 |b|^{-(\lambda_4+1)}} \big)^\ell,
\end{aligned}
\]
where $\lambda_5:= 2\varrho_4\log(2C_6)$ and $\lambda_6:=(2\varrho_4)^{-1}$. Substituting back into the sum \eqref{equ.sumLDn} to calculate its norm, and noting that there are exactly $\binom{n}{\ell}$ distinct terms containing $\ell$ copies of $D$, we obtain that
\[
\begin{aligned}
	\|\mathcal{L}_{a+ib}^n\|_{(b)}&\le \sum_{\ell=0}^{n} \binom{n}{\ell}\frac{3}{2}|b|^{\lambda_5}e^{-n\lambda_6 |b|^{-(\lambda_4+1)}}\cdot  \big(\frac{27}{2}C_2C_7|a||b|^{\lambda_5+\lambda_0+1} e^{\lambda_6 |b|^{-(\lambda_4+1)}} \big)^\ell\\
 &= \frac{3}{2}|b|^{\lambda_5}e^{-n\lambda_6 |b|^{-(\lambda_4+1)}}\cdot \Big(1+ \frac{27}{2}C_2C_7|a||b|^{\lambda_5+\lambda_0+1} e^{\lambda_6 |b|^{-(\lambda_4+1)}}\Big)^n\\
 &\le \frac{3}{2}|b|^{\lambda_5} \exp\Big(-n \big(\lambda_6 |b|^{-(\lambda_4+1)}- \frac{27}{2}C_2C_7|a||b|^{\lambda_5+\lambda_0+1} e^{\lambda_6 |b|^{-(\lambda_4+1)}} \big)\Big).
\end{aligned}
\]
The last ``$\le$'' follows from the fact $(1+x)^n\le e^{nx}$ for all $x>-1$. Since $e^{\lambda_6 |b|^{-(\lambda_4+1)}}\le e^{\frac{1}{6\varrho_4}}$ for all $|b|\ge \wt b\ge 3$,
there exists $\lambda_7>0$ such that 
\[
\frac{27}{2}C_2C_7|a||b|^{\lambda_5+\lambda_0+1} e^{\lambda_6 |b|^{-(\lambda_4+1)}} \le \frac{1}{2} \lambda_6 |b|^{-(\lambda_4+1)}
\]
for all $|b|\ge \wt b$ and $|a|\le |b|^{-\lambda_7}$. This gives
\[
\|\mathcal{L}_{a+ib}^n\|_{(b)}\le \frac{3}{2}|b|^{\lambda_5} e^{-n \cdot \frac{1}{2} \lambda_6 |b|^{-(\lambda_4+1)}}.
\]

Using \eqref{eq.relnorms} again, we obtain
\[
\|\mathcal{L}_{a+ib}^n\|_{\mathscr{L}}\le C_2 9|b|^{\lambda_0}  \|\mathcal{L}_{a+ib}^n\|_{(b)}\le  \frac{27}{2}C_2 |b|^{\lambda_5+\lambda_0} e^{-n \cdot \frac{1}{2} \lambda_6 |b|^{-(\lambda_4+1)}}.
\]
The proof is completed by taking $\wt C_0:=\frac{27}{2}C_2$ and choosing a constant $\wt\lambda_0\ge\max\{\lambda_7, \lambda_5+\lambda_0\}$ such that $|\wt b|^{\wt\lambda_0-(\lambda_4+1)}\ge \frac{2}{\lambda_6}$.
\end{proof}

%%%%%%%%%%%%%%%%%%%%%%%%%%%%%%%%%%%%%%%%%%%%%%%%%%%%%%%%%%%%%%%%%%%

\section{Rapid mixing for SPR Borel homoclinic classes with good asymptotics}\label{sec.rapidflow}

In this section, we return to flows on manifolds, and show that the equilibrium state on a SPR Borel homoclinic class having good asymptotics is rapidly mixing.

\subsection{Appropriate TMF for Borel homoclinic classes}\label{subsec.goodasy}

Consider $C^\alpha$ vector fields ($\alpha\ge 2$) on $M$. Let $p_0$ be a hyperbolic periodic point  
and $q$ be a transverse homoclinic point of $p_0$. By using the shadowing lemma, it is proved in \cite{FMT07} that for any vector field $X$ in a $C^1$-open and $C^\alpha$-dense set, there exist constants $\Pi>0$, $\gamma\in(0,1)$, $\Theta\in[0,\pi)$, $\Omega>0$, and a sequence of periodic points $\{p_N\}_{N\in \bZ^+}$ with $p_N\to q$,  sequences $E_N$ and $\zeta_N$ such that, the orbits of all $p_N$ are contained in the same transitive hyperbolic set and
\begin{equation}\label{eq.expressperiod}
\tau_N=N\tau_0+\Pi+E_N\cdot \gamma^N\cos(N\Theta+\zeta_N)+o(\gamma^N),
\end{equation}
where $\tau_N$ is the period of $p_N$, $|E_N|<\Omega$ for every $N\geq 1$ and either $\Theta=0$ and $\zeta_N\equiv 0$, or  $\Theta\in(0, \pi)$ and $\zeta_N\in(\Theta_0-\pi/12,\Theta_0+\pi/12)$ for every $N\geq 1$, where $\Theta_0$ is a fixed constant independent of $N$.

We say that the hyperbolic periodic point $p_0$ has {\it good asymptotics} if for the sequence of periodic points $p_N\to q$ given above, we have $\liminf_{N\to\infty}|E_N|> 0$. Moreover, we say that a Borel homoclinic class $B$ has {\it good asymptotics} if it contains a periodic point having good asymptotics.

\begin{pro}{\rm \cite{FMT07}}\label{pro.goodasymptotics}
Let $X$ be a $C^1$ vector field on $M$ and $p_0$
be a hyperbolic periodic point of $X$ exhibiting a transverse homoclinic point. Then there exist a $C^1$ neighborhood $\U^1$ of $X$ satisfying the following property: for any $\alpha\ge 2$, there is a $C^2$-open and $C^\alpha$-dense subset $\U^\alpha$ in the set of all $C^\alpha$ vector fields in $\U^1$  such that for any $Y\in \U^\alpha$, the continuation $p_0^Y$ of $p_0$ has good asymptotics.
\end{pro}
\begin{proof}[Sketch of the Proof]
Let $q$ be a  transverse homoclinic point of $p_0$. The union of the orbit of $p_0$ and the orbit of $q$ is a hyperbolic set. There exists a closed neighborhood $\overline U$ of this hyperbolic set and  a $C^1$ neighborhood $\U^1$ of $X$
such that for any $Z\in \U^1$,  the maximal invariant set in $\overline U$ of $Z$ is a transitive hyperbolic set containing $p_0^Z$, where  $p_0^Z$ is  the continuation of $p_0$. Let $q^Z$ be a transverse homoclinic point of $p_0^Z$ contained in this maximal invariant set and depending continuously on $Z$.

	Fix a local transverse section containing $p_0^Z$, and consider the Poincar\'e return map $\mathbf P$ defined on a small neighborhood of $p_0^Z$ in this transverse section. 
	There is a $C^1$-open and $C^\alpha$-dense subset $\wt{\U}^1$ in the set of all $C^\alpha$ vector fields in $\U^1$ such that for any $Z\in \wt{\U}^1$, the Poincar\'e return map $\mathbf P$ can be $C^1$-linearized  in a small neighborhood of $p_0^Z$, and the linearization depends continuously on $\mathbf P$ in the $C^2$ topology. By using the shadowing lemma, 
	we can obtain a periodic point $p_N$ whose orbit is contained in $\overline U$, and 
$p_N$ shadows a pseudo-orbit consisting of several orbit segments from the Poincar\'e return map and an orbit segment from the orbit of $q^Z$. The orbit of $p_N$ is contained in the maximal invariant set of $\overline U$, which is a transitive hyperbolic set. 

Furthermore, by using the argument in \cite[\S\,4.2 and \S\,4.3]{FMT07}, for any  $Z\in\wt{\U}^1$ and $C^\alpha$ neighborhood $\U(Z)$, there exists a vector field $Z'\in \U(Z)$ and a $C^2$-open neighborhood $\U'(Z')$ of $Z'$ such that \eqref{eq.expressperiod} and $\liminf_{N\to\infty}|E_N|> 0$ hold for all $Y\in \U'(Z')$. Taking the union of all such $C^2$-open neighborhoods $\U'(Z')$ for all $Z\in\wt{\U}^1$, we obtain the desired $C^2$-open and $C^\alpha$-dense subset $\U^\alpha$.
\end{proof}

In the remaining part of this section, we always assume that $X$ is a $C^{1+}$ vector field on $M$, $B$ is a Borel homoclinic class and $p_0\in B$ has good asymptotics. Moreover, we assume that  $X$ is SPR on $B$ for a H\"older continuous potential $\vartheta :M\to\mathbb{R}$. By Theorem~\ref{thm.main.spr}, there exists a unique   equilibrium state $\nu $ of $\vartheta$ for $X$  on $B$. 

 We first construct a TMF coding $B$ such that the appropriate one-sided TMF satisfies the setup in the previous section. 
 
 \begin{lem}\label{lem.appropriateonesided}
 	There exists a TMF $(\Sigma_r,\sigma_r)$ and a factor map ${\pi}_{{r}}:{\Sigma}_{{r}}\to B$ satisfying the following properties:
 	\begin{enumerate}
 	\item {\rm about the factor map $\pi_r$:} $\pi_r$ is a H\"older continuous finite-to-one map such that $\pi_r(\Sigma_r)$ has full $\mu$-measure and $\pi_r\circ \sigma_{r,t}=\varphi_t\circ \pi_r$ for all $t$;
 	\item {\rm about the two-sided TMF $(\Sigma_r,\sigma_r)$:} $(\Sigma_r,\sigma_r)$ is topologically transitive; it has good asymptotics; it is SPR for $\vartheta\circ \pi_r$; the unique equilibrium state $\bar\nu$ of $\vartheta\circ \pi_r$ for $(\Sigma_r,\sigma_r)$ satisfies $\nu=\bar\nu\circ \pi_r^{-1}$;
 	\item {\rm about the base $(\Sigma,\sigma)$:} $(\Sigma,\sigma)$ is topologically mixing and locally compact; it is SPR for the H\"older continuous potential $\phi:=\Delta_{\vartheta\circ \pi_r}-P_{\rm TOP}(\sigma_r,\vartheta\circ \pi_r)\cdot r$ (see Lemma~\ref{lem.relCT}); the projection measure $\mu:=\bar\nu_{_{\Sigma}}$ is the unique equilibrium state of $\phi$ for $(\Sigma,\sigma)$ (see Proposition~\ref{pro.relationmme});
 	\item {\rm about the associated one-sided TMS $(\Sigma^+,\sigma)$:} $(\Sigma^+,\sigma)$ is topologically mixing and locally compact; it is SPR for $\phi^{+}$, where $\phi^{+}:\Sigma^{+}\to\mathbb{R}$ satisfies $\phi^{+}\circ\mathbf{p}^{+}=\phi+\varsigma_\phi-\varsigma_\phi\circ\sigma$ for some H\"older continuous function $\varsigma_\phi :\Sigma\to\mathbb{R}$, and ${\mathbf{p}}^{+}: \Sigma\to\Sigma^+$ is the natural semi-conjugacy; the projection measure $\mu^{+}:=\mu\circ (\mathbf{p}^{+})^{-1}$ is the unique equilibrium state of $\phi^+$ for $(\Sigma^+,\sigma)$;
 	\item {\rm about the associated one-sided TMF $(\Sigma^+_{r^+},\sigma_{r^+})$:} $(\Sigma^+_{r^+},\sigma_{r^+})$ has good asymptotics; the roof function $r^+$ is  H\"older continuous with $\inf r^+>0$, and satisfies $r^+\circ \mathbf{p}^{+}=r+\varsigma-\varsigma\circ\sigma$ for some H\"older continuous functions $\varsigma :\Sigma\to\mathbb{R}$.
    \end{enumerate}
 \end{lem}
 \begin{proof}
 	The proof proceeds in two stages of recoding: the first uses the spectral decomposition of the irreducible base TMS to produce a topologically mixing one, while the second takes a sufficiently large iterate of the shift to ensure that the one-sided roof function obtained from Sinai's lemma admits a strictly positive infimum. All key properties including the SPR property, unique equilibrium state and good asymptotics are preserved under these (semi-)conjugacies. The proof is placed in Appendix~\ref{appsec.decay}. 
 \end{proof}
  
The one-sided TMF $(\Sigma^+_{r^+},\sigma_{r^+})$ constructed in  Lemma~\ref{lem.appropriateonesided} satisfies all the standing assumptions in \S\,\ref{sec.EstimRurlle}, and every result established in that section applies directly to $(\Sigma^+_{r^+},\sigma_{r^+})$.

\subsection{Rapid mixing for singular flows}

Now we prove the main result of this section.

\begin{thm}\label{thm.rapidmixing}
	Let  $X$ be a $C^{1+}$ vector field  on $M$ and $\vartheta$ be a H\"older continuous potential on $M$.  If $X$ is SPR for $\vartheta$ on a Borel homoclinic class $B$ and $B$ has good asymptotics, then the unique equilibrium state of $\vartheta$ for $X$  on $B$ is rapidly mixing.
	\end{thm}
\begin{proof}
Through the preceding derivations, we have obtained the estimates for the complex Ruelle operator of an appropriate one-sided TMS. The remaining proof follows the standard arguments as presented in \cite{Pol85,Dol98,Melbourne2007}, and other related references.

Denote by $\nu$ the unique equilibrium state of $\vartheta$ for $X$  on $B$. Let $(\Sigma_r,\sigma_r)$, $\pi_r:\Sigma_r\to B$ and the associated $(\Sigma^+_{r^+},\sigma_{r^+})$, $\varsigma$, $\bar\nu$,  $\mu$, $\mu^+$, and $\phi^+$  be given by Lemma~\ref{lem.appropriateonesided}. Denote $\bar r:=r^+\circ \mathbf{p}^+$, where $\mathbf{p}^+:\Sigma\to \Sigma^+$ is the natural semi-conjugacy.
Fix $\theta>0$ such that all of $\pi_r$, $\varsigma$, $\phi^+$, $r$, $r^+$, and $\bar r$ are $\theta$-H\"older continuous. By Lemma~\ref{lem.mixing}, the unique equilibrium state $\bar\nu$ of $\vartheta\circ \pi_r$ for $(\Sigma_r,\sigma_r)$ is mixing.

Take any $C^m$ functions $F$ and $G$ on $M$ for some large $m$ which will be determined later. Define $f:=F\circ \pi_r$ and $g:= G\circ \pi_r$. Without loss of generality, we assume that $\int_{\Sigma_r} fd\bar\nu=\int_{\Sigma_r} gd\bar\nu=0$, and otherwise we only need to replace $f$ and $g$ with $f-\int_{\Sigma_r} fd\bar\nu$ and $g-\int_{\Sigma_r} gd\bar\nu$, respectively.
Then it suffices to prove that the correlation function for $f$ and $g$  
\[
{\rho}_{{f,g}}(t):=\int_{\Sigma_r} f\cdot (g\circ \sigma_{r,t})d\bar\nu
\]
decays rapidly. The Laplace transform of ${\rho}_{f,g}$ can be written as:
\begin{equation}\label{equ.laplace}
  \widehat{{\rho}}_{f,g}(c)=\frac{1}{\int_{\Sigma} r d\mu}
\sum_{n=0}^{\infty}\int_{\Sigma} e^{-cr_n}f_c\cdot g_c\circ \sigma^nd{\mu}+ R_{f,g}(c),
\end{equation}
where $f_c(\ul{x})=\int_0^{r(\ul{x})}e^{cu}f(\ul{x},u)du$, $g_c(\ul{x})=\int_0^{r(\ul{x})}e^{-cu}g(\ul{x},u)du$ and $R_{f,g}$ is an analytic part which has no impacts on the decay rate.

Denote by 
$\partial^\ell_t f$ the $\ell$-th order derivative of $f$ along the flow direction, $0\le \ell<m$. Note that $g$ and $\partial_t^\ell f$ can be regarded as a $\theta$-H\"older continuous function on $\Sigma_r$ for each $0\le \ell<m$. Integrating $f_c(\ul{x})$ by parts $\ell$ times and substituting back into \eqref{equ.laplace}, we obtain
\begin{equation}\label{equ.lpartlaplace}
  \widehat{{\rho}}_{f,g}(c)=\frac{(-1)^\ell}{c^\ell\int_{\Sigma} r d\mu}
\sum_{n=0}^{\infty}\int_{\Sigma} e^{-cr_n}f^{(\ell)}_c\cdot g_c\circ \sigma^nd{\mu}+ R'_{f,g}(c),
\end{equation}
where $f^{(\ell)}_c(\ul{x})=\int_0^{r(\ul{x})}e^{cu}(\partial^\ell_t f)(\ul{x},u)du$,  and $R'_{f,g}$ consists of an analytic part and boundary terms from integration by parts, 
which has no impacts on the decay rate.

For any $c$ with $|\Re c|=|a|\le 1$, we have $|f_c^{(\ell)}(\ul{x})|\le \sup r\cdot e^{\sup r}\cdot \sup|\partial^{\ell}_t f|$, and
\[
	|f_c^{(\ell)}(\ul{x})-f_c^{(\ell)}(\ul{y})| \le \int_0^{r(\ul{y})}e^{|a|u}|(\partial^{\ell}_t f)(\ul{x},u)-\partial^{\ell}_t f(\ul{y},u)|du+\big(|{r(\ul{x})}-{r(\ul{y})}| \cdot e^{\sup r}\cdot \sup|\partial^{\ell}_t f|\big).
\]
Thus $\|f_c^{(\ell)}\|_{\theta}\le e^{\sup r} \|r\|_\theta \cdot \|\partial^{\ell}_t f\|_\theta \le e^{\sup r} \|r\|_\theta \cdot(1+\|\pi_r\|_\theta) \|F\|_{C^m}$. Similarly,  $\|g_c\|_{\theta}\le e^{\sup r} \|r\|_\theta \cdot (1+\|\pi_r\|_\theta)\|G\|_{C^m}$.

Since the one-sided function $\bar r$ on $\Sigma$ satisfies $\bar r=r+\varsigma-\varsigma\circ\sigma$, the integral term in the main part of \eqref{equ.lpartlaplace} can be written as
\[
 e^{-cr_n}f_c^{(\ell)}\cdot g_c\circ \sigma^nd{\mu}= e^{-c\bar{r}_n}(e^{c\varsigma}f_c^{(\ell)})\cdot (e^{-c\varsigma}g_c)\circ \sigma^nd{\mu}.
\]

 Denote $t_\pm(\ul{x},\ul{y}):=\min \{|j|:x_j\neq y_j\}$, where $\min \emptyset:=+\infty$.
For any $k\in \bZ^+$, define 
\[
\begin{aligned}
f_{c,k}^{(\ell)}(\ul x)&:=\sup\{(e^{c\varsigma}f_c^{(\ell)})\circ\sigma^k(\ul y):t_\pm(\sigma^k(\ul{x}), \sigma^k(\ul{y}))> k\},\\
 g_{c,k}(\ul x)&:=\sup\{(e^{-c\varsigma}g_c)\circ \sigma^k(\ul y):t_\pm(\sigma^k(\ul{x}), \sigma^k(\ul{y}))> k\}.	
\end{aligned}
\]

Similar to \cite[Lemma~5.5]{Melbourne2007}, we have the following lemma.

\begin{lem}\label{lem.onesideapprox}
	Both  $f_{c,k}^{(\ell)}$ and $g_{c,k}$ are one-sided $\theta$-H\"older continuous functions. For any $c=a+ib$ satisfying $|a|\le 1$ and $|b|\ge 3$, we have
\begin{enumerate}
	\item $\sup|f_{c,k}^{(\ell)}|=\sup|e^{c\varsigma}f_c^{(\ell)}|\le e^{\sup |\varsigma|}\cdot \sup|f_c^{(\ell)}|$ and $\sup|g_{c,k}|=\sup|e^{-c\varsigma}g_c|\le e^{\sup |\varsigma|}\cdot \sup|g_c|$;
\item $\|f_{c,k}^{(\ell)}\|_{\theta}\le 3 e^{\sup |\varsigma|} e^{2k\theta} \cdot \sup|f_c^{(\ell)}| $ and $\|g_{c,k}\|_{\theta}\le 3 e^{\sup |\varsigma|} e^{2k\theta} \cdot \sup|g_c|$;
\item $\sup|f_{c,k}^{(\ell)}-(e^{c\varsigma}f_c^{(\ell)})\circ\sigma^k|\le2|b|(1+\|\varsigma\|_\theta)e^{\sup |\varsigma|}\cdot e^{-k\theta}\|f_c^{(\ell)}\|_\theta$ and $|g_{c,k}-(e^{-c\varsigma}g_c)\circ\sigma^{k}|\le 2|b|(1+\|\varsigma\|_\theta) e^{\sup |\varsigma|}\cdot e^{-k\theta}\|g_c\|_\theta$. 
\end{enumerate}
\end{lem}

\begin{proof}
We only prove the conclusion of $f_{c,k}^{(\ell)}$. The proof for $g_{c,k}$ is the same.
 If $\mathbf{p}^+(\ul{x})=\mathbf{p}^+(\ul{y})$, then $t_\pm(\sigma^k(\ul{x}), \sigma^k(\ul{y}))> k$ for all $k\in \bZ^+$, which implies $f_{c,k}^{(\ell)}(\ul{x})=f_{c,k}^{(\ell)}(\ul{y})$. Thus $f_{c,k}^{(\ell)}$ is a one-sided function. Item~(1) follows from the definition directly.

Note that if $t_\pm(\ul{x},\ul{y})> 2k$ then  $f_{c,k}^{(\ell)}(\ul{x})=f_{c,k}^{(\ell)}(\ul{y})$. Thus 
\[
\frac{|f_{c,k}^{(\ell)}(\ul{x})-f_{c,k}^{(\ell)}(\ul{y})|}{e^{-t_\pm(\ul{x},\ul{y})\cdot \theta}}\le 2\sup|f_{c,k}^{(\ell)}|\cdot  e^{2k \theta}\le 2 e^{2k\theta} e^{\sup |\varsigma|}\cdot \sup|f_c^{(\ell)}|.
\]
So we have $\|f_{c,k}^{(\ell)}\|_{\theta}\le (1+2 e^{2k\theta}) e^{\sup |\varsigma|}\cdot \sup|f_c^{(\ell)}|\le 3 e^{\sup |\varsigma|} e^{2k\theta} \cdot \sup|f_c^{(\ell)}|$. This proves Item~(2).

For any $\ul{x},\ul{y}\in \Sigma$ satisfying $t_\pm(\sigma^k(\ul{x}),\sigma^k(\ul{y}))> k$, we have
\[
\begin{aligned}
 &\, \big|(e^{c\varsigma}f_c^{(\ell)})\circ\sigma^k(\ul x)-(e^{c\varsigma}f_c^{(\ell)})\circ\sigma^k(\ul y)\big|\\
  \le &\, \big| e^{c\varsigma(\sigma^k(\ul{x}))} \big|\cdot \big| f_c^{(\ell)}(\sigma^k(\ul{x}))-f_c^{(\ell)}(\sigma^k(\ul{y}))\big| + \big| e^{c\varsigma(\sigma^k(\ul{x}))}-e^{c\varsigma(\sigma^k(\ul{y}))} \big|\cdot \big| f_c^{(\ell)}(\sigma^k(\ul{y}))\big|\\
  \le &\, e^{\sup |\varsigma|}\cdot\|f_c^{(\ell)}\|_\theta\cdot e^{-k\theta}+ e^{\sup |\varsigma|}\cdot |c|\cdot\|\varsigma\|_\theta\cdot e^{-k\theta}\cdot \sup |f_c^{(\ell)}|\\
  \le &\, 2|b|(1+\|\varsigma\|_\theta)e^{\sup |\varsigma|}\cdot e^{-k\theta}\|f_c^{(\ell)}\|_\theta.\quad\quad\quad\quad\quad\quad  (\text{since } 1<|c|<2|b|)
\end{aligned}
\]
Then Item~(3) follows from the fact 
\[
|f_{c,k}^{(\ell)}(\ul{x})-(e^{c\varsigma}f_c^{(\ell)})\circ\sigma^k(\ul{x})|\le \sup\{|(e^{c\varsigma}f_c^{(\ell)})\circ\sigma^k(\ul x)-(e^{c\varsigma}f_c^{(\ell)})\circ\sigma^k(\ul y)|: t_\pm(\sigma^k(\ul{x}), \sigma^k(\ul{y}))> k \}.
\]
\end{proof}

The integral term in the main part of \eqref{equ.lpartlaplace} can be further written as
\begin{equation}\label{eq.dcnk}\begin{aligned}
\int_{\Sigma} e^{-c\bar{r}_n}(e^{c\varsigma}f_c^{(\ell)})\cdot (e^{-c\varsigma}g_c)\circ \sigma^n d\mu
&=\int_{\Sigma} e^{-c\bar{r}_n\circ \sigma^k}(e^{c\varsigma}f_c^{(\ell)})\circ \sigma^k\cdot (e^{-c\varsigma}g_c)\circ \sigma^{n+k}d\mu\\
&=I_1(c,n,k)+I_2(c,n,k)+I_3(c,n,k),
\end{aligned}\end{equation}
where 
\[\begin{aligned}
& I_1(c,n,k):=\int_{\Sigma}e^{-c\bar{r}_n\circ \sigma^k}(e^{c\varsigma}f_{c}^{(\ell)})\circ \sigma^k \cdot \big((e^{-c\varsigma}g_c)\circ \sigma^k-g_{c,k}\big)\circ \sigma^nd\mu, \\
& I_2(c,n,k):=\int_{\Sigma}e^{-c\bar{r}_n\circ \sigma^k}\big((e^{c\varsigma}f_c^{(\ell)})\circ\sigma^k-f_{c,k}^{(\ell)}\big)\cdot (g_{c,k}\circ \sigma^{n})d\mu, \\
& I_3(c,n,k):=\int_{\Sigma}e^{-c\bar{r}_n\circ \sigma^k}f_{c,k}^{(\ell)}\cdot g_{c,k}\circ\sigma^nd\mu.
\end{aligned}\]

For any $c=a+ib$ satisfying $|b|\ge 3$ and $|a|\le1$, it follows from Lemma~\ref{lem.onesideapprox} and above estimates that
\[
\begin{aligned}
	|I_1(c,n,k)|,\ |I_2(c,n,k)| &\le e^{n|a|\sup \bar{r}}\cdot e^{\sup |\varsigma|} \cdot 2|b|(1+\norm{\varsigma}_\theta) e^{\sup |\varsigma|}e^{-k\theta}\cdot \|f_c^{(\ell)}\|_\theta \|g_c\|_\theta\\
	& \le \wt C_1 e^{n|a|\sup \bar{r}-k\theta}|b|\cdot \norm{F}_{C^m} \norm{G}_{C^m}
\end{aligned}
\] 
for some constant $\wt C_1$ which only depends on the $
\theta$-norms of $\pi_r$, $r$, $\bar r$ and $\varsigma$.

On the other hand, since $f_{c,k}^{(\ell)}$ and $g_{c,k}$ are one-sided functions on $\Sigma$, there exist $f_{c,k}^{(\ell),+}, g_{c,k}^+:\Sigma^+\to \bR$ such that $f_{c,k}^{(\ell)}=f_{c,k}^{(\ell),+}\circ \mathbf{p}^+$ and $g_{c,k}=g_{c,k}^+ \circ \mathbf{p}^+$. Moreover, $\|f_{c,k}^{(\ell),+}\|_\theta=\|f_{c,k}^{(\ell)}\|_\theta$ and $\|g_{c,k}^+\|_\theta=\|g_{c,k}\|_\theta$. 

Then $I_3(c,n,k)$ can be projected to $\Sigma^+$ and written as
\[
I_3(c,n,k)=\int_{\Sigma^+}e^{-c r^+_n}(e^{c r^+_k}f_{c,k}^{(\ell),+})\cdot(e^{-c r^+_k} g^+_{c,k})\circ \sigma^nd{\mu^+}=\int_{\Sigma^+}\mathcal{L}_{\bar{\phi},-c}^n (e^{c r^+_k}f_{c,k}^{(\ell),+})\cdot(e^{-c r^+_k} g^+_{c,k})d{\mu^+},
\]
where $\bar \phi$ is given by Proposition~\ref{pro.BCS1to1}.
By Proposition~\ref{pro.Idabinv}, for any $c=a+ib$ satisfying $|b|\ge \wt{b}\ge 3$ and $|a|\leq |b|^{-\wt{\lambda}_0}\le 1$, we have $\big\|\mathcal{L}_{\bar{\phi},-c}^n\big\|_{\mathscr{L}}\le \wt C_0|b|^{\wt \lambda_0}e^{-n|b|^{-\wt\lambda_0}}$. It follows from Lemma~\ref{lem.propUp12norm}~\ref{ite.Pfisint} and \ref{ite.nonemptypres} that
\[
\bigg|\int_{\Sigma^+}\mathcal{L}_{\bar\phi,-c}^n (e^{c r^+_k}f_{c,k}^{(\ell),+})d{\mu^+}\bigg|\le  \frac{C_2 \wt C_0|b|^{\wt \lambda_0}e^{-n|b|^{-\wt\lambda_0}} \| e^{c r^+_k}f_{c,k}^{(\ell),+}\|_{\mathscr{L}}}{\|\mathbbm 1\|_{\mathscr L}}\le C_2 \wt C_0|b|^{\wt \lambda_0}e^{-n|b|^{-\wt\lambda_0}} \| e^{c r^+_k}f_{c,k}^{(\ell),+}\|_\theta.
\]
By standard (but somewhat lengthy) calculations, we may obtain that $|e^{-c r^+_k} g^+_{c,k}|\le \wt C_2e^{k\sup \bar r}\sup |g_c|$ and $ \| e^{c r^+_k}f_{c,k}^{(\ell),+}\|_{\theta}\le \wt C_2e^{2k\theta}e^{k\sup \bar r}|b|\sup |f_c^{(\ell)}|$ for some constant $\wt C_2$ which only depends on the $
\theta$-norms of $\pi_r$, $\bar r$ and $\varsigma$. Thus
\[
|I_3(c,n,k)|\le \wt C_3 |b|^{\wt\lambda_0+1}\cdot \exp\Big({2(\theta+ \sup \bar r)k-n|b|^{-\wt\lambda_0}}\Big)\cdot\norm{F}_{C^m} \norm{G}_{C^m},
\]
for some constant $\wt C_3$ which only depends on the $
\theta$-norms of $\pi_r$, $r$, $\bar r$ and $\varsigma$.

For any $n\in \bN_0$, take $k:=\big\lfloor \frac{n|b|^{-\wt \lambda_0}}{4(\theta+\sup \bar r)}\big\rfloor +1$. It follows that
\[
|I_3(c,n,k)|\le \wt C_3 |b|^{\wt\lambda_0+1}\cdot \exp\big(2(\theta+ \sup \bar r)\big) \exp\big(-\frac{1}{2}n|b|^{-\wt\lambda_0}\big)\cdot\norm{F}_{C^m} \norm{G}_{C^m}.
\]
Moreover, denote $\epsilon:=\min\big\{1,\frac{\theta}{8\sup \bar r(\theta+\sup \bar r)}\big\}$. For any $|a|\le \epsilon|b|^{-\wt \lambda_0}$,  we have
\[
|I_1(c,n,k)|,\ |I_2(c,n,k)|  \le \wt C_1 \exp\big(-\frac{\theta}{8(\theta+\sup \bar r)}n|b|^{-\wt \lambda_0}\big)\cdot|b|\cdot \norm{F}_{C^m} \norm{G}_{C^m}.
\]

Let $\delta':=\frac{\theta}{8(\theta+\sup \bar r)}$. Combining the above estimates, there exists a constant $\wt C_4$ which only depends on the $
\theta$-norms of $\pi_r$, $r$, $\bar r$ and $\varsigma$ such that for any $c=a+ib$ satisfying $|b|\ge \wt b$ and $|a|\le \epsilon|b|^{-\wt \lambda_0}$,
\[
\bigg|\int_{\Sigma} e^{-cr_n}f_c^{(\ell)}\cdot g_c\circ \sigma^nd{\mu}\bigg|\le \wt C_4 |b|^{\wt\lambda_0+1}\exp\big(-\delta' n|b|^{-\wt\lambda_0}\big)\cdot\norm{F}_{C^m} \norm{G}_{C^m}.
\]

Summing over $n$ and using the fact $\frac{1}{1-e^{-x}}\le \frac{2}{x}$ for all $0<x\le \frac{1}{4}$, the main part of $ \widehat{{\rho}}_{f,g}$ in 
\eqref{equ.lpartlaplace} has the following estimate:
\[
  \bigg|\frac{(-1)^\ell}{c^\ell\int_{\Sigma} r d\mu}
\sum_{n=0}^{\infty}\int_{\Sigma} e^{-cr_n}f^{(\ell)}_c\cdot g_c\circ \sigma^nd{\mu}\bigg|\le  \wt C_5 |b|^{2\wt\lambda_0+1-\ell}\norm{F}_{C^m} \norm{G}_{C^m}
\] 
for some constant $\wt C_5$ which only depends on the $
\theta$-norms of $\pi_r$, $r$, $\bar r$ and $\varsigma$. For any $A>0$, take $m\in \bZ^+$ such that $m>2\wt\lambda_0+3+A$, and take $\ell=m-1$. By a standard inverse Laplace transform estimate (see \cite[\S\,10]{Dol98}), we have $|\rho_{f,g}(t)|\le C \norm{F}_{C^m} \norm{G}_{C^m} |t|^{-A}$ for some constant $C$. The proof is completed.
\end{proof}

%%%%%%%%%%%%%%%%%%%%%%%%%%%%%%%%%%%%%%%%%%%%%%%%%%%%%%%%%%%%%%%%%%%
%%%%%%%%%%%%%%%%%%%%%%%%%%%%%%%%%%%%%%%%%%%%%%%%%%%%%%%%%%%%%%%%%%%

\part{Applications}\label{part.app}

%%%%%%%%%%%%%%%%%%%%%%%%%%%%%%%%%%%%%%%%%%%%%%%%%%%%%%%%%%%%%%%%%%%

\section{Three-dimensional flows}\label{sec.3flow}

In this section, we consider $C^\infty$-smooth flows on a $3$-dimensional closed manifold $M^3$.

\subsection{SPR property for three-dimensional flows}

Let $K\subset M$. Denote $\Var(\vartheta|_K)=\sup_{x\in K}\vartheta -\inf_{x\in K}\vartheta $. The main result of this subsection is the following theorem. Theorem~\ref{main.3flowmme} follows immediately from this result by taking $K=M$.

\begin{thm}\label{thm.3SPRset}
Let $X$ be a $C^\infty$-smooth vector field on  $M^3$ and  $\vartheta :M^3\to\mathbb{R}$ be a H\"older continuous function. If $K$ is an invariant Borel set  satisfying:
\begin{itemize}
\item $\Var(\vartheta|_{\ol K})<h_{\rm top}(X|_{\ol K})$ (which implies that $h_{\rm top}(X|_{\ol K})>0$); and
\item $P_{\rm TOP}(X|_K,\vartheta)=P_{\rm top}(X|_{\overline{K}},\vartheta)$, 
\end{itemize}
  then $X$ is SPR for $\vartheta$ on $K$.
\end{thm}

From now on, we fix a H\"older continuous function $\vartheta :M^3\to\mathbb{R}$ and  an invariant Borel set $K$ satisfying the assumptions of Theorem~\ref{thm.3SPRset}, that is,  $\Var(\vartheta|_{\ol K})<h_{\rm top}(X|_{\ol K})$  and  $P_{\rm TOP}(X|_K,\vartheta)=P_{\rm top}(X|_{\overline{K}},\vartheta)$. 
We begin with several elementary facts. The notation $Y_n\xrightarrow{C^\alpha} X$  means that $Y_n$ converges to $X$ in the $C^\alpha$ topology, and $\mu_n \rightharpoonup \mu$ denotes the weak-$*$ convergence.
Denote $\chi_0:=\frac{1}{2}(h_{\rm top}(X|_{\ol K})-\Var(\vartheta|_{\ol K}))>0$.  

\begin{lem}\label{lem.regularhyper} We have the following properties.
  \begin{enumerate}
  \item For any $C^1$ vector field $Y$ and any ergodic $Y$-invariant measure $\nu$ with $\nu({\ol K})=1$, if $P_{\nu}(Y,\vartheta)> P_{\rm top}(X|_{\ol K},\vartheta)-\chi_0$, then $\nu$ is a regular saddle-type $\chi_0$-hyperbolic measure for $Y$. 
  \item If there is a sequence of $C^\infty$ vector fields $Y_n$ with $\nu_n\in \mathbb P(Y_n)$ such that $Y_n\xrightarrow{C^\infty} X$, $\nu_n \rightharpoonup  \mu\in\mathbb P(X|_{\ol K})$ and $\limsup_{n\to\infty} P_{\nu_n}(Y_n,\vartheta) \ge  P_{\mathrm{TOP}}(X|_K,\vartheta)$, then $\mu$ is a regular saddle-type $\chi_0$-hyperbolic $X$-invariant measure. 
  \item For any $\xi>0$, there exists a $C^\infty$ neighborhood $\U$ of $X$, real numbers $r>0$ and $P<P_{\rm TOP}(X|_K,\vartheta)$ such that, if $\nu$ is a $Y$-invariant measure for some $Y\in \U$ satisfying $\nu({\ol K})=1$ and $P_\nu(Y,\vartheta)>P$, then $\nu(M(\Sing(Y),r))< \xi$. 
  \end{enumerate}
\end{lem}
\begin{proof}
By the variational principle for flows, we have
\[
P_{\rm top}(X|_{\overline{K}},\vartheta)=\sup_{\nu \in \mathbb P(X|_{\ol K})} ( h_\nu(X)+\int \vartheta d\nu )\ge \sup_{\nu \in \mathbb P(X|_{\ol K})}  h_\nu(X)+\inf_{x\in {\ol K}} \vartheta(x)=h_{\rm top}(X|_{\ol K})+ \inf_{x\in {\ol K}} \vartheta(x).
\]	
Let $\nu\in \Pe(Y)$ satisfy $\nu({\ol K})=1$ and $P_{\nu}(Y,\vartheta)> P_{\rm top}(X|_{\ol K},\vartheta)-\chi_0$. It then follows that
\[
\begin{aligned}
h_{\nu}(Y)\ge P_{\nu}(Y,\vartheta)-\sup_{x\in {\ol K}} \vartheta(x) >\, & P_{\rm top}(X|_{\ol K},\vartheta)-\chi_0-\sup_{x\in {\ol K}} \vartheta(x)\\ 
\ge\, & h_{\rm top}(X|_{\ol K})+ \inf_{x\in {\ol K}} \vartheta(x)-\chi_0-\sup_{x\in {\ol K}} \vartheta(x)=\chi_0>0.
\end{aligned}
\]

By Ruelle's inequality, $\nu$ is a saddle-type $\chi_0$-hyperbolic measure for $Y$. Moreover, since any ergodic measure with positive metric entropy cannot be supported on a single singularity, $\nu$ is regular for $Y$. This completes the proof of Item~(1).

We now turn to Item~(2). Assume that there exists $\nu_n\in \mathbb P(Y_n)$ satisfying $Y_n\xrightarrow{C^\infty}  X$,  $\nu_n \rightharpoonup   \mu\in\mathbb P(X|_{\ol K})$ and $\limsup_{n\to\infty} P_{\nu_n}(Y_n,\vartheta) \ge  P_{\mathrm{TOP}}(X|_K,\vartheta)$. By the upper-semicontinuity theorem of Newhouse \cite{New}, we obtain $\limsup_{n\to \infty}h_{\nu_n}(Y_n)\le h_\mu(X)$, which implies $\limsup_{n\to \infty}P_{\nu_n}(Y_n,\vartheta)\le P_\mu(X,\vartheta)$. Thus $P_\mu(X,\vartheta)\ge P_{\mathrm{TOP}}(X|_K,\vartheta)$. Combining with $P_\mu(X,\vartheta)\le P_{\mathrm{top}}(X|_{\ol K},\vartheta)=P_{\mathrm{TOP}}(X|_K,\vartheta)$, we obtain that $\mu$ is an equilibrium state of $\vartheta$ for $X$ on $\ol K$. In particular, almost every ergodic component of $\mu$ is also an equilibrium state. By Item~(1), every such ergodic component is a saddle-type $\chi_0$-hyperbolic measure for $X$, so $\mu$ itself is saddle-type $\chi_0$-hyperbolic.

 It remains to show that $\mu$ is regular. Suppose for contradiction that $a:=\mu(\Sing(X)\cap {\ol K})>0$. Then $\mu$ admits a decomposition ${\mu}=a{\mu}_{0}+(1-a){\mu}_{1}$, where $\mu_0\in \mathbb P(X|_{\Sing(X)\cap {\ol K}})$ and $ \mu_1\in\mathbb P(X|_{\ol K})$ with $\mu_1(\Sing(X))=0$. Since 
 \[
 P_{\mu_0}(X)\le h_{\mu_0}(X)+\sup_{x\in {\ol K}} \vartheta(x)=\sup_{x\in {\ol K}} \vartheta(x), \text{ and}
 \]
\begin{equation}\label{equ.pressureandchi0}
P_{\mathrm{TOP}}(X|_K,\vartheta)=P_{\rm top}(X|_{\overline{K}},\vartheta)\ge h_{\rm top}(X|_{\ol K})+ \inf_{x\in {\ol K}} \vartheta(x)=\sup_{x\in {\ol K}} \vartheta(x)+2\chi_0,
\end{equation}

we have
\[
P_{\mu_0}(X)\le P_{\mathrm{TOP}}(X|_K,\vartheta) - 2\chi_0,
\text{ and}
\]
\[
P_{\mu}(X)=aP_{\mu_0}(X)+(1-a)P_{\mu_1}(X)\le a(P_{\mathrm{TOP}}(X|_K,\vartheta)-2\chi_0)+(1-a)P_{\mathrm{top}}(X|_{\ol K},\vartheta)< P_{\mathrm{TOP}}(X|_K,\vartheta).
\]
 
This contradiction proves that $\mu$ is regular. The proof of Item~(2) is completed.

For item (3), the statement for the special case $Y=X$ has been established in \cite[Lemma~5.4]{PYY25a}, and the general case follows by a similar argument.  
Fix any $\xi>0$. We first claim that there exist $\U$, $r>0$ and $P<P_{\rm TOP}(X|_K,\vartheta)$ such that $\nu(M(\Sing(X),2r))< \xi$ for any $Y\in \U$, $\nu\in \mathbb{P}(Y)$ with $\nu({\ol K})=1$ and $P_\nu(Y,\vartheta)>P$.

Suppose for contradiction that the claim fails. Then there exists $Y_n\xrightarrow{C^\infty} X$ with $\nu_n\in\mathbb{P}(Y_n)$ such that $\nu_n({\ol K})=1$, $P_{\nu_n}(Y_n,\vartheta)>P_{\rm TOP}(X|_K,\vartheta)-1/n$ and $\nu_n(M(\Sing(X),1/n))\ge\xi$. Taking a subsequence if necessary, we assume that $\nu_n \rightharpoonup \mu$. Then $\mu$ is an $X$-invariant measure with $\mu({\ol K})=1$. By item (2), $\mu$ is regular. However, $\mu(M(\Sing(X),1/m))\ge \limsup_{n\to \infty}\nu_n(M(\Sing(X),1/m))\ge  \xi$ implies that $\mu(\Sing(X))\ge  \xi$. This contradiction proves the claim.

Finally, note that $\Sing(Y)\subset M(\Sing(X),r)$ whenever $Y$ is sufficiently $C^0$-close to $X$. By shrinking $\U$ if necessary, we may ensure that $M(\Sing(Y),r)\subset M(\Sing(X),2r)$ for all $Y\in \U$. This immediately yields the desired estimate in Item~(3).	
\end{proof}

We remark that when proving the SPR property of $X$, we do not need to change the vector field. However, when we prove in the next subsection that open and dense vector fields admit good asymptotics, it is necessary to consider the perturbed system. To this end, most of the results in this subsection will be stated in a more general form that considers both $X$ and nearby vector fields.

Let $Y$ be a $C^1$ vector field on $M$. Denote by $\varphi^Y_t$, $\psi^Y_t$ and $\psi^{Y,*}_t$ the flow, linear Poincar\'e flow and scaled linear Poincar\'e flow generated by $Y$ respectively. Denote by $\Pes^Y_\chi(\ell,\varepsilon)$ the Pesin block for $Y$. For simplicity, we generally omit $X$ in the corresponding notations when $Y=X$.

Define
\[
G^1(M^3):=\{(x,E): x\in M^3 \text{ and $E$ is a $1$-dimensional linear subspace of $T_xM^3$}\}.
\]
Then $G^1(M^3)$ forms a smooth fibre bundle over $M^3$, which is called the {\it $1$-dimensional Grassmannian bundle} of $M$. Denote by $\widetilde{\pi}$ the natural projection from $G^1(M^3)$ to $M^3$, that is, $\widetilde{\pi}(x,E)=x$.

The flow $\varphi^Y_t$ and its tangent flow $d\varphi^Y_t$ induce a continuous flow $\widetilde{\varphi}^Y_t$ on $G^1(M^3)$ as
\[
\widetilde{\varphi}^Y_t(x,E)=(\varphi^Y_t(x),d\varphi^Y_t(E)).
\]

For any $\wt \varphi^Y$-invariant probability measure $\wt \mu$ on $G^1(M^3)$, the probability measure $\mu:=\wt \mu\circ \wt\pi^{-1}$ is $Y$-invariant, which is called the {\it projection} of $\wt\mu$.

Assume that $\mu$ is a regular saddle-type $Y$-invariant hyperbolic measure. For $\mu$-a.e. $x \in M^3$, let $T_xM^3 = E^{Y,-}_x \oplus E^{Y,0}_x \oplus E^{Y,+}_x$ be the Oseledets splitting with respect to the tangent flow $d\varphi^Y$, where $E^{Y,0}_x=\langle Y(x) \rangle$ corresponds to the zero Lyapunov exponent, $E^{Y,-}_x$ and $E^{Y,+}_x$ are the $1$-dimensional subspaces corresponding to the negative and positive Lyapunov exponents respectively. Define
\[
\widetilde{\mu}^{Y,+}:=\int_{G^1(M^3)} \delta_{(x,E^{Y,+}_x)}d\mu(x) \quad \text{and}\quad \widetilde{\mu}^{Y,-}:=\int_{G^1(M^3)} \delta_{(x,E^{Y,-}_x)}d\mu(x).
\]
These two $\wt\varphi^Y$-invariant measures $\widetilde{\mu}^{Y,+}$ and $\widetilde{\mu}^{Y,-}$ are called the {\it unstable lift} and {\it stable lift} of $\mu$ for $Y$, respectively. They satisfy $\widetilde{\pi}_{*}(\widetilde{\mu}^{Y,\pm}) = \mu$. Moreover, if $\mu$ is ergodic for $Y$, then both $\widetilde{\mu}^{Y,+}$ and $\widetilde{\mu}^{Y,-}$ are ergodic for $\wt \varphi^Y$.  Define
\[
\begin{aligned}
\widetilde{\mathscr{L}}_K^{+}:= \{\, \widetilde{\mu} : & \text{ there exists a sequence of $C^\infty$ vector fields } Y_n \text{ with measures } \nu_n \in \Pe(Y_n) \text{ such that } \\ & \ Y_n\xrightarrow{C^\infty} X,\ \nu_n(K)=1,\ \lim_{n\to \infty}P_{\nu_n}(Y_n,\vartheta) \ge  P_{\mathrm{TOP}}(X|_K,\vartheta) \text{ and }\widetilde{\nu}_n^{Y_n,+} \rightharpoonup \widetilde{\mu} \text{ on } G^1(M^3)\}.
\end{aligned}
\]

Without loss of generality, we may assume that  $P_{\nu_n}(Y_n,\vartheta)>P_{\mathrm{TOP}}(X|_K,\vartheta)-\chi_0=P_{\mathrm{top}}(X|_{\ol K},\vartheta)-\chi_0$ for all $n$ in the definition of $\widetilde{\mathscr{L}}_K^{+}$. Then by Lemma~\ref{lem.regularhyper} (1), $\nu_n$ is a regular saddle-type hyperbolic measure for $Y_n$. Thus the unstable lift $\widetilde{\nu}_n^{Y_n,+}$ is well defined.

\begin{pro}\label{pro.continu}
The set $\widetilde{\mathscr{L}}^{+}_K$ is weak-$*$ compact in the set of all probability measures on $G^1(M^3)$. 
For any $\widetilde{\mu}\in\widetilde{\mathscr{L}}_K^{+}$,   the projection $\mu=\widetilde{\pi}_{*}(\widetilde{\mu})$  is an $X$-invariant regular saddle-type $\chi_0$-hyperbolic measure satisfying $\widetilde{\mu}^{X,+}=\widetilde{\mu}$.
\end{pro}

To prove the above proposition,  we require the following result. Similar results were first established in \cite{BCS22b} for surface diffeomorphisms. An improved version for surface diffeomorphisms was subsequently proved in \cite{Bur}. Following the same method, \cite[Theorem~3.3]{Zan} extended these results to higher-dimensional diffeomorphisms with exactly one positive Lyapunov exponent for the case $f_n=f$.
 The same proof yields the following proposition. Note that the results in \cite{BCS22b,Bur,Zan} are stated for  general $C^r$ ($r>1$) diffeomorphisms.

\begin{pro}\label{thm.zang}
Let $f,f_n$ be $C^\infty$ diffeomorphisms on $M$ such that $f_n\to f$ in the $C^\infty$ topology. Assume that $\nu_{n}$ is an  $f_n$-ergodic measure having exactly one positive Lyapunov exponent. Denote by $\widetilde{\nu}_n^{f_n,+}:=\int_{G^1(M)} \delta_{(x,E^{f_n,+}_x)}d\nu_n(x)$ the 
unstable lift of $\nu_n$ for $f_n$, where $E^{f_n,+}_x$ is the $1$-dimensional subspace with respect to the unique positive Lyapunov exponent.

If the unstable lifts $\wt{\nu}_{n}^{f_n,+}$ converge in the  weak-$*$ topology to some $\wt f$-invariant measure $\wt{\mu}$ whose projection $\mu$ also has exactly one positive Lyapunov exponent, then for any $h>0$ there exists an $f$-invariant measure $\mu_1$ with unstable lift $\wt{\mu}_1^{f,+}$, a $\wt f$-invariant measure $\wt\mu_0$, and  $a\in[0,1]$ such that  $\wt\mu$ admits the decomposition $\wt{\mu}=a\wt{\mu}_{1}^{f,+}+(1-a)\wt{\mu}_{0}$ satisfying
\[
\limsup_{n\to\infty}h_{\nu_{n}}(f_n)\leq a h_{\mu_{1}}(f)+(1-a)h.
\]
\end{pro}

\begin{proof}[Proof of Proposition~\ref{pro.continu}]
Since $\widetilde{\mathscr{L}}^{+}_K$ consists of limit points, it is compact.	Let $\widetilde{\mu}\in\widetilde{\mathscr{L}}_K^{+}$ and $\nu_n\in \Pe(Y_n)$ be the regular saddle-type hyperbolic ergodic measure given by the definition of $\widetilde{\mathscr{L}}_K^{+}$. Since $Y_n\xrightarrow{C^\infty} X$ and $\widetilde{\nu}_n^{Y_n,+} \rightharpoonup \widetilde{\mu}$, we have ${\nu}_n \rightharpoonup {\mu}:=\wt\mu\circ{\wt\pi}^{-1}\in \mathbb{P}(X)$. Moreover, it follows from $\nu_n(K)=1$ that $\mu({\ol K})=1$.
Thus by Lemma~\ref{lem.regularhyper}~(2), $\mu$ is a regular saddle-type $\chi_0$-hyperbolic measure for $X$. 

It remains to prove $\wt \mu^{X,+}=\wt \mu$.
Without loss of generality, we assume that $P_{\nu_n}(X,\vartheta)> P_{\rm TOP}(X|_{K},\vartheta)-\chi_0=P_{\rm top}(X|_{\ol K},\vartheta)-\chi_0$ for all $n$. Then by Lemma~\ref{lem.regularhyper}~(1), $\nu_n$ is a regular saddle-type $\chi_0$-hyperbolic measure. Note that  $\nu_{n}$ have exactly one positive Lyapunov exponent for $d\varphi^{Y_n}$, and  $\mu$  have exactly one positive Lyapunov exponent for $d\varphi:=d\varphi^X$.
	
Consider the time-$T$ map $\varphi^{Y_n}_{_T}$ ($T>0$) of the flow generated by $Y_n$. The ergodic $Y_n$-invariant measure $\nu_n$ is $\varphi^{Y_n}_{_T}$-invariant for all $T>0$. Moreover, $\nu_n$ is $\varphi^{Y_n}_{_T}$-ergodic for all but at most countably many $T>0$. So there exists $T>0$ such that $\nu_n$ is $\varphi^{Y_n}_{_T}$-ergodic for all $n$. Fix this $T$.

The time-$T$ maps $\varphi^{Y_n}_{_T}$ and $\varphi_{_T}(=\varphi^X_{_T})$ are $C^\infty$ diffeomorphisms on $M^3$ satisfying $\varphi^{Y_n}_{_T}\xrightarrow{C^\infty} \varphi_{_T}$. At $\nu_n$-a.e. (resp. $\mu$-a.e.) $x\in M^3$, the Lyapunov exponents of diffeomorphism $\varphi^{Y_n}_{_T}$ (resp. $\varphi_{_T}$) are precisely $T$ times the corresponding exponents of $d\varphi^{Y_n}$ (resp. $d\varphi$), and the corresponding Oseledets splittings coincide. Consequently,  $\nu_{n}$ and  $\mu$ have exactly one positive Lyapunov exponent of $\varphi^{Y_n}_{_T}$ and $\varphi_{_T}$, respectively. Furthermore, $\widetilde{\nu}_n^{\varphi^{Y_n}_{_T},+}=\widetilde{\nu}_n^{Y_n,+}\rightharpoonup {\wt\mu}$, and $\widetilde{\mu}^{\varphi_{_T},+}=\widetilde{\mu}^{X,+}$.
	
Take $h=\chi_0 T$. By Theorem~\ref{thm.zang}, there exists a $\varphi_{_T}$-invariant measure $\mu_1$, a $\wt \varphi_{_T}$-invariant measure $\wt\mu_0$, and  $a\in[0,1]$ such that  $\wt{\mu}=a\wt{\mu}_{1}^{\varphi_{_T},+}+(1-a)\wt{\mu}_{0}$ and 
\[
\limsup_{n\to\infty}h_{\nu_{n}}(\varphi^{Y_n}_{_T})\leq a h_{\mu_{1}}(\varphi_{_T})+(1-a)\chi_0 T.
\]
It implies that
\begin{equation}\label{equ.pressureinequ}
\limsup_{n\to \infty} P_{\nu_n}(Y_n,\vartheta)=\limsup_{n\to \infty} \Big(\frac{1}{T}h_{\nu_n}(\varphi^{Y_n}_{_T})+\int \vartheta d\nu_n\Big)\le a\cdot \frac{1}{T}h_{\mu_{1}}(\varphi_{_T})+(1-a)\chi_0 +\int \vartheta d\mu. 
\end{equation}

 Applying the pushforward $\wt\pi_*$ to the formula $\wt{\mu}=a\wt{\mu}_{1}^{+}+(1-a)\wt{\mu}_{0}$ yields ${\mu}=a{\mu}_{1}+(1-a){\mu}_{0}$, where $\mu_0:=\wt\mu_0\circ{\wt\pi}^{-1}$. Define 
\[
 m_1:=\frac{1}{T}\int_0^T \mu_1\circ \varphi_{-t}\, dt  \quad \text{and} \quad  m_0:=\frac{1}{T}\int_0^T \mu_0\circ \varphi_{-t}\,  dt.
 \]
 Then $m_0$ and $m_1$ are $X$-invariant measures. Since $\mu$ is $X$-invariant, we have
 \[
 \mu=\frac{1}{T}\int_0^T \mu\circ \varphi_{-t}\,  dt=a m_1+(1-a)m_0.
 \]
 Moreover, since $h_{\mu_1\circ \varphi_{-t}}(\varphi_{_T})=h_{\mu_1}(\varphi_{_T})$ for all $t$, we obtain 
 \[
 h_{m_1}(\varphi_{_T})=\frac{1}{T}\int_0^T h_{\mu_1\circ \varphi_{-t}}(\varphi_{_T})dt=h_{\mu_1}(\varphi_{_T}).
 \]
 
It follows from \eqref{equ.pressureinequ} that
 \begin{equation}\label{equ.pressuredecomp}
 \begin{aligned}
\limsup_{n\to \infty} P_{\nu_n}(Y_n,\vartheta) & \le  a\cdot \frac{1}{T}h_{m_1}(\varphi_{_T})+(1-a)\chi_0 +\int \vartheta d\mu\\ & = a( h_{m_{1}}(\varphi_1)+\int\vartheta dm_1) +(1-a)(\chi_0 +\int \vartheta dm_0). 
\end{aligned}
\end{equation}

If $a=0$, then by \eqref{equ.pressureandchi0} and \eqref{equ.pressuredecomp},
\[
P_{\mathrm{TOP}}(X|_K,\vartheta)\le\limsup_{n\to \infty} P_{\nu_n}(Y_n,\vartheta)\le \chi_0+\int \vartheta d\mu\le \chi_0+\sup_{x\in \ol K} \vartheta \le P_{\mathrm{TOP}}(X|_K,\vartheta)-\chi_0.
\]
This contradiction proves  $a>0$. Thus $m_1\in\mathbb P(X|_{\ol K})$  follows from $\mu\in\mathbb P(X|_{\ol K})$. 
So we have $h_{m_{1}}(\varphi_1)+\int\vartheta dm_1\le P_{\mathrm{top}}(X|_{\ol K},\vartheta)=P_{\mathrm{TOP}}(X|_K,\vartheta)$. 

If $a<1$, then $m_0({\ol K})=1$. Therefore by \eqref{equ.pressureandchi0} and \eqref{equ.pressuredecomp} again, we have
\[
P_{\mathrm{TOP}}(X|_K,\vartheta)\le\limsup_{n\to \infty} P_{\nu_n}(Y_n,\vartheta)\le aP_{\mathrm{TOP}}(X|_K,\vartheta)+(1-a)(P_{\mathrm{TOP}}(X|_K,\vartheta)-\chi_0).
\]
This contradiction implies that $a=1$. Thus $\mu=\mu_1$ and $\wt\mu=\wt{\mu}_{1}^{\varphi_{_T},+}={\wt{\mu}}^{\varphi_{_T},+}=\wt\mu^{X,+}$. 
\end{proof}

Next we consider the {\it scaled tangent flow} $(d\varphi^Y_t)^*: TM\to TM$, which is defined as
\[
(d\varphi^Y_t)^*(v):=\frac{d_x\varphi^Y_t(|Y(x)|v)}{|Y(\varphi^Y_t(x))|}=\frac{|Y(x)|}{|Y(\varphi^Y_t(x))|}d_x\varphi^Y_t(v)=\frac{d_x\varphi^Y_t(v)}{\|d_x\varphi^Y_t|_{\langle Y(x)\rangle }\|},
\]
where $x\in M\setminus \Sing(Y)$ and $v\in T_x M$.

For a positive number $\chi$, two positive integers $k,\Theta$,  and a $C^1$ vector field $Y$, define
\[
\widetilde{U}^{Y,+}_{k,\Theta}(\chi):=\{(x,E)\in G^1(M^3): x\in M\setminus\Sing(Y),\ \|(d_x\varphi^Y_{{-k}})^*|_{_{E}}\|<e^{-k\chi },\ \|(d_x\varphi^Y_{{-\Theta k}})^*|_{_{E}}\|<e^{-\Theta k\chi } \}.
\]
Note that $\widetilde{U}^{Y,+}_{k,\Theta}(\chi)$ is open in $G^1(M^3)$. Furthermore, $\widetilde{U}^{Y,+}_{k,\Theta}(\chi)$ has the following properties.

\begin{lem}\label{lem.UforYandX}
For any $k_0, \Theta\in \bZ^+$, $0<\chi_2<\chi_1$  and $r>0$, there exists a $C^1$ neighborhood $\U_1$ of $X$ such that
\[
\wt{U}^{X,+}_{k,\Theta}(\chi_1)\setminus \wt\pi^{-1}(M(\Sing(X),r))\subset \wt{U}^{Y,+}_{k,\Theta}(\chi_2)\setminus \wt\pi^{-1}(M(\Sing(Y),r/2))
\]
for any $Y\in\U_1$ and $k\in\{1,2,\ldots, k_0\}$.	
\end{lem}
\begin{proof}
	We prove this lemma by contradiction. Suppose that there exist numbers $k_0, \Theta\in \bZ^+$, $0<\chi_2<\chi_1$, $r>0$, and $Y_n\xrightarrow{C^1} X$, $1\le k_n\le k_0$ and $(x_n,E_n)\in \wt{U}^{X,+}_{k_n,\Theta}(\chi_1)\setminus \wt\pi^{-1}(M(\Sing(X),r))$ such that $(x_n,E_n)\not \in \wt{U}^{Y_n,+}_{k_n,\Theta}(\chi_2)\setminus \wt\pi^{-1}(M(\Sing(Y_n),r/2))$. By taking a subsequence, we assume that $k_n=k$ is a constant and $(x_n,E_n)\to (x,E)$ in $G^1(M^3)$. 
	
	By the $C^0$ continuity, we have $M(\Sing(Y_n),r/2)\subset M(\Sing(X),r)$ for all sufficiently large $n$. This implies that $x_n\not\in M(\Sing(Y_n),r/2)$, and then $(x_n,E_n)\not \in \wt{U}^{Y_n,+}_{k,\Theta}(\chi_2)$. It follows from the definition of $\wt{U}^{Y_n,+}_{k,\Theta}(\chi_2)$ that $\|(d_{x_n}\varphi^{Y_n}_{{-k}})^*|_{_{E_n}}\|\ge e^{-k\chi_2 }$ or $\|(d_{x_n}\varphi^{Y_n}_{{-\Theta k}})^*|_{_{E_n}}\|\ge e^{-\Theta k\chi_2 }$.  Without loss of generality, we assume that the first inequality holds for all sufficiently large $n$. Taking $n\to+\infty$, and noting that  $x\not \in\Sing(X)$, we have $\|(d_{x}\varphi_{_{{-k}}})^*|_{_{E}}\|\ge e^{-k\chi_2}$.
	
	On the other hand, it follows from $(x_n,E_n)\in \wt{U}^{X,+}_{k_n,\Theta}(\chi_1)$ that 
	$ \|(d_{x_n}\varphi_{_{{-k}}})^*|_{_{E_n}}\|<e^{-k \chi_1 }$. Taking $n\to+\infty$, we have $\|(d_{x}\varphi_{_{{-k}}})^*|_{_{E}}\|\le e^{-k\chi_1 }$. It contradicts the fact $\chi_2<\chi_1$. The proof is completed.
\end{proof}

To establish the connection between the sets $\widetilde{U}^{Y,+}_{k,\Theta}(\chi)$ and the Pesin blocks, we require the following technical lemma.

\begin{lem}\label{lem.psiforU}
For any $0<\chi_4<\chi_3$ and $L>1$, there exists a positive integer $\Theta=\Theta(\chi_4,\chi_3,L)$ such that, for  any $C^1$ vector field $Y$ satisfying $\|\nabla Y\|\le L$, any positive integer $k$, and any $(x,E)\in  \wt{U}^{Y,+}_{k,\Theta}(\chi_3)$, we have $E\cap \langle Y(x) \rangle=\{0\}$ and
\[
\big\|\psi^{Y,*}_{-\Theta k}\big|_{\proj^Y_{_{\N^Y}}(E)}\big\|<e^{-\Theta k\chi_4},
\]
where $\proj^Y_{_{\N^Y}}$ is the orthogonal projection onto the normal bundle $\N^Y$ along the flow direction $\langle Y \rangle$.
\end{lem}
\begin{proof}
Let $0<\chi_4<\chi_3$ and $L>1$ be arbitrary. Since 
\[
\lim_{k\to +\infty} \frac{1}{k} \log \frac{1+e^{2Lk}}{1-e^{-k\chi_3}}=2L,
\]
there exists a positive integer $\Theta=\Theta(\chi_4,\chi_3,L)$ such that
\begin{equation}\label{equ.Theta}
\frac{1}{k} \log \frac{1+e^{2Lk}}{1-e^{-k\chi_3}}<(\chi_3-\chi_4)\Theta
\end{equation}
for any $k\in \bZ^+$.

Suppose $(x,E)\in  \wt{U}^{Y,+}_{k,\Theta}(\chi_3)$ for some $C^1$ vector field $Y$ with $\|\nabla Y\|\le L$ and some positive integer $k$. Let $v\in E$ be a unit vector. Denote $v_{_{\N}}:=\proj^Y_{_{\N^Y}}(v)$ and $v_{_{Y}}:=v-v_{_{\N}}\in \langle Y(x)\rangle$. 

Note that 
\[
|(d_{x}\varphi^Y_{_{{-k}}})^*(v_{_Y})|=\frac{|d_x\varphi^Y_{_{-k}}(v_{_Y})|}{\|d_x\varphi^Y_{_{-k}}|_{\langle Y(x)\rangle }\|}=|v_{_Y}|,
\]
and
\[
|(d_{x}\varphi^Y_{_{{-k}}})^*(v_{_{\N}})|=\frac{|d_x\varphi^Y_{_{-k}}(v_{_{\N}})|}{\|d_x\varphi^Y_{_{-k}}|_{\langle Y(x)\rangle }\|}\le \|d_{\varphi^Y_{-k}(x)}\varphi^Y_{_{k}}\|\cdot\|d_x\varphi^Y_{_{-k}}\|\cdot |v_{_{\N}}|\le e^{2Lk}|v_{_{\N}}|.
\]
The last ``$\le $'' above is due to the Gronwall inequality. By the reverse triangle inequality,
\begin{equation}\label{equ.dphivandvN}
  |(d_{x}\varphi^Y_{_{{-k}}})^*(v)|\ge |(d_{x}\varphi^Y_{_{{-k}}})^*(v_{_Y})|-|(d_{x}\varphi^Y_{_{{-k}}})^*(v_{_{\N}})|\ge |v_{_Y}|-e^{2Lk}|v_{_{\N}}|\ge 1-|v_{_{\N}}|-e^{2Lk}|v_{_{\N}}|.
\end{equation}

On the other hand, by the definition of $\wt{U}^{Y,+}_{k,\Theta}(\chi_3)$, we have $|(d_{x}\varphi^Y_{_{{-k}}})^*(v)|< e^{-k\chi_3}$ and $|(d_x\varphi^Y_{_{{-\Theta k}}})^*(v)|<e^{-\Theta k\chi_3}$. Combining with \eqref{equ.dphivandvN}, we obtain
\[
|v_{_{\N}}|>\frac{1-e^{-k\chi_3}}{1+e^{2Lk}}>0.
\]
This implies $E\cap \langle Y(x) \rangle=\{0\}$. Recall that $\psi^{Y}_{-n}(v_{_{\N}})=\proj^Y_{_{\N^Y}}(d_{x}\varphi^Y_{_{{-n}}}(v))$. So we have
\[
|\psi^{Y,*}_{_{-\Theta k}}(v_{_{\N}})|\le |(d_{x}\varphi^Y_{_{{-\Theta k}}})^*(v)|<e^{-\Theta k\chi_3}.
\]
Then
\[
\big\|\psi^{Y,*}_{_{-\Theta k}}\big|_{\proj^Y_{_{\N^Y}}(E)}\big\|=\frac{|\psi^{Y,*}_{_{-\Theta k}}(v_{_{\N}})|}{|v_{_{\N}}|}<\frac{1+e^{2Lk}}{1-e^{-k\chi_3}}\cdot e^{-\Theta k\chi_3}<e^{-\Theta k\chi_4}.
\]
The last ``$<$'' above is due to \eqref{equ.Theta}. The proof is completed.
\end{proof}

Next we consider the measure of $\wt{U}^{X,+}_{k,\Theta}(\chi)$.

\begin{lem}\label{lem.cptmeas}
For any $\delta\in(0,1/2)$ and $\Theta\in\bZ^+$, there exist a neighborhood $\widetilde{\mathscr{V}}^{+}_K$ of $\widetilde{\mathscr{L}}^{+}_K$ in the set of all probability measures on $G^1(M^3)$, a positive integer $k_0$ and a positive number $r$ such that for all $\widetilde{\nu}\in\widetilde{\mathscr{V}}_K^{+}$, there is $k\in\{1,\dots,k_0\}$ satisfying $\widetilde{\nu}\left(\widetilde{U}^{X,+}_{k,\Theta}(\chi_0)\setminus \wt \pi ^{-1}(M(\Sing(X),r))\right)> 1-\delta^2$.
\end{lem}
\begin{proof}
Let $\delta\in(0,1/2)$ and $\Theta\in\bZ^+$ be arbitrary. Take any measure  $\widetilde{\mu}\in\widetilde{\mathscr{L}}^{+}_K$. It follows from Proposition~\ref{pro.continu} that $\mu=\widetilde{\pi}_{*}(\widetilde{\mu})$ is an $X$-invariant regular saddle-type $\chi_0$-hyperbolic measure satisfying $\widetilde{\mu}=\widetilde{\mu}^{X,+}$. Similar to the discussion in \S\,\ref{subsec.Lyapunov}, since $\mu$ is regular $\chi_0$-hyperbolic and $E^{X,+}_x$ is the subspace in the Oseledets splitting corresponding to the positive Lyapunov exponent, for $\mu$-a.e. $x\in M\setminus \Sing(X)$, we have
 \[
 \lim_{n\to \pm \infty}\frac{1}{n}\log\|(d_x\varphi_n)^*|_{E^{X,+}_x}\|=\lim_{t\to \pm \infty}\frac{1}{t}\log\|d_x\varphi_t|_{E^{X,+}_x}\|>\chi_0.
 \]
 Then using the same argument as in the classical Pesin theory, there exists $k_{\widetilde{\mu}}\geq1$ such that 
\[
\mu (\{x\in M\setminus \Sing(X): \|(d_x\varphi_{-n})^*|_{E^{X,+}_x}\|<e^{-n\chi_0} \text{ for all } n\ge k_{\widetilde{\mu}} \})>1-\frac{\delta^2}{2}.
\]
Note that $\widetilde{\mu}=\widetilde{\mu}^{X,+}$ implies $(x,E)=(x, E^{X,+}_x)$ holds for $\widetilde{\mu}$-a.e. $(x,E)$. So we have
\[
\wt\mu (\widetilde{U}^{X,+}_{k_{\widetilde{\mu}},\Theta}(\chi_0))\ge \wt\mu (\{(x,E)\in G^1(M^3):x\in M\setminus \Sing(X),\ \|(d_x\varphi_{-n})^*|_{E}\|<e^{-n\chi_0},\ \forall\, n\ge k_{\widetilde{\mu}} \})>1-\frac{\delta^2}{2}.
\]

On the other hand, $\mu(\Sing(X))=0$ implies that there is $r_{\widetilde{\mu}}>0$ such that $\mu(M(\Sing(X),r_{\widetilde{\mu}}))<\delta^2/2$. Thus  $\wt\mu(\wt\pi^{-1}(M(\Sing(X),r_{\widetilde{\mu}})))<\delta^2/2$.

Since $\widetilde{U}^{X,+}_{k_{\widetilde{\mu}},\Theta}(\chi_0) \setminus \wt\pi^{-1}(M(\Sing(X),r_{\widetilde{\mu}}))$ is open in $G^1(M^3)$, there exists a neighborhood $\widetilde{\mathscr{V}}_{\widetilde{\mu}}$ of $\wt\mu$ such that
\[
\widetilde{\nu}\left(\widetilde{U}^{X,+}_{k_{\widetilde{\mu}},\Theta}(\chi_0) \setminus \wt\pi^{-1}(M(\Sing(X),r_{\widetilde{\mu}}))\right)> 1-\delta^2
\]
 for all $\wt\nu\in \widetilde{\mathscr{V}}_{\widetilde{\mu}}$. 
  By the compactness of $\widetilde{\mathscr{L}}^{+}_K$, there are finitely many $\wt\mu_1,\dots,\wt\mu_s\in \widetilde{\mathscr{L}}^{+}_K$ such that $\widetilde{\mathscr{L}}^{+}_K\subset  \bigcup_{j=1}^s \widetilde{\mathscr{V}}_{\widetilde{\mu}_j}$.
 Then $\widetilde{\mathscr{V}}_K^+:= \bigcup_{j=1}^s \widetilde{\mathscr{V}}_{\widetilde{\mu}_j}$ is a neighborhood of $\widetilde{\mathscr{L}}^{+}_K$. Let $k_0:=\max_{1\le j\le s}k_{\wt\mu_j}$ and $r:=\min_{1\le j\le s}r_{\wt\mu_j}$. For any $\wt\nu \in \widetilde{\mathscr{V}}_K^+$, there is $1\le j\le s$ such that $\wt\nu \in \widetilde{\mathscr{V}}_{\widetilde{\mu}_j}$. So we have 
 \[
\widetilde{\nu}\left(\widetilde{U}^{X,+}_{k_{\widetilde{\mu}_j},\Theta}(\chi_0) \setminus \wt\pi^{-1}(M(\Sing(X),r))\right)\ge \widetilde{\nu}\left(\widetilde{U}^{X,+}_{k_{\widetilde{\mu}_j},\Theta}(\chi_0) \setminus \wt\pi^{-1}(M(\Sing(X),r_{\widetilde{\mu}_j}))\right)> 1-\delta^2.
\]\end{proof}

The following proposition is a key step in proving the SPR property. Its diffeomorphism case was proved in \cite[Theorem~3.1]{BCS25} (also see \cite[\S4]{Zan}). The proof, which closely follows the diffeomorphism case, is provided in Appendix~\ref{appsec.spr}.

\begin{pro}\label{pro.bigpressurebigpesin}
	For any $\tau'\in (0,1)$, $\chi\in (0,\chi_0)$ and $\varepsilon>0$, there exist  numbers $\ell_0>0$, $P_0<P_{\rm TOP}(X|_K,\vartheta)$ and a $C^\infty$ neighborhood $\U$ of $X$ such that, for any $Y\in \U$  and any  $\nu\in \Pe(Y)$ satisfying $\nu(K)=1$ and $P_\nu(Y,\vartheta)>P_0$, we have $\nu(\Pes^Y_\chi(\ell_0,\varepsilon))>\tau'$. 
	\qed
\end{pro}

Now we prove the  main theorem of this subsection.

\begin{proof}[Proof of Theorem~\ref{thm.3SPRset}]	
For any given $\tau\in (0,1)$, choose some $\tau'\in (\tau,1)$. For this $\tau'$ and any $\chi\in (0,\chi_0)$, $\varepsilon>0$, let $\ell_0>0$ and $P_0<P_{\rm TOP}(X|_K,\vartheta)$ be given by Proposition~\ref{pro.bigpressurebigpesin}. (Since we only consider $X$ in this theorem, the $C^\infty$ neighborhood $\U$ is not used.) By Lemma~\ref{lem.regularhyper}~(3), increasing $P_0<P_{\rm TOP}(X|_K,\vartheta)$ if necessary, there exists $r>0$ such that if $\nu$ is an $X$-invariant measure satisfying $\nu({ K})=1$ and $P_\nu(X,\vartheta)>P_0$, then $\nu(M(\Sing(X),r))< \tau'-\tau$. 	

 Denote 
\[
\Lambda:=\Pes_\chi(\ell_0,\varepsilon) \setminus {\rm Int}(M(\Sing(X),r)),
\]
 where ${\rm Int}(M(\Sing(X),r))$ is the interior of ${M}(\Sing(X),r)$. Recall that $\ol{\Pes_\chi(\ell,\varepsilon)}\setminus\Pes_\chi(\ell,\varepsilon)\subset\Sing(X)$ (see Lemma~\ref{lem.cptofpesinblock} for details). Thus $\Lambda$  is a compact set. Moreover, for any $\nu\in \Pe(X|_K)$ with $P_\nu(X,\vartheta)>P_0$, we have $\nu(\Lambda)>\tau$. This completes the proof.	
\end{proof}

Based on Theorem~\ref{thm.3SPRset} and Theorem~\ref{thm.main.spr}, we immediately obtain the following corollary.

\begin{cor}\label{thm.main.3flowmme}
	Let $M^3$ be a closed $3$-dimensional $C^\infty$ Riemannian manifold and $X$ be a $C^{\infty}$ vector field on $M^3$.
\begin{itemize}
  \item If $B$ is a Borel homoclinic class of $X$ and $\vartheta$ is a  H\"older continuous potential on $M^3$ satisfying ${\rm Var} (\vartheta) < h_{\rm TOP}(X|_B)=h_{\rm top}(X|_{\ol{B}})$, then $X$ is SPR for $\vartheta$ on $B$. Consequently, there exists a unique   equilibrium state of $\vartheta$  for $X$  on $B$. 
  \item If $\vartheta$ is a  H\"older continuous potential on $M^3$ satisfying ${\rm Var} (\vartheta) < h_{\rm top}(X)$, then $X$ is SPR for $\vartheta$ on $M^3$. Consequently, there exist only finitely many ergodic equilibrium states of $\vartheta$ for $X$. \qed
\end{itemize}
\end{cor}

%%%%%%%%%%%%%%%%%%%%%%%%%%%%%%%%%%%%%%%%%%%%%%%%%%%%%%%%%%%%%%%%%%

\subsection{Good asymptotics for three-dimensional flows}

In this subsection we consider the existence of good asymptotics. We first state some facts for Pesin blocks. These facts are valid in all dimensions greater than or equal to $3$.
As in the classical Pesin theory, we only consider the case where $\varepsilon$ is much smaller than $\chi$.

Using the classical argument for the compactness of hyperbolic sets, the splitting $\mathcal{N}^Y_y=\N^{Y,s}_y\oplus \N^{Y,u}_y$ in the definition of Pesin blocks is uniquely determined, and the decomposition varies continuously with respect to $y\in \Pes_\chi^Y(\ell,\varepsilon)$ and with respect to $Y$ in the $C^1$ topology. Precisely, we have the following lemma. A concise proof is included in Appendix~\ref{appsec.spr}.

\begin{lem}\label{lem.cptofpesinblock}
Let $0<\varepsilon< \chi$ and $\ell\ge 1$. Then for any $C^1$ vector field $Y$ on $M$, and any $y\in \Pes_\chi^Y(\ell,\varepsilon)$,  the splitting $\mathcal{N}^Y_y=\N^{Y,s}_y\oplus \N^{Y,u}_y$ in the definition of Pesin blocks is uniquely determined. Moreover, if there are $Y_k \xrightarrow{C^1} Y$ and  $y_k\in \Pes^{Y_k}_\chi(\ell,\varepsilon)$ such that $y_k\to x\in M\setminus \Sing(Y)$, then $x\in \Pes^Y_\chi(\ell,\varepsilon)$ and $\N^{Y_k,s}_{y_k}\to \N^{Y,s}_x$, $\N^{Y_k,u}_{y_k}\to \N^{Y,u}_x$.	
\qed
\end{lem}

Recall that the {\it local stable manifold} of $x$ for $\varphi_t$, denoted by $W_{\loc}^{s}(x)$, is the set of points $y\in W^{s}(x)$ such that $\dist (\varphi_t(x),\varphi_t(y))$ is less than some small constant for any $t\ge 0$. The {\it local unstable manifold} $W_{\loc}^{u}(x)$ is defined similarly. We also denote by $W_{\loc}^{Y,s}(x)$ and $W_{\loc}^{Y,u}(x)$ the local stable and unstable manifolds for $\varphi^Y_t$.

\begin{lem}[Pesin~\cite{Pes76}]\label{lem.Pesinmanifold}
Let $X$ be a $C^\alpha$ vector field ($\alpha>1$) on $M$ and $0<\varepsilon\ll \chi$, $\ell\ge 1$. Then 
there exists a $C^\alpha$ neighborhood $\U $ of $X$ such that for any $Y\in \U$ and any $y\in \Pes_\chi^Y(\ell,\varepsilon)$, $W_{\loc}^{Y,s}(y)$ and $W_{\loc}^{Y,u}(y)$ are $C^\alpha$ embedded discs.  Moreover, if there are $Y_k \xrightarrow{C^\alpha} Y\in\U$ and  $y_k\in \Pes^{Y_k}_\chi(\ell,\varepsilon)$ such that $y_k\to x\in M\setminus \Sing(Y)$, then $W^{Y_k,s}_{\loc}(y_k)\to W^{Y,s}_{\loc}(x)$ and $W^{Y_k,u}_{\loc}(y_k)\to W^{Y,u}_{\loc}(x)$ in the $C^1$ topology.	
\end{lem}

For a $C^\alpha$ vector field $X$ on $M$ ($\alpha>1$), denote (recall that we omit ``$X$'' in corresponding notations for simplicity)
\[
\begin{aligned}
\widehat{\Pes}_\chi(\ell,\varepsilon):= \big\{x\in \Pes_\chi(\ell,\varepsilon):\ & \text{there exists a sequence of hyperbolic periodic points $p_n$ of $X$ such that}\\
&  p_n\to x, \text{ the orbit of each $p_n$ exhibits a transverse homoclinic point,}\\
&  W_{\loc}^{s}(p_n)\to W_{\loc}^{s}(x)  \text{ and }  W_{\loc}^{u}(p_n)\to W_{\loc}^{u}(x) \text{ in the $C^1$ topology} \big\}.
\end{aligned}
\]
Then $\widehat{\Pes}_\chi(\ell,\varepsilon)$ is relatively closed in $M\setminus \Sing(X)$ whenever $\varepsilon\ll\chi$. Moreover, for any $\chi$-hyperbolic measure, $\widehat{\Pes}_\chi(\ell,\varepsilon)$ is a full-measure subset of ${\Pes}_\chi(\ell,\varepsilon)$.

\begin{lem}\label{lem.PEShat}
	Let $X$ be a $C^\alpha$ vector field ($\alpha>1$) on $M$ and $0<\varepsilon\ll \chi$, $\ell\ge 1$. 
\begin{enumerate}
  \item If there exists a sequence of $y_k\in \widehat{\Pes}_\chi(\ell,\varepsilon)$ such that $y_k\to x\in M\setminus \Sing(X)$, then $x\in \widehat\Pes_\chi(\ell,\varepsilon)$.
  \item $\mu(\widehat\Pes_\chi(\ell,\varepsilon))=\mu(\Pes_\chi(\ell,\varepsilon))$ holds for any regular $\chi$-hyperbolic $X$-invariant measure $\mu$.
\end{enumerate}
\end{lem}
\begin{proof}
	First assume that $y_k\in \widehat{\Pes}_\chi(\ell,\varepsilon)$ such that $y_k\to x\in M\setminus \Sing(X)$. By Lemma~\ref{lem.cptofpesinblock}, we have $x\in \Pes_\chi(\ell,\varepsilon)$. According to Lemma~\ref{lem.Pesinmanifold}, $W^{s/u}_{\loc}(y_k)\to W^{s/u}_{\loc}(x)$ in the $C^1$ topology. And by the definition of $\widehat{\Pes}_\chi(\ell,\varepsilon)$, each $W^{s/u}_{\loc}(y_k)$ is $C^1$ approximated by the local stable and unstable manifolds of hyperbolic periodic points that admit transverse homoclinic points, respectively. Consequently, $x\in \widehat{\Pes}_\chi(\ell,\varepsilon)$. Item (1) is proved.
	
	For Item (2), the analogous result in the diffeomorphism case was proved in \cite{Kat80}. By the same argument together with the closing lemma for singular flows from \cite{LLL2024}, Item (2) follows readily in the singular flow case.
\end{proof}

\begin{lem}\label{lem.homorelated}
Let $X$ be a $C^\alpha$ vector field ($\alpha>1$) on $M$ and $0<\varepsilon\ll \chi$, $\ell\ge 1$. For any $x\in \widehat\Pes_\chi(\ell,\varepsilon )$, there exist a $C^\alpha$ neighborhood $\U_x$ of $X$, a hyperbolic periodic point $p_x$ of $X$ exhibiting a transverse homoclinic point, and a positive number $\delta_x>0$ such that, if $Y\in \U_x$ and $y\in \Pes^Y_\chi(\ell,\varepsilon)$ satisfy $\dist(x,y)<\delta_x$, then $W^{Y,s}_{\loc}(y)$ intersects $W^{Y,u}_{\loc}(p_x^Y)$ transversely and $W^{Y,u}_{\loc}(y)$ intersects $W^{Y,s}_{\loc}(p_x^Y)$ transversely, where $p_x^Y$ is the continuation of $p_x$.
\end{lem}
\begin{proof}
Let $X$, $0<\varepsilon\ll \chi$ and $\ell\ge 1$ be given. For any $x\in \widehat\Pes_\chi(\ell,\varepsilon )$, by the definition, there exists a hyperbolic periodic point $p_x$ exhibiting a transverse homoclinic point such that $W^{s}_{\loc}(x)$ intersects $W^{u}_{\loc}(p_x)$ transversely and $W^{u}_{\loc}(x)$ intersects $W^{s}_{\loc}(p_x)$ transversely.
	
	According to classical hyperbolic theory, the local stable manifold $W^{Y,s}_{\loc}(p_x^Y)$ and  local unstable manifold $W^{Y,u}_{\loc}(p_x^Y)$ of the hyperbolic periodic point $p_x^Y$ depend continuously on $Y$ in the $C^1$ topology. Furthermore, by Lemma~\ref{lem.Pesinmanifold}, the manifolds $W^{Y,s}_{\loc}(y)$ and $W^{Y,u}_{\loc}(y)$ depend continuously on $Y$ and $y$ in the $C^1$ topology as $Y$ varies in the $C^\alpha$ topology and $y\in\Pes^Y_\chi(\ell,\varepsilon)$.
	
	Since the transverse intersection of submanifolds is $C^1$ robust, we conclude that there exist a $C^\alpha$ neighborhood $\U_x$ of $X$ and a positive number $\delta_x>0$ satisfying the statement of this lemma.	
\end{proof}

\smallskip

We now return to smooth vector fields on a closed $3$-dimensional manifold $M^3$.

 \begin{lem}\label{lem.pertasym}
Let  $\vartheta :M^3\to\mathbb{R}$ be a H\"older continuous function. Assume that $X$ is a $C^\infty$ vector field on $M^3$ satisfying $\Var (\vartheta)<h_{\rm top}(X)$. Then there exists $P_X<P_{\rm top}(X,\vartheta)$ such that for any $C^{\infty}$ neighborhood $\mathcal{U}_X$ of $X$, there exist $Z\in\mathcal{U}_X$ and a $C^{\infty}$ neighborhood $\mathcal{U}_Z\subset \mathcal{U}_X$ of $Z$ satisfying:  for any $Y\in\mathcal{U}_Z$ and any $Y$-ergodic measure $\nu$ with $P_{\nu}(Y,\vartheta)>P_X$, the unique Borel homoclinic class carrying $\nu$ has good asymptotics.
\end{lem}
\begin{proof}
Let $\vartheta$ and $X$ be given such that $\Var (\vartheta)<h_{\rm top}(X)$. Define $\chi_0:=\frac{1}{2}(h_{\rm top}(X)-\Var(\vartheta))>0$. By the variational principle,
 \[
 P_{\rm top}(X,\vartheta)=\sup\{P_{\mu}(X,\vartheta): \mu\in \mathbb P(X)\}=P_{\rm TOP}(X|_M,\vartheta).
 \]

Fix $0<\chi<\chi_0$ and $0<\varepsilon\ll \chi$. By Proposition~\ref{pro.bigpressurebigpesin}, with $\tau'=0.9$ and $K=M$,  there exist $\ell_0>0$, $P_0<P_{\rm top}(X,\vartheta)$, and a $C^\infty$ neighborhood $\U$ of $X$ such that for any $Y\in \U$  and any  $\nu\in \Pe(Y)$ satisfying  $P_\nu(Y,\vartheta)>P_0$, we have $\nu(\Pes^Y_\chi(\ell_0,\varepsilon))>0.9$. 
Next, by Lemma~\ref{lem.regularhyper}~(3), setting $\xi=0.1$ and $K=M$, we may shrink the $C^\infty$ neighborhood $\U$ and increase $P_0<P_{\rm top}(X,\vartheta)$ if necessary, such that there exists $r>0$ satisfying $\nu(M(\Sing(Y),r))< 0.1$ for any $Y\in \U$ and any $\nu\in \Pe(Y)$ with $P_\nu(Y,\vartheta)>P_0$. Therefore,  
\begin{equation}\label{eq.measureofPES}
  \nu(\Pes^Y_\chi(\ell_0,\varepsilon)\setminus M(\Sing(Y),r))>0.8.
\end{equation}

Denote $\widehat\Lambda:=\widehat\Pes_\chi(\ell_0,\varepsilon)\setminus {\rm Int}(M(\Sing(X),r))$. By Lemma~\ref{lem.PEShat}~(1), $\widehat\Lambda$ is a compact set. By 
Lemma~\ref{lem.PEShat}~(2) and \eqref{eq.measureofPES}, $\nu(\widehat
\Lambda)>0.8$ for any $\nu\in \Pe(X)$ with $P_\nu(X,\vartheta)>P_0$. In particular, if $\mu$ is an equilibrium state of $\vartheta$ for $X$, then  almost every ergodic component of $\mu$ is an equilibrium state of $\vartheta$ for $X$, which implies $\mu(\widehat\Lambda)>0.8$.

According to Lemma~\ref{lem.homorelated}, for any $x\in \widehat\Lambda$, there exist a $C^\infty$ neighborhood $\U_x$ of $X$, a hyperbolic periodic point $p_x$ of $X$ exhibiting a transverse homoclinic point, and a positive number $\delta_x>0$ such that, if $Y\in \U_x$ and $y\in \Pes^Y_\chi(\ell_0,\varepsilon)$ satisfy $\dist(x,y)<\delta_x$, then $W^{Y,s}_{\loc}(y)$ intersects $W^{Y,u}_{\loc}(p_x^Y)$ transversely and $W^{Y,u}_{\loc}(y)$ intersects $W^{Y,s}_{\loc}(p_x^Y)$ transversely.

By the compactness of $\widehat\Lambda$, there are finitely many $x_1,\ldots, x_k\in \widehat\Lambda$ such that 
\[
\widehat\Lambda\subset \bigcup_{j=1}^k {\rm Int}(M(x_j,\delta_{x_j}))=:U.
\]

We claim that there exist a $C^\infty$ neighborhood $\U_0\subset \U $ of $X$ and $P_0\le P_X<P_{\rm top}(X,\vartheta)$ such that for any  $Y\in \U_0$ and any $\nu\in \Pe(Y)$, if $P_\nu(Y,\vartheta)>P_X$ then $\nu (U)>0.7$. Otherwise, there exist $Y_n\xrightarrow{C^\infty} X$ and $\nu_n\in \Pe(Y_n)$ with $\liminf_{n\to \infty} P_{\nu_n}(Y_n,\vartheta)\ge P_{\rm top}(X,\vartheta)$ and $\nu_n (U)\le 0.7$.   Without loss of generality, we assume that $\nu_n \rightharpoonup \mu$. Then $\mu(U)\le 0.7$ by the openness of $U$. But since $\mu$ is an equilibrium state of $\vartheta$ for $X$, we have  $\mu(U)\ge\mu(\widehat\Lambda)>0.8$. This contradiction proves this claim.

After reducing $\mathcal{U}_0$ if necessary, we may assume that the continuation $p_{x_j}^Y$ of $p_{x_j}$ ($j=1,\ldots,k$) is well defined for any $Y\in\U_0$, and each $p_{x_j}^Y$ still has a transverse homoclinic point. Moreover, we also assume that $\U_0\subset \U_{x_j}$ for each $j=1,\ldots,k$.

Let $\mathcal{U}_X$ be a $C^{\infty}$ neighborhood  of $X$.  By Proposition~\ref{pro.goodasymptotics}, there exist $Z_1\in \U_X\cap\U_0$ and a $C^\infty$ neighborhood $\U_{Z_1}\subset \U_X\cap\U_0$ of $Z_1$ such that $p_{x_1}^Y$ has good asymptotics for any $Y\in \U_{Z_1}$. Since $p_{x_2}^{Z_1}$ still has a transverse homoclinic point, we apply Proposition~\ref{pro.goodasymptotics} again, there exist $Z_2\in \U_{Z_1}$ and a $C^\infty$ neighborhood $\U_{Z_2}\subset \U_{Z_1}$ of $Z_2$ such that $p_{x_2}^Y$ has good asymptotics for any $Y\in \U_{Z_2}$.

Repeat the above process, and we finally obtain $Z\in \U_X\cap\U_0$ and a $C^\infty$ neighborhood $\U_{Z}\subset \U_X\cap\U_0$ of $Z$ such that $p_{x_j}^Y$ ($j=1,\ldots,k$) has good asymptotics for any $Y\in \U_{Z}$.
 
Take any $Y\in\mathcal{U}_Z$ and any $Y$-ergodic measure $\nu$ with $P_{\nu}(Y,\vartheta)>P_X$. By \eqref{eq.measureofPES} and the above claim, we have $\nu(\Pes^Y_\chi(\ell_0,\varepsilon)\cap U)>0.5$. Thus there is $1\le j\le k$ such that $\nu(\Pes^Y_\chi(\ell_0,\varepsilon)\cap {\rm Int}(M(x_j,\delta_{x_j})))>0$.
By Lemma~\ref{lem.homorelated}, each point in $\Pes^Y_\chi(\ell_0,\varepsilon)\cap {\rm Int}(M(x_j,\delta_{x_j}))$ is homoclinic related to $p_{x_j}^Y$. Hence $p_{x_j}^Y$ is contained in the unique Borel homoclinic class carrying $\nu$. The proof is completed. 
\end{proof}

Now we prove Theorem~\ref{main.3flowrm}. Recall that  $\X_{\vartheta}^\infty(M^3)$ is the set of all $C^\infty$ vector fields on $M^3$ satisfying  $\Var(\vartheta)< h_{\rm top}(X)$.

\begin{thm}{\rm (Theorem~\ref{main.3flowrm})}\label{thm.main.3flowrm}
	If $\vartheta$ is a  H\"older continuous potential on $M^3$, then there exists a $C^\infty$-open and dense subset $\U$ of $\X_{\vartheta}^\infty(M^3)$ such that for every $Y\in \U$,  every ergodic equilibrium state of $\vartheta$ for $Y$ is rapidly mixing.
\end{thm}
\begin{proof}
 For any $X\in \X_{\vartheta}^\infty(M^3)$, let $P_0=P_0(X)<P_{\rm top}(X,\vartheta)$ be given by Lemma~\ref{lem.pertasym}. It is proved in \cite{LY12} (see also \cite{Kat80,LLL2024}) that there exists a hyperbolic horseshoe $\Lambda_0$ for $X$ whose topological entropy can be made arbitrarily close to $h_{\rm top}(X)$. Then by using the technique in \cite{Gel}, we may also require that the topological pressure of $\vartheta$ on $\Lambda_0$
is arbitrarily close to $P_{\rm top}(X,\vartheta)$. Since the topological entropy and the topological pressure of $\vartheta$ vary lower semi-continuously on a hyperbolic horseshoe, and the supremum of lower semi-continuous functions remains lower semi-continuous, we conclude that the topological entropy and the topological pressure of $\vartheta$ on $M^3$ vary lower semi-continuously with respect to the vector field. Thus there is a $C^\infty$ neighborhood $\mathcal{U}_2$ of $X$ such  that $h_{\rm top}(Y)> \Var(\vartheta)$ and $P_{\rm top}(Y,\vartheta)>P_0$ for all $Y\in\mathcal{U}_2$. We obtain $\U_2\subset \X_{\vartheta}^\infty(M^3)$.

For any $Y\in\U_2$, by Corollary~\ref{thm.main.3flowmme}, $Y$ is SPR for $\vartheta$ on $M^3$. Moreover, there exist only finitely many ergodic equilibrium states of $\vartheta$ for $Y$. Each ergodic equilibrium state is carried by a Borel homoclinic class, and $Y$ is SPR on such Borel homoclinic class for $\vartheta$.

By Lemma~\ref{lem.pertasym}, there exist a sequence of vector fields $Z_{X,n}\in\mathcal{U}_2$ converging to $X$ in the $C^\infty$ topology and a corresponding sequence of $C^{\infty}$-open neighborhoods $\mathcal{U}_{Z_{X,n}}\subset \mathcal{U}_2$ of $Z_{X,n}$ such that for every $n$, any $Y\in\mathcal{U}_{Z_{X,n}}$ and any $Y$-ergodic measure $\nu$ satisfying $P_{\nu}(Y,\vartheta)>P_0$, the unique Borel homoclinic class carrying $\nu$ has good asymptotics. In particular, every Borel homoclinic class carrying an ergodic equilibrium state of $\vartheta$ for $Y\in\mathcal{U}_{Z_{X,n}}$ has good asymptotics. 

Define $\U:=\bigcup_{X\in \X_{\vartheta}^\infty(M^3)}\bigcup_{n\ge 1}\U_{Z_{X,n}}$. Then $\U$ is a $C^\infty$-open and dense subset of $\X_{\vartheta}^\infty(M^3)$. It follows from Theorem~\ref{thm.rapidmixing} that $\U$ satisfies the statement of this theorem.
\end{proof}

%%%%%%%%%%%%%%%%%%%%%%%%%%%%%%%%%%%%%%%%%%%%%%%%%%%%%%%%%%%%%%%%%%%
\section{Star flows}\label{sec.starflow}

In this section, we assume $M$ is a smooth closed Riemannian manifold of dimension at least $3$, and $X$ is a star vector field on $M$.

\subsection{Basic properties of star flows}

Recall that a $C^1$ vector field $X$ has the  {\it star property} if there exists a $C^1$ neighborhood $\U$ of $X$ such that for all $Y\in\U$, all periodic orbits and all singularities of $Y$ are hyperbolic. For star vector fields, we have the following fundamental result \cite{Liao1}:

\begin{thm}{\rm \cite{Liao1}}\label{thm.liao}
	For any star vector field $ X$, there is a $C^1$ neighborhood $\mathcal{U}$ of $X$ and numbers $\eta, t_0>0$ such that for any periodic orbit $\gamma$ of $Y \in \mathcal{U}$ with period $T \geq t_0$, if $\mathcal{N}_{\gamma}={\overline\N}^{s} \oplus {\overline\N}^{u}$ is the hyperbolic splitting with respect to $\psi_{t}^{Y}$ then
\begin{enumerate}
    \item for every $x\in \gamma$ and $t \geq t_0$, we have
   \[
   \|\psi_{t}^{Y}|_{{\overline\N}^{s}(x)}\|\cdot \|\psi_{-t}^{Y}|_{{\overline\N}^{u}(\varphi^Y_t(x))}\| \leq \mathrm{e}^{-2\eta t};
    \]
    \item for every $x \in \gamma$, we have
    \[
    \prod_{j=0}^{\lfloor T/t_0\rfloor -1} \|\psi_{t_0}^{Y}|_{{\overline\N}^{s}(\varphi_{jt_0}^{Y}(x))}\| \leq \mathrm{e}^{-\eta T}  \quad \text{and} \quad  \prod_{j=0}^{\lfloor T/t_0\rfloor -1} \|\psi_{-t_0}^{Y}|_{{\overline\N}^{u}(\varphi_{(j+1)t_0}^{Y}(x))}\| \leq \mathrm{e}^{-\eta T}.
    \]
  \end{enumerate}
\end{thm}

Based on the above lemma, it is essentially proved in \cite[Theorem~5.6]{SGW} by using the Ergodic Closing Lemma \cite{Man82,Wen96} that for any star vector field $X$, there exist a $C^1$ neighborhood $\mathcal U$ of $X$ and a constant $\chi_0> 0$ such that for every vector field $Y\in \mathcal U$, any $Y$-invariant regular measure must be $\chi_0$-hyperbolic. It implies that $\mu (\Pes^Y_{\chi_0})=1$ for any $Y$-invariant regular measure $\mu$, where $\Pes^Y_{\chi_0}$ is defined in \S\,\ref{subsec.Pesinblock}.

Recall that $O(Y) \subset M \setminus \Sing(Y)$ is defined in \S\,\ref{subsec.Lyapunov} via the Oseledec Theorem and the Poincar\'e Recurrence Theorem.
 Denote by $\N^{Y,u}_x$ (resp. $\N^{Y,s}_x$) the direct sum of subspaces at $x\in O(Y)$ whose Lyapunov exponents are positive (resp. negative). 
Then there is a measurable invariant splitting $\N^Y_x=\N_x^{Y,s}\oplus \N_x^{Y,u}$ for $\psi^Y$ (and $\psi^{Y,*}$). 

It is clear that $\N^{Y,u}$ and $\N^{Y,s}$ coincide with $\overline\N^u$ and $\overline\N^s$ in Theorem~\ref{thm.liao} on periodic orbits. Using the same standard argument as in the proof of Lemma~\ref{lem.cptofpesinblock}, we know that both $\N^{Y,u}$ and $\N^{Y,s}$ vary continuously if their dimensions remain unchanged.
Denote $u^Y(x)=\dim \N_x^{Y,u}$. Then $u^Y$ is invariant along a flow orbit.
If $\nu$ is a regular ergodic measure for $Y$, then $u^Y(x)$ is a constant for $\nu$-a.e. $x$. We denote this constant by $u^Y(\nu)$.

Denote by $\X^*(M)$ the set of $C^1$ star vector fields. The following property holds for {\it multi-singular hyperbolic} vector fields, which constitute an open dense set of $\X^*(M)$. See \cite{SGW,BD,CDYZ}.

\begin{lem}{\rm \cite[Proposition~4.2]{CDYZ}}\label{lem.domsplstar}
There is a $C^1$-open and dense subset $\U^*$ of $\X^*(M)$ satisfying: for any $X\in\U^*$ and any chain recurrent class $C$ of $X$, there exist a $C^1$ neighborhood $\U$ of $X$, an open neighborhood $U$ of $C$, and a constant $i$ such that, for any $Y\in \U$, any invariant set $\Lambda$ of $\varphi^Y_t$ in $U$, and any regular $Y$-invariant measure $\nu$ supported on $\Lambda$, $u^Y(x)=i$ for $\nu$-a.e. $x$.
\end{lem}

Moreover, we have the following lemma.

\begin{lem}\label{lem.indexstar}
Let $\U^*$ be given by Lemma~\ref{lem.domsplstar}. 	Assume that there are $Y_n, X\in \U^*$ such that $Y_n\xrightarrow{C^1} X$, and regular measures $\nu_n\in \Pe(Y_n)$, $\mu\in \mathbb P(X)$ such that $\nu_n \rightharpoonup \mu$. Then there exists a constant $i$ such that for each big enough $n$, $u^{Y_n}(\nu_n)=i=u^X(x)$ for $\mu$-a.e. $x$.
\end{lem}
\begin{proof}
	Since $\nu_n$ is ergodic for $Y_n$, there exists a chain recurrent class $C_n$ of $Y_n$ such that $\supp(\nu_n)\subset C_n$. Recall that the chain recurrent class is upper-semicontinuous with respect to the flow in the $C^1$ topology under the Hausdorff topology. Under the weak-$*$ convergence of measures, the support of the limit measure $\mu$ is contained in the Hausdorff limit of the supports of $\nu_n$, which lies in a  chain recurrent class $C$ of $X$.
	
		Let $U$ be the open neighborhood of $C$ given by Lemma~\ref{lem.domsplstar}. By using the upper-semicontinuity of chain recurrent classes again, we have $C_n\subset U$ for each big enough $n$. Then according to Lemma~\ref{lem.domsplstar}, $u^{Y_n}(y)$ is a constant $i$ for  $\nu_n$-a.e. $y$ and $u^X(x)$ is the same constant $i$ for $\mu$-a.e. $x$. 
\end{proof}

For pressures of star vector fields, we have the following lemma.

\begin{lem}\label{lem.continustar} 
Let  $\vartheta :M\to\mathbb{R}$ be a continuous function. Assume that $Y_n, X\in \X^*(M)$ satisfy $Y_n\xrightarrow{C^1} X$, and $\nu_n\in \Pe(Y_n)$, $\mu\in \mathbb P(X)$ satisfy $\nu_n \rightharpoonup \mu$. Then
\begin{enumerate}
  \item $\limsup_{n\to \infty}P_{\nu_n}(Y_n,\vartheta)\le P_{\mu}(X,\vartheta)$. 
  \item $\liminf_{n\to \infty} P_{\rm top}(Y_n,\vartheta)\ge P_{\rm top}(X,\vartheta)$.
\end{enumerate}	
\end{lem}
\begin{proof}
Item (1) follows directly as a corollary of the upper-semicontinuity of the metric entropy. It is proved in \cite{SYY} that if $X$ is away from {\it homoclinic tangency} and $\mu(\Sing(X))=0$, then the metric entropy is  upper-semicontinuous. Later, this property was extended in \cite{PYY25b} to the $C^1$-open and dense subset of star vector fields. More recently, it is further extended in \cite{LYYZ} that the upper-semicontinuity holds for all vector fields away from homoclinic tangency. 

The lower-semicontinuity of the topological entropy for star vector fields is proved in \cite[Theorem~B]{LSWW}. By using the technique in \cite{Gel} (see the proof of Theorem~\ref{thm.main.3flowrm}), one may show that the topological pressure for star vector fields is also lower-semicontinuous.
\end{proof}

\subsection{Star flows with pressure gap}

 Let  $\vartheta$ be a  H\"older continuous potential. Recall that $\X_\vartheta^{*}(M)$ is the set of all $C^2$ star vector fields satisfying $\sup_{x\in \Sing(X)} \vartheta(x)< P_{\rm top}(X)$.
The main theorem of this section is the following:

\begin{thm}{\rm (Theorem~\ref{main.starrm})}\label{thm.main.starrm}
	 For any H\"older continuous potential $\vartheta$, there exists a $C^2$-open and $C^1$-dense subset $\U$ of $\X_\vartheta^{*}(M)$ such that for every $X\in \U$, $X$ is SPR for $\vartheta$, and every ergodic equilibrium state of $\vartheta$ is rapidly mixing.
\end{thm}

The proof of Theorem~\ref{thm.main.starrm} follows a similar method to that of Theorem~\ref{thm.main.3flowrm}. The main difficulty lies in handling the higher-dimensional case. However, thanks to the star condition, the linear Poincar\'e flow admits a uniform dominated splitting $\N^{Y,s}\oplus \N^{Y,u}$, which allows us to avoid lifting measures to the Grassmannian bundle.

Let $X\in \X_\vartheta^{*}(M)$ be arbitrary. Define
\[
\begin{aligned}
\mathscr{L}:=\, & \{\mu:\exists\, \nu_n \in \Pe(Y_n)\text{ s.t. }  Y_n\xrightarrow{C^1} X,\ \nu_n \rightharpoonup \mu \text{ and } \lim_{n\to \infty}P_{\nu_n}(Y_n,\vartheta) \ge   P_{\mathrm{top}}(X,\vartheta)\}, \text{ and}\\
	\mathscr{L}^i:=\, & \{\mu:\exists\, \nu_n \in \Pe(Y_n)\text{ s.t. } u^{Y_n}(\nu_n)=i,\ Y_n\xrightarrow{C^1} X,\ \nu_n \rightharpoonup \mu \text{ and } \lim_{n\to \infty}P_{\nu_n}(Y_n,\vartheta) \ge   P_{\mathrm{top}}(X,\vartheta)\}.
\end{aligned}
\]

By Lemma~\ref{lem.indexstar}, $\mathscr{L}^i\cap \mathscr{L}^j=\emptyset$ if $i\neq j$. Thus $\mathscr{L}=\biguplus_{i=1}^{d-2} \mathscr{L}^i$, where $d$ is the dimension of $M$. All the sets $\mathscr{L}$ and $\mathscr{L}^i$ are compact ($\mathscr{L}^i$ may be empty). 
For any $\mu\in \mathscr{L}$, according to the upper-semicontinuity of metric pressure of star vector fields (Lemma~\ref{lem.continustar} (1)),  $\mu$ is an equilibrium state of $\vartheta$ for $X$. Then similar to Lemma~\ref{lem.regularhyper}, by using the pressure gap condition $\sup_{x\in\Sing(X)}\vartheta(x) <P_{\rm top}(X,\vartheta)$, we conclude that each $\mu\in \mathscr{L}$ is regular.

\begin{lem}\label{lem.cptmeasstar}
Let $X\in \U^*$, where $\U^*$ is given by Lemma~\ref{lem.domsplstar}. For any $\delta\in(0,1/2)$, there exist a $C^1$ neighborhood $\U\subset \U^*$ of $X$, a neighborhood ${\mathscr{V}}$ of ${\mathscr{L}}$, a positive integer $k_0$ and a positive number $r$ satisfying: for any $Y\in \U$ and ${\nu}\in{\mathscr{V}}\cap \Pe(Y)$, there is $k\in \{1,\ldots, k_0\}$ and a set $U^+\subset M\setminus M(\Sing(X),r)$ such that $\nu(U^+)>1-\delta^2$ and $\|\psi^{Y,*}_{-k}|_{\N^{Y,u}_x}\|<e^{-k\chi_0/2}$ holds for all $x\in U^+$.
\end{lem}
\begin{proof}
First, we claim that for any $\delta\in(0,1/2)$ there exist $r>0$ and a neighborhood ${\mathscr{V}}_0$ of ${\mathscr{L}}$ such that $\nu(M(\Sing(X),2r))\le \delta^2/4$ for any  $\nu\in \mathscr{V}_0$. If it is not true, then there are  $\nu_n\rightharpoonup  \mu\in {\mathscr{L}}$ 
and $r_n\to 0$ satisfying $\nu_n(M(\Sing(X),r_n))>\delta^2/4$. It implies that $\mu(\Sing(X))\ge \delta^2/4$, contradicting the regularity of $\mu$. Take  a $C^1$ neighborhood $\U_0\subset \U^*$ of $X$, such that for any $Y\in \U_0$ and any $\nu\in \mathscr{V}_0$,  $\nu(M(\Sing(Y),r))\le \delta^2/4$.

Since $\mathscr{L}$ is the disjoint union of $\mathscr{L}^i$, reduce $\mathscr{V}_0$ if necessary, we may assume that $\mathscr{V}_0$ is the disjoint union of $\mathscr{V}_0^i$, where each $\mathscr{V}_0^i$ is a neighborhood of $\mathscr{L}^i$. Moreover, by Lemma~\ref{lem.indexstar}, reduce $\mathscr{V}_0$ and $\U_0$ if necessary, we may also assume that $u^Y(\nu)=i$ for any $Y\in \U_0$ and any $Y$-ergodic measure $\nu\in\mathscr{V}_0^i$.

\smallskip
\nt {\bf Claim.} {\it For any $\mu\in \mathscr{L}$ there exist a $C^1$ neighborhood $\U_{\mu}\subset\U_0$ of $X$, a neighborhood $\mathscr{V}_{\mu}\subset \mathscr{V}_0$ of $\mu$, an open set $V^+_{\mu}\subset M\setminus M(\Sing(X),r)$ and a positive integer $k_{\mu}>0$ such that for any $Y\in \U_{\mu}$ and any $Y$-ergodic measure $\nu\in\mathscr{V}_{\mu}$, we have $\nu(V^+_{\mu})> 1-\delta^2$ and $\|\psi^{Y,*}_{-k_{\mu}}|_{\N^{Y,u}_x}\|<e^{-k_{\mu}\chi_0/2}$ holds for $\nu$-full measure points $x$ in $V^+_{\mu}$.}  
\smallskip
\begin{proof}	
	Recall that $\mu$ is $\chi_0$-hyperbolic. Thus  there exist $\ell_0\ge 1$ and $\varepsilon_0\ll \chi_0$ such that $\mu(\Pes_{\chi_0}(\ell_0,\varepsilon_0))>1-\delta^2/4$. Denote $\Lambda:=\Pes_{\chi_0}(\ell_0,\varepsilon_0)\setminus {\rm Int}(M(\Sing(X),2r))$. Then $\Lambda$ is a compact set satisfying $\mu(\Lambda)>1-\delta^2/2$.
	By the Definition of Pesin blocks, there is $k_{\mu}$ such that $\|\psi^*_{-k_{\mu}}|_{\N^u_x}\|<e^{-k_{\mu}\chi_0/2}$ for any $x\in\Lambda$.

	Assume that $\mu\in {\mathscr{L}}^i$. We state that there is a $C^1$ neighborhood $\U_{\mu}\subset\U_0$ of $X$ and an open neighborhood $V^+_{\mu}\subset M\setminus M(\Sing(X),r)$ of $\Lambda$ such that 
	 for any $Y\in \U_{\mu}$ and any $Y$-ergodic measure $\nu\in\mathscr{V}_0^i$, $\|\psi^{Y,*}_{-k_{\mu}}|_{\N^{Y,u}_x}\|<e^{-k_{\mu}\chi_0/2}$ holds for $\nu$-full measure points $x$ in $V^+_{\mu}$. Otherwise, there exist  sequences of $Y_n\xrightarrow{C^1} X$ and of $x_n\in O(Y_n)$  with $u^{Y_n}(x_n)=i$ such that $x_n\to x\in \Lambda$ and $\|\psi^{Y_n,*}_{-k_{\mu}}|_{\N^{Y_n,u}_{x_n}}\|\ge e^{-k_{\mu}\chi_0/2}$. 
	Since $x\not\in \Sing(X)$, by the continuity of the dominated splitting on the normal bundle implied by Theorem~\ref{thm.liao}, we have $\N^{Y_n,u}_{x_n}\to \N^{u}_x$. It implies $\|\psi^*_{-k_{\mu}}|_{\N^u_{x}}\|\ge e^{-k_{\mu}\chi_0/2}$. This contradiction proves the statement.
	 
	 Since $V^+_\mu$ is open, there is a neighborhood $\mathscr{V}_{\mu}\subset \mathscr{V}_0$ of $\mu$ such that $\nu(V^+_\mu)>1-\delta^2$ for any $\nu\in  \mathscr{V}_{\mu}$.
The proof is finished.
\end{proof}

Since $\mathscr{L}$ is compact, there exist finitely many $\mu_1,\ldots,\mu_m$ such that $\mathscr{L} \subset \bigcup_{j=1}^m \mathscr{V}_{\mu_j}$. Define $\U:=\bigcap_{j=1}^m \U_{\mu_j}$,  ${\mathscr{V}} := \bigcup_{j=1}^m \mathscr{V}_{\mu_j}$ and $k_0 := \max\{k_{\mu_1},\ldots, k_{\mu_m}\}$.  
Then for any $Y\in \U$, $\nu \in {\mathscr{V}} \cap \Pe(Y)$, there exists $1 \le j \le m$ such that $\nu \in \mathscr{V}_{\mu_j}$. We just need to let $k=k_{\mu_j}$ and $U^+$ be  the $\nu$-full measure subset of $V^+_{\mu_j}$ in the above claim.	
\end{proof}

\begin{proof}[Proof of Theorem~\ref{thm.main.starrm}]
Based on Lemma~\ref{lem.cptmeasstar}, we may reproduce the proof procedure of Theorem~\ref{thm.main.3flowrm} to establish Theorem~\ref{thm.main.starrm}. The details are omitted.
\end{proof}

%%%%%%%%%%%%%%%%%%%%%%%%%%%%%%%%%%%%%%%%%%%%%%%%%%%%%%%%%%%%%%%%%%%
\appendix

\part*{Appendix: standard proofs}

\addcontentsline{toc}{part}{Appendix: standard proofs}

\section{Supplementary proofs for Part~II}\label{appsec.coding}

\begin{proof}[Proof of Proposition~\ref{strmnfd}]
	The proof follows the argument in \cite[Proposition~4.8]{LMN} and \cite[Proposition~4.8]{BCL}. The main new technical ingredient here is the control of the distortion of the Poincar\'e return time between nearby orbits of the singular flow.

	 For any $ \xi\in\bR^{d^s(x_0)}(p_0^s)$, define $\wh t_n(\xi):=t_n(\Psi_{x_0}( \xi,G( \xi)))$. Then we have
 $\wh t_0\equiv 0$ and
	\[
	\wh t_n(\xi)=\sum_{i=0}^{n-1}\tau_{x_i}(g^{+}_{x_0,i}(\Psi_{x_0}( \xi,G( \xi)))) \quad \text{and} \quad \wh t_{-n}(\xi)=-\sum_{i=-n}^{-1}\tau_{x_i}(g^{-}_{x_0,i}(\Psi_{x_0}( \xi,G( \xi)))),
	\]
for any $n\ge 1$.		
Let $x=\Psi_{x_0}( \xi_0,G( \xi_0))$ be the shadowed point, where $ \xi_0\in\bR^{d^s(x_0)}(p_0^s)$.	Define $\Delta_n:\bR^{d^s(x_0)}(p_0^s)\to\mathbb{R}$ by $\Delta_n( \xi):=\wh t_n(\xi)-\wh t_n(\xi_0)$. We claim that the sequence $\{\Delta_n(\xi)\}_{n\ge 0}$ is uniformly Cauchy, and hence converges uniformly. Indeed, for any $ \xi\in\bR^{d^s(x_0)}(p_0^s)$ and $0<n<m$,
	\begin{equation}\label{Delta}
		|\Delta_n( \xi)-\Delta_m( \xi)|=\sum_{k=n+1}^{m}|\tau_{x_k}(g^{+}_{x_0,k}(\Psi_{x_0}( \xi,G( \xi))))-\tau_{x_k}(g^{+}_{x_0,k}(\Psi_{x_0}( \xi_0,G( \xi_0))))|.
	\end{equation}
	
	Assume that $x_k\in D_{j_k}$. For any $\xi'\in \bR^{d^s(x_k)}(p_k^s)$ and $y:=\Psi_{x_k}( \xi',G( \xi'))$, by \eqref{eq.boundtimedg} and \eqref{eq.BdPsi}, we have $|\tau_{x_k}(y)-\tau_{x_k}(x_k)|\le \frac{\dist(y,x_k)}{|X(o_{j_k})|}\le 8p_k^s$. It follows from $\tau_{x_k}(x_k)=R_{\mathbb D}(x_k)$ that $\frac{1}{2}\inf R_{\mathbb D}\le \tau_{x_k}(y)\le \frac{4}{3}\sup R_{\mathbb D}$ for all $\varepsilon>0$ small enough. Combining with Proposition~\ref{invmnfd}~(3a), each term in \eqref{Delta} is bounded by
	\[
		\big|\tau_{x_k}(g^{+}_{x_0,k}(\Psi_{x_0}( \xi,G( \xi))))-\tau_{x_k}(g^{+}_{x_0,k}(\Psi_{x_0}( \xi_0,G( \xi_0))))\big|
		\leq\frac{1}{|X(o_{j_k})|}\cdot16 p^s_0 |X(o_{j_k})| e^{-\frac{2\chi}{3}\wh t_k(\xi)}\le  16p^s_0 e^{-\frac{\chi\inf R_{\mathbb D}}{3}k}.
	\]
It follows immediately that ${\Delta_n(\xi)}$ converges uniformly to a well-defined function $\Delta:\bR^{d^s(x_0)}(p_0^s)\to\mathbb{R}$. Define $E_n:=\Delta-\Delta_n$. A similar argument shows that $|E_n( \xi)|<\varepsilon e^{-\frac{\chi}{2}n}$ for all $n\geq0$ and $ \xi\in\bR^{d^s(x_0)}(p_0^s)$.
	
	Let $\wh{V}^s:=\{\varphi_{\Delta( \xi)}(\Psi_x( \xi,G( \xi))): \xi\in\bR^{d^s(x_0)}(p_0^s)\}$.
	For any $\wh{y}, \wh{z}\in \wh{V}^s$, there exist $y=\Psi_x( \xi_1,G( \xi_1))$ and $z=\Psi_x( \xi_2,G( \xi_2))$ for some $ \xi_1,\xi_2\in\bR^{d^s(x_0)}(p_0^s)$, such that $\wh{y}=\varphi_{\Delta( \xi_1)}(y)$ and $\wh{z}=\varphi_{\Delta( \xi_2)}(z)$. For any $t\geq0$, there exists a unique $n\geq0$ such that $t_{n}(\xi_0)\le t< t_{n+1}(\xi_0)$. Therefore,
	\begin{equation}\label{strongdis}
	\begin{aligned}
			\dist(\varphi_t(\wh{y}),\varphi_t(\wh{z}))
			&\leq
			\dist(\varphi_{t-t_n(\xi_0)+E_n(\xi_1)}(g^{+}_{x,n}(y)),\varphi_{t-t_n(\xi_0)+E_n(\xi_1)}(g^{+}_{x,n}(z)))\\
			&+\dist(\varphi_{t-t_n(\xi_0)+E_n(\xi_1)}(g^{+}_{x,n}(z)),\varphi_{t-t_n(\xi_0)+E_n(\xi_2)}(g^{+}_{x,n}(z))).
	\end{aligned}
	\end{equation}
	
	Since $|t-t_n(\xi_0)+E_n(\xi_i)|<\rho+\varepsilon$, 
	by  Proposition~\ref{invmnfd}~(3a) and the fact that $\norm{d\varphi_t}\le e^{\frac{|t|}{2}}$, we have
	\[
		\dist(\varphi_{t-t_n(\xi_0)+E_n(\xi_1)}(g^{+}_{x,n}(y)),\varphi_{t-t_n(\xi_0)+E_n(\xi_1)}(g^{+}_{x,n}(z)))
		\leq e^{\frac{\rho+\varepsilon }{2}}\cdot 16p_0^s |X(o_{j_n})|\cdot e^{-\frac{\chi\inf R_{\mathbb D}}{3}n}.
	\] 
	Recall that $|E_n(\xi)|<\varepsilon e^{-\frac{\chi}{2}n}< \varepsilon e^{-\frac{\chi\inf R_{\mathbb D}}{3}n}$. By Lemma~\ref{lem.C1Poincare}, we have
	\[
	\dist(\varphi_{t-t_n(\xi_0)+E_n(\xi_1)}(g^{+}_{x,n}(z)),\varphi_{t-t_n(\xi_0)+E_n(\xi_2)}(g^{+}_{x,n}(z)))\leq \frac{10}{9}|X(o_{j_n})|\cdot 2\varepsilon e^{-\frac{\chi\inf R_{\mathbb D}}{3}n}.
	\]
	Substituting these two bounds into inequality \eqref{strongdis}, we obtain that for all $\varepsilon>0$ small enough, 
	\[
	\dist(\varphi_t(\wh{y}),\varphi_t(\wh{z}))\leq(16p_0^s\cdot e^{\frac{\rho+\varepsilon }{2}}+\frac{20}{9}\varepsilon )|X(o_{j_n})|\cdot e^{-\frac{\chi\inf R_{\mathbb D}}{3}n}<|X(x_n)|\cdot e^{-\frac{\chi\inf R_{\mathbb D}}{3}(n+1)}.
	\]
	Combined with the bound $t\leq\frac{4}{3} (n+1)\sup R_{\mathbb D}$, this immediately yields 
	\[
	\dist(\varphi_t(\wh{y}),\varphi_t(\wh{z}))< |X(x_n)|\cdot e^{-\frac{\chi\inf R_{\mathbb D}}{4\sup R_{\mathbb D}}t}.\]
	
This completes the proof.
\end{proof}

\vspace{5pt}

\begin{proof}[Proof of Lemma~\ref{impv.srat}]	
	Without loss of generality, we assume that $|\omega_0|=1$. By Lemma~\ref{lem.comH} and \eqref{eq.flowGI}, for all $\varepsilon>0$ small enough, there exists a uniform constant $J>1$ such that 
	\begin{equation}\label{eq.bdvectors}
		|\omega_1|,\  |\pi^s_{y_0,x_0}(\omega_0)|,\ |\pi^s_{y_1,x_1}(\omega_1)|\in(J^{-1},J).
	\end{equation}

	Let $\tau_0 \in (0,2\rho)$  satisfy $g^{-}_{x_1}(x_1)=f^{-1}(x_1)=\varphi_{-\tau_0}(x_1)$. Denote $\lambda_0:=\Phi^*_{-\tau_0}(\pi^s_{y_1,x_1}(\omega_1))\in\H^s_{f^{-1}(x_1)}$. Then
	\begin{equation}\label{ssratio}
		\frac{s(y_0,\omega_0)}{s(x_0,\pi^s_{y_0,x_0}(\omega_0))}=\frac{s(f^{-1}(x_1),\lambda_0)}{s(x_0,\pi^s_{y_0,x_0}(\omega_0))}\cdot\frac{s(y_0,\omega_0)}{s(f^{-1}(x_1),\lambda_0)}. 
	\end{equation}
We have the following estimates:	
	\begin{itemize}
		\item Part 1: $\frac{s(f^{-1}(x_1),\lambda_0)}{s(x_0,\pi^s_{y_0,x_0}(\omega_0))}=e^{\pm Q(x_0)^{\beta/4}}$.
	\end{itemize}
	
	By the definitions of $s(\cdot,\cdot)$, $C(\cdot)$, and $\wt{C}(\cdot)$, we have 
	\[
	\begin{aligned}
	s(x_0,\pi^s_{y_0,x_0}(\omega_0))&=|C(x_0)^{-1}(\pi^s_{y_0,x_0}(\omega_0))|=|\wt{C}(x_0)^{-1}\circ d_{x_0}\exp_{o_i}^{-1}(\pi^s_{y_0,x_0}(\omega_0))|,\text{ and}\\	
	s(f^{-1}(x_1),\lambda_0)&=|C(f^{-1}(x_1))^{-1}(\lambda_0)|=|\wt{C}(f^{-1}(x_1))^{-1}\circ d_{f^{-1}(x_1)}\exp_{o_i}^{-1}(\lambda_0)|.
	\end{aligned}
	\]
	Thus
	\[
		\Norm{\frac{s(f^{-1}(x_1),\lambda_0)}{s(x_0,\pi^s_{y_0,x_0}(\omega_0))}-1}\le\frac{\big|\wt{C}(x_0)^{-1}\circ d_{x_0}\exp_{o_i}^{-1}(\pi^s_{y_0,x_0}(\omega_0))-\wt{C}(f^{-1}(x_1))^{-1}\circ d_{f^{-1}(x_1)}\exp_{o_i}^{-1}(\lambda_0)\big|}{|C(x_0)^{-1}(\pi^s_{y_0,x_0}(\omega_0))|}.
	\]
	By \eqref{eq.bdvectors} and Lemma~\ref{lem.propofC}~\ref{fC.norm}, the denominator is uniformly bounded from below by $1/J$.
	To bound the numerator, we split it into three parts:
	\[
	\big|\wt{C}(x_0)^{-1}\circ d_{x_0}\exp_{o_i}^{-1}(\pi^s_{y_0,x_0}(\omega_0))-\wt{C}(f^{-1}(x_1))^{-1}\circ d_{f^{-1}(x_1)}\exp_{o_i}^{-1}(\lambda_0)\big|\le s_1+s_2+s_3,
	\]
	where 
	\[\begin{aligned}
		&s_1:=\|\wt{C}(x_0)^{-1}-\wt{C}(f^{-1}(x_1))^{-1}\|\cdot |d_{x_0}\exp_{o_i}^{-1}(\pi^s_{y_0,x_0}(\omega_0))|,\\
		&s_2:=\|\wt{C}(f^{-1}(x_1))^{-1}\|\cdot|d_{x_0}\exp_{o_i}^{-1}(\pi^s_{y_0,x_0}(\omega_0))-d_{y_0}\exp_{o_i}^{-1}(\omega_0)|,\\
		&s_3:=\|\wt{C}(f^{-1}(x_1))^{-1}\|\cdot|d_{y_0}\exp_{o_i}^{-1}(\omega_0)-d_{f^{-1}(x_1)}\exp_{o_i}^{-1}(\lambda_0)|.
	\end{aligned}\]
	
	\smallskip
	\noindent$\circ$ Estimate on $s_1$:
	By \eqref{eq.bdvectors}, we have  $|d_{x_0}\exp_{o_i}^{-1}(\pi^s_{y_0,x_0}(\omega_0))|\leq 2J$. By Proposition~\ref{pro.overlap}~\ref{ol.Ccomp}, we obtain $\big\|\wt{C}(x_0)^{-1}-\wt{C}(f^{-1}(x_1))^{-1}\big\|\le (\varepsilon Q(x_0)\cdot\varepsilon Q(x_0))^3\le \varepsilon^{6} Q(x_0)^6$.
Therefore, $s_1\le 2J \varepsilon^{6} Q(x_0)^6$.	
	
	\smallskip
	\noindent$\circ$ Estimate on $s_2$:
	By Lemma~\ref{lem.comH}, we have $|d_{x_0}\exp_{o_i}^{-1}(\pi^s_{y_0,x_0}(\omega_0))-d_{y_0}\exp_{o_i}^{-1}(\omega_0)|\leq \varepsilon^{1/16}\cdot (\varepsilon Q(x_0))^{5\beta/16}$. 
	It follows from Proposition~\ref{pro.overlap}~\ref{ol.Ccomp} that $\big\|\wt{C}(f^{-1}(x_1))^{-1}\big\|\le \big\|\wt{C}(x_0)^{-1}\big\|e^{(\varepsilon Q(x_0)\cdot\varepsilon Q(x_0))^3}$.
Combining these estimates with Lemma~\ref{lem.estimateQ}~\ref{est.QC}, we conclude that 
$s_2\le e^{\varepsilon^{6+36/\beta}}\varepsilon^{\frac{3+5\beta}{16}}Q(x_0)^{\frac{7\beta}{24}}$.

 	\smallskip
	\noindent$\circ$ Estimate on $s_3$:
	Recall that $\lambda_0=\Phi^*_{-\tau_0}(\pi_{y_1,x_1}^s(\omega_1))$ and $\omega_0=\Phi^*_{-\tau_1}(\omega_1)$. By Lemma~\ref{hpsi}~\ref{Phi.rel}, we have 
	\[
		d_{f^{-1}(x_1)}\exp_{o_i}^{-1}(\lambda_0)=\frac{|X(o_j)|}{|X(o_i)|}\cdot  d_{\wt{x_1}^*}\wt g^{-,*}_{x_1}(\omega_1^{x_1}) \quad \text{and} \quad 
		d_{y_0}\exp_{o_i}^{-1}(\omega_0)=\frac{|X(o_j)|}{|X(o_i)|}\cdot  d_{\wt{y_1}^*}\wt g^{-,*}_{x_1}(\omega_1^{y_1}),
	\]
	where $\omega_1^{x_1}=d_{x_1}\exp^{-1}_{o_j}(\pi^s_{y_1,x_1}(\omega_1))$, $\omega_1^{y_1}=d_{y_1}\exp^{-1}_{o_j}(\omega_1)$. Since $\frac{|X(o_j)|}{|X(o_i)|}\in (\frac{9}{10},\frac{10}{9})$, we obtain
	\[
	|d_{y_0}\exp_{o_i}^{-1}(\omega_0)-d_{f^{-1}(x_1)}\exp_{o_i}^{-1}(\lambda_0)|\le \tfrac{10}{9}\big(\|d_{\wt {y_1}^*}\wt g^{-,*}_{x_1}\|\cdot|\omega_1^{x_1}-\omega_1^{y_1}|+\|d_{\wt {x_1}^*}\wt g^{-,*}_{x_1}-d_{\wt {y_1}^*}\wt g^{-,*}_{x_1}\|\cdot|\omega_1^{x_1}|\big).
	\]
	Recall that  $\|d\exp^\pm\|\le 2$. It follows from \eqref{eq.boundtimedg} and Lemma~\ref{lem.comH} that $\|d_{\wt z^*}\wt g^{-,*}_y\|\leq 80/9$ and $|\omega_1^{x_1}-\omega_1^{y_1}|\leq \varepsilon^{1/16}\cdot (p_1^s\wedge p_1^u)^{5\beta/16}$. By \eqref{eq.GPO22}, we have $p_1^u\le e^{\varepsilon\rho}p_0^u\le e^{\varepsilon\rho}\varepsilon Q(x_0)$. Therefore, $\|d_{\wt {y_1}^*}\wt g^{-,*}_{x_1}\|\cdot|\omega_1^{x_1}-\omega_1^{y_1}|\le \frac{80}{9}\varepsilon^{\frac{1}{16}}e^{\frac{5\beta\varepsilon\rho}{16}}Q(x_0)^{\frac{5\beta}{16}}$.
	
	On the other hand, by \eqref{eq.bdvectors} and \eqref{eq.gstar}, we have $|\omega_1^{x_1}|\leq 2J$ and $\|d_{\wt {x_1}^*}\wt g^{-,*}_{x_1}-d_{\wt {y_1}^*}\wt g^{-,*}_{x_1}\|\leq K_F\cdot |\wt {x_1}^*-\wt {y_1}^*|^{\beta}$. By Lemma~\ref{lem.C1Poincare}~\ref{fb.sectPoincare}, $|\wt {x_1}^*-\wt {y_1}^*|\le \frac{40}{9}|\wt{f^{-1}(x_1)}^*-\wt{y_0}^*|$. It follows from $\wt{f^{-1}(x_1)}^*, \wt{y_0}^*\in \wt\Psi^*_{x_0}(R[p_0^s\wedge p_0^u])$ that $|\wt{f^{-1}(x_1)}^*-\wt{y_0}^*|\le 8\varepsilon Q(x_0)$.	 Combining the above estimates with Lemma~\ref{lem.estimateQ}~\ref{est.QC}, we conclude that $\|d_{\wt {x_1}^*}\wt g^{-,*}_{x_1}-d_{\wt {y_1}^*}\wt g^{-,*}_{x_1}\|\cdot|\omega_1^{x_1}|\le 2JK_F(\frac{320}{9})^\beta \varepsilon^{\beta} Q(x_0)^\beta$.
	
	\smallskip	

	To sum up, we obtain that for all $\varepsilon>0$ small enough, $\Big|\frac{s(f^{-1}(x_1),\lambda_0)}{s(x_0,\pi^s_{y_0,x_0}(\omega_0))}-1\Big|<\frac{1}{2}Q(x_0)^{\beta/4}$ and thus
	$\frac{s(f^{-1}(x_1),\lambda_0)}{s(x_0,\pi^s_{y_0,x_0}(\omega_0))}=e^{\pm Q(x_0)^{\beta/4}}$.
	
	\begin{itemize}
		\item { Part 2}: $\frac{s(y_0,\omega_0)}{s(f^{-1}(x_1),\lambda_0)}=e^{\pm(\zeta-2 Q(x_0)^{\beta/4})}$.
	\end{itemize}
	
	By the definition of $s(\cdot,\cdot)$, we have
	
	\[
	s(f^{-1}(x_1),\lambda_0)^2=I_1+e^{2\chi\tau_0}s(x_1, \pi_{y_1,x_1}^s(\omega_1))^2  \quad \text{and} \quad  s(y_0,\omega_0)^2=I_2+ e^{2\chi\tau_1}s(y_1,\omega_1)^2,
	\]
	where $I_1=4e^{2\rho}\int_{0}^{\tau_0}e^{2\chi t}|\Phi^*_{t}(\lambda_0)|^{2}dt$ and $I_2=4e^{2\rho}\int_{0}^{\tau_1}e^{2\chi t}|\Phi^*_{t}(\omega_0)|^{2}dt$. By the same argument as \eqref{eq.bdvectors}, we may assume without loss of generality that $|\Phi^*_{t}(\lambda_0)|$, $|\Phi^*_{t}(\omega_0)|\in(J^{-1},J)$ for all $t\in [0,2\rho]$. Since $\tau_0, \tau_1<2\rho$, there exists a constant $J_0=J_0(\chi)>1$ such that $I_1,I_2\in[J_0^{-1},J_0]$.	Moreover, since $\frac{s(y_1,\omega_1)}{s(x_1,\pi_{y_1,x_1}^s(\omega_1))}=e^{\pm\zeta}$, we have
	\begin{equation}\label{est.srat}
		\frac{s(y_0,\omega_0)^2}{s(f^{-1}(x_1),\lambda_0)^2}\le \frac{I_2+e^{2\chi\tau_1}e^{2\zeta}s(x_1, \pi_{y_1,x_1}^s(\omega_1))^2}{I_1+e^{2\chi\tau_0}s(x_1, \pi_{y_1,x_1}^s(\omega_1))^2}=e^{2\zeta+2\chi(\tau_1-\tau_0)}\Big(1-\frac{I_1-e^{-2\zeta+2\chi(\tau_0-\tau_1)}I_2}{s(f^{-1}(x_1),\lambda_0)^2}\Big).
	\end{equation}
	
By \eqref{eq.boundtimedg}, \eqref{eq.bdPsi},  Lemma~\ref{lem.estimateQ}~\ref{est.QC}, and the fact $f^{-1}(x_1), y_0\in \Psi_{x_0}(R[p_0^s\wedge p_0^u])$, we have
	\begin{equation}\label{im.cl0}
		|\tau_0-\tau_1|\le \frac{1}{|X(o_i)|}\dist(f^{-1}(x_1),y_0)\le 16\varepsilon Q(x_0)\le 16 \varepsilon^{1+\frac{6}{\beta}} \ll 1.
	\end{equation}
	
	\smallskip
	\smallskip
	\nt {\bf Claim.} {\it $|I_1-I_2| \le \varepsilon I_2$ for all $\varepsilon>0$ small enough.}
	\begin{proof}
	We first have
	\[
	\begin{aligned}
	|I_1-I_2|\le\,& \Big|4e^{2\rho}\int_{0}^{\tau_0}e^{2\chi t}|\Phi^*_{t-\tau_0}(\pi^s_{y_1,x_1}(\omega_1))|^2dt-4e^{2\rho}e^{2(\tau_1-\tau_0)\chi}\int_{\tau_0-\tau_1}^{\tau_0}e^{2\chi t}|\Phi^*_{t-\tau_0}(\omega_1)|^{2}dt\Big|\\
    \le\,& 4e^{2\rho}\int_{0}^{\tau_0}e^{2\chi t}\Big||\Phi^*_{t-\tau_0}(\pi^s_{y_1,x_1}(\omega_1))|^2-|\Phi^*_{t-\tau_0}(\omega_1)|^{2}\Big|dt+4e^{2\rho}\Big|\int_{\tau_0-\tau_1}^{0}e^{2\chi t}|\Phi^*_{t-\tau_0}(\omega_1)|^{2}dt\Big|\\
    &+4e^{2\rho}\Big|1-e^{2(\tau_1-\tau_0)\chi}\Big|\int_{\tau_0-\tau_1}^{\tau_0}e^{2\chi t}|\Phi^*_{t-\tau_0}(\omega_1)|^{2}dt\\
    \le\,& 4e^{2\rho}\int_{0}^{\tau_0}e^{2\chi t}\Big(|\Phi^*_{t-\tau_0}(\pi^s_{y_1,x_1}(\omega_1))|+|\Phi^*_{t-\tau_0}(\omega_1)|\Big)\Big||\Phi^*_{t-\tau_0}(\pi^s_{y_1,x_1}(\omega_1))|-|\Phi^*_{t-\tau_0}(\omega_1)|\Big|dt\\
    &+4e^{2\rho}|\tau_1-\tau_0|e^{4\chi \rho}J^{2}+4e^{2\rho}\cdot 4\chi|\tau_1-\tau_0|\cdot4\rho e^{4\chi \rho}J^{2}.	
    \end{aligned}
	\]	
Denote $\lambda_1:=\pi^s_{y_1,x_1}(\omega_1)$. By Lemma~\ref{lem.comH}~\ref{comH3} and \eqref{eq.GPO22}, we have $\big||\lambda_1|-|\omega_1|\big|\le\big|\pi^s_{y_1,x_1}(\omega_1)-T_{y_1,x_1}(\omega_1)\big|<\varepsilon^{1/16}\cdot (p_1^s \wedge p_1^u)^{5\beta/16}\cdot J\le \varepsilon^{1/16} \cdot e^{5\beta\varepsilon\rho/16}\varepsilon^{5\beta/16} \varepsilon^{15/8}\cdot J$.	
 By Proposition~\ref{hpsi}~\ref{Phi.hol}, we conclude that
\[
\begin{aligned}
\Big||\Phi^*_{t-\tau_0}\big(\frac{\lambda_1}{|\lambda_1|}\big)|-|\Phi^*_{t-\tau_0}\big(\frac{\omega_1}{|\omega_1|}\big)|\Big|&\leq {K}_{F3}\bigg(\Big(\frac{\dist(x_1,y_1)}{|X(x_1)|}\Big)^{\beta}+\Big|\frac{\lambda_1}{|\lambda_1|}-T_{y_1,x_1}\big(\frac{\omega_1}{|\omega_1|}\big)\Big|\bigg)\\
&\le {K}_{F3}\bigg(\Big(\frac{80}{9}\cdot (p_1^s \wedge p_1^u)\Big)^{\beta}+\Big|\frac{\lambda_1}{|\lambda_1|}-\frac{\lambda_1}{|\omega_1|}\Big|+\frac{|\lambda_1-T_{y_1,x_1}(\omega_1)|}{|\omega_1|}\bigg)\\
&\le {K}_{F3}\bigg(\Big(\frac{80}{9}\Big)^\beta e^{\varepsilon\rho\beta}\varepsilon^\beta \varepsilon^6 + 2J^2 \varepsilon^{1/16} e^{5\beta\varepsilon\rho/16}\varepsilon^{5\beta/16} \varepsilon^{15/8} \bigg).
\end{aligned}
\]	
	
Note that $\big||\Phi^*_{t-\tau_0}(\lambda_1)|-|\Phi^*_{t-\tau_0}(\omega_1)|\big|\le 2J$ and
\[
\Big||\Phi^*_{t-\tau_0}(\lambda_1)|-|\Phi^*_{t-\tau_0}(\omega_1)|\Big|\le |\lambda_1|	\Big||\Phi^*_{t-\tau_0}\big(\frac{\lambda_1}{|\lambda_1|}\big)|-|\Phi^*_{t-\tau_0}\big(\frac{\omega_1}{|\omega_1|}\big)|\Big|+ \frac{\big||\lambda_1|-|\omega_1|\big|}{|\omega_1|}|\Phi^*_{t-\tau_0}(\omega_1)|.
\]
Therefore, for all $\varepsilon>0$ small enough, we have $|I_1-I_2|\le J_0^{-1} \varepsilon\le \varepsilon I_2$.
	\end{proof}
	
Recall that $\zeta\ge \sqrt{\varepsilon}$. It follows from \eqref{im.cl0} that 	
	$2\zeta-2\chi(\tau_0-\tau_1)\ge \sqrt{\varepsilon}$ for all $\varepsilon>0$ small enough. Thus
	\[
	I_1-e^{-2\zeta+2\chi(\tau_0-\tau_1)}I_2\ge (1-e^{-2\zeta+2\chi(\tau_0-\tau_1)})I_2-|I_1-I_2|\ge (1-e^{-\sqrt{\varepsilon}})I_2-\varepsilon I_2\ge (\sqrt{\varepsilon}/2-\varepsilon)I_2\ge \varepsilon J_1^{-1}.
	\]	
On the other hand, by Proposition~\ref{pro.overlap}~\ref{ol.Ccomp} and \eqref{eq.bdPsi}, we have 
\[
s(f^{-1}(x_1),\lambda_0)^2=|\wt{C}(f^{-1}(x_1))^{-1}\circ d_{f^{-1}(x_1)}\exp_{o_i}^{-1}(\lambda_0)|^2\le 4J \big\|{C}(x_0)^{-1}\big\|e^{(\varepsilon Q(x_0)\cdot\varepsilon Q(x_0))^3}.
\]
It follows from Lemma~\ref{lem.estimateQ}~\ref{est.QC} that $8Q(x_0)^{\beta/4}s(f^{-1}(x_1),\lambda_0)^2\le \varepsilon J_1^{-1}$. Then by \eqref{est.srat} and \eqref{est.srat}, for all $\varepsilon>0$ small enough, we conclude that
\[
\frac{s(y_0,\omega_0)^2}{s(f^{-1}(x_1),\lambda_0)^2}\le e^{2\zeta+2\chi(\tau_1-\tau_0)}\big(1-8Q(x_0)^{\beta/4}\big)\le e^{2\zeta-4Q(x_0)^{\beta/4}}.
\] 

The lower bound for the ratio follows by swapping the roles of numerator and denominator and applying the identical estimates. As a summary, we have $\frac{s(y_0,\omega_0)}{s(f^{-1}(x_1),\lambda_0)}=e^{\pm(\zeta-2 Q(x_0)^{\beta/4})}$.
	
	Substituting the estimates of Part 1 and Part 2 into \eqref{ssratio}, we complete the proof directly.
\end{proof}

\vspace{5pt}

\begin{proof}[Proof of Lemma~\ref{impv.fin}]
The proof of this lemma is similar to that of \cite[Lemma 6.5]{LMN}. Here we only present the outline of the proof. Recall that $s(\cdot,\cdot)$, which is defined in \eqref{equ.extenddefs}, depends on $\chi$. We use the notation $s_{\chi}(\cdot,\cdot)$ throughout the proof to make this dependence explicit. Note that $s_{\chi}(x,v)=\sup_{\chi'<\chi}s_{\chi'}(x,v)$. Take an arbitrary $\chi'\in(0,\chi)$. Denote $\ol{\chi}:=(\chi+\chi')/2$, and choose $\gamma\in(0,1)$ such that $\gamma^{\beta/3}<\inf \{e^{-t\ol{\chi}}-e^{-t\chi}:t\in[\inf R_{\mathbb{D}},2\rho]\}$. 

Assume that $z\in V^s\cap\NUH_{\chi}^{\#}$ and denote $z_n:=g^{+}_{x_0,n}(z)$, where $g^{+}_{x_0,n}$ is defined in \S\,\ref{subsec.graphtran}. It is possible that $z_n\notin\mathbb{D}$ but $z_n\in\widehat{\mathbb{D}}$. The coordinate maps $\wt\Psi^{*}_{z_n}$ and $\Psi_{z_n}$ can be defined similarly and satisfy corresponding properties stated in Section~\ref{subsec.scalpesinchart}.
By the definition of $\NUH_{\chi}^{\#}$ and Lemma~\ref{lem.estimateQ}~\ref{est.QC}, there are $q>0$ and an increasing sequence $\{n_k\}_{k\geq0}$ such that for every $k\geq0$,
\begin{equation}\label{eq.olzq}
Q(f^{n_k}(z))>q(f^{n_k}(z))\geq q  \quad \text{and} \quad \norm{C(f^{n_k}(z))^{-1}}<q^{-\beta/48}.
\end{equation}

Define $\wt v^*_{\bar{\chi},n}:=\wt\Psi^{*,\gamma q^s(z_n),\gamma q^u(z_n)}_{z_n}$. 
Note that $\{\wt v^*_{\bar{\chi},n}\}$ is not an $\varepsilon$-gpo for either $\NUH_{\chi}$ or $\NUH_{\ol{\chi}}$. However, since Proposition~\ref{pro.Pesinlift} remains true and the chart distortion is  dominated by the hyperbolic rate of the Poincar\'e maps, the graph transform can be performed along $\{\wt v^*_{\bar{\chi},n}\}$. The detailed proof is in \cite[Appendix~A]{ALP}. 

Denote by $\wt{V}^{*,s}_{\ol{\chi},n}:=\wt V^{*,s}[\{\wt v^{*}_{\bar{\chi},\ell}\}_{\ell\ge n}]$
the $s$-admissible manifold at $\wt{v}^*_{\bar{\chi},\ell}$. Let $V^s_{\ol{\chi},n}\subset \widehat{\mathbb{D}}$ be the projection of $\wt{V}^{*,s}_{\ol{\chi},n}$, and $G_n$ be the corresponding representing function.
By using the same argument as that of Proposition~\ref{2mnfd}, we have $g^{+}_{x_0,n}(V^s)\subset V^s_{\bar{\chi},n}$ for all sufficiently large $n$.

We first estimate $s(\cdot)$ on $V^s_{\ol{\chi},0}$. For any $y\in V^s_{\ol{\chi},0}$ and $\omega\in T_y  V^s_{\ol{\chi},0}$, denote $y_n=g^{+}_{x_0,n}(y)$ and $\omega_n=d_yg^{+}_{x_0,n}(\omega)$. Write $\Psi_{z_n}^{-1}(y_n)=\theta_n=(\theta_n^s,G_n(\theta_n^s))$ and $\omega_n=d_{\theta_n}\Psi_{z_n}(\xi_n,d_{\theta_n^s}G_n(\xi_n))$. By the graph transform, we have 
$(\xi_{n+1},d_{\theta_{n+1}^s}G_{n+1}(\xi_{n+1}))=d_{y_n}g^{+}_{x_n} (\xi_n,d_{\theta_n^s}G_n(\xi_n))$. It follows from Proposition~\ref{pro.overlift} and \ref{ite.AMcomd2} that
\[
|\xi_{n+1}|\leq (e^{-\chi\tau_{x_n}(z_n)}+2\varepsilon(\gamma q^s(z_n))^{\beta/3})|\xi_n|<e^{-\ol{\chi}\tau_{x_n}(z_n)}|\xi_n|
\]
where $\tau_{x_n}$ is the time function from chart  $\Psi^{p^s_n,p^u_n}_{x_n}$ to $\Psi^{p^s_{n+1},p^u_{n+1}}_{x_{n+1}}$.
For any $y\in \Psi^{p^s_0,p^u_0}_{x_0}$, denote $t_{x_n,\ell}(y):=\tau_{x_n}(g^{+}_{x_0,n}(y))+\dots+\tau_{x_{n+\ell-1}}(g^{+}_{x_0,{n+\ell-1}}(y))$ whenever it is well-defined. Recall that $g^+_{x_n,\ell}=g^+_{x_{n+\ell-1}}\circ\dots\circ g^+_{x_n}$.

Assume that $x_n\in D_{j_n}$.  For any $n,\ell \ge 0$, $y'\in V^s_{\ol{\chi},n}$, and $\omega'\in T_{y'}V^s_{\ol{\chi},n}$ with $|\omega'|=1$, by repeating the proof of Proposition~\ref{invmnfd}, we have 
\begin{equation}\label{propty.newchi}
|d_{y'}g^{+}_{x_n,\ell}(\omega')|\leq 32\|C(x_n)^{-1}\|\cdot \frac{|X(o_{j_{n+\ell}})|}{|X(o_{j_n})|}\cdot e^{-\ol{\chi}\cdot t_{x_n,\ell}(z)}.
\end{equation}

By \eqref{eq.olzq}, \eqref{propty.newchi}, \eqref{eq.boundtimedg}, \eqref{eq.flowGI}, and Lemma~\ref{lem.estimateQ}~\ref{est.QC}, through standard calculations, we obtain for $n=n_k$,
\begin{equation}\label{eq.simprove}
\begin{aligned}
	s^2_{\chi'}(y',\omega')&=4e^{2\rho}\sum_{\ell=0}^{\infty}e^{2\chi' t_{x_n,\ell}(y')}\int_{0}^{t_{x_n,\ell+1}(y')-t_{x_n,\ell}(y')}e^{2\chi' t}|\Phi^*_t(d_{y'}g^{+}_{x_n,\ell}(\omega'))|^2dt\\
	&\le 8192\rho e^{12\rho} \varepsilon^{1/4}q^{-\beta/24}e^{4\chi'\varepsilon^{1+6/\beta}(1-e^{-\inf R_{\wh{\mathbb D}}\cdot \chi/2})^{-1} }(1-e^{-(\chi-\chi')\inf R_{\wh{\mathbb D}}})^{-1}.
\end{aligned}
\end{equation}

On the other hand, by an identical argument, we can prove that Lemma~\ref{impv.srat} still holds for $s_{\chi'}(\cdot)$. See \cite[Corollary 5.5]{ALP} and \cite[Corollary 6.4]{LMN} for more details. Applying this Improvement Lemma along $z_{n_{k+\ell-1}}\to\cdots\to z_{n_{k+\ell}}$ and using \eqref{eq.simprove}, the ratio improves a constant if it is outside $[e^{-\sqrt{\varepsilon}},e^{\sqrt{\varepsilon}}]$. Taking $\ell\to +\infty$, we obtain that 
\[
\frac{s_{\chi'}(y',\omega')}{s_{\chi'}(z_{n_k},\pi^s_{y',z_{n_k}}(\omega'))}=e^{\pm\sqrt{\varepsilon}}.
\]

Take $k\ge 0$ such that $g^{+}_{x_0,n_k}(V^s)\subset V^s_{\bar{\chi},n_k}$. For any $y\in V^s$ and $\omega\in T_yV^s$ with $|\omega|=1$, we have
\[
s^2_{\chi'}(y,\omega)\le 4e^{2\rho}\int_0^{t_{x_0,k}(y)}e^{2\chi t}|\Phi^*_t(\omega)|^2dt+e^{2\chi t_{x_0,k}(y)}s^2_{\chi'}(y_{n_k},\omega_{n_k}).
\] 
By \eqref{eq.boundtimedg}, \eqref{eq.simprove}, and Lemma~\ref{invmnfd}~(3), we can obtain that $s_{\chi'}(y,\omega)$ is upper bounded by a constant which is independent of $y$, $\omega$, and $\chi'$. Then this constant is an upper bound of $s_{\chi}(y)$ for all $y\in V^s$.
\end{proof}

%%%%%%%%%%%%%%%%%%%%%%%%%%%%%%%%%%%%%%%%%%%%%%%%%%%%%%%%%%%%%%%%%%%%%%%%%%%%%%%%%%%%%%%%%%%

\section{Supplementary proofs for Part~III}\label{appsec.decay}

\begin{proof}[Proof of Lemma~\ref{lem.FMT}]
Recall that $\tau_N$ is a period of $(\ul{p_N},0)$ with respect to $\sigma_r$, and there exist constants $\Pi>0$, $\gamma\in(0,1)$, $\Theta\in[0,2\pi)$, $\Omega>0$ and sequences $\{E_N\}_{N\ge 0}$, $\{\zeta_N\}_{N\ge 0}$ such that
\begin{equation}\label{equ.doogasymp}
	\tau_N=N\tau_0+\Pi+E_N\cdot \gamma^N\cos(N\Theta+\zeta_N)+o(\gamma^N),
\end{equation}
$|E_N|<\Omega$ for every $N\geq 1$, $\liminf_{N\to\infty}|E_N|> 0$, and either  $\Theta=0$ and $\zeta_N\equiv 0$, or  $\Theta\in(0, \pi)$ and $\zeta_N\in(\Theta_0-\pi/12,\Theta_0+\pi/12)$ for some constant $\Theta_0$ and every $N\geq 1$.

There exists $\psi_N\in \bZ^+$ such that $\sigma^{\psi_N}(\ul{p_N})=\ul{p_N}$ and $\tau_N=r_{\psi_N}(\ul{p_N})$. It follows that $\tau_N\ge \psi_N\cdot \inf r$.
For any $k\ge 3$, $\ell\in\bZ^+$ and $\ul{x}\in \mathcal{O}$, by substituting $\sigma^{(\ell-j-1)m_k}(\ul{x})$ for $\ul{x}$ in Lemma~\ref{lem.dirichlet}, we have
\begin{align}
&\, \big|e^{ib_kr_{\ell m_k}(\ul{x})}e^{-i\omega_1^k(\sigma^{\ell m_k}(\ul{x}))}-e^{i\cdot 2\pi \ell\frac{t_k}{u_k}}e^{-i\omega_1^k(\ul{x})}\big|\notag\\
\le &\,\sum_{j=0}^{\ell-1} \big|e^{i\cdot 2\pi j\frac{t_k}{u_k}}\big|\cdot \big|e^{ib_kr_{(\ell-j-1)m_k}(\ul{x})}\big|\cdot \big|e^{ib_kr_{m_k}(\sigma^{(\ell-j-1)m_k}(\ul{x}))}e^{-i\omega_1^k(\sigma^{(\ell-j) m_k}(\ul{x}))}-e^{i\cdot 2\pi \frac{t_k}{u_k}}e^{-i\omega_1^k(\sigma^{(\ell-j-1)m_k}(\ul{x}))}\big|\notag\\
\le &\, \ell\cdot 20 u_k^{-1}|b_k|^{-\frac{\lambda}{2}}.\label{equ.iteratel}
\end{align}

Take $\ell=u_k\cdot\psi_N$ and $\ul{x}=\ul{p_N}$ in \eqref{equ.iteratel}, note that $\sigma^{\ell m_k}(\ul{p_N})=\ul{p_N}$ and $r_{\ell m_k}(\ul{p_N})=u_km_k\tau_N$, we have 
\[
\big|e^{ib_ku_km_k\tau_N}-1\big|=\big|e^{ib_ku_km_k\tau_N}e^{-i\omega_1^k(\ul{p_N})}-e^{i\cdot 2\pi \psi_Nt_k}e^{-i\omega_1^k(\ul{p_N})}\big|\leq 20\psi_N |b_k|^{-\frac{\lambda}{2}}\le  \frac{20\tau_N}{\inf r}\cdot |b_k|^{-\frac{\lambda}{2}}.
\]
Hence
\begin{equation}\label{equ.goodasypN}
	\dist({b_ku_km_k\tau_N},2\pi \bZ)\le \frac{\pi}{2}\cdot \frac{20}{\inf r}\cdot \tau_N|b_k|^{-\frac{\lambda}{2}}.
\end{equation}
Similarly, take $\ell=u_k\cdot N\psi_0$ and $\ul{x}=\ul{p_0}$ in \eqref{equ.iteratel}, we may get
\begin{equation}\label{equ.goodasyp0}
	\dist({b_ku_km_k N\tau_0},2\pi \bZ)\le \frac{\pi}{2}\cdot \frac{20}{\inf r}\cdot N\tau_0|b_k|^{-\frac{\lambda}{2}}.	
\end{equation}

By \eqref{equ.doogasymp}, \eqref{equ.goodasypN} and \eqref{equ.goodasyp0}, there exists $\wt C>0$ independent of $k$ and $N$ such that
\begin{equation}\label{equ.goodasypdiffer}
\begin{aligned}
	&\dist(b_ku_km_k(\Pi+E_N\cdot \gamma^N\cos(N\Theta+\zeta_N)+o(\gamma^N)),2\pi \bZ)\\
	=\,&\dist(b_ku_km_k(\tau_N -N\tau_0),2\pi \bZ)\le \wt{C} N|b_k|^{-\frac{\lambda}{2}}.	
\end{aligned}
\end{equation}

Take $\wt\varrho>0$ big enough such that
\[
|b_k|^{-(\wt{\varrho}-1)\log \gamma}\ge 2\Omega |b_k|^{1+\lambda}\cdot \varrho\log |b_k| 
\]
for all $k\ge 3$.
Take $N(k):=\lfloor \wt{\varrho}\log|b_k|\rfloor$ and recall that $u_k\le |b_k|^{\lambda/2}$, $m_k\le \varrho \log|b_k|$ and $|E_N|<\Omega$. Then for any $k\ge 3$,
\[
\big|b_ku_km_kE_{N(k)}\cdot \gamma^{N(k)}\cos(N(k)\Theta+\zeta_{N(k)})\big|\le \frac{1}{2}|b_k|^{-\frac{\lambda}{2}}.
\]
Combining with \eqref{equ.goodasypdiffer}, the following estimate holds for all $k$ big enough
\[
\dist(b_ku_km_k \Pi,2\pi \bZ)\le \wt{C}\cdot (\wt{\varrho}\log|b_k|)\cdot|b_k|^{-\frac{\lambda}{2}}+|b_k|^{-\frac{\lambda}{2}}.
\]
Substitute the inequality above back into \eqref{equ.goodasypdiffer}, we get
\begin{equation}\label{equ.estforcos}
	\dist(b_ku_km_k(E_N\cdot \gamma^N\cos(N\Theta+\zeta_N)+o(\gamma^N)),2\pi \bZ)\le 
	\wt{C} N|b_k|^{-\frac{\lambda}{2}}+\wt{C} \wt{\varrho}\log|b_k|\cdot|b_k|^{-\frac{\lambda}{2}}+|b_k|^{-\frac{\lambda}{2}}.
\end{equation}

Take  $M(k):=\lfloor \frac{\log(|b_k|u_km_k\Omega)}{-\log\gamma}\rfloor +1$. Then we have $\gamma\leq |b_k|u_km_k\Omega\cdot \gamma^{M(k)}\leq 1$. For any given $j\ge 0$ and big enough $k$, by taking $N=M(k)+j$ in \eqref{equ.estforcos} we have
\[
\begin{aligned}
	&\big|b_ku_km_kE_{M(k)+j}\cdot\gamma^{M(k)+j}\cos((M(k)+j)\Theta+\zeta_{M(k)+j})\big|\\
	=\,&\dist(b_ku_km_k(E_{M(k)+j}\cdot \gamma^{M(k)+j}\cos((M(k)+j)\Theta+\zeta_{M(k)+j})+o(\gamma^{M(k)+j})),2\pi \bZ)\\
	\le\,&
	\wt{C} (M(k)+j)|b_k|^{-\frac{\lambda}{2}}+\wt{C} \wt{\varrho}\log|b_k|\cdot|b_k|^{-\frac{\lambda}{2}}+|b_k|^{-\frac{\lambda}{2}}\xrightarrow{k\to \infty} 0.
\end{aligned}
\]
It follows from the choice of $M(k)$ and the assumption $\liminf_{N\to\infty}|E_N|> 0$ that for any fixed $j\ge 0$,
\begin{equation}\label{eq.cosMk}
  \lim_{k\to\infty}\cos((M(k)+j)\Theta+\zeta_{M(k)+j})=0.
\end{equation}

By the assumption, if $\Theta=0$, then $\zeta_N\equiv 0$. This contradicts to \eqref{eq.cosMk}. Thus $\Theta\in(0, \pi)$ and $\zeta_N\in(\Theta_0-\pi/12,\Theta_0+\pi/12)$ for some constant $\Theta_0$ and every $N\geq 1$. Take $j=0$ and $j=\ell\in\bZ^+$ in \eqref{eq.cosMk}, we obtain that $\lim_{k\to\infty}\dist(\ell\Theta+\zeta_{M(k)+\ell}-\zeta_{M(k)},\pi\mathbb{Z})=0$. Since $|\zeta_{M(k)+\ell}-\zeta_{M(k)}|<\pi/6$, we have $\dist(\ell\Theta,\pi\mathbb{Z})\le \pi/6$ for all $\ell\in \bZ^+$. This is impossible for any $\Theta\in(0, \pi)$. We complete the proof.

\end{proof}

\begin{proof}[Proof of Lemma~\ref{lem.appropriateonesided}]
 Let $(\wh\Sigma_{\wh r},\wh \sigma_{\wh r})$ be the irreducible TMF given by Theorem~\ref{thm.SPRcoding}. It is SPR for the H\"older continuous potential $ \vartheta \circ \wh\pi_{\wh r}$, where $\wh\pi_{\wh r}:\wh\Sigma_{\wh r}\to B$ is the H\"older continuous factor map given by Theorem~\ref{thm.strongthm}. It follows from the proof of Theorem~\ref{thm.main.mme}
that there exists a unique equilibrium state $\wh\nu $ of the H\"older continuous potential $ \vartheta \circ \wh\pi_{\wh r}$ for $(\wh\Sigma_{\wh r},\wh \sigma_{\wh r})$, and $\nu=\wh\nu\circ \wh\pi_{\wh r}^{-1}$.
Moreover, the base $\wh\Sigma$ is irreducible since  $\wh\Sigma_{\wh r}$ is irreducible. 

Denote by $\hbar$ the period of $(\wh\Sigma,\wh\sigma)$. By Lemma~\ref{lem.spectdecomp}, 
the subsystem obtained by restricting $\wh\sigma^{\hbar}$ to $\wh\Sigma(0)$ is topologically conjugate to a topologically mixing TMS $(\wt \Sigma, \wt\sigma)$. (If $\hbar=1$, then $\wh\Sigma(0)=\wh \Sigma$.) Denote by $\wt\pi:\wt\Sigma\to \wh\Sigma(0)$ this conjugacy. 
In fact, $(\wt \Sigma, \wt\sigma)$ is constructed by taking all cylinders of length $\hbar$ starting at the $0$-th position of a bilateral symbolic sequence in $\wh\Sigma(0)$ in the original symbolic space as new symbols, and identifying $\wh\sigma^{\hbar}$ with the new left shift map $\wt\sigma$.

Write $\wt r:=\wh r_{\hbar}\circ \wt\pi$, where ${\wh r}_{\hbar}$ is the $\hbar$-th Birkhoff sum of ${\wh r}$. Then $\wt r$ is a H\"older continuous roof function on $\wt\Sigma$ with $\inf \wt r>0$. It is easy to check that the suspension flow $(\wt\Sigma_{\wt r},\wt\sigma_{\wt r})$ is topologically conjugate to $(\wh\Sigma_{\wh r},\wh \sigma_{\wh r})$, and the conjugacy $\wt \pi_{\wt r}:\wt\Sigma_{\wt r}\to \wh\Sigma_{\wh r}$ is Lipschitz.

Let $(\wt\Sigma^+,\wt\sigma)$ be the associated one-sided TMS of $(\wt\Sigma,\wt\sigma)$, and let $\wt{\mathbf{p}}^{+}: \wt\Sigma\to\wt\Sigma^+$ be the natural semi-conjugacy defined as $\wt{\mathbf{p}}^{+}\big((\wt v_n)_{n\in\mathbb{Z}}\big)=(\wt v_n)_{n\geq 0}$. By Lemma~\ref{Sinailem}, 
there are H\"older continuous functions ${\wt r}^{+}:\wt\Sigma^+\to\mathbb{R}$ and $\wt\varsigma : \wt\Sigma\to\mathbb{R}$ such that ${\wt r}^{+}\circ\wt{\mathbf{p}}^{+}={\wt r}+\wt\varsigma-\wt\varsigma\circ\wt\sigma$. Let $K(\wt\varsigma)$ be a bound of $\wt\varsigma$. Take $n_0\in \bZ^+$ such that $n_0\cdot \inf {\wt r}>2K(\wt\varsigma)$. 
Since ${\wt r}^{+}_{n_0}\circ \wt{\mathbf{p}}^{+}={\wt r}_{n_0}+\wt\varsigma-\wt\varsigma\circ\wt\sigma^{n_0}$, we know that ${\wt r}^{+}_{n_0}$ is a  H\"older continuous function on $\wt\Sigma^+$ such that $\inf {\wt r}^{+}_{n_0}>0$.

Recoding cylinders of length $n_0$ in $\wt\Sigma$ as new symbolics, we can construct a TMS
$(\Sigma,\sigma)$ which is topologically conjugate to $(\wt\Sigma,\wt\sigma^{n_0})$. Let $\bar\pi$ be the conjugacy. Since $(\wt \Sigma, \wt\sigma)$ is topologically mixing, so is $(\wt\Sigma,\wt\sigma^{n_0})$, and hence $(\Sigma,\sigma)$. It follows that the associated one-sided TMS $(\Sigma^+,\sigma)$ is topologically mixing. Denote $r:={\wt r}_{n_0}\circ \bar\pi$. Then $(\Sigma_r,\sigma_r)$ is semi-conjugate to $(\wt\Sigma_{\wt r},\wt\sigma_{\wt r})$, and the semi-conjugacy $\bar\pi_r: \Sigma_r\to \wt\Sigma_{\wt r}$ is an $n_0$-to-$1$ Lipschitz map. Moreover, $(\Sigma_r,\sigma_r)$ is topologically transitive since $(\Sigma,\sigma)$ is topologically mixing.

By the construction, $\bar\pi$ can be naturally regarded as a conjugacy from $\Sigma^+$ to $\wt\Sigma^+$. Let $r^+:=\wt r_{n_0}^+\circ \bar\pi:\Sigma^+\to\bR$, $\varsigma:=\wt\varsigma\circ \bar\pi:\Sigma\to \bR$ and $\mathbf{p}^{+}: \Sigma\to\Sigma^+$ be the natural semi-conjugacy. Then we have
\begin{equation}\label{equ.onesidesinai}
	r^+\circ \mathbf{p}^{+}=r+\varsigma-\varsigma\circ\sigma.
\end{equation}
 Thus $(\Sigma^+_{r^+},\sigma_{r^+})$ is a one-sided TMF based on a topologically mixing one-sided TMS $(\Sigma^+,\sigma)$, and the roof function $r^+$ is   H\"older continuous satisfying $\inf r^+>0$.

Denote $\pi_r:=\wh\pi_{\wh r}\circ \wt \pi_{\wt r}\circ \bar\pi_r$. It is a H\"older continuous finite-to-one map satisfying $\pi_r\circ \sigma_{r,t}=\varphi_t\circ \pi_r$. And $\pi_r(\Sigma_r)$ has full $\nu$-measure since $\wh\pi_{\wh r}(\wh\Sigma_{\wh r})$ has full $\nu$-measure.

Recall that $(\wh\Sigma_{\wh r},\wh \sigma_{\wh r})$ is SPR for the H\"older continuous potential $\vartheta \circ \wh\pi_{\wh r}$, and $\wh\nu$ is the unique equilibrium state of $\vartheta \circ \wh\pi_{\wh r}$ for $(\wh\Sigma_{\wh r},\wh \sigma_{\wh r})$. Since $\wt \pi_{\wt r}:\wt\Sigma_{\wt r}\to \wh\Sigma_{\wh r}$ is a conjugacy, we have that $(\wt\Sigma_{\wt r},\wt\sigma_{\wt r})$ is SPR for the H\"older continuous potential $\vartheta \circ \wh\pi_{\wh r}\circ \wt \pi_{\wt r}$ and $\wh\nu\circ \wt \pi_{\wt r}$ is the unique equilibrium state of $\vartheta \circ \wh\pi_{\wh r}\circ \wt \pi_{\wt r}$ for $(\wt\Sigma_{\wt r},\wt\sigma_{\wt r})$. 

Since the semi-conjugacy $\bar\pi_r: \Sigma_r\to \wt\Sigma_{\wt r}$ is an $n_0$-to-$1$ map, we conclude that $(\Sigma_r, \sigma_r)$ is SPR for the H\"older continuous potential $\vartheta\circ \pi_r=\vartheta \circ \wh\pi_{\wh r}\circ \wt \pi_{\wt r}\circ \bar\pi_r:\Sigma_r\to\bR$ (refer to the proof of Theorem~\ref{thm.SPRcoding}). Moreover, it follows from the topological transitivity of $(\Sigma_r, \sigma_r)$ that there exists a unique equilibrium state $\bar\nu$ of $\vartheta\circ \pi_r$ for $(\Sigma_r, \sigma_r)$, and it satisfies $\wh\nu\circ \wt \pi_{\wt r}=\bar\nu\circ \bar\pi_r^{-1}$. Thus $\nu=\bar\nu\circ \pi_r^{-1}$.

By Lemma~\ref{lem.relCT}, $(\Sigma,\sigma)$ is SPR for the H\"older continuous potential $\phi:=\Delta_{\vartheta\circ \pi_r}-P_{\rm TOP}(\sigma_r,\vartheta\circ \pi_r)\cdot r$.
By Proposition~\ref{pro.relationmme}, the projection measure $\mu:=\bar\nu_{_{\Sigma}}$ is the unique equilibrium state of $\phi$ for $(\Sigma,\sigma)$. By Lemma~\ref{Sinailem}, there are H\"older continuous functions $\phi^{+}:\Sigma^{+}\to\mathbb{R}$ and $\varsigma_\phi :\Sigma\to\mathbb{R}$ such that $\phi^{+}\circ\mathbf{p}^{+}=\phi+\varsigma_\phi-\varsigma_\phi\circ\sigma$. And the projection measure $\mu^{+}:=\mu\circ (\mathbf{p}^{+})^{-1}$ is the unique equilibrium state of $\phi^+$ for $(\Sigma^+,\sigma)$. Since $P_{\rm TOP}(\sigma,\phi^{+}\circ\mathbf{p}^{+})=P_{\rm TOP}(\sigma,\phi)$, we have that $(\Sigma,\sigma)$ is also SPR for $\phi^{+}\circ\mathbf{p}^{+}$. Hence $(\Sigma^+,\sigma)$ is SPR for $\phi^{+}$.

Next we consider the good asymptotic property. Let $p_N$ be given by the definition of the good asymptotic property for the Borel homoclinic class $B$, such that \eqref{eq.expressperiod} holds and $\liminf_{N\to\infty}|E_N|> 0$. Denote by $\ul{\wh{p_N}}$ the canonical coding lift of $p_N$ in $\wh\Sigma$, which is a bilateral symbolic sequence naturally generated by sequentially arranging all the $R\in\mathscr{R}$ that the orbit of $p_N$ under the flow $\varphi_t$ traverses in chronological order (see Theorem~\ref{thm.strongthm}~(8c)). Without loss of generality, we may assume that $\ul{\wh{p_N}}\in \wh\Sigma(0)$. Note that the period of $(\ul{\wh{p_N}},0)$ with respect to $\wh\sigma_{\wh r}$ is $\tau_N$. It follows from the construction above that, by recoding $\ul{\wh{p_N}}$ starting from the $0$-th position into consecutive blocks of $n_0\hbar$ adjacent symbols, one obtains $\ul{p_N}\in \Sigma$, and $n_0\hbar \tau_N$ is a period (may be not the positively minimal one) of $(\ul{p_N},0)$ with respect to $\sigma_r$. Note that 
 \[
 n_0\hbar\tau_N=N(n_0\hbar\tau_0)+(n_0\hbar\Pi)+(n_0\hbar E_N)\cdot \gamma^N\cos(N\Theta+\zeta_N)+ o(\gamma^N)
 \]
and $\liminf_{N\to\infty}|n_0\hbar E_N|> 0$.

Collect all symbols that appear in $\ul{\wh{p_N}}$ with $N\ge 0$ and denote them as the set $\wh S$. Since all flow orbits of $p_N$ are contained in a transitive hyperbolic set, applying the same argument as in the proof of Theorem~\ref{thm.SPRcoding} shows that $\wh S$ is a finite set.
Moreover, let $S$ be the collection of all symbols that appear in $\ul{{p}_N}$ with $N\ge 0$.
According to the construction above, $S$ is also a finite set. It concludes that $(\Sigma_{r},\sigma_{r})$ has good asymptotics.

It is clear that the periodic points $\{\mathbf{p}^{+}(\ul{{p}_N})\}$ of $(\Sigma^+,\sigma)$ satisfy the same property. Thus
 $(\Sigma^+_{r^+},\sigma_{r^+})$ also has good asymptotics.
\end{proof}

%%%%%%%%%%%%%%%%%%%%%%%%%%%%%%%%%%%%%%%%%%%%%%%%%%%%%%%%%%%%%%%%%%%%%%%%%%%%%%%%%%%%%%%%%%%

\section{Supplementary proofs for Part~IV}\label{appsec.spr}

We first recall some useful facts given by \cite[\S\,2]{BCS25}. Let $ (\Omega, \mathcal{F}, \nu)$ be a probability space and $T:\Omega \to \Omega$ be an invertible map preserving $\nu$. Let  $\varpi : \Omega \to \mathbb{R}$ be a function and $\varpi_n$ be the  $n$-th Birkhoff sum of $\varpi$.

\begin{lem}{\rm \cite[Lemma 2.18]{BCS25}}\label{lem.pliss}
If $\varpi \in L^{\infty}(\nu)$, $\gamma_0 \in \mathbb{R}$ and $\kappa\in [0,1]$ satisfy $\nu \{x \in \Omega : \varpi(x) > \gamma_0\} \leq \kappa$, then for every $\gamma_1<\gamma_0$,
\begin{equation}
\nu\{x : \frac{1}{n}\varpi_{n}(x) \geq \gamma_1 \text{ for any } n \geq 1 \} \geq \frac{\int \varpi d\nu - \gamma_1 -\kappa\|\varpi - \gamma_0\|_{\infty}}{\gamma_0- \gamma_1}.
\end{equation}
\end{lem}

Let $\Pi:\Omega \to (0, \infty)$ be a function and $\varepsilon>0$. The {\it  $\varepsilon$-tempered envelope} of $\Pi$ is the function $\Pi_{\varepsilon} : \Omega \to (0, \infty]$ defined by
 $\Pi_{\varepsilon} (x) := \sup \{ e^{-|n| \varepsilon} \Pi( T^{n} x ) : n \in \mathbb{Z} \}$.
If $\Pi_{\varepsilon}$ is finite at $x$, then $\Pi_{\varepsilon}$ is finite at each $T^n(x)$, and the restriction of $\Pi_{\varepsilon}$ to the orbit of $x$ is the smallest function defined on this orbit such that $ \Pi_{\varepsilon} \geq \Pi $ and $ e^{-\varepsilon} \Pi_{\varepsilon} \leq \Pi_{\varepsilon} \circ T \leq e^{\varepsilon} \Pi_{\varepsilon} $.

\begin{lem}{\rm \cite[Lemma 2.20]{BCS25}}\label{lem.envelope}
Assume that $\nu$ is a $T$-invariant probability measure and $\Pi : \Omega \to (0, \infty)$ is measurable. If there is $c_0>0$ such that $e^{-c_0} \leq \frac{\Pi\circ T(x)}{\Pi(x)} \leq e^{c_0}$ for all $x\in \Omega$, then for any $\varepsilon > 0$ and any $C>0$,  we have $\Pi_{\varepsilon} (x)$ is finite for $\nu$-a.e. $x\in \Omega$, and
   \[
    \nu\{x : \Pi_{\varepsilon} (x) > C\} \leq 4 \max\{\frac{c_{0}}{\varepsilon}, 1\}\cdot \nu\{x : \Pi(x) > C\}.
   \]
\end{lem}

The above lemma was proved in \cite[Lemma 2.20 (2)]{BCS25} under the assumption of ergodicity. However, ergodicity of $\nu $ is not needed here. Indeed, by \cite[Lemma III.8]{Man81} (see also the ``fact'' in \cite[Proposition 3.5]{BCL}), the estimate $e^{-c_0} \leq \frac{\Pi\circ T(x)}{\Pi(x)} \leq e^{c_0}$ implies that $\lim_{n\to \pm\infty} \frac{1}{n}\log \Pi(T^nx)=0$ for $\nu$-a.e. $x\in \Omega$. The remainder of the proof is identical to that in \cite[Lemma 2.20 (2)]{BCS25}.

Let $Y$ be a $C^1$ vector field on a closed manifold $M$. Recall that $\N^Y$ is the normal bundle of $Y$ over $M\setminus \Sing(Y)$. For any $\chi>0$ and $n_0\in \bZ^+$, define
\[\begin{aligned}
J^Y_{n_0}(\chi):=\{x\in M\setminus \Sing(Y):& \text{ there exists a splitting } \N^Y_x=F^{Y,-}_x\oplus F^{Y,+}_x \text{ such that}\\ & \|\psi^{Y,*}_{jn_0}|_{F^{Y,-}_x}\|\le e^{-\chi j n_0} \text{ and } \|\psi^{Y,*}_{-jn_0}|_{F^{Y,+}_x}\|\le e^{-\chi j n_0} \text{ for any } j\ge 0\}
\end{aligned}\]

The following proposition is proved by the same method as \cite[Proposition 2.21]{BCS25} (also see \cite[\S4]{Zan}). It holds in any dimension at least $3$. 

\begin{pro}\label{pro.Pntopesin}
For any $\chi>0$, $0<\varepsilon<\chi/2$ and $L\ge 1$, there exists $c_0=c_0(\chi,\varepsilon,L)$ satisfying: for any $n_0\in\bZ^+$ there is $\ell=\ell(\chi,L,n_0)$ such that for any $C^1$ vector field $Y$ on $M$ with $\|\nabla Y\|\le L$ and any $Y$-invariant measure $\nu$, we have
	\[
	\nu(M\setminus \Pes^Y_{\chi}(\ell,\varepsilon))\le c_0\cdot\nu(M\setminus J^Y_{n_0}(\chi)).
	\]
\end{pro}
\begin{proof}
Take any $\chi>0$, $0<\varepsilon<\chi/2$, $n_0\in\bZ^+$, $L\ge 1$. Let $Y$ be a $C^1$ vector field satisfying $\|\nabla Y\|\le L$. Recall that $\|\psi^{Y,*}_t\|\le \|(d\varphi^{Y}_t)^*\| \le e^{2L|t|}$ for any $t$ (by the Gronwall inequality). 
Assume that $\nu $ is a $Y$-invariant measure. By the Ergodic Decomposition Theorem, it suffices to prove the statement for the case where $\nu$ is an ergodic measure for $Y$. (Note that $\nu$ is still invariant with respect to the time-$1$ map $\varphi^Y_1$.) Moreover, if $\nu(J^Y_{n_0}(\chi))=0$, then the conclusion is trivial. Next we assume that $\nu(J^Y_{n_0}(\chi))>0$. Then by the definition of $J^Y_{n_0}(\chi)$, $\nu$ is a regular hyperbolic ergodic measure.

Recall that $O(Y)\subset M\setminus \Sing(Y)$ is the set of recurrent Oseledec regular points for $\psi^{Y,*}$. For any $x\in O(Y)$, denote by $\N_x^{Y,-}$ the direct sum of subspaces whose Lyapunov exponents are negative and by $\N_x^{Y,+}$ the direct sum of subspaces whose Lyapunov exponents are positive. It is clear that the splitting $\N^Y_x=F^{Y,-}_x\oplus F^{Y,+}_x$ coincides with $\N^Y_x= \N_x^{Y,-}\oplus \N_x^{Y,+}$ for any $x\in J^Y_{n_0}(\chi)\cap O(Y)$.

For each $x\in O(Y)$, define
\[
\Pi^s(x):=\sup_{j\ge 0}\, (e^{\chi j}\|\psi_j^{Y,*}|_{_{\N^{Y,-}_x}}\|),\quad \Pi^u(x):=\sup_{j\ge 0} \,(e^{\chi j}\|\psi_{-j}^{Y,*}|_{_{\N^{Y,+}_x}}\|), \quad\text{and}\quad  \Pi(x):=\max \{\Pi^s(x),\Pi^u(x)\}.
\]
By Gronwall inequality, we have
\[
\|\psi^{Y,*}_{j\pm 1}|_{\N^{Y,\Xi}_x}\|  \le\|\psi^{Y,*}_j|_{\N^{Y,\Xi}_{\varphi^Y_{\pm 1}(x)}}\|\cdot \|\psi^{Y,*}_{\pm 1}|_{\N^{Y,\Xi}_x}\|\le e^{2L} \|\psi^{Y,*}_j|_{\N^{Y,\Xi}_{\varphi^Y_{\pm 1}(x)}}\|,
\]
for any $j\in\bZ$, $\Xi=+$ or $-$.  It follows that
\[
\begin{aligned}
	 e^{\chi-2L}\cdot\Pi^s(\varphi^Y_1(x)) & \le \Pi^s(x)\le \max\{1,e^{2L+\chi}\cdot \Pi^s(\varphi^Y_1(x))\}\text{ and,}\\
	 e^{-\chi-2L}\cdot \Pi^u(\varphi^Y_1(x)) & \le\Pi^u(x)\le \max\{1, e^{2L-\chi}\cdot \Pi^u(\varphi^Y_1(x))\}.
\end{aligned}
\]
Thus for any $x\in O(Y)$,
\[
e^{-(2L+\chi)}\le \frac{\Pi(\varphi^Y_1(x))}{\Pi(x)} \le e^{2L+\chi}. 
\]
By Lemma~\ref{lem.envelope}, for any $\varepsilon>0$ and $C>0$, the envelope $\Pi_\varepsilon$ with respect to $T=\varphi^Y_1$ is finite $\nu $ almost everywhere, and
\begin{equation}\label{equ.Pienvelope}
   \nu\{x\in O(Y) : \Pi_{\varepsilon} (x) > C\} \leq 4 \max\{\frac{2L+\chi}{\varepsilon}, 1\}\cdot \nu\{x\in O(Y) : \Pi(x) > C\}.
\end{equation}

Next, for any $x\in O(Y)$, we define
\[
\widehat\Pi(x):=(\angle(\N_x^{Y,-},\N_x^{Y,+}))^{-1}.
\]
Recall that the distortion of angles is bounded by the norm of the linear operator. Since $\|\psi^{Y,*}_{\pm 1}\|\le e^{2L}$, we have
\[
e^{-4L}\le\frac{\widehat\Pi(\varphi^Y_1(x))}{\widehat\Pi(x)}\le e^{4L}.
\]
Then by Lemma~\ref{lem.envelope} again, for any $\varepsilon>0$ and $\widehat C>0$, the envelope $\widehat\Pi_\varepsilon$ of $\widehat\Pi$ with respect to $T=\varphi^Y_1$ is finite $\nu $ almost everywhere, and
\begin{equation}\label{equ.whPienvelope}
   \nu\{x\in O(Y) : \widehat\Pi_{\varepsilon} (x) > \widehat C\} \leq 4 \max\{\frac{4L}{\varepsilon}, 1\}\cdot \nu\{x\in O(Y) : \widehat\Pi(x) > \widehat C\}.
\end{equation}
 
It follows from $\max_{|t|\le n_0} \|\psi^{Y,*}_t\|\le e^{2n_0L} $ that 
\begin{equation}\label{equ.Pn0Pi}
  J^Y_{n_0}(\chi)\cap O(Y)\subset \{x\in O(Y): \Pi(x)\le e^{2(n_0+\chi)L}\}.
\end{equation}
For any $x\in O(Y)$ such that $\Pi_\varepsilon (x)\le e^{2(n_0+\chi)L}$, any $m, n\in\mathbb{Z}$, $n\ge 0$, we have
\begin{equation}\label{equ.Pesinnorm}
  \begin{aligned}
\|\psi_n^{Y,*}|_{\N^{Y,-}_{\varphi^Y_m(x)}}\|& \le  e^{-\chi n}\cdot\Pi(\varphi^{Y}_m(x))\le \Pi_\varepsilon (x)\cdot e^{-\chi n+\varepsilon|m|}\le e^{2(n_0+\chi)L}\cdot e^{-\chi n+\varepsilon|m|} \text{, and}\\
\|\psi_{-n}^{Y,*}|_{\N^{Y,+}_{\varphi^Y_m(x)}}\|& \le  e^{-\chi n}\cdot\Pi(\varphi^{Y}_m(x))\le \Pi_\varepsilon (x)\cdot e^{-\chi n+\varepsilon|m|}\le e^{2(n_0+\chi)L}\cdot e^{-\chi n+\varepsilon|m|}.
\end{aligned}
\end{equation}
In particular, we have
\[
\|\psi_n^{Y,*}|_{\N^{Y,-}_{x}}\| \le e^{2(n_0+\chi)L-\chi n}  \quad \text{and} \quad  
\|\psi_{-n}^{Y,*}|_{\N^{Y,+}_{\varphi^Y_n(x)}}\|  \le  e^{2(n_0+\chi)L-(\chi-\varepsilon) n}.
\]

Let $n_1:=\lfloor \frac{4(n_0+\chi)L+2}{\chi}\rfloor +1$. Then $ e^{2(n_0+\chi)L-\chi n_1} < e^{2(n_0+\chi)L-(\chi-\varepsilon) n_1}\le e^{-1}$ (recall $\varepsilon<\chi/2$).
Thus for any unit vectors $v\in \N^{Y,-}_{x}$ and $w\in \N^{Y,+}_{x}$ with $\angle(v,w)\in [0,\pi/2]$, we have
\[
2\sin\frac{\angle(v,w)}{2}= |v-w|\ge \|\psi_{n_1}^{Y,*}\|^{-1}\cdot |\psi_{n_1}^{Y,*}(v-w)|\ge e^{-2n_1L}(|\psi_{n_1}^{Y,*}(w)|-|\psi_{n_1}^{Y,*}(v)|)\ge e^{-2n_1L}(e-e^{-1}).
\]
Hence $\angle(v,w)\ge e^{-2n_1L}(e-e^{-1})> e^{-2n_1L}$. It implies $\widehat\Pi(x)\le  e^{2n_1L}$. Consequently,
\begin{equation}\label{equ.PiwhPi}
\{x\in O(Y):\widehat\Pi(x)> e^{2n_1L}\}\subset \{x\in O(Y):\Pi_\varepsilon (x)>e^{2(n_0+\chi)L}\}.
\end{equation}
Moreover, for any $x\in O(Y)$ such that $\widehat\Pi_\varepsilon(x)\le  e^{2n_1L}$ and any $m\in \bZ$,
\begin{equation}\label{equ.Pesinangle}
  \angle(\N_{\varphi^Y_m(x)}^{Y,-},\N_{\varphi^Y_m(x)}^{Y,+})=\widehat\Pi(\varphi^Y_m(x))^{-1}\ge \widehat\Pi_\varepsilon (x)^{-1}\cdot e^{-\varepsilon |m|}\ge e^{-2n_1L}\cdot e^{-\varepsilon |m|}.
\end{equation}

Denote $\ell:=\max\{ e^{2n_1L}, e^{2(n_0+\chi)L}\}$. By \eqref{equ.Pesinnorm} and \eqref{equ.Pesinangle}, 
\begin{equation}\label{equ.PiwhPiPesin}
  \{x\in O(Y):\Pi_\varepsilon (x)\le e^{2(n_0+\chi)L}\}\cap \{x\in O(Y):\widehat\Pi_\varepsilon(x)\le  e^{2n_1L}\}\subset \Pes^Y_\chi(\ell, \varepsilon).
\end{equation}

Combining \eqref{equ.Pienvelope}--\eqref{equ.Pn0Pi}, \eqref{equ.PiwhPi} and \eqref{equ.PiwhPiPesin}, we obtain
\[
\begin{aligned}
\nu(M\setminus \Pes^Y_\chi(\ell, \varepsilon))& \le \nu(\{x\in O(Y):\Pi_\varepsilon (x)> e^{2(n_0+\chi)L}\}\cup \{x\in O(Y):\widehat\Pi_\varepsilon(x)>  e^{2n_1L}\})\\
& \le  4 \max\{\frac{2L+\chi}{\varepsilon}, 1\}\cdot \nu(\{x\in O(Y):\Pi (x)> e^{2(n_0+\chi)L}\})\\
& \ \ \ \, + 4 \max\{\frac{4L}{\varepsilon}, 1\}\cdot	\nu(\{x\in O(Y):\widehat\Pi(x)>  e^{2n_1L}\})\\
& \le  4 \max\{\frac{2L+\chi}{\varepsilon}, 1\}\cdot \nu(\{x\in O(Y):\Pi (x)> e^{2(n_0+\chi)L}\})\\
& \ \ \ \, + 4 \max\{\frac{4L}{\varepsilon}, 1\}\cdot	\nu(\{x\in O(Y):\Pi_\varepsilon (x)>e^{2(n_0+\chi)L}\})\\
& \le (1+ 4 \max\{\frac{4L}{\varepsilon}, 1\})\cdot 4 \max\{\frac{2L+\chi}{\varepsilon}, 1\}\cdot \nu(\{x\in O(Y):\Pi (x)> e^{2(n_0+\chi)L}\})\\
& \le (1+ 4 \max\{\frac{4L}{\varepsilon}, 1\})\cdot 4 \max\{\frac{2L+\chi}{\varepsilon}, 1\}\cdot  \nu(M\setminus J^Y_{n_0}(\chi)).
\end{aligned}
\]
This proposition is proved by taking $c_0:=(1+ 4 \max\{\frac{4L}{\varepsilon}, 1\})\cdot 4 \max\{\frac{2L+\chi}{\varepsilon}, 1\}$, which is independent of $n_0$.
\end{proof}

Similar to $\widetilde{\mathscr{L}}_K^{+}$ and $\widetilde{U}^{Y,+}_{k,\Theta}(\chi)$, we define
\[
\begin{aligned}
\widetilde{\mathscr{L}}_K^{-}:= \{\, \widetilde{\mu} : & \text{ there exists a sequence of $C^\infty$ vector fields } Y_n \text{ with measures } \nu_n \in \Pe(Y_n) \text{ such that } \\ & \ Y_n\xrightarrow{C^\infty} X,\ \nu_n(K)=1,\ \lim_{n\to \infty}P_{\nu_n}(Y_n,\vartheta) \ge  P_{\mathrm{TOP}}(X|_K,\vartheta) \text{ and }\widetilde{\nu}_n^{Y_n,-} \rightharpoonup \widetilde{\mu} \text{ on } G^1(M^3)\},
\end{aligned}
\]
and 
\[
\widetilde{U}^{Y,-}_{k,\Theta}(\chi):=\{(x,E)\in G^1(M^3): x\in M\setminus\Sing(Y),\ \|(d_x\varphi^Y_{{k}})^*|_{_{E}}\|<e^{-k\chi },\ \|(d_x\varphi^Y_{{\Theta k}})^*|_{_{E}}\|<e^{-\Theta k\chi } \}.
\]

The same argument shows that the symmetric unstable counterparts of Lemma~\ref{lem.UforYandX}--\ref{lem.cptmeas} hold for $\widetilde{U}^{Y,-}_{k,\Theta}(\chi)$ and $\widetilde{\mathscr{L}}_K^{-}$. On the basis of Lemma~\ref{lem.UforYandX}--\ref{lem.cptmeas} and Proposition~\ref{pro.Pntopesin}, Proposition~\ref{pro.bigpressurebigpesin} is proved using an argument similar to that of \cite[Theorem~3.1]{BCS25}.

\begin{proof}[Proof of Proposition~\ref{pro.bigpressurebigpesin}]
Let $0<\chi<\chi_0$, $\varepsilon>0$ and $0<\tau'<1$ be arbitrary. Since $\Pes^Y_\chi(\ell,\varepsilon)$ increases with $\varepsilon$, we may only consider the case $\varepsilon<\chi/2$. Take $\chi <\chi_4<\chi_3< \chi_0$. We first fix some parameters.
\setlist[enumerate]{label={\rm (p\arabic*)}}
\begin{enumerate}
  \item\label{item.L} $L$: Take $L>1$ such that $\|\nabla X\|\le L/2$.
  \item\label{item.c0} $c_0=c_0(\chi,\varepsilon,L)$: For $\chi$, $\varepsilon$, and $L$, by Proposition~\ref{pro.Pntopesin}, there exists $c_0=c_0(\chi,\varepsilon,L)$ satisfying: for any $n_0\in\bZ^+$ there is $\ell=\ell(\chi,L,n_0)$ such that for any $C^1$ vector field $Y$ on $M$ with $\|\nabla Y\|\le L$ and any $Y$-invariant measure $\nu$, we have
  \[
  \nu(M\setminus \Pes^Y_{\chi}(\ell,\varepsilon))\le c_0\cdot\nu(M\setminus J^Y_{n_0}(\chi)).
  \]
  \item\label{item.Theta} $\Theta=\Theta(\chi_4,\chi_3,L)$: For $\chi_4$, $\chi_3$, and $L$, by Lemma~\ref{lem.psiforU}, there is a positive integer $\Theta=\Theta(\chi_4,\chi_3,L)$ such that for  any $C^1$ vector field $Y$ with $\|\nabla Y\|\le L$, any positive integer $k$, any $(x_1,E_1)\in  \wt{U}^{Y,+}_{k,\Theta}(\chi_3)$ and $(x_2,E_2)\in  \wt{U}^{Y,-}_{k,\Theta}(\chi_3)$, we have  $E_1\cap \langle Y(x) \rangle=E_2\cap \langle Y(x) \rangle=\{0\}$,
\[
\big\|\psi^{Y,*}_{-\Theta k}\big|_{\proj^Y_{_{\N^Y}}(E_1)}\big\|<e^{-\Theta k\chi_4}  \quad \text{and} \quad  \big\|\psi^{Y,*}_{\Theta k}\big|_{\proj^Y_{_{\N^Y}}(E_2)}\big\|<e^{-\Theta k\chi_4}.
\]
  \item\label{item.delta} $\delta$: Denote $\delta:=\min\Big\{\frac{\chi_4-\chi}{2(2L+\chi_4)}, \frac{1-\tau'}{2c_0}\Big\}$.
  \item\label{item.Vk0} $\widetilde{\mathscr{V}}^{+}_K$, $\widetilde{\mathscr{V}}^{-}_K$ and $k_0$, $r$: For $\delta$ and $\Theta$, by Lemma~\ref{lem.cptmeas}, there exist neighborhoods $\widetilde{\mathscr{V}}^{+}_K$ of $\widetilde{\mathscr{L}}^{+}_K$ and $\widetilde{\mathscr{V}}^{-}_K$ of $\widetilde{\mathscr{L}}^{-}_K$, a positive integer $k_0$ and a positive number $r$ such that for all $\widetilde{\nu}\in\widetilde{\mathscr{V}}_K^{\Xi}$ ($\Xi=+$ or $-$), there is $k\in\{1,\dots,k_0\}$ satisfying 
  \[
  \widetilde{\nu}\left(\widetilde{U}^{X,\Xi}_{k,\Theta}(\chi_0)\setminus \wt \pi ^{-1}(M(\Sing(X),r))\right)> 1-\delta^2.
  \]
  \item\label{item.n0l0} $n_0$ and $\ell_0$: Denote $n_0:=\big\lfloor \frac{2(2L+\chi_4)\Theta k_0}{\chi_4-\chi}\big\rfloor +1$ and $\ell_0:=\ell(\chi,L,n_0)$, where $\ell(\cdot,\cdot,\cdot)$ is given by Item~\ref{item.c0}.
  \item\label{item.U1} $\U_1$: For $\chi_1=\chi_0$, $\chi_2=\chi_3$, and $k_0, \Theta\in \bZ^+$, $r>0$, by Lemma~\ref{lem.UforYandX}, there exists a $C^1$ neighborhood $\U_1$ of $X$ such that
\[
\wt{U}^{X,\Xi}_{k,\Theta}(\chi_0)\setminus \wt\pi^{-1}(M(\Sing(X),r))\subset \wt{U}^{Y,\Xi}_{k,\Theta}(\chi_3)\setminus \wt\pi^{-1}(M(\Sing(Y),r/2))
\]
for any $Y\in\U_1$, any $k\in\{1,2,\ldots, k_0\}$, and for both $\Xi=+$ and $\Xi=-$. Moreover, reduce $\U_1$ if necessary, we assume $\|\nabla Y\|\le L$ for any $Y\in \U_1$.
  \item\label{item.UP0} $\U$ and $P_0$: For $\widetilde{\mathscr{V}}^{+}_K$ and $\widetilde{\mathscr{V}}^{-}_K$, by the definition of $\widetilde{\mathscr{L}}^{+}_K$ and $\widetilde{\mathscr{L}}^{-}_K$, there exist a $C^\infty$ neighborhood $\U $ of $X$ and constant $P_0<P_{\mathrm{TOP}}(X|_K,\vartheta)$ such that, for any $Y\in \U$ and any $\nu\in \Pe(Y)$ satisfying $\nu(K)=1$ and $P_{\nu}(Y,\vartheta)>P_0$, we have $\wt\nu^{Y,+}\in \widetilde{\mathscr{V}}^{+}_K$ and $\wt\nu^{Y,-}\in \widetilde{\mathscr{V}}^{-}_K$. Reduce $\U$ if necessary, we assume $\U\subset \U_1$.
\end{enumerate}
\setlist[enumerate]{label={\rm (\arabic*)}}

In what follows, we verify that $\U$, $P_0$ (given by Item~\ref{item.UP0}) and $\ell_0$ (given by Item~\ref{item.n0l0}) satisfy the requirements of the proposition.

Take arbitrary $Y\in \U$ and $\nu\in \Pe(Y)$ satisfying $\nu(K)=1$ and $P_{\nu}(Y,\vartheta)>P_0$. By Items~\ref{item.U1} and \ref{item.UP0}, $\wt\nu^{Y,+}\in \widetilde{\mathscr{V}}^{+}_K$ and $\|\nabla Y\|\le L$. Then by Item~\ref{item.Vk0}, there is $k\in\{1,\dots,k_0\}$ such that 
\[
  \wt\nu^{Y,+}\left(\widetilde{U}^{X,+}_{k,\Theta}(\chi_0)\setminus \wt \pi ^{-1}(M(\Sing(X),r))\right)> 1-\delta^2,
  \]
where $\delta$ is given by Item~\ref{item.delta}. By Item~\ref{item.U1},
\[
 \wt\nu^{Y,+}(\wt{U}^{Y,+}_{k,\Theta}(\chi_3))> 1-\delta^2.
  \]

Denote (recall $\widetilde{\varphi}^Y_t(x,E)=(\varphi^Y_t(x),d\varphi^Y_t(E))$)
\[
\wt J^+:=\left\{(x,E)\in G^1(M^3): \#\{0\le j <n: \wt\varphi^{Y}_{-j}(x,E)\in \wt{U}^{Y,+}_{k,\Theta}(\chi_3) \} \ge (1-\delta)n \text{ for all } n\ge 0 \right\}.
\]
According to Lemma~\ref{lem.pliss}, for $\kappa=0$, $\gamma_0=1$, $\gamma_1=1-\delta$ and $\varpi=$ the characteristic function of $\wt{U}^{Y,+}_{k,\Theta}(\chi_3)$,  we have
\begin{equation}\label{equ.measureofwtJ}
  \wt\nu^{Y,+}(\wt J^+)>1-\delta.
\end{equation}

Take any $(x,E)\in \wt J^+$ and $n\ge 0$. We estimate $\|\psi^{Y,*}_{-n}|_{\proj^Y_{_{\N^Y}} (E)}\|$ by considering disjoint subintervals of $(-n,0]$ of length $\Theta k$ whose right endpoint $-j$ satisfy $\wt \varphi^Y_{-j}(x,E)\in\wt{U}^{Y,+}_{k,\Theta}(\chi_3)$. Precisely, define $a_1:=0$ (note that $(x,E)\in \wt{U}^{Y,+}_{k,\Theta}(\chi_3)$ since $(x,E)\in \wt J^+$) and 
\[
a_i:=\min  \{j\ge a_{i-1}+\Theta k: \wt \varphi^Y_{-j}(x,E)\in\wt{U}^{Y,+}_{k,\Theta}(\chi_3)\}, \ \forall\, i\ge 2.
\]

Denote $F^{Y,+}_x:=\proj^Y_{_{\N^Y}}(E)$. Then $\psi^Y_{-j}(F^{Y,+}_x)=\proj^Y_{_{\N^Y}}(d\varphi^Y_{-j}(E))$. By Item~\ref{item.Theta}, we have
\begin{equation}\label{equ.psiai}
\big\|\psi^{Y,*}_{-\Theta k}\big|_{\psi^Y_{-a_i}(F^{Y,+}_x)}\big\|<e^{-\Theta k\chi_4}, \ \forall\, i\ge 1.
\end{equation}
Define $I:=\max \{i\ge 1: a_i+\Theta k \le n\}$ and $V:=\bigcup_{i=1}^I [a_i,a_i+ \Theta k)$. Note that $\varphi^Y_{-j}(x,E)\not\in\wt{U}^{Y,+}_{k,\Theta}(\chi_3)$ for any $j\in [0,n-\Theta k]\setminus V$. Thus, by the definition of $\wt J^+$,
\[
\#([0,n)\setminus V)\le \delta (n-\Theta k+1)+\Theta k\le \delta n+\Theta k_0, \text{ and then } \# V=\Theta k I\ge n-\delta n-\Theta k_0.  
\]

Combining with \eqref{equ.psiai}, we get
\[
\begin{aligned}
	\big\|\psi^{Y,*}_{-n}\big|_{F^{Y,+}_x}\big\| & \le \prod_{i=1}^I \big\|\psi^{Y,*}_{-\Theta k}\big|_{\psi^Y_{-a_i}(F^{Y,+}_x)}\big\| \cdot \prod_{j\in [0,n)\setminus V} \big\|\psi^{Y,*}_{-1}\big|_{\psi^Y_{-j}(F^{Y,+}_x)}\big\|\\ 
	& \le e^{-\Theta k I \chi_4}\cdot e^{2L\cdot \#([0,n)\setminus V)}\le e^{-(n-\delta n-\Theta k_0)\chi_4+2L(\delta n +\Theta k_0)}=e^{(-\chi_4+(\chi_4+2L)\delta)n+(\chi_4+2L)\Theta k_0}.
\end{aligned}
\]
It follows from Item~\ref{item.delta} that $(\chi_4+2L)\delta\le (\chi_4-\chi)/2$. Thus,
\[
  \big\|\psi^{Y,*}_{-n}\big|_{F^{Y,+}_x}\big\|\le e^{-n\chi }\cdot e^{-\frac{\chi_4-\chi}{2}n+(\chi_4+2L)\Theta k_0}.
\]
By Item~\ref{item.n0l0}, we know that  $\frac{\chi_4-\chi}{2}n_0>(\chi_4+2L)\Theta k_0$. So, 
\[
\big\|\psi^{Y,*}_{-n}\big|_{F^{Y,+}_x}\big\|\le e^{-n\chi}
\]
for all $n\ge n_0$. Note that the above inequality also holds for $n=0$. 

Consequently, we have
\[
\wt J^+\subset\big\{ (x,E)\in G^1(M^3): x\in M^3\setminus \Sing(Y), \big\|\psi^{Y,*}_{-jn_0}\big|_{\proj^Y_{_{\N^Y}}(E)}\big\|\le e^{- jn_0\chi} \text{ for any integer } j\ge 0 \big\}
\]
By \eqref{equ.measureofwtJ} and Item~\ref{item.delta},
\[
\wt\nu^{Y,+}(\wt J^+)>1-\delta\ge 1-\frac{1-\tau'}{2c_0}.
\]
Thus,
\[
\wt\nu^{Y,+}(\big\{ (x,E)\in G^1(M^3): x\in M^3\setminus \Sing(Y), \big\|\psi^{Y,*}_{-jn_0}\big|_{\proj^Y_{_{\N^Y}}(E)}\big\|\le e^{- jn_0\chi}, \ \forall\, j\ge 0 \big\})>1-\frac{1-\tau'}{2c_0}.
\]

Since $(x,E)=(x, E^{X,+}_x)$ holds for $\widetilde{\nu}^{Y,+}$-a.e. $(x,E)\in G^1(M^3)$, we have $\proj^Y_{_{\N^Y}}(E)=\N^{Y,+}_x$, $\widetilde{\nu}^{Y,+}$-a.e.. Therefore,
\[
\nu(\big\{  x\in M^3\setminus \Sing(Y): \big\|\psi^{Y,*}_{-jn_0}\big|_{\N^{Y,+}_x}\big\|\le e^{- jn_0\chi}, \ \forall\, j\ge 0 \big\})>1-\frac{1-\tau'}{2c_0}.
\]
Similarly, 
\[
\nu(\big\{  x\in M^3\setminus \Sing(Y): \big\|\psi^{Y,*}_{jn_0}\big|_{\N^{Y,-}_x}\big\|\le e^{- jn_0\chi}, \ \forall\, j\ge 0 \big\})>1-\frac{1-\tau'}{2c_0}.
\]
Thus,
\[
\nu(J^Y_{n_0}(\chi))>1-\frac{1-\tau'}{c_0}. 
\]

Finally, by Item~\ref{item.c0}, we conclude 
\[
 \nu(\Pes^Y_{\chi}(\ell_0,\varepsilon))=1- \nu(M\setminus \Pes^Y_{\chi}(\ell_0,\varepsilon))\ge 1-c_0\cdot\nu(M\setminus J^Y_{n_0}(\chi))>\tau'.
  \]
\end{proof}

\begin{proof}[Proof of Lemma~\ref{lem.cptofpesinblock}]
Let $y\in \Pes^Y_\chi(\ell,\varepsilon)$. Then there exists a splitting $\mathcal{N}^{Y}_{y}=\N^{Y,s}_{y}\oplus \N^{Y,u}_{y}$ such that for any $m\in \bZ$ and $n\in \mathbb{N}_0$,
\begin{align}
  &  \|\psi_n^{Y,*}|_{\psi_m^{Y,*}(\N^{Y,s}_{y})}\|  \le \ell e^{-\chi n}e^{\varepsilon|m|}, \label{equ.contract} \\
  &   \|\psi_{-n}^{Y,*}|_{\psi_m^{Y,*}(\N^{Y,u}_{y})}\|  \le \ell e^{-\chi n}e^{\varepsilon|m|}, \text{ and}  \label{equ.expand} \\
  & \angle (\psi_m^{Y,*}(\N^{Y,s}_{y}),\psi_m^{Y,*}(\N^{Y,u}_{y}))  \ge \ell^{-1} e^{-\varepsilon|m|}. \label{equ.angle}
\end{align}

Setting $m=0$ in \eqref{equ.contract} and $m=n$ in \eqref{equ.expand}, we obtain that for any $n\in \mathbb{N}_0$,
 \begin{equation}\label{equ.Nuexpanding}
   \|\psi_n^{Y,*}|_{\N^{Y,s}_{y}}\|  \le \ell e^{-\chi n}  \quad \text{and} \quad   \|\psi_{-n}^{Y,*}|_{\psi_{n}^{Y,*}(\N^{Y,u}_{y})}\|  \le \ell e^{-(\chi-\varepsilon) n}. 
\end{equation}
Recall that $\varepsilon<\chi$. Therefore, as $n\to +\infty$, we have $|\psi_n^{Y,*}(v)|\to 0$ for every $v\in \N^{Y,s}_{y}$ and $|\psi_n^{Y,*}(v)|\to +\infty$ for every $v\in \N^Y_y\setminus \N^{Y,s}_{y}$. It implies that $\N^{Y,s}_{y}$ is unique. By a symmetrical argument, $\N^{Y,u}_{y}$ is also unique. 

Next we assume that $Y_k\xrightarrow{C^1} Y$ and $y_k\to x\in M\setminus\Sing(Y)$, where $y_k\in \Pes^{Y_k}_\chi(\ell,\varepsilon)$. Extract a subsequence $k_j$ such that both $\{\N^{Y_{k_j},s}_{y_{_{k_j}}}\}$ and $\{\N^{Y_{k_j},u}_{y_{_{k_j}}}\}$ converge. Denote by $F^s, F^u\subset \N^Y_x$ the limit subspaces respectively. Replace $Y$ by $Y_{k_j}$ and $y$ by $y_{_{k_j}}$ in \eqref{equ.Nuexpanding}, and letting $j\to +\infty$, we have
\begin{equation}\label{equ.contractexpandonx}
  \|\psi_n^{Y,*}|_{{F^s}}\|  \le \ell e^{-\chi n}  \quad \text{and} \quad   \|\psi_{-n}^{Y,*}|_{\psi_{n}^{Y,*}(F^u)}\|  \le \ell e^{-(\chi-\varepsilon) n}, 
\end{equation}
for any $n\in \mathbb{N}_0$. It follows from $\varepsilon<\chi$ that $F^s\cap F^u=\{0\}$. Therefore, $\N^Y_x=F^s\oplus F^u$.

We claim that $\N^{Y_{k},s}_{y_{{k}}}\to F^s$ (and  $\N^{Y_{k},u}_{y_{{k}}} \to F^u$ symmetrically). 
Otherwise, there is another subsequence $k'_j$ such that $\N^{Y_{k'_j},s}_{y_{_{k'_j}}}\to G^s\neq F^s$. Without loss of generality, we assume $\dim F^s\le \dim G^s$. Then we may choose $v\in G^s\setminus F^s$. Similar to \eqref{equ.contractexpandonx}, we have $|\psi^{Y,*}_n(v)|\le \ell e^{-\chi n}\to 0$ as $n\to +\infty$. On the other hand,
let $v=v^s+v^u$, where $v^s\in F^s$ and $v^u\in F^u$. Note that $v^u\neq 0$. Then by \eqref{equ.contractexpandonx} we obtain
    \[
    |\psi^{Y,*}_n(v)|\ge |\psi^{Y,*}_n(v^u)|-|\psi^{Y,*}_n(v^s)|\ge \ell^{-1} e^{(\chi-\varepsilon )n}-\ell e^{-\chi n}\to +\infty  \text{ as } n\to +\infty.
    \]
This contradiction proves the claim.
 
Denote $\N_x^{Y,s}:=F^s$ and $\N_x^{Y,u}:=F^u$. Letting $k\to +\infty$ in \eqref{equ.contract}--\eqref{equ.angle}, we conclude that $x\in \Pes^Y_\chi(\ell,\varepsilon)$. The proof is completed.
\end{proof}

%%%%%%%%%%%%%%%%%%%%%%%%%%%%%%%%%%%%%%%%%%%%%%%%%%%%%%%%%%%%%%%%%%%

\bibliographystyle{amsplain}

\begin{thebibliography}{10}

\addcontentsline{toc}{section}{References}

\bibitem{Abr59}
L. M. Abramov, On the entropy of a flow, {\it Dokl. Akad. Nauk SSSR}, {\bf  128} (1959) 873--875.


\bibitem{AW}
R. L. Adler and B. Weiss, Entropy, a complete metric invariant for automorphisms of the torus, {\it  Proc. Nat. Acad. Sci. U.S.A.}, {\bf 57} (1967), 1573--1576.


\bibitem{ABS}
V. S. Afra\u{\i}movi\v{c}, V. V. Bykov, and L. P. Sil'nikov, The origin and structure of the Lorenz attractor, {\it  Dokl. Akad. Nauk. SSSR}, {\bf 234} (1977), no. 2, 336--339.


 \bibitem{AK1942} 
 W. Ambrose and S. Kakutani, Structure and continuity of measurable flows, {\it Duke Math. J.}, {\bf 9} (1942), 25--42.

\bibitem{Aok}
N. Aoki, The set of Axiom A diffeomorphisms with no cycles, {\it  Bol. Soc. Brasil. Mat. (N.S.)}, {\bf  23} (1992), no. 1-2, 21--65.

\bibitem{ALP}
E. Araujo, Y. Lima and M. Poletti, Symbolic dynamics for nonuniformly hyperbolic maps with singularities in high dimension, {\it  Mem. Amer. Math. Soc.}, {\bf 301} (2024), no. 1511, vi+117 pp.


\bibitem{AM2016} 
V. Ara\'ujo and I. Melbourne, Exponential decay of correlations for nonuniformly hyperbolic flows with a $C^{1+\alpha}$  stable foliation, including the classical Lorenz attractor, {\it Ann. Henri Poincar\'e}, {\bf 17} (2016), no. 11, 2975--3004.


 \bibitem{AM2019} 
 V. Ara\'ujo and I. Melbourne, Mixing properties and statistical limit theorems for singular hyperbolic flows without a smooth stable foliation, {\it Adv. Math.}, {\bf 349} (2019), 212--245.

\bibitem{AV}
V. Ara\'ujo and P. Varandas, Robust exponential decay of correlations for singular-flows, {\it  Comm. Math. Phys.}, {\bf  311} (2012), 215--246. Erratum, {\it Comm. Math. Phys.}, {\bf 341} (2016), no. 2, 729--731.


 \bibitem{AGY2006} 
 A. Avila, S. Gou\"ezel and J. Yoccoz, Exponential mixing for the Teichm\"uller flow, {\it Publ. Math. Inst. Hautes \'Etudes Sci.}, No. {\bf 104} (2006), 143--211. 
 
 
 
 
\bibitem{BDL}
V. Baladi, M. F. Demers and C. Liverani, Exponential decay of correlations for finite horizon Sinai billiard flows, {\it Invent. Math.}, {\bf  211} (2018), no. 1, 39--177.


\bibitem{BV2005} 
V. Baladi and B. Vall\'ee, Exponential decay of correlations for surface semi-flows without finite Markov partitions, {\it Proc. Amer. Math. Soc.}, {\bf  133} (2005), 865--874.

\bibitem{BBM}
P. B\'alint, O. Butterley and I. Melbourne, Polynomial decay of correlations for flows, including Lorentz gas examples, {\it Comm. Math. Phys.}, {\bf 368} (2019), 55-111.


\bibitem{BI2006} 
L. Barreira and G. Iommi, Suspension flows over countable Markov shifts, {\it J. Stat. Phys.}, {\bf 124} (2006), no.1, 207--230.

\bibitem{BP07}
L. Barreira and Y. Pesin, {\it  Nonuniform hyperbolicity}, volume 115 of Encyclopedia of Mathematics and its Applications, Cambridge University Press, Cambridge, 2007.

\bibitem{BS00}
L. Barreira and B. Saussol, Multifractal analysis of hyperbolic flows, {\it  Comm. Math. Phys.}, {\bf 214} (2000), no. 2, 339--371.

\bibitem{Ben}
S. Ben Ovadia, Symbolic dynamics for non-uniformly hyperbolic diffeomorphisms of compact smooth manifolds, {\it  J. Mod. Dyn.}, {\bf 13} (2018), 43--113.


\bibitem{BD}
C. Bonatti and A. da Luz, Star flows and multisingular hyperbolicity, {\it  J. Eur. Math. Soc. (JEMS)}, {\bf 23} (2021), no. 8, 2649--2705.

\bibitem{Bow70}
R. Bowen, Markov partitions for Axiom A diffeomorphisms, {\it  Amer. J. Math.}, {\bf 92} (1970), 725--747.

\bibitem{Bow71}
R. Bowen, Periodic points and measures for Axiom A diffeomorphisms, {\it Trans. Amer. Math. Soc.}, {\bf 154} (1971), 377--397.

\bibitem{Bow73}
R. Bowen, Symbolic dynamics for hyperbolic flows, {\it  Amer. J. Math.}, {\bf 95} (1973), 429--460.


\bibitem{Bow74}
R. Bowen, Maximizing entropy for a hyperbolic flow, {\it  Math. Systems Theory}, {\bf 7} (1974), no. 4, 300--303.

\bibitem{Bow75}
R. Bowen, Some systems with unique equilibrium states, {\it  Math. Systems Theory}, {\bf 8} (1975), no. 3, 193--202.

\bibitem{Bow78}
R. Bowen, {\it On Axiom A diffeomorphisms}, American Mathematical Society, Providence, R.I., 1978, Regional Conference Series in Mathematics, No. 35.


\bibitem{Bow08}
R. Bowen, {\it  Equilibrium states and the ergodic theory of Anosov diffeomorphisms}, Lecture Notes in Math., 470, Springer-Verlag, Berlin, 2008. viii+75 pp.


\bibitem{BW}
R. Bowen and P. Walters, Expansive one-parameter flows, {\it  J. Differential Equations}, {\bf  12} (1972), 180--193.


\bibitem{Bur}
D. Burguet, Maximal measure and entropic continuity of Lyapunov exponents for $C^r$ surface diffeomorphisms with large entropy, {\it  Ann. Henri Poincar\'e}, {\bf  25} (2024), no. 2, 1485--1510.

\bibitem{BLY}
D. Burguet, C. Luo and D. Yang, Effective SPR property for surface diffeomorphisms and three-dimensional vector fields, {\it  Preprint}, 2025. arXiv:2512.03515.


\bibitem{BCFT}
K. Burns, V. Climenhaga, T. Fisher, and D. J. Thompson, Unique equilibrium states for geodesic flows in nonpositive curvature, {\it  Geom. Funct. Anal.}, {\bf 28} (2018), no. 5, 1209--1259.

\bibitem{BMW}
K. Burns, H. Masur,  A. Wilkinson, The Weil-Petersson geodesic flow is ergodic, {\it  Ann. of Math. (2)}, {\bf 175} (2012), no.2, 835--908.

 \bibitem{BW20} 
 O. Butterley and K. War, Open sets of exponentially mixing Anosov flows, {\it J. Eur. Math. Soc. (JEMS)}, {\bf 22} (2020), no. 7, 2253--2285.


\bibitem{Buz97}
J. Buzzi, Intrinsic ergodicity of smooth interval maps, {\it  Israel J. Math.}, {\bf 100} (1997), 125--161.

\bibitem{Buz05}
J. Buzzi, Subshifts of quasi-finite type, {\it  Invent. Math.}, {\bf 159} (2005), no. 2, 369--406.


\bibitem{BCL}
J. Buzzi, S. Crovisier, and Y. Lima, Symbolic dynamics for large non-uniformly hyperbolic sets of three dimensional flows, {\it  Adv. Math.}, {\bf 479} (2025), part A, Paper No. 110410, 91 pp.

\bibitem{BCLLY}
J. Buzzi, S. Crovisier, Y. Lima, C. Luo, and D. Yang, Symbolic dynamics for non-uniformly hyperbolic flows, {\it  Preprint}, 2026. arXiv:2608.14095.


\bibitem{BCS22}
 J. Buzzi, S. Crovisier, and O. Sarig, Measures of maximal entropy for surface diffeomorphisms, {\it  Ann. of Math. (2)}, {\bf 195} (2022), no. 2, 421--508.

\bibitem{BCS22b}
J. Buzzi, S. Crovisier, and O. Sarig, Continuity of Lyapunov exponents and entropy for surface diffeomorphisms, {\it  Invent. Math.}, {\bf  230} (2022), no. 2, 767--849.

\bibitem{BCS25}
J. Buzzi, S. Crovisier, and O. Sarig, Strong positive recurrence and exponential mixing for diffeomorphisms,  {\it  Preprint}, 2025. arXiv:2501.07455.

\bibitem{BF}
J. Buzzi and T. Fisher, Entropic stability beyond partial hyperbolicity, {\it  J. Mod. Dyn.}, {\bf 7} (2013), no. 4, 527--552.

\bibitem{BFSV}
 J. Buzzi, T. Fisher, M. Sambarino, and C. V\'asquez, Maximal entropy measures for certain partially hyperbolic, derived from Anosov systems, {\it  Ergodic Theory Dynam. Systems}, {\bf 32} (2012), no.1, 63--79.
 
 \bibitem{BuzLY}
 J. Buzzi, C. Luo, and D. Yang, Continuity properties of ergodic measures of maximal entropy for $C^r$ surface diffeomorphisms, {\it  Preprint}, 2024. arXiv:2412.19658.
 

\bibitem{CKW}
V. Climenhaga, G. Knieper, and K. War, Uniqueness of the measure of maximal entropy for geodesic
flows on certain manifolds without conjugate points, {\it  Adv. Math.}, {\bf 376} (2021), Paper No. 107452, 44 pp.

\bibitem{CT14}
V. Climenhaga and D. J.  Thompson, Intrinsic ergodicity via obstruction entropies, {\it  Ergodic Theory Dynam. Systems}, {\bf 34} (2014), no. 6, 1816--1831.

\bibitem{CT16}
Vaughn Climenhaga and D. J. Thompson, Unique equilibrium states for flows and homeomorphisms with non-uniform structure, {\it  Adv. Math.}, {\bf 303} (2016), 745--799.

\bibitem{CT18}
V. Climenhaga, T. Fisher, and D. J. Thompson, Unique equilibrium states for Bonatti-Viana diffeomorphisms, {\it  Nonlinearity}, {\bf 31} (2018), no. 6, 2532--2570.


\bibitem{CS}
V. Cyr and O. Sarig, Spectral gap and transience for Ruelle operators on countable Markov shifts, {\it  Comm.
Math. Phys.}, {\bf  292} (2009), no. 3, 637--666.


\bibitem{CDYZ}
S. Crovisier, A. da Luz, D. Yang, and J. Zhang, On the notions of singular domination and (multi)singular hyperbolicity, {\it Sci. China Math.}, {\bf 63} (2020), no. 9, 1721--1744.


\bibitem{CY}
	S. Crovisier and D. Yang, On the nonlinear Poincar\'e flow, {\it  Nonlinearity}, {\bf 38} (2025), no. 3, Paper No. 035022, 25 pp.

\bibitem{DaL}
A. da Luz, Star flows with singularities of different indices,  {\it  Preprint}, 2018. arXiv:1806.09011.


\bibitem{DV2021} 
D. Daltro and P. Varandas, Exponential decay of correlations for Gibbs measures and semiflows over $C^{1+\alpha}$ piecewise expanding maps, {\it Ann. Henri Poincar\'e}, {\bf 22} (2021), no. 7, 2137--2159.

\bibitem{Dol98a}
D. Dolgopyat, On decay of correlations in Anosov flows, {\it Ann. of Math. (2)}, {\bf  147} (1998), no. 2, 357--390.

\bibitem{Dol98}
D. Dolgopyat, Prevalence of rapid mixing in hyperbolic flows, {\it Ergodic Theory Dynam. Systems}, {\bf 18} (1998), no. 5, 1097--1114.


\bibitem{FMT07} 
M. Field, I. Melbourne and A. T\"or\"ok, Stability of mixing and rapid mixing for hyperbolic flows, {\it Ann. of Math. (2)},  {\bf 166} (2007), no. 1, 269--291. 


\bibitem{GY} S. Gan and D. Yang, Morse-Smale systems and horseshoes for three dimensional singular flows, {\it  Ann. Sci. \'{E}c. Norm. Sup\'{e}r (4)},  {\bf 51} (2018), no. 1, 39--112.
	
\bibitem{GW}
S. Gan and L. Wen, Nonsingular star flows satisfy Axiom A and the no-cycle condition, {\it Invent. Math.}, {\bf  164} (2006), no. 2, 279--315.
	
\bibitem{Gel}	
K. Gelfert, Horseshoes for diffeomorphisms preserving hyperbolic measures, {\it  Math. Z.},
{\bf 283} (2016), 685--701.
	
	
\bibitem{GLP}	
P. Giulietti, C. Liverani,  M. Pollicott,  Anosov flows and dynamical zeta functions, {\it Ann. of Math. (2)}, {\bf 178} (2013), no. 2, 687--773.	
	
\bibitem{Guc76}	
J. Guckenheimer, A strange, strange attractor,  In {\it The Hopf Bifurcation Theorems and Its Applications.} Applied Mathematical Series, vol. 19. New York: Springer-Verlag, 368--381, 1976.	
	
\bibitem{GucW}
J. Guckenheimer and R. F. Williams, Structural stability of Lorenz attractors, {\it  Inst. Hautes \'Etudes Sci. Publ. Math.}, No. {\bf 50} (1979), 59--72.	
	
	
\bibitem{Gur}
B. M. Gurevich, Shift entropy and Markov measures in the space of paths of a countable graph, {\it  Dokl. Akad. Nauk SSSR}, {\bf 192} (1970), 963--965.	


\bibitem{Gur96}
B. M. Gurevich, Stably recurrent nonnegative matrices, {\it Uspekhi Mat. Nauk}, {\bf  51} (1996), no. 3(309), 195--196; translation in {\it Russian Math. Surveys} {\bf 51} (1996), no. 3, 551--552.


\bibitem{GS}
B. M. Gurevich and S. V. Savchenko, Thermodynamic formalism for symbolic Markov chains with a countable number of states, {\it  Uspekhi Mat. Nauk}, {\bf  53} (1998), no. 2(320), 3--106; translation in {\it Russian Math. Surveys} {\bf 53} (1998), no. 2, 245--344.


\bibitem{GZ}
B. M. Gurevich and A. S. Zargaryan, Conditions for the existence of a maximal measure for a countable symbolic Markov chain, {\it  Vestnik Moskov. Univ. Ser. I Mat. Mekh.}, (1988), no. 5, 14--18, 103; translation in {\it Moscow Univ. Math. Bull.}, {\bf 43} (1988), no. 5, 18--23.



\bibitem{Hay}
S. Hayashi, Diffeomorphisms in ${\mathcal F}^1(M)$ satisfy Axiom A, {\it Ergodic Theory Dynam. Systems}, {\bf 12} (1992), no. 2, 233--253.


\bibitem{Hof}
F. Hofbauer, On intrinsic ergodicity of piecewise monotonic transformations with positive entropy, {\it Israel J. Math.}, {\bf 34}(1979), no. 3, 213--237.


\bibitem{IV2021}
G. Iommi and A. Velozo, Measures of maximal entropy for suspension flows, {\it Math. Z.}, {\bf  297} (2021), no. 3-4, 1473--1482.

\bibitem{Kat80}
A. Katok, Lyapunov exponents, entropy and periodic orbits for diffeomorphisms, {\it Publ. Math. Inst. Hautes \'Etudes Sci.}, {\bf 51} (1980), 137--173.

\bibitem{KSSV}
M. Kunzinger, H. Schichl, R. Steinbauer, and J. A. Vickers, Global Gronwall estimates for integral curves on Riemannian manifolds, {\it  Rev. Mat. Complut.}, {\bf  19} (2006), no. 1, 133--137.

\bibitem{LLS}
F. Ledrappier, Y. Lima, and O. Sarig, Ergodic properties of equilibrium measures for smooth three dimensional flows, {\it  Comment. Math. Helv.}, {\bf  91} (2016), no. 1, 65--106.



\bibitem{LP2023} 
J. Li and W. Pan, Exponential mixing of geodesic flows for geometrically finite hyperbolic manifolds with cusps, {\it Invent. Math.}, {\bf 231} (2023), no. 3, 931--1021.

\bibitem{LGW} 
M. Li, S. Gan and L. Wen, Robustly transitive singular sets via approach of an extended linear Poincar\'e flow, {\it Discrete Contin. Dyn. Syst.}, {\bf 13} (2005), no. 2, 239--269.


\bibitem{LLL2024} 
M. Li, C. Liang and X. Liu, A closing lemma for non-uniformly hyperbolic singular flows, {\it Comm. Math. Phys.}, {\bf 405} (2024), no. 8, Paper No. 195, 35 pp.


\bibitem{LSWW}
M. Li, Y. Shi, S. Wang and X. Wang, Measures of intermediate entropies for star
vector fields, {\it  Israel J. Math.}, {\bf  240} (2020), no. 2, 791--819.

\bibitem{LYYZ}
M. Li, F. Yang, J. Yang and R. Zheng, A countable partition for singular flow based on filtration structure: construction and application, {\it  Preprint}, 2024.

\bibitem{LY12}
Z. Lian, L.-S. Young, Lyapunov exponents, periodic orbits, and horseshoes for semiflows on Hilbert spaces, {\it J. Am. Math. Soc.}, {\bf 25} (2012), no. 3, 637--665.

\bibitem{Liao}
S. Liao, Certain ergodic properties of a differential system on a compact differentiable manifold, {\it Front. Math. China}, {\bf  1} (2006), no. 1, 1--52.

\bibitem{Liao1} 
S. Liao, A basic property of a certain class of differential systems, {\it Acta Math. Sinica}, {\bf 22} (1979), no. 3, 316--343.

\bibitem{Liao81}
S. Liao, Obstruction sets.(II), {\it  Beijing Daxue Xuebao}, (1981), no.2, 1--36.


\bibitem{Liao2} 
S. Liao, Some uniformity properties of ordinary differential systems and a generalization of an existence theorem for periodic orbits, {\it Beijing Daxue Xuebao}, (1985), no. 2, 1--19.
	
\bibitem{Liao3} 
S. Liao, On $(\eta,d)$-contractible orbits of vector fields, {\it Systems Sci. Math. Sci.}, {\bf 2} (1989), no. 3, 193--227.


\bibitem{Lim}
Y. Lima, Symbolic dynamics for one dimensional maps with nonuniform expansion, {\it  Ann. Inst. H. Poincar\'e C Anal. Non Lin\'eaire}, {\bf 37} (2020), no. 3, 727--755.	

	
\bibitem{LM}
Y. Lima and C. Matheus, Symbolic dynamics for non-uniformly hyperbolic surface maps with discontinuities, {\it Ann. Sci. \'Ec. Norm. Sup\'er (4)},  {\bf 51} (2018), no. 1, 1-38.

\bibitem{LMN}
Y. Lima, J. Mongez and J. Nascimento, Symbolic dynamics for non-uniformly hyperbolic flows in high dimension, {\it  Preprint}, 2025. arXiv:2509.09050.


\bibitem{LP}
Y. Lima and M. Poletti, Homoclinic classes of geodesic flows on rank $1$ manifolds, {\it  Proc. Amer. Math. Soc.}, {\bf 153} (2025), no. 4, 1611--1620.


\bibitem{LS}
Y. Lima and O. Sarig, Symbolic dynamics for three-dimensional flows with positive topological entropy, {\it  J. Eur. Math. Soc. (JEMS)}, {\bf 21} (2019), no. 1, 199--256.



\bibitem{Liv04}
 C. Liverani, On contact Anosov flows, {\it Ann. of Math. (2)},  {\bf 159} (2004), no. 3, 1275--1312.


\bibitem{Lor}
E. N. Lorenz, Deterministic nonperiodic flow, {\it  J. Atmospheric. Sci.}, {\bf 20} (1963),130--141.

\bibitem{LY26} C. Luo and D. Yang, Strong Positive recurrence for potential and exponential mixing of equilibrium states of surface diffeomorphisms, {\it  Preprint}, 2026, arXiv: 2602.04275.

\bibitem{Man81}
R. Ma\~n\'e, {\it Lyapounov exponents and stable manifolds for compact transformations}, Geometric dynamics (Rio de Janeiro, 1981), Lecture Notes in Math., Springer-Verlag, Berlin, 1983.


\bibitem{Man82}
R. Ma\~n\'e, An ergodic closing lemma, {\it Ann. of Math. (2)},  {\bf 116}(1982), no. 3, 503--540.

\bibitem{Man88}
R. Ma\~n\'e, A proof of the $C^1$ stability conjecture, {\it  Inst. Hautes \'Etudes Sci. Publ. Math.}, No. {\bf 66} (1988), 161--210.


\bibitem{Melbourne2007} 
I. Melbourne, Rapid decay of correlations for nonuniformly hyperbolic flows, {\it Trans. Amer. Math. Soc.}, {\bf 359} (2007), no. 5, 2421--2441.

\bibitem{Mel18}
I. Melbourne, Superpolynomial and polynomial mixing for semiflows and flows, {\it Nonlinearity}, {\bf 31} (2018), no. 10, R268--R316.


\bibitem{MM}
R. Metzger and C. A. Morales, Sectional-hyperbolic systems, {\it Ergodic Theory  Dynam. Systems}, {\bf  28} (2028), no. 5, 1587--1597.

\bibitem{MPP99}
C. A. Morales, M. J. Pacifico and E. R. Pujals, Singular hyperbolic systems, {\it  Proc. Amer. Math. Soc.}, {\bf  127} (1999), no. 11, 3393--3401.

\bibitem{New72}
S. E. Newhouse, Hyperbolic limit sets, {\it  Trans. Amer. Math. Soc.}, {\bf  167} (1972), 125--150.

\bibitem{New}
S. E. Newhouse, Continuity properties of entropy, {\it Ann. of Math. (2)},   {\bf  129} (1989), no. 2, 215--235.

\bibitem{Ose}
V. I. Oseledec, A multiplicative ergodic theorem. Characteristic  Ljapunov, exponents of dynamical systems, {\it Trudy Moskov. Mat. Ob\v{s}\v{c}.},  {\bf 19} (1968), 179--210.

\bibitem{PYY22}
M. J. Pacifico, F. Yang and J. Yang, Existence and uniqueness of equilibrium states for systems with specification at a fixed scale: an improved Climenhaga-Thompson criterion, {\it Nonlinearity}, {\bf 35} (2022), no. 12, 5963--5992.

\bibitem{PYY25a}
M. J. Pacifico, F. Yang and J. Yang, Equilibrium states for the classical Lorenz attractor and sectional-hyperbolic attractors in higher dimensions, {\it  Duke Math. J.}, {\bf 174} (2025), no. 10, 1901--2010.


\bibitem{PYY25b}
M. J. Pacifico, F. Yang and J. Yang, An ergodic spectral decomposition theorem for singular star flows, {\it  Preprint}, 2025. arXiv:2506.19989.


\bibitem{PYYY}
M. J. Pacifico, F. Yang, J. Yang and G. Yao, An improved Climenhaga-Thompson criterion for locally maximal sets, {\it  Preprint}, 2025. arXiv:2503.23173.


\bibitem{Par}
W. Parry, Intrinsic Markov chains, {\it  Trans. Amer. Math. Soc.}, {\bf 112} (1964), 55--66.

\bibitem{PP}
W. Parry and M. Pollicott, Zeta functions and the periodic orbit structure of hyperbolic dynamics, {\it Ast\'erisque}, No. 187-188 (1990), 268 pp.

\bibitem{Pes76}
J. B. Pesin, Families of invariant manifolds that correspond to nonzero characteristic exponents, {\it  Izv. Akad. Nauk SSSR Ser. Mat.}, {\bf  40} (1976), 1332--1379, 1440.

\bibitem{Pol85}
 M. Pollicott, On the rate of mixing of Axiom A flows, {\it Invent. Math.}, {\bf  81} (1985), no. 3, 413--426.

\bibitem{PS}
M. Pollicott and R. Sharp, Exponential error terms for growth functions on negatively curved surfaces, {\it Amer. J. Math.}, {\bf 120} (1998), no. 5,  1019--1042.


\bibitem{Rat69}
M. E. Ratner, Markov decomposition for an y-flow on a three-dimensional manifold, {\it  Mat. Zametki}, {\bf 6} (1969), 693--704.

\bibitem{Rat73}
M. E. Ratner,  Markov partitions for Anosov flows on $n$-dimensional manifolds, {\it  Israel J. Math.}, {\bf 15} (1973), 92--114.


\bibitem{Rue83}
D. Ruelle, Flots qui ne m\'elangent pas exponentiellement, {\it C. R. Acad. Sci. Paris S\'er. I Math.}, {\bf  296} (1983), no. 4, 191--193.

\bibitem{RH}
R. R\"uhr and O. Sarig, Effective intrinsic ergodicity for countable state topological markov shifts, {\it  Israel J. Math.}, {\bf 251} (2022), 679--735.

\bibitem{Sarig1999} 
O. Sarig, Thermodynamic formalism for countable Markov shifts, {\it Ergodic Theory Dynam. Systems}, {\bf 19} (1999), no. 6, 1565--1593.

\bibitem{Sar01}
O. Sarig, Phase transitions for countable Markov shifts, {\it  Comm. Math. Phys.}, {\bf  217} (2001), no. 3, 555--577.

\bibitem{Sar11}
O. Sarig, Bernoulli equilibrium states for surface diffeomorphisms, {\it J. Mod. Dyn.}, {\bf 5} (2011), 593--608.

\bibitem{Sar13}
O. Sarig, Symbolic dynamics for surface diffeomorphisms with positive entropy, {\it  J. Amer. Math. Soc.}, {\bf 26} (2013), no. 2, 341--426.


\bibitem{SGW}
Y. Shi, S. Gan and L. Wen, On the singular hyperbolicity of star flows, {\it J. Mod. Dyn.}, {\bf  8} (2014), no. 2, 191--219.

\bibitem{SY}
Y. Shi and F. Yang, Star flows with singular equilibrium states, in preparation.

\bibitem{SYY}
Y. Shi, F. Yang and J. Yang, A countable partition for singular flows, and its application on the entropy theory, {\it Israel J. Math.}, {\bf 249} (2022), no. 1, 375--429. 


\bibitem{Sin68}
Ja. G. Sina\u{\i}, Construction of Markov partitionings, {\it  Funkcional. Anal. i Prilo\v{z}en.}, {\bf 2} (1968), no. 3, 70--80.

\bibitem{Sin}
Ja. G. Sina\u{\i}, Gibbs measures in ergodic theory, {\it  Uspehi Mat. Nauk}, {\bf 27} (1972), no. 4(166), 21--64.


\bibitem{Tsu18} 
M. Tsujii, Exponential mixing for generic volume-preserving Anosov flows in dimension three, {\it J. Math. Soc. Japan}, {\bf 70} (2018), no. 2,  757--821.
	
\bibitem{TZ23} 
M. Tsujii and Z. Zhang, Smooth mixing Anosov flows in dimension three are exponentially mixing, {\it Ann. of Math. (2)},  {\bf 197} (2023), no. 1, 65--158.

\bibitem{Ver}
D. Vere-Jones, Geometric ergodicity in denumerable Markov chains, {\it  Quart. J. Math. Oxford Ser. (2)},  {\bf  13} (1962), 7--28.

\bibitem{Wal}
P. Walters, {\it  An introduction to ergodic theory}, Grad. Texts in Math.,  79, Springer-Verlag, New York-Berlin, 1982. ix+250 pp.

\bibitem{Wen96}
L. Wen, On the $C^1$ stability conjecture for flows, {\it J. Differential Equations}, {\bf 129} (1996), 334--357.

\bibitem{WW}
X. Wen and L. Wen, A rescaled expansiveness for flows, {\it Trans. Amer. Math. Soc.}, {\bf 371} (2019), no. 5, 3179--3207. 

\bibitem{Young1998} 
L.-S. Young, Statistical properties of dynamical systems with some hyperbolicity, {\it Ann. of Math. (2)},  {\bf 147} (1998), no. 3, 585--650.

\bibitem{You99}
L.-S. Young, Recurrence times and rates of mixing, {\it  Israel J. Math.}, {\bf  110} (1999), 153--188.

\bibitem{Zan}
Y. Zang, Measures of maximal entropy for $C^\infty$ three-dimensional flows, {\it  Preprint}, 2025. arXiv:2503.21183.

%%%%%%%%%%%%%%%%%%%%%%%%%%%%%%%%%%%%

\end{thebibliography}

\noindent Ming Li,
\smallskip

\noindent  School of Mathematical Sciences and LPMC

\noindent  Nankai University, Tianjin 300071, People's Republic of China

\noindent limingmath@nankai.edu.cn

\bigskip

\noindent Xingzhong Liu,
\smallskip

\noindent  School of Mathematical Sciences and Insititute of Mathematics and Interdisciplinary Sciences

\noindent  Tianjin Normal University, Tianjin 300387, People's Republic of China

\noindent lxzmathematics@yeah.net

\end{document}